\documentclass[10pt]{article}
\usepackage{amsmath}
\usepackage{amssymb,latexsym,xcolor,dsfont,enumitem,mathrsfs,stmaryrd,bm,mathtools} %,cite}
\usepackage{cite}
\usepackage[colorlinks=true,linkcolor=blue!50!black,pagecolor=blue!50!black,citecolor=green!50!black]{hyperref}
\usepackage[auth-sc]{authblk}
\usepackage{multirow}
\usepackage{longtable}
\usepackage[left=2.1cm,right=2.1cm, top=2.5cm,bottom=2.5cm,bindingoffset=0cm]{geometry}

\usepackage{tikz}
\usepackage{pgf}
\usetikzlibrary{calc}
\usetikzlibrary{intersections}
\usetikzlibrary{fadings}

\usepackage[cp1251]{inputenc}
\usepackage[T1]{fontenc}
\usepackage[russian,english]{babel}

\newlength{\templength}
\newlength{\templengtha}
\newlength{\templengthb}
\newlength{\templengthc}
\newlength{\templengthd}
\newlength{\templengthe}
\newlength{\templengthf}
\newlength{\templengths}
\newlength{\templengthw}

\definecolor{mathcolor}{RGB}{0,75,105} %{RGB}{0,0,0}%{RGB}{0,75,105} % чтобы выделить "свою" математику
\definecolor{argcolor}{RGB}{0,105,50} % {RGB}{0,0,0}%{RGB}{0,105,50} % чтобы выделить аргументы в "своей" математике
\definecolor{speccolor}{RGB}{0,40,125} %{RGB}{0,0,0}%{RGB}{0,40,125} % чтобы выделить "свой" текст
\newcommand{\dismath} [1] {{\color{mathcolor}#1}}
\newcommand{\distext} [1] {{\color{speccolor}#1}}
\newcommand{\argument}[1] {{\color{argcolor!50!black}{#1}}}

\SetEnumitemKey{midsep}{itemsep=0pt}

\setlist{midsep}

\newcommand{\defnotion}       [1]       {\distext{\textit{#1\/}}}

\newcommand{\remember}        [1]       {}
  \newcommand{\Rem}           [1]       {\remember{#1}}

\newcommand{\set}             [1]       {\dismath{\{ \argument{#1} \}}}

\newcommand{\tuple}           [1]       {\dismath{(\argument{#1})}}

\newcommand{\otuple}          [1]       {\dismath{\langle\argument{#1}\rangle}}

\newcommand{\intmodels}       [0]   {\mathrel{\dismath{|\hspace{-1.25pt}|\hspace{-2.8pt}{-}}}}
\newcommand{\clmodels}        [0]   %{\modmodels}
                                    {\mathrel{\dismath{|\hspace{-1.25pt}{\models}}}}
  \newcommand{\imodels}   [0]{\intmodels}
  
\newcommand{\nameKFrame}      [1]       {\dismath{\mathfrak{#1}}}       % имя шкалы Крипке
\newcommand{\nameKModel}      [1]       {\dismath{\mathfrak{#1}}}       % имя модели Крипке
\newcommand{\kframe}          [1]       {\nameKFrame{#1}}       % имя шкалы Крипке
\newcommand{\kmodel}          [1]       {\nameKModel{#1}}       % имя модели Крипке
\newcommand{\kFrame}          [1]       {\dismath{\bfrak{#1}}}       % имя шкалы Крипке
\newcommand{\kModel}          [1]       {\dismath{\bfrak{#1}}}       % имя модели Крипке
\newcommand{\cModel}          [1]       {\dismath{\bm{#1}}}      % имя классической модели

\newcommand{\logic}           [1]       {\dismath{\mathbf{#1}}}         % имя логики
\newcommand{\lang}            [1]       {\dismath{\mathcal{#1}}}        % имя языка
\newcommand{\sclass}          [1]       {\dismath{\mathscr{#1}}}        % имя класса структур
  \newcommand{\scls}[1]{\sclass{#1}}
  \newcommand{\Scls}[1]{\Sclass{#1}}

\newcommand{\numbers}         [1]       {\dismath{\mathds{#1}}}
\newcommand{\numN}                      {\numbers{N}}
\newcommand{\numNp}                     {\numbers{N}^+}
\newcommand{\numNpp}                    {\numbers{N}^{++}}

\newcommand{\Pow}             [1]       {\dismath{\mathscr{P}(\argument{#1})}}  % множество подмножеств

\newcommand{\boldfrak}        [1]       {\boldsymbol{\frak{#1}}}

\newcommand{\bfrak}           [1]       {\boldfrak{#1}}

\newcommand{\implication}               {\to}
  \newcommand{\imp} {\implication}
\newcommand{\conjunction}             {\wedge}   % команда имеется в wasysym
  \newcommand{\con} {\conjunction}
\newcommand{\disjunction}               {\vee}
  \newcommand{\dis} {\disjunction}

\newcommand{\equivalence}               {\leftrightarrow}
  \newcommand{\lra} {\equivalence}

\newcommand{\leftsquare}[1]{%
\begin{tikzpicture}[scale=#1]
\draw (0,0)--(0,1)--(1,1)--(1,0)--cycle;
\draw [fill, color=black!25] (0,0)--(0.5,0.5)--(0,1);
\draw (0,0)--(0.5,0.5)--(0,1);
\draw (0,1)--(0,0);
\end{tikzpicture}
}
\newcommand{\rightsquare}[1]{%
\begin{tikzpicture}[scale=#1]
\draw (0,0)--(0,1)--(1,1)--(1,0)--cycle;
\draw [fill, color=black!25] (1,0)--(0.5,0.5)--(1,1);
\draw (1,0)--(0.5,0.5)--(1,1);
\draw (1,1)--(1,0);
\end{tikzpicture}
}

\newcommand{\upsquare}[1]{%
\begin{tikzpicture}[scale=#1]
\draw (0,0)--(0,1)--(1,1)--(1,0)--cycle;
\draw [fill, color=black!25] (0,1)--(0.5,0.5)--(1,1);
\draw (0,1)--(0.5,0.5)--(1,1);
\draw (1,1)--(0,1);
\end{tikzpicture}
}

\newcommand{\downsquare}[1]{%
\begin{tikzpicture}[scale=#1]
\draw (0,0)--(0,1)--(1,1)--(1,0)--cycle;
\draw [fill, color=black!25] (0,0)--(0.5,0.5)--(1,0);
\draw (0,0)--(0.5,0.5)--(1,0);
\draw (1,0)--(0,0);
\end{tikzpicture}
}

\newcommand{\leftsq} {\mathop{\leftsquare {0.180}}}
\newcommand{\rightsq}{\mathop{\rightsquare{0.180}}}
\newcommand{\upsq}   {\mathop{\upsquare   {0.180}}}
\newcommand{\downsq} {\mathop{\downsquare {0.180}}}

\renewcommand{\iff}                     {\mathrel{\dismath{\Longleftrightarrow}}} %{\Leftarrow\!\!\!\Rightarrow}
\newcommand{\imply}                     {\mathrel{\dismath{\Longrightarrow}}}     %{=\!\!\!\Rightarrow}
\newcommand{\bydef}                     {\mathrel{\dismath{\leftrightharpoons}}}  %{=\!\!\!\Rightarrow}

\makeatletter
\newcommand{\mref}{\@ifnextchar({\mref@i}{\mref@i({\Box},{p})}}
\def\mref@i(#1,#2){#1#2 \implication #2}
\makeatother
    \makeatletter
    \newcommand{\mrefp}{\@ifnextchar({\mrefp@i}{\mrefp@i({p})}}
    \def\mrefp@i(#1){\mref(\Box,#1)}
    \makeatother
\makeatletter
\newcommand{\FOref}{\@ifnextchar({\FOref@i}{\FOref@i({x},{P})}}
\def\FOref@i(#1,#2){\forall #1\,#2(#1,#1)}
\makeatother
    \makeatletter
    \newcommand{\FOrefp}{\@ifnextchar({\FOrefp@i}{\FOrefp@i({P})}}
    \def\FOrefp@i(#1){\FOref(x,#1)}
    \makeatother
\makeatletter
\newcommand{\FOrefi}{\@ifnextchar({\FOrefi@i}{\FOrefi@i({x},{P})}}
\def\FOrefi@i(#1,#2){\forall #1\,#1#2#1}
\makeatother
    \makeatletter
    \newcommand{\FOrefip}{\@ifnextchar({\FOrefip@i}{\FOrefip@i({P})}}
    \def\FOrefip@i(#1){\FOrefi(x,#1)}
    \makeatother

\makeatletter
\newcommand{\mtra}{\@ifnextchar({\mtra@i}{\mtra@i({\Box},{p})}}
\def\mtra@i(#1,#2){#1#2 \implication #1#1#2}
\makeatother
    \makeatletter
    \newcommand{\mtrap}{\@ifnextchar({\mtrap@i}{\mtrap@i({p})}}
    \def\mtrap@i(#1){\mtra(\Box,#1)}
    \makeatother
\makeatletter
\newcommand{\FOtra}{\@ifnextchar({\FOtra@i}{\FOtra@i({x},{y},{z},{P})}}
\def\FOtra@i(#1,#2,#3,#4){\forall #1\forall #2\forall #3\,(#4(#1,#2)\conjunction #4(#2,#3) \implication #4(#1,#3))}
\makeatother
    \makeatletter
    \newcommand{\FOtrap}{\@ifnextchar({\FOtrap@i}{\FOtrap@i({P})}}
    \def\FOtrap@i(#1){\FOtra(x,y,z,#1)}
    \makeatother
\makeatletter
\newcommand{\FOtrai}{\@ifnextchar({\FOtrai@i}{\FOtrai@i({x},{y},{z},{P})}}
\def\FOtrai@i(#1,#2,#3,#4){\forall #1\forall #2\forall #3\,(#1#4#2\conjunction #2#4#3 \implication #1#4#3)}
\makeatother
    \makeatletter
    \newcommand{\FOtraip}{\@ifnextchar({\FOtraip@i}{\FOtraip@i({P})}}
    \def\FOtraip@i(#1){\FOtrai(x,y,z,#1)}
    \makeatother

\makeatletter
\newcommand{\msym}{\@ifnextchar({\msym@i}{\msym@i({\Box},{\Diamond},{p})}}
\def\msym@i(#1,#2,#3){#3 \implication #1#2#3}
\makeatother
    \makeatletter
    \newcommand{\msymp}{\@ifnextchar({\msymp@i}{\msymp@i({p})}}
    \def\msymp@i(#1){\msym(\Box,\Diamond,#1)}
    \makeatother
\makeatletter
\newcommand{\FOsym}{\@ifnextchar({\FOsym@i}{\FOsym@i({x},{y},{P})}}
\def\FOsym@i(#1,#2,#3){\forall #1\forall #2\,(#3(#1,#2)\implication #3(#2,#1))}
\makeatother
    \makeatletter
    \newcommand{\FOsymp}{\@ifnextchar({\FOsymp@i}{\FOsymp@i({P})}}
    \def\FOsymp@i(#1){\FOsym(x,y,#1)}
    \makeatother
\makeatletter
\newcommand{\FOsymi}{\@ifnextchar({\FOsymi@i}{\FOsymi@i({x},{y},{P})}}
\def\FOsymi@i(#1,#2,#3){\forall #1\forall #2\,(#1#3#2 \implication #2#3#1)}
\makeatother
    \makeatletter
    \newcommand{\FOsymip}{\@ifnextchar({\FOsymip@i}{\FOsymip@i({P})}}
    \def\FOsymip@i(#1){\FOsymi(x,y,#1)}
    \makeatother

\makeatletter
\newcommand{\meuc}{\@ifnextchar({\meuc@i}{\meuc@i({\Box},{\Diamond},{p})}}
\def\meuc@i(#1,#2,#3){#2#3 \implication #1#2#3}
\makeatother
    \makeatletter
    \newcommand{\meucp}{\@ifnextchar({\meucp@i}{\meucp@i({p})}}
    \def\meucp@i(#1){\meuc(\Box,\Diamond,#1)}
    \makeatother
\makeatletter
\newcommand{\FOeuc}{\@ifnextchar({\FOeuc@i}{\FOeuc@i({x},{y},{z},{P})}}
\def\FOeuc@i(#1,#2,#3,#4){\forall #1\forall #2\forall #3\,(#4(#1,#2)\con #4(#1,#3)\implication #4(#2,#3))}
\makeatother
    \makeatletter
    \newcommand{\FOeucp}{\@ifnextchar({\FOeucp@i}{\FOeucp@i({P})}}
    \def\FOeucp@i(#1){\FOeuc(x,y,z,#1)}
    \makeatother
\makeatletter
\newcommand{\FOeuci}{\@ifnextchar({\FOeuci@i}{\FOeuci@i({x},{y},{z},{P})}}
\def\FOeuci@i(#1,#2,#3,#4){\forall #1\forall #2\forall #3\,(#1#4#2 \con #1#4#3 \implication #2#4#3)}
\makeatother
    \makeatletter
    \newcommand{\FOeucip}{\@ifnextchar({\FOeucip@i}{\FOeucip@i({P})}}
    \def\FOeucip@i(#1){\FOeuci(x,y,z,#1)}
    \makeatother

\makeatletter
\newcommand{\mser}{\@ifnextchar({\mser@i}{\mser@i({\Box},{\Diamond},{p})}}
\def\mser@i(#1,#2,#3){#1#3 \implication #2#3}
\makeatother
    \makeatletter
    \newcommand{\mserp}{\@ifnextchar({\mserp@i}{\mserp@i({p})}}
    \def\mserp@i(#1){\mser(\Box,\Diamond,#1)}
    \makeatother
\makeatletter
\newcommand{\FOser}{\@ifnextchar({\FOser@i}{\FOser@i({x},{y},{P})}}
\def\FOser@i(#1,#2,#3){\forall #1\exists #2\,#3(#1,#2)}
\makeatother
    \makeatletter
    \newcommand{\FOserp}{\@ifnextchar({\FOserp@i}{\FOserp@i({P})}}
    \def\FOserp@i(#1){\FOser(x,y,#1)}
    \makeatother
\makeatletter
\newcommand{\FOseri}{\@ifnextchar({\FOseri@i}{\FOseri@i({x},{y},{P})}}
\def\FOseri@i(#1,#2,#3){\forall #1\exists #2\,#1#3#2}
\makeatother
    \makeatletter
    \newcommand{\FOserip}{\@ifnextchar({\FOserip@i}{\FOserip@i({P})}}
    \def\FOserip@i(#1){\FOseri(x,y,#1)}
    \makeatother

\makeatletter
\newcommand{\mla}{\@ifnextchar({\mla@i}{\mla@i({\Box},{p})}}
\def\mla@i(#1,#2){#1(#1#2 \implication #2) \implication #1#2}
\makeatother
    \makeatletter
    \newcommand{\mlap}{\@ifnextchar({\mlap@i}{\mlap@i({p})}}
    \def\mlap@i(#1){\mla(\Box,#1)}
    \makeatother

\makeatletter
\newcommand{\mgrz}{\@ifnextchar({\mgrz@i}{\mgrz@i({\Box},{p})}}
\def\mgrz@i(#1,#2){#1(#1(#2 \implication #1#2) \implication #2) \implication #2}
\makeatother
    \makeatletter
    \newcommand{\mgrzp}{\@ifnextchar({\mgrzp@i}{\mgrzp@i({p})}}
    \def\mgrzp@i(#1){\mgrz(\Box,#1)}
    \makeatother

\makeatletter
\newcommand{\mwgrz}{\@ifnextchar({\mwgrz@i}{\mwgrz@i({\Box},{p})}}
\def\mwgrz@i(#1,#2){#1^+(#1(#2 \implication #1#2) \implication #2) \implication #2}
\makeatother
    \makeatletter
    \newcommand{\mwgrzp}{\@ifnextchar({\mwgrzp@i}{\mwgrzp@i({p})}}
    \def\mwgrzp@i(#1){\mwgrz(\Box,#1)}
    \makeatother

\makeatletter
\newcommand{\ExtDiamond}{\@ifnextchar({\ExtDiamond@i}{\ExtDiamond@i({p})}}
\def\ExtDiamond@i(#1){{\Diamond\!\!\!\!\Diamond}_{#1}}
\makeatother

\renewcommand{\dismath} [1] {#1}
\renewcommand{\distext} [1] {#1}
\renewcommand{\argument}[1] {#1}

\newcounter{\theequation}[section]
\renewcommand{\theequation}{\thesection.\arabic{equation}}

\newcommand{\md}       [0]{\mathop{\mbox{\textnormal{\texttt{md}}}}}
\newcommand{\sub}      [0]{\mathop{\mathit{sub}}}
\renewcommand{\kModel} [1]{\kmodel{#1}}
\renewcommand{\Scls}   [1]{\scls{#1}}

\newcommand{\ckf}      [0]{\mathop{\mbox{\textnormal{\texttt{KF}}}}}
\newcommand{\caf}      [0]{\mathop{\mbox{\textnormal{\texttt{AF}}}}}
\newcommand{\fin}      [0]{\mathop{\mbox{\textnormal{\texttt{fin}}}}}
\newcommand{\QML}      [0]{\mathop{\mbox{\textnormal{\texttt{QML}}}}}
\newcommand{\QMLext}   [2]{\mathop{\mbox{\textnormal{\texttt{QML}}}^{\mathit{#1}}_{\mathit{#2}}}}
\newcommand{\QMLe}     [0]{\QMLext{e}{all}}
\newcommand{\QMLc}     [0]{\QMLext{c}{all}}
\newcommand{\QMLed}    [0]{\QMLext{e}{dfin}}
\newcommand{\QMLcd}    [0]{\QMLext{c}{dfin}}
\newcommand{\QMLew}    [0]{\QMLext{e}{wfin}}
\newcommand{\QMLcw}    [0]{\QMLext{c}{wfin}}
\newcommand{\QSIL}     [0]{\mathop{\mbox{\textnormal{\texttt{QSIL}}}}}
\newcommand{\QSILext}  [2]{\mathop{\mbox{\textnormal{\texttt{QSIL}}}^{\mathit{#1}}_{\mathit{#2}}}}

\newcommand{\QSILed}   [0]{\QSILext{e}{dfin}}
\newcommand{\QSILcd}   [0]{\QSILext{c}{dfin}}
\newcommand{\QSILew}   [0]{\QSILext{e}{wfin}}
\newcommand{\QSILcw}   [0]{\QSILext{c}{wfin}}
\newcommand{\aug}      [2]{\mathop{\mbox{\textnormal{\texttt{aug}}}^{\mathit{#1}}_{\mathit{#2}}}}

\makeatletter
\def\thmstyle{\it} % style of text in theorem environment
\def\@begintheorem#1#2{\it \trivlist \item[\hskip
        \labelsep{\bf #1\ #2.}]\thmstyle}
\def\@opargbegintheorem#1#2#3{\it \trivlist \item[\hskip
        \labelsep{\bf #1\ #2\ (#3).}]\thmstyle}
\makeatother

\newtheorem{theorem}{{\indent}Theorem}[section]

\newtheorem{lemma}[theorem]{{\indent}Lemma}
\newtheorem{proposition}[theorem]{{\indent}Proposition}
\newtheorem{sublemma}[theorem]{{\indent}Sublemma}

\newtheorem{corollary}[theorem]{{\indent}Corollary}

\newtheorem{remark}[theorem]{{\indent}Remark}
\newtheorem{conjecture}[theorem]{{\indent}Conjecture}

\newcommand{\numeral}[1]{\overline{#1}} %{\ulcorner\!#1\!\urcorner}

\definecolor{cg0}{gray}{1.00}          % It's mine, use it here
\definecolor{cg1}{gray}{0.90}          % It's mine, use it here
\definecolor{cg2}{gray}{0.80}          % It's mine, use it here
\definecolor{cg3}{gray}{0.70}          % It's mine, use it here
\definecolor{cg4}{gray}{0.60}          % It's mine, use it here
\definecolor{cg5}{gray}{0.50}          % It's mine, use it here
\definecolor{cg6}{gray}{0.40}          % It's mine, use it here
\definecolor{cg7}{gray}{0.30}          % It's mine, use it here
\definecolor{cg8}{gray}{0.20}          % It's mine, use it here
\definecolor{cg9}{gray}{0.10}          % It's mine, use it here
 
\newcommand{\nicearrow}[3]
{
}

\newcommand{\drawtileflattm} [9]
{
\draw [white, opacity = 0, name path = diag 1] #1--#3;
\draw [white, opacity = 0, name path = diag 2] #2--#4;
\draw [name intersections = {of = diag 1 and diag 2, by = {tcenter}}];

\foreach \c in {0.96}
{
  \coordinate (ttl)     at ($(tcenter)+\c*#4-\c*(tcenter)$);
  \coordinate (tbr)     at ($(tcenter)+\c*#2-\c*(tcenter)$);
  \coordinate (tbl)     at ($(tcenter)+\c*#1-\c*(tcenter)$);
  \coordinate (ttr)     at ($(tcenter)+\c*#3-\c*(tcenter)$);
  \coordinate (sttl)    at ($(tcenter)+0.4*(ttl)-0.4*(tcenter)$);
  \coordinate (stbr)    at ($(tcenter)+0.4*(tbr)-0.4*(tcenter)$);
  \coordinate (stbl)    at ($(tcenter)+0.4*(tbl)-0.4*(tcenter)$);
  \coordinate (sttr)    at ($(tcenter)+0.4*(ttr)-0.4*(tcenter)$);
}

\coordinate (ledge) at ($0.5*#4+0.5*#1$);
\coordinate (bedge) at ($0.5*#1+0.5*#2$);
\coordinate (redge) at ($0.5*#2+0.5*#3$);
\coordinate (tedge) at ($0.5*#3+0.5*#4$);

\draw (ttl)--(tbl)--(tbr)--(ttr)--cycle;
\draw (sttl)--(stbl)--(stbr)--(sttr)--cycle;
\draw (ttl)--(sttl);
\draw (ttr)--(sttr);
\draw (tbl)--(stbl);
\draw (tbr)--(stbr);

\node [] at (tcenter) {\phantom{$y^l_l$}{#9}\phantom{$y^l_l$}};
\node [fill = white, draw = black!25, rounded corners] at (ledge)   {{$\phantom{iR^i_iy}$}};
\node [fill = white, draw = black!25, rounded corners] at (bedge)   {{$\phantom{iR^i_iy}$}};
\node [fill = white, draw = black!25, rounded corners] at (redge)   {{$\phantom{iR^i_iy}$}};
\node [fill = white, draw = black!25, rounded corners] at (tedge)   {{$\phantom{iR^i_iy}$}};
\node [] at (ledge)   {\phantom{$y^l_l$}{#5}\phantom{$y^l_l$}};
\node [] at (bedge)   {\phantom{$y^l_l$}{#6}\phantom{$y^l_l$}};
\node [] at (redge)   {\phantom{$y^l_l$}{#7}\phantom{$y^l_l$}};
\node [] at (tedge)   {\phantom{$y^l_l$}{#8}\phantom{$y^l_l$}};
}

\newcommand{\drawtileflattms} [4]
{
\draw [white, opacity = 0, name path = diag 1] #1--#3;
\draw [white, opacity = 0, name path = diag 2] #2--#4;
\draw [name intersections = {of = diag 1 and diag 2, by = {tcenter}}];

\foreach \c in {0.96}
{
  \coordinate (ttl)     at ($(tcenter)+\c*#4-\c*(tcenter)$);
  \coordinate (tbr)     at ($(tcenter)+\c*#2-\c*(tcenter)$);
  \coordinate (tbl)     at ($(tcenter)+\c*#1-\c*(tcenter)$);
  \coordinate (ttr)     at ($(tcenter)+\c*#3-\c*(tcenter)$);
  \coordinate (sttl)    at ($(tcenter)+0.4*(ttl)-0.4*(tcenter)$);
  \coordinate (stbr)    at ($(tcenter)+0.4*(tbr)-0.4*(tcenter)$);
  \coordinate (stbl)    at ($(tcenter)+0.4*(tbl)-0.4*(tcenter)$);
  \coordinate (sttr)    at ($(tcenter)+0.4*(ttr)-0.4*(tcenter)$);
}

\coordinate (ledge) at ($0.5*#4+0.5*#1$);
\coordinate (bedge) at ($0.5*#1+0.5*#2$);
\coordinate (redge) at ($0.5*#2+0.5*#3$);
\coordinate (tedge) at ($0.5*#3+0.5*#4$);

\draw (ttl)--(tbl)--(tbr)--(ttr)--cycle;
\draw (sttl)--(stbl)--(stbr)--(sttr)--cycle;
\draw (ttl)--(sttl);
\draw (ttr)--(sttr);
\draw (tbl)--(stbl);
\draw (tbr)--(stbr);

}

\newcommand{\drawtileflattmslanted} [4]
{
\draw [white, opacity = 0, name path = diag 1] #1--#3;
\draw [white, opacity = 0, name path = diag 2] #2--#4;
\draw [name intersections = {of = diag 1 and diag 2, by = {tcenter}}];

\foreach \c in {0.96}
{
  \coordinate (ttl)     at ($(tcenter)+\c*#4-\c*(tcenter)$);
  \coordinate (tbr)     at ($(tcenter)+\c*#2-\c*(tcenter)$);
  \coordinate (tbl)     at ($(tcenter)+\c*#1-\c*(tcenter)$);
  \coordinate (ttr)     at ($(tcenter)+\c*#3-\c*(tcenter)$);
  \coordinate (sttl)    at ($(tcenter)+0.4*(ttl)-0.4*(tcenter)$);
  \coordinate (stbr)    at ($(tcenter)+0.4*(tbr)-0.4*(tcenter)$);
  \coordinate (stbl)    at ($(tcenter)+0.4*(tbl)-0.4*(tcenter)$);
  \coordinate (sttr)    at ($(tcenter)+0.4*(ttr)-0.4*(tcenter)$);
}

\coordinate (ledge) at ($0.5*#4+0.5*#1$);
\coordinate (bedge) at ($0.5*#1+0.5*#2$);
\coordinate (redge) at ($0.5*#2+0.5*#3$);
\coordinate (tedge) at ($0.5*#3+0.5*#4$);

\draw (ttl)--(tbl)--(tbr)--(ttr)--cycle;
\draw (sttl)--(stbl)--(stbr)--(sttr)--cycle;
\draw (ttl)--(sttl);
\draw (ttr)--(sttr);
\draw (tbl)--(stbl);
\draw (tbr)--(stbr);
}

\newcommand{\circleone} [4]
{
\draw [fill = #3] #1 circle [radius = #2];
}

\newcommand{\circletwo} [4]
{
\circleone{#1}{#2}{#3}{#4}
\draw [] #1 circle [radius = #2+#4];
}

\newcommand{\circlethree} [4]
{
\circletwo{#1}{#2}{#3}{#4}
\draw [] #1 circle [radius = #2+2*#4];
}

\newcommand{\circletile} [6]
{
\filldraw[fill=#3, draw=black] #1 -- ($#1+#2*(-0.707,+0.707)$)
                               arc [start angle=135, end angle=225, radius=#2] -- cycle;
\filldraw[fill=#4, draw=black] #1 -- ($#1+#2*(-0.707,-0.707)$)
                               arc [start angle=225, end angle=315, radius=#2] -- cycle;
\filldraw[fill=#5, draw=black] #1 -- ($#1+#2*(+0.707,-0.707)$)
                               arc [start angle=-45, end angle= 45, radius=#2] -- cycle;
\filldraw[fill=#6, draw=black] #1 -- ($#1+#2*(+0.707,+0.707)$)
                               arc [start angle= 45, end angle=135, radius=#2] -- cycle;
}

\newcommand{\drawtileflatsmall} [8]
{
\coordinate (stcenter) at ($0.5*#1+0.5*#3$);

\fill [#5, fill opacity=0.75] #1--(stcenter)--#4--cycle;
\fill [#6, fill opacity=0.75] #2--(stcenter)--#1--cycle;
\fill [#7, fill opacity=0.75] #3--(stcenter)--#2--cycle;
\fill [#8, fill opacity=0.75] #4--(stcenter)--#3--cycle;

\draw #1 -- #3;
\draw #2 -- #4;
\draw #1 -- #2 -- #3 -- #4 -- cycle;
}

\newcommand{\drawhalftileflatsmallh} [6]
{
\coordinate (stcenter) at ($0.5*#1+0.5*#3$);

\fill [#5, fill opacity=0.75] #1--(stcenter)--#4--cycle;
\fill [#6, fill opacity=0.75] #3--(stcenter)--#2--cycle;

\draw #1 -- #3 -- #2 -- #4 -- cycle;
}

\newcommand{\drawhalftileflatsmallv} [6]
{
\coordinate (stcenter) at ($0.5*#1+0.5*#3$);

\fill [#5, fill opacity=0.75] #2--(stcenter)--#1--cycle;
\fill [#6, fill opacity=0.75] #4--(stcenter)--#3--cycle;

\draw #1 -- #3 -- #4 -- #2 -- cycle;
}

\begin{document}

%\begin{frontmatter} 

\title{Recursive inseparability \\ 
of classical theories of a binary predicate \\ 
and non-classical logics of a unary predicate\thanks{The work is supported by the Basic Research Program of the HSE University.}}
%with three variables
%}
%\tnotetext[t1]{The work is supported by the Basic Research Program of the HSE University.}
%\tnotetext[t1]{The paper was prepared within the framework of the project ``International academic cooperation'' HSE University.}
\author{Mikhail Rybakov}
\affil{{Higher School of Modern Mathematics MIPT}, {HSE University}}
%\ead{m\_rybakov@mail.ru}
\date{}

%\begin{keyword}
%first-order theory
%\sep modal predicate logic 
%\sep superintuitionistic predicate logic 
%\sep monadic fragment
%\sep dyadic fragment
%\sep recursive separability
%\sep recursive enumerability
%\sep decidability
%\MSC[2020] 03B10 \sep 03B20 \sep 03B25 \sep 03B45 %\sep 03B55 \sep 03D35
%\end{keyword}

%\end{frontmatter}

\maketitle

\begin{abstract}
The paper considers algorithmic properties of classical and non-classical first-order logics and theories in bounded languages. 
The main idea is to prove the undecidability of various fragments of classical and non-classical first-order logics and theories indirectly~--- by extracting it as a consequence of the recursive inseparability of  special problems associated with them. 
%
%The main idea is to demonstrate the undecidability of various fragments of classical and nonclassical first-order logics and theories indirectly. This is achieved by extracting it as a consequence of the recursive inseparability of certain specialized problems associated with them.
%
First, we propose a domino problem, which makes it possible to catch the recursive inseparability of two sets. Second, using this problem, we prove that the classical first-order logic of a binary predicate and the theory of its finite models where the predicate is symmetric and irreflexive are recursively inseparable in a language with a single binary predicate letter and three variables (without constants and equality). Third, we prove, for an infinite class of logics, that the monadic fragment of a modal predicate logic and the logic of the class of its finite Kripke frames are recursively inseparable in languages with a single unary predicate letter and two individual variables; the same result is obtained if we replace the condition of finiteness of frames with the condition of finiteness of domains allowed in frames. Forth, we expand the results to a wide class of superintuitionistic predicate logics. In particular, it is proved that the positive fragments of the intuitionistic predicate logic and the logic of the class of finite intuitionistic Kripke frames are recursively inseparable in the language with a single unary predicate letter and two individual variables. The technique used and the results obtained allow us to answer some additional questions about the decidability of special monadic fragments of some modal and superintuitionistic predicate logics.
\end{abstract}

%\newpageafter{abstract}

%\maketitle

\newpage

\tableofcontents

\newpage

%\input{RecInsep-30-text-Overleaf}
%%%%%%%%%%%%%%%%%%%%%%%%%%%%%%%%%%%%%%%
\section{Introduction}
\label{sec:introduction}
\setcounter{equation}{0}

\subsection{Issues under consideration}

We consider algorithmic properties of classical and non-classical first-order logics and theories in restricted languages. It is well known that the classical predicate logic $\logic{QCl}$ is undecidable~\cite{Church36}, and to prove this, it is sufficient to use a single binary predicate letter and three individual variables~\cite[Section~4.8~(ii)]{TG87} (see also~\cite{Suranyi43} for three variables and~\cite[Chapter~21]{BBJ07} for a single binary predicate letter); the same is true for many classical theories~\cite{ELTT:1965, Sper:2016, MR:2022:DoklMath, MR:2023:LI, MR:2023:HeraldTSU}. At the same time, the monadic fragment of $\logic{QCl}$ is decidable. The decidability remains if we add the equality~\cite[Chapter~21]{BBJ07}, allow only formulas with at most two variables~\cite{Mortimer75,GKV97} or use only formulas of guarded fragments~\cite{Gradel99}. In general, the classical decision problem~\cite{BGG97} has been transformed into a classification problem when the purpose of the study is to search for boundaries within which the decidability or undecidability is still preserved. As for non-classical logics, they are usually undecidable even if the language contains only monadic predicate letters~\cite{Kripke62,MMO65,MR:2002:LI,MR:2017:LI} or two individual variables~\cite{KKZ05} or even both a single monadic letter and two-three individual variables~\cite{RSh19SL,RShJLC20a,RShJLC21b}; known results on decidability are obtained under fairly strong restrictions on language or semantics~\cite{HWZ00, HWZ01, WZ01, WZ02, MR:2017:LI, RShsubmitted, ARSh:2023:arXiv}. 
In this paper, some general methods for obtaining results on the undecidability of fragments of classical and non-classical logics and theories will be proposed. With their help, we will get answers to some questions.

The paper is based on three questions, and our purpose is, in particular, to answer them. In addition, there are several other issues that have also been resolved.

The first question was posed by Evgeny Zolin in 2019 during a conference and consists of the following:
\begin{itemize}
\item
\textit{Are the monadic fragments of modal predicate logics defined by classes of finite Kripke frames decidable in languages with finitely many individual variables?}
\end{itemize}
The question is answered only partially. So, it is proved that the monadic fragment of a logic defined by a finite Kripke frame is decidable~\cite{MR:2017:LI}; the same is true for every logic $\logic{QAlt}_n$, where $n\in\numN$,\footnote{We use $\numN$, $\numNp$, and $\numNpp$ for, respectively, the sets of natural numbers including~$0$, positive natural numbers, and positive natural numbers without~$1$.} defined as a logic of the class of Kripke frames where each world sees at most $n$ worlds~\cite{MR:2017:LI}, even if we add the equality to the language~\cite{RShsubmitted}. Also, it is proved that the monadic fragments of logics defined by the classes of finite Kripke frames of logics such as $\logic{QK}$, $\logic{QT}$, $\logic{QD}$, $\logic{QKB}$, $\logic{QKTB}$, $\logic{QK4}$, $\logic{QS4}$, and some others are undecidable in languages with a single unary predicate letter and three individual variables~\cite{RShJLC20a}; similar results are obtained for superintuitionistic predicate logics~\cite{RShJLC21b}. It should be noted that the methods used in~\cite{RShJLC20a,RShJLC21b} give us nothing about logics such as $\logic{QS5}$, $\logic{QK45}$, $\logic{QKD45}$, $\logic{QK4B}$ and logics containing formulas bounding the depth of Kripke frames; 
%they also give us nothing about fragments with two individual variables. 
in addition, they do not cover fragments with two individual variables.
%
%Additionally, they do not cover fragments with two individual variables.
%
As a result, we must admit that the question has not received a satisfactory answer. We shall fix~it.

%The second question, firstly, was of interest to the author initially, and secondly, was asked by Stanislav Speranski in 2022 during a private conversation (concerning some near, more general, problems):
%
The second question initially stimulated the author's interest and was later posed by Stanislav Speranski during a private conversation in 2022, concerning some broader, more general issues:
\begin{itemize}
%{
%\widowpenalty 10000
%\raggedbottom
%\interlinepenalty 10000
\item
\textit{Is the classical theory of finite models for the symmetric irreflexive binary relation decidable in a language with a single binary predicate letter and three individual variables?}
%}
\end{itemize}
It is known that the theory of the symmetric irreflexive binary relation is undecidable~\cite{NerodeShore80,Kremer97,Sper:2016} and even in a language with a single binary predicate letter and three individual variables (without constants and equality)~\cite{MR:2022:DoklMath}. Although the methods of~\cite{MR:2022:DoklMath} allow us to prove undecidability for many classical theories of a binary predicate, in particular, defined by classes of finite models, they are not directly applicable to the theory asked. Notice that if the fragment of this theory is undecidable, then we can readily answer the question on the decidability of logics defined by finite Kripke frames of $\logic{QS5}$, $\logic{QK45}$, $\logic{QKD45}$, or $\logic{QK4B}$ in languages with a single unary predicate letter and three variables.

The third question was posed by Valentin Shehtman during a conference in 2023, and was as follows:
\begin{itemize}
\item
\textit{Are the monadic fragment of a modal predicate logic and the monadic fragment of the logic of its finite Kripke frames recursively separable?}
\end{itemize}
%Thanks to this question, there was a more general view of the situation with the previous two questions. 
%
This question provided a broader perspective on the situation, based on the insights gained from the previous two questions.
For the classical first-order logic, the Trakhtenbrot theorem~\cite{Trakhtenbrot50,Trakhtenbrot53} says that $\logic{QCl}$ and the theory of finite models are recursively inseparable and, as a corollary, undecidable. Thus, this theorem shows us another way of proving the undecidability. We are going to go this way to answer all the questions raised (and not only).

First of all, we answer the question of Stanislav Speranski. To this end, we propose a tiling problem allowing us to catch recursive inseparability, and then describe it by first-order formulas with a single binary predicate letter and three individual variables; as a result, we obtain a proof of the Trakhtenbrot theorem for this language. After this is done, we modify the construction so that it can be expanded to the theory of symmetric irreflexive binary predicate. Then we use the results to answer the question of Valentin Shehtman and, as a corollary, to the question of Evgeny Zolin in the case of three variables in the language. Along the way, we get the answers to some close questions. In particular, we answer a question similar to that posed by Valentin Shehtman:
\begin{itemize}
\item
\textit{Are the monadic fragment of a modal predicate logic and the monadic fragment of the logic of its frames with finite domains recursively separable?}
\end{itemize}
Also, using the technique of~\cite{KKZ05}, we then show that for many logics the number of variables used to prove the recursive inseparability of certain problems can be reduced to two, which gives us a more exhaustive answer to the question of Evgeny Zolin.

In addition, we answer similar questions for the intuitionistic predicate logic $\logic{QInt}$ and some of its extensions. So, it is known that the intuitionistic predicate logic is undecidable in the language with a single unary predicate letter and two individual variables~\cite{RSh19SL} and the logic of the class of finite intuitionistic frames is undecidable (even not recursively enumerable) in the language with a single unary predicate letter and three individual variables~\cite{RShJLC21b}. 
%using the recursive inseparability of some problems for intuitionistic logic, 
Below, we get the answer to the following question:
\begin{itemize}
\item
\textit{Are the intuitionistic predicate logic and the logic of its finite Kripke frames recursively separable in the language containing a single unary predicate letter and two individual variables?}
\end{itemize}
We shall show that the answer is ``No'' even for the positive fragments; and the same for the intuitionistic predicate logic and the logic of its frames with finite domains. It is not difficult to show that the intuitionistic predicate logic can be replaced by any superintuitionistic predicate logic contained in $\logic{QKC}$, the predicate logic of the weak law of the excluded middle, 
without affecting the truth of the results.
%and the results remain true. 
This is so since the positive fragments of $\logic{QInt}$ and $\logic{QKC}$ coincide. It should be noted that this observation was significantly used in earlier works~\cite{RSh19SL,RShJLC21b}, where the undecidability of the monadic fragments was obtained only for the superintuitionistic logics contained in~$\logic{QKC}$. Thus, the following question remained unanswered:
\begin{itemize}
\item
\textit{Is there a superintuitionistic predicate logic not contained in\/~$\logic{QKC}$ such that every its superintuitionistic sublogic is undecidable in the language with a single unary predicate letter and two individual variables?}
\end{itemize}
Using an approach based on obtaining the recursive inseparability of problems, we show that there are infinitely many such logics. Moreover, for each of them, the logic and the logic of its finite Kripke frames are recursively inseparable in a language with a single unary predicate letter and two individual variables.

Another goal of the author is to obtain very simple, clear, and short proofs of the results obtained. This seems important because earlier papers containing solutions to close issues often also contain too many technical details. The abundance of such details makes it difficult to identify essential points in order to use them in other studies, and also creates certain difficulties for including the results obtained in the educational process. With this in mind, the reader will find quite a few simple statements below, leading step by step to solving the questions posed.
The results were presented in~\cite{MR:2023:MR}.

\subsection{Short explanation of the technical part}

The technical basis of the constructions used below is made up of well-known ideas that interact and intertwine in a certain way. So, we use Turing machines, tiling problems, relativization, translations, embeddings, and semantical methods. Now we pay attention only to the ideas of simulating predicate letters in the modal first-order language by formulas with a single unary predicate letter.
% and in the classical first-order language with a single binary predicate letter.

In earlier author's papers (including joint works) considering issues of algorithmic complexity of non-classical logics of a unary predicate, some ideas originating in works related to the complexity of propositional logics~\cite{BS93,Spaan93,Halpern95,ChRyb03,Rybakov06,Rybakov07,Rybakov08} were used in a straightforward way~\cite{MR:2002:LI,RSh19SL,RShJLC20a,RSh20AiML,RShJLC21b,RShJLC21c,RybIGPL22}. This led to the fact that in order to prove that the fragment of a logic is undecidable, first an undecidable problem was modeled in it using an unlimited set of unary predicate letters (originally, propositional variables), and then all these predicate letters were simulated by formulas containing a single unary predicate letter (originally, a single propositional variable). So, to simulate propositional variables $p_1,\ldots,p_n$, formulas like $\Diamond^k p$, where $1\leqslant k\leqslant n$, were used (slightly more complicated, but it does not matter now) and to simulate $P_k(x)$, an analog like $\Diamond^k P(x)$ was taken. 
%
%Such a choice made it possible to quite easily transfer some techniques from the field of propositional logics to the field of predicate logics. 
%
Making such a choice facilitated the transfer of techniques from the field of propositional logic to the field of predicate logic quite easily.
But we had to pay for this choice, for example, by losing the logics containing formulas bounding the depth of Kripke frames. We are going to win them back.

\begin{figure}
\centering
\begin{tikzpicture}[scale=2.25]

\coordinate (w0) at (0,0*0.4);
\coordinate (w1) at (0,1*0.4);
\coordinate (w2) at (0,2*0.4);
\coordinate (w3) at (0,3*0.4);
\coordinate (w4) at (0,4*0.4);

\shade [ball color=black] (w0) circle [radius = 1.75pt];
\shade [ball color=black] (w1) circle [radius = 1.75pt];
\shade [ball color=black] (w3) circle [radius = 1.75pt];
\shade [ball color=black] (w4) circle [radius = 1.75pt];
\draw  [>=latex, ->, shorten >= 4pt, shorten <= 4pt, color=black] (w0)--(w1);
\draw  [>=latex, ->, shorten >= 8pt, shorten <= 4pt, color=black] (w1)--(w2);
\draw  [>=latex, ->, shorten >= 4pt, shorten <= 8pt, color=black] (w2)--(w3);
\draw  [>=latex, ->, shorten >= 4pt, shorten <= 4pt, color=black] (w3)--(w4);

\node [] at ($(w2)+(0,0.032)$) {$\vdots$};
\draw (0.64,0.2) arc [start angle = -124, end angle = 221, x radius = 0.84, y radius = 0.64] -- ($(w0)+(0.032,0.016)$) -- cycle;

\coordinate (a1) at (0.60+0*0.18,1*0.4);
\coordinate (a2) at (0.60+1*0.18,1*0.4);
\coordinate (a3) at (0.60+2*0.18,1*0.4);
\coordinate (a4) at (0.60+3*0.18,1*0.4);
\coordinate (a5) at (0.60+4*0.18,1*0.4);
\coordinate (a6) at (0.60+5*0.18,1*0.4);

\coordinate (w0') at (3,0*0.4);

\coordinate (b1) at ($(a1)+(w0')-(w0)+(0,0.18)$);
\coordinate (b2) at ($(b1)+(0,0.18)$);
\coordinate (b3) at ($(b2)+(0,0.18)$);
\coordinate (b4) at ($(b3)+(0,0.18)$);

\foreach \x in {1,3,4}
{
\filldraw [color=black!64] ($(b1)+\x*(0.18,0)$) circle [x radius = 0.04, y radius = 0.04];
\filldraw [color=black!64] ($(b3)+\x*(0.18,0)$) circle [x radius = 0.04, y radius = 0.04];
\filldraw [color=black!64] ($(b4)+\x*(0.18,0)$) circle [x radius = 0.04, y radius = 0.04];
%\node [color=black!64] at ($(b2)+\x*(0.18,0)+(0,0.032)$) {$\vdots$};
\draw [color=black!64] ($(a1)+(w0')-(w0)+\x*(0.18,0)$)--($(b1)+\x*(0.18,0)$);
\draw [color=black!64] ($(b3)+\x*(0.18,0)$)--($(b4)+\x*(0.18,0)$);
\draw [color=black!64, densely dashed, shorten <= 1pt] ($(b1)+\x*(0.18,0)$)--($(b3)+\x*(0.18,0)$);
}

\foreach \x in {0,1}
{
\draw  [>=latex, ->, shorten >= 6pt, shorten <= 2pt, color=black!64] ($(a1)+\x*(w0')-\x*(w0)$)--($(a6)+(0.09,0)+\x*(w0')-\x*(w0)$);
\filldraw [color=black!64] ($(a1)+\x*(w0')-\x*(w0)$) circle [x radius = 0.04, y radius = 0.04];
\filldraw [color=black!64] ($(a2)+\x*(w0')-\x*(w0)$) circle [x radius = 0.04, y radius = 0.04];
\filldraw [color=black!64] ($(a3)+\x*(w0')-\x*(w0)$) circle [x radius = 0.04, y radius = 0.04];
\filldraw [color=black!64] ($(a4)+\x*(w0')-\x*(w0)$) circle [x radius = 0.04, y radius = 0.04];
\filldraw [color=black!64] ($(a5)+\x*(w0')-\x*(w0)$) circle [x radius = 0.04, y radius = 0.04];
\node [color=black!64] at ($(a6)+(0.16,0)+\x*(w0')-\x*(w0)$) {$\cdots$};
}

\shade [ball color=black] (w0') circle [radius = 1.75pt];

\draw ($(0.64,0.2)+(w0')-(w0)$) arc [start angle = -124, end angle = 221, x radius = 0.84, y radius = 0.64] -- ($(w0')+(0.032,0.016)$) -- cycle;

\draw [dashed, rounded corners] ($(w1)+(-0.16,-0.18)$) 
                             -- ($(w1)+(+0.16,-0.18)$) 
                             -- ($(w4)+(+0.16,+0.16)$)
                             -- ($(w4)+(-0.16,+0.16)$)
                             -- cycle;

\draw [dashed, >=latex, ->, shorten >= 0pt, shorten <= 16pt] ($(w2)+(0.06,0)$) .. controls ($(w2)+(0.21,1.0)$) and ($(w2)+(0.8,1.0)$) .. ($(w2)+(1.0,0)$);

\node [] at ($(w2)+(2.5,-0.1)$) {$\Longrightarrow$};

\node [below left ] at (w0)  {$w$};
\node [right]       at (w0)  {{}\hspace{2.8em}the domain of~$w$};
\node [below left ] at (w0') {$w$};
\node [right]       at (w0') {{}\hspace{2.8em}the domain of~$w$};

\end{tikzpicture}
\caption{Simulating worlds by individuals}
\label{fig:0}
\end{figure}
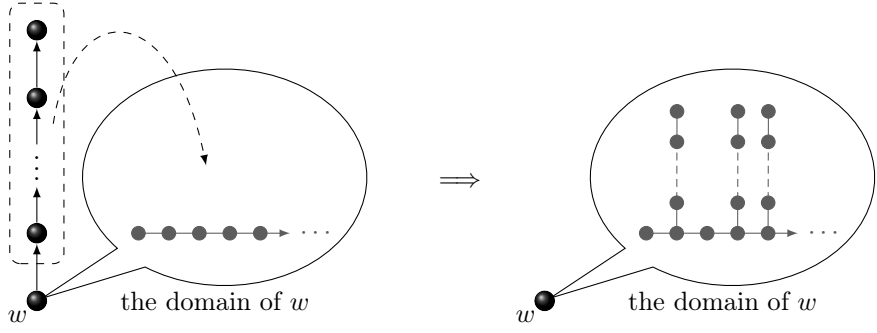

The idea is to move some simulating constructions from the modal predicate language into its classical fragment. For the semantic part of such modeling, this means that we replace some structures formed by possible worlds with structures formed by elements in the domain of a frame. For example, the formula $\Diamond^k P(x)$, being true at a world $w$ of a Kripke model, says that there exists a sequence of $k+1$ worlds starting with~$w$, and its last element satisfies~$P(x)$. We can say something similar using a binary predicate letter~$R$: 
$$
\begin{array}{lcl}
\varphi_k(x) 
  & = 
  & \exists x_1\ldots\exists x_k\,(R(x,x_1)\wedge R(x_1,x_2)\wedge\ldots\wedge R(x_{k-1},x_k)\wedge P(x_k)),
\end{array}
$$
see Figure~\ref{fig:0}. The last formula is modality-free, but it contains a binary predicate letter and $k$ new individual variables. We are to eliminate them. To eliminate all except two variables, we can just alternate the variables. So, the formula
$$
\begin{array}{lcl}
\varphi_3(x) 
  & = 
  & \exists x_1\exists x_2\exists x_3\,(R(x,x_1)\wedge R(x_1,x_2)\wedge R(x_{2},x_3)\wedge P(x_3))
\end{array}
$$
is equivalent in $\logic{QCl}$ to the formula
$$
\begin{array}{lcl}
\varphi'_3(x) 
  & = 
  & \exists y\,(R(x,y)\wedge \exists x\, (R(y,x)\wedge \exists y\,( R(x,y)\wedge P(y))))
\end{array}
$$
containing only two individual variables; similarly for any other $k\in\numNp$. To eliminate the binary predicate letter $R$, following the Kripke construction~\cite{Kripke62}, we can simulate $R(x,y)$ by a formula like $\Diamond(P_1(x)\wedge P_2(y))$ or even $\Diamond(P(x)\wedge \Diamond P(y))$. Notice that in the last case we lose extensions of $\logic{QT}$ and logics of depth two (if the depth of a singleton is one). If $R$ corresponds to a symmetric irreflexive binary relation, then formula $R(x,y)$ can be simulated also by a formula like $\neg\Diamond(P(x)\wedge P(y))$ or a similar formula~\cite{MR:2017:LI}, and we do not lose that logics. Thus, to answer the questions concerning monadic fragments of logics of frames with finite domains (and, as we shall see, of classes of finite Kripke frames, too), it is useful to know properties of the theory of finite symmetric irreflexive binary predicate.

Of course, in the constructions below, there are some additional details that we have to take into account, and this example just illustrates some simple principles that we are going to use throughout the paper.

\subsection{Structure of the paper}

We begin with describing preliminary constructions is Section~\ref{sec:prelim} we shall use throughout the paper. Namely, we shall define a convenient modification of Turing machines and then a suitable domino problem that actually describes computations of such machines. 

Then, in Section~\ref{sec:classical}, we use the domino problem to prove a generalization of the Trakhtenbrot theorem for a number of theories of a binary predicate in the language containing a single binary predicate letter and three variables. Our main aim is to prove that the classical predicate logic and the theory of finite models for symmetric irreflexive binary relation are recursively inseparable in this language. 

Section~\ref{sec:modal} contains corollaries of the results presented in Section~\ref{sec:classical} and additional theorems. So, using the mentioned results, we shall prove that almost all modal predicate logics and logics of their finite Kripke frames (or the ones of frames with finite domains) are recursively inseparable in the language containing a single unary predicate letter and three individual variables; 
%
%then, with some exception, the same result will be obtained 
%
in addition, with certain exceptions, similar conclusions will be drawn
for logics in the language with a single unary predicate letter and two individual variables.

In Section~\ref{sec:int}, we show that the same is true also for a number of superintuitionistic predicate logics and even for their positive fragments. 

We make some remarks about complexity in Section~\ref{sec:complexity} and then conclude in Section~\ref{sec:conclusion}.

\section{Preliminary constructions}
\label{sec:prelim}
\setcounter{equation}{0}

\subsection{Turing machines}

For our purposes, we define a modification of Turing machines~\cite{LP98,Sipser12}.
We consider single-tape deterministic Turing machines. Their special features are a finite set of halting states instead of one halting state and instructions beginning with halting states. These instructions work infinitely when we reach a halting state; they are useful to make the tilings we shall consider infinite.

Thus, a \defnotion{Turing machine} is a tuple $M = \langle \Sigma, Q, q_0, F, \delta \rangle$, where $\Sigma$ is a finite alphabet such that $\Box, \# \in \Sigma$ ($\Box$ is the \defnotion{blank symbol} and $\#$ is the \defnotion{end tape marker symbol}), $Q$ is a finite set of \defnotion{states}; $q_0 \in Q$ is the \defnotion{initial state}; $F$ is a set of \defnotion{halting states}, $F\subseteq Q$; and $\delta$ is a \defnotion{program}.  The program $\delta$ is a function $\delta\colon Q\times\Sigma\to Q\times\Sigma\times\{L,S,R\}$; it satisfies some conditions defined below. If $\delta\colon{q}{s}\mapsto{q'}{s'}{\Delta}$, then 
\begin{itemize}
\item $q' \ne q_0$; 
\item $s=\#$ if, and only if, $s'=\#$; 
\item $\Delta\ne L$ whenever $s=\#$; 
\item $q'=q$, $s'=s$, and $\Delta=S$ whenever $q\in F$.
\end{itemize}

A \defnotion{configuration} of a machine $M = \langle \Sigma, Q, q_0, F, \delta \rangle$ is an $\omega$-word $vqv'$, where $q \in Q$ and $vv'=a_0a_1a_2\ldots{}$ is an $\omega$-word over $\Sigma$ satisfying the following conditions: 
\begin{itemize}
\item there exists $k\in\numNp$ such that $a_i = \Box$, for every $i\geqslant k$;
\item $a_i = \#$ if, and only if, $i=0$.
\end{itemize}

A Turing machine $M$ can be thought of as a computing device equipped with a tape divided into an infinite sequence of cells $c_0,c_1,c_2,\ldots{}$, each containing a symbol from $\Sigma$, with one cell being scanned by a movable head.  Then, a configuration $vqv'$ of $M$ represents a computation instant at which the tape contains the symbols of the word $vv'$, $M$ is in state $q$, and the head is scanning the cell containing the first symbol of~$v'$.  An \defnotion{instruction} $\delta\colon{q}{s}\mapsto{q'}{s'}{\Delta}$ is applicable to this configuration just in case $M$ is in state $q$ and is scanning a cell containing~$s$.  As a result of applying this instruction, $M$ enters state $q'$, replaces $s$ with $s'$ in the cell, and either moves one cell to the left or to the right or stays put, depending on whether $\Delta$ is $L$, $R$ or~$S$, respectively. Given a word $x$ over $\Sigma\setminus\{\Box,\#\}$ as an \defnotion{input}, $M$ consecutively executes the instructions of~$\delta$ starting from the configuration $q_0\# x \Box\Box\Box\ldots{}$; if $M$ reaches a configuration whose state component $q$ is in $F$, then $M$ \defnotion{halts} on $x$, which is denoted by $q!M(x)$; if $M$ does not halt on $x$, we write $\neg !M(x)$. Without a loss of generality, we may assume that the cell being scanned when $M$ halts is $c_0$ (which contains~$\#$). Notice that then instruction $\delta\colon q\#\mapsto {q\#}S$ can be applied providing us with the same configuration. This means that even if $M$ halts, we may consider the infinite computation of $M$, in which $M$ loops the same halting configuration.

We will use Turing machines as a computational model for partial recursive functions on~$\numN$. To do this, we have to encode natural numbers in some finite alphabet. For $m\in\numN$, let 
$$
\begin{array}{lcl}
\numeral{m} & = & \underbrace{||\ldots|}_{\mathclap{\mbox{$m$ times}}},
\end{array}
$$
i.e., $\numeral{0}$ is the empty word, $\numeral{1}={|}$, $\numeral{2}={||}$, $\numeral{3}={|||}$, etc.; the word $\numeral{m}$ is the \defnotion{code} of~$m$.

\subsection{Recursive separability}

Let $X$ and $Y$ be subsets of $\numN$ such that $X\cap Y = \varnothing$. Then $X$ and $Y$ are called \defnotion{recursively separable} if for some recursive subset $Z$ of $\numN$ both $X\subseteq Z$ and $Y\cap Z=\varnothing$ hold; if there is no such $Z$, then $X$ and $Y$ are called \defnotion{recursively inseparable}. 

%Notice that if $X$ and $Y$ are recursively inseparable, then it follows from the definition, that $X\cap Y=\varnothing$.
We shall deal with theories and logics regarding them as special sets of formulas. Usually, theories intersect; therefore, the notion of recursive separability is not applicable directly to them (i.e., to their G\"{o}del numbers). However, it is applicable if we consider theories $T_1$ and $T_2$ such that $T_1\subseteq T_2$ or $T_2\subseteq T_1$. If, say, $T_1\subseteq T_2$, then we may consider $T_1$-validity problem and $T_2$-refutability problem whose intersection is, clearly, empty.  

This observation leads to the following natural expansion of the definition of recursive separability. Let $X$ and $Y$ be subsets of $\numN$ such that $X\subset Y$. Then we call $X$ and $Y$ \defnotion{recursively separable} if $X$ and $\numN\setminus Y$ are recursively separable, i.e., $X\subseteq Z\subseteq Y$ for some recursive~$Z$; otherwise we call $X$ and $Y$ \defnotion{recursively inseparable}.

%Notice that if $X\cap Y = \varnothing$, then $X$ and $Y$ are recursively separable in the second sense if, and only if, $X$ and $Y$ are recursively separable in the first sense; if $X\subseteq Y$, then $X$ and $Y$ are recursively separable in the second sense if, and only if, $X$ and $\numN\setminus Y$ are recursively separable in the first sense. Also, clearly, if $X$ and $Y$ are recursively inseparable (in any of the senses), then both $X$ and $Y$ are not recursive (otherwise, we can take $Z=X$ or $Z=Y$ to recursively separate $X$ and~$Y$). Below, considering recursive separability of subsets of~$\numN$, we will assume disjoint sets, while considering theories, we will assume included sets of formulas.

%Let $T_1$ and $T_2$ be two $\lang{L}$-theories, closed under certain inference rules, where $\lang{L}$ is a language (the set of all words we call $\lang{L}$-formulas). Let also $T_1\subseteq T_2$. We call $T_1$ and $T_2$ \defnotion{recursively separable} if 

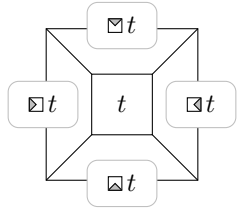
\begin{figure}
\centering
\begin{tikzpicture}[scale=2.1]
\drawtileflattm{(0,0)}{(1,0)}{(1,1)}{(0,1)}{$\leftsq t$}{$\downsq t$}{$\rightsq t$}{$\upsq t$}{$t$}
\end{tikzpicture}
\caption{Tile with marks}
\label{fig:1}
\end{figure}

It is known that there exist non-intersecting recursively enumerable sets $\mathbb{X}$ and $\mathbb{Y}$ which are recursively inseparable~\cite[Theorem~3.3]{Papadimitriou}. Then there is a partial recursive function $f_{\mathbb{XY}}\colon\numN\to\numN$ distinguishing $\mathbb{X}$ and~$\mathbb{Y}$:
$$
\begin{array}{lcl}
f_{\mathbb{XY}}(x) 
  & = 
  & \left\{
      \begin{array}{rl}
        0 & \mbox{if $x\in \mathbb{X}$;} \\
        1 & \mbox{if $x\in \mathbb{Y}$;} \\
        \mbox{undefined} & \mbox{if $x\not\in \mathbb{X}\cup\mathbb{Y}$.} \\
      \end{array}
    \right.
\end{array}
$$
This function is computable by some Turing machine $M_0 = \langle \Sigma_0, Q_0, q_0, F_0, \delta_0 \rangle$. We may assume that $M_0$ has two halting states $q_{\mathbb{X}}$ and $q_{\mathbb{Y}}$ such that, for every $m\in \numN$,
\begin{itemize}
\item
if $m\in\mathbb{X}$, then $q_{\mathbb{X}}!M_0(\numeral{m})$;
\item
if $m\in\mathbb{Y}$, then $q_{\mathbb{Y}}!M_0(\numeral{m})$;
\end{itemize}
also, notice that
\begin{itemize}
\item
if $m\not\in\mathbb{X}\cup\mathbb{Y}$, then $\neg !M_0(\numeral{m})$.
\end{itemize}

For further constructions, we could explicitly specify $\mathbb{X}$, $\mathbb{Y}$, and $M_0$ explicitly; but we shall give general constructions that allow us to vary them. However, for the rest of the text, let $\mathbb{X}$, $\mathbb{Y}$, and $M_0$ be fixed.

\subsection{Tiling problem we shall consider}
\label{subsec:tiling_problems}

%\subsection{Original tiling problem}

For our purposes, let us represent $M_0$ by sets of square tile types~--- one set for every input $\numeral{m}$, where $m\in\numN$. The aim is to replace the computations of $M_0$ with $\numN\times\numN$ tilings.\footnote{For tiling problems consult~\cite{Harel86}.}

We may think of a \defnotion{tile} as a $1 \times 1$ square, with a fixed orientation, whose edges are marked by words over some finite alphabet (we will use $\Sigma_0\cup Q_0$ enriched by some technical symbols). 
A~\defnotion{tile type} $t$ consists of a specification of a \defnotion{mark} (i.e., a word) for each edge; we write $\leftsq t$, $\rightsq t$, $\upsq t$, and $\downsq t$ for the marks of, respectively, the left, the right, the top, and the bottom edges of the tiles of type~$t$, see Figure~\ref{fig:1}.

Let $T = \{t_0, \ldots, t_{n}\}$ be a set of tile types. Informally, a \defnotion{$T$-tiling} is an arrangement of tiles, whose types are in~$T$, on a grid so that the edge marks of adjacent tiles match, both horizontally and vertically. Formally, define $T$-tiling as a function $f\colon \numN \times \numN \to T$ such that for all $i, j \in \numN$,
\begin{itemize}[leftmargin=3em]
\item[$(1)$] $\rightsq f(i,j) = \leftsq f(i+1,j)$;
\item[$(2)$] $\upsq f(i,j) = \downsq f(i,j+1)$;
\end{itemize}
see Figure~\ref{fig:2}.

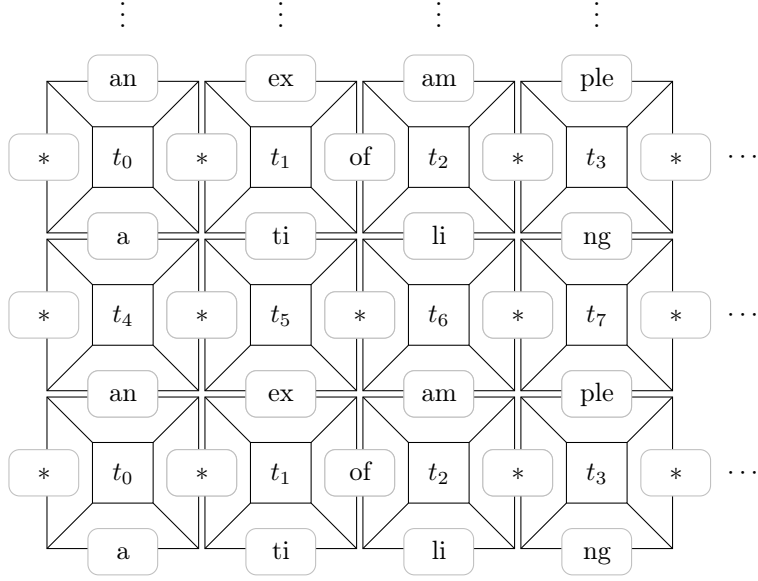
\begin{figure}
\centering
\begin{tikzpicture}[scale=2.1]

\drawtileflattm{(0,0)}{(1,0)}{(1,1)}{(0,1)}{$\ast$}{a} {$\ast$}{an} {$t_0$}
\drawtileflattm{(1,0)}{(2,0)}{(2,1)}{(1,1)}{$\ast$}{ti}{of}    {ex} {$t_1$}
\drawtileflattm{(2,0)}{(3,0)}{(3,1)}{(2,1)}{of}    {li}{$\ast$}{am} {$t_2$}
\drawtileflattm{(3,0)}{(4,0)}{(4,1)}{(3,1)}{$\ast$}{ng}{$\ast$}{ple}{$t_3$}

\drawtileflattm{(0,1)}{(1,1)}{(1,2)}{(0,2)}{$\ast$}{an} {$\ast$}{a} {$t_4$}
\drawtileflattm{(1,1)}{(2,1)}{(2,2)}{(1,2)}{$\ast$}{ex} {$\ast$}{ti}{$t_5$}
\drawtileflattm{(2,1)}{(3,1)}{(3,2)}{(2,2)}{$\ast$}{am} {$\ast$}{li}{$t_6$}
\drawtileflattm{(3,1)}{(4,1)}{(4,2)}{(3,2)}{$\ast$}{ple}{$\ast$}{ng}{$t_7$}

\drawtileflattm{(0,2)}{(1,2)}{(1,3)}{(0,3)}{$\ast$}{a} {$\ast$}{an} {$t_0$}
\drawtileflattm{(1,2)}{(2,2)}{(2,3)}{(1,3)}{$\ast$}{ti}{of}    {ex} {$t_1$}
\drawtileflattm{(2,2)}{(3,2)}{(3,3)}{(2,3)}{of}    {li}{$\ast$}{am} {$t_2$}
\drawtileflattm{(3,2)}{(4,2)}{(4,3)}{(3,3)}{$\ast$}{ng}{$\ast$}{ple}{$t_3$}

\node [right=16pt] at (4,0.5) {$\cdots$};
\node [right=16pt] at (4,1.5) {$\cdots$};
\node [right=16pt] at (4,2.5) {$\cdots$};

\node [above=16pt] at (0.5,3) {$\vdots$};
\node [above=16pt] at (1.5,3) {$\vdots$};
\node [above=16pt] at (2.5,3) {$\vdots$};
\node [above=16pt] at (3.5,3) {$\vdots$};

\end{tikzpicture}
\caption{Example of a tiling}
\label{fig:2}
\end{figure}

Now, we define special tile types for $M_0$ and its possible inputs. We start with $t_0$, which is defined by
\begin{itemize}
\item $\leftsq t_0 = \otimes$, $\rightsq t_0 = \numeral{0}$, $\upsq t_0 = q_0\#$, $\downsq t_0 = \otimes$,
\end{itemize}
where $\otimes$ is a new symbol. Next, for every $s\in\Sigma\setminus\{\#\}$, define $t_s^\ast$ by
\begin{itemize}
\item $\leftsq t_s^\ast = \rightsq t_s^\ast = \ast$, $\upsq t_s^\ast = \downsq t_s^\ast = s$,
\end{itemize}
where $\ast$ is a new symbol, and $t_{\#}^\ast$ by
\begin{itemize}
\item $\leftsq t_{\#}^\ast = \otimes$, $\rightsq t_{\#}^\ast = \ast$, $\upsq t_{\#}^\ast = \downsq t_{\#}^\ast = \#$,
\end{itemize}
see Figure~\ref{fig:3}. 

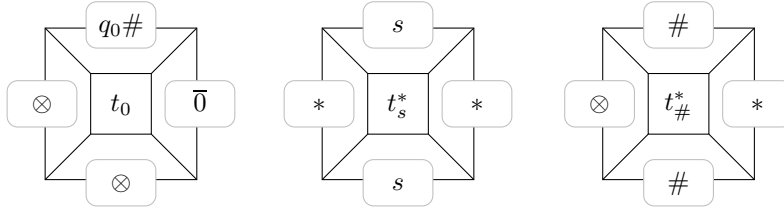
\begin{figure}
\centering
\begin{tikzpicture}[scale=2.1]
\drawtileflattm{(0,0)}{(1,0)}{(1,1)}{(0,1)}{{$\otimes$}}{{$\otimes$}}{{$\numeral{0}$}}{{$q_0\#$}}{$t_0$}
\end{tikzpicture}
\hspace{1em}
\begin{tikzpicture}[scale=2.1]
\drawtileflattm{(0,0)}{(1,0)}{(1,1)}{(0,1)}{{$\ast$}}{{$s$}}{{$\ast$}}{{$s$}}{$t_s^\ast$}
\end{tikzpicture}
\hspace{1em}
\begin{tikzpicture}[scale=2.1]
\drawtileflattm{(0,0)}{(1,0)}{(1,1)}{(0,1)}{{$\otimes$}}{{$\#$}}{{$\ast$}}{{$\#$}}{$t_{\#}^\ast$}
\end{tikzpicture}
\caption{Tile types $t_0$, $t_s^\ast$, and $t_{\#}^\ast$}
\label{fig:3}
\end{figure}

For every $q\in Q_0$ and $s\in\Sigma_0$, we define the tile types according to instruction $\delta\colon qs\mapsto q's'\Delta$. There are three cases. 

First: $\delta_0\colon qs\mapsto q's'S$, i.e., $\Delta=S$. Then define $t_{qs}$ by
\begin{itemize}
\item $\leftsq t_{qs} = \rightsq t_{qs} = \ast$, $\upsq t_{qs} = q's'$, $\downsq t_{qs} = qs$, where $s\ne\#$;
\item $\leftsq t_{q\#} = \otimes$, $\rightsq t_{q\#} = \ast$, $\upsq t_{q\#} = q'\#$, $\downsq t_{q\#} = q\#$,
\end{itemize}
see Figure~\ref{fig:4}. 

\begin{figure}
\centering
\begin{tikzpicture}[scale=2.1]
\drawtileflattm{(0,0)}{(1,0)}{(1,1)}{(0,1)}{{$\ast$}}{{$qs$}}{{$\ast$}}{{$q's'$}}{$t_{qs}$}
\end{tikzpicture}
\hspace{1em}
\begin{tikzpicture}[scale=2.1]
\drawtileflattm{(0,0)}{(1,0)}{(1,1)}{(0,1)}{{$\otimes$}}{{$q\#$}}{{$\ast$}}{{$q'\#$}}{$t_{q\#}$}
\end{tikzpicture}
\caption{Tile types for instructions $\delta_0\colon qs\mapsto q's'S$ and $\delta_0\colon q\#\mapsto q'\#S$}
\label{fig:4}
\end{figure}
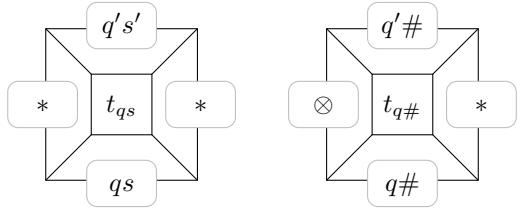

Second: $\delta_0\colon qs\mapsto q's'R$, i.e., $\Delta=R$. Then define $t_{qs}$ and $t_{qs}^{a}$, for every $a\in\Sigma_0\setminus\{\#\}$, by
\begin{itemize}
\item 
$\leftsq t_{qs} = \ast$,
$\rightsq t_{qs} = qs$, 
$\upsq t_{qs} = s'$, 
$\downsq t_{qs} = qs$, where $s\ne\#$;
\item 
$\leftsq t_{q\#} = \otimes$,
$\rightsq t_{q\#} = q\#$, 
$\upsq t_{q\#} = \#$, 
$\downsq t_{q\#} = q\#$;
\item 
$\leftsq t_{qs}^{a} = qs$,
$\rightsq t_{qs}^{a} = \ast$, 
$\upsq t_{qs}^{a} = q'a$, 
$\downsq t_{qs}^{a} = a$,
\end{itemize}
see Figure~\ref{fig:5}. 

\begin{figure}
\centering
\begin{tikzpicture}[scale=2.1]
\drawtileflattm{(0,0)}{(1,0)}{(1,1)}{(0,1)}{{$\ast$}}{{$qs$}}{{$qs$}}{{$s'$}}{$t_{qs}$}
\drawtileflattm{(1,0)}{(2,0)}{(2,1)}{(1,1)}{{$qs$}}{{$a$}}{{$\ast$}}{{$q'a$}}{$t_{qs}^{a}$}
\end{tikzpicture}
\hspace{1em}
\begin{tikzpicture}[scale=2.1]
\drawtileflattm{(0,0)}{(1,0)}{(1,1)}{(0,1)}{{$\otimes$}}{{$q\#$}}{{$q\#$}}{{$\#$}}{$t_{q\#}$}
\drawtileflattm{(1,0)}{(2,0)}{(2,1)}{(1,1)}{{$q\#$}}{{$a$}}{{$\ast$}}{{$q'a$}}{$t_{q\#}^{a}$}
\end{tikzpicture}
\caption{Tile types for instructions $\delta_0\colon qs\mapsto q's'R$ and $\delta_0\colon q\#\mapsto q'\#R$}
\label{fig:5}
\end{figure}
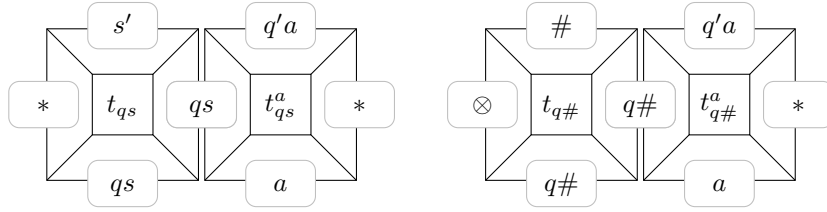

Third: $\delta_0\colon qs\mapsto q's'L$, i.e., $\Delta=L$. Notice that $s\ne\#$. Then define $t_{qs}$ and $t_{qs}^{a}$, for every $a\in\Sigma_0$, by
\begin{itemize}
\item 
$\leftsq t_{qs} = qs$,
$\rightsq t_{qs} = \ast$, 
$\upsq t_{qs} = s'$, 
$\downsq t_{qs} = qs$;
\item 
$\leftsq t_{qs}^{a} = \ast$,
$\rightsq t_{qs}^{a} = qs$, 
$\upsq t_{qs}^{a} = q'a$, 
$\downsq t_{qs}^{a} = a$, where $a\ne\#$;
\item 
$\leftsq t_{qs}^{\#} = \otimes$,
$\rightsq t_{qs}^{\#} = qs$, 
$\upsq t_{qs}^{\#} = q'\#$, 
$\downsq t_{qs}^{\#} = \#$,
\end{itemize}
see Figure~\ref{fig:6}. 

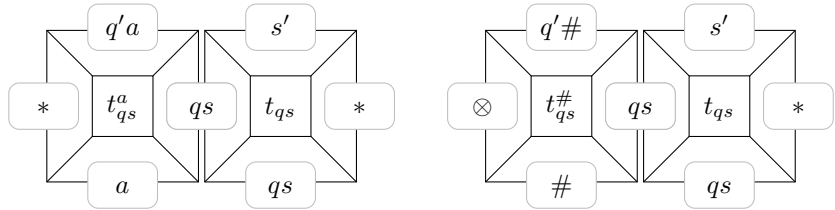
\begin{figure}
\centering
\begin{tikzpicture}[scale=2.1]
\drawtileflattm{(0,0)}{(1,0)}{(1,1)}{(0,1)}{{$\ast$}}{{$a$}}{{$qs$}}{{$q'a$}}{$t_{qs}^{a}$}
\drawtileflattm{(1,0)}{(2,0)}{(2,1)}{(1,1)}{{$qs$}}{{$qs$}}{{$\ast$}}{{$s'$}}{$t_{qs}$}
\end{tikzpicture}
\hspace{1em}
\begin{tikzpicture}[scale=2.1]
\drawtileflattm{(0,0)}{(1,0)}{(1,1)}{(0,1)}{{$\otimes$}}{{$\#$}}{{$qs$}}{{$q'\#$}}{$t_{qs}^{\#}$}
\drawtileflattm{(1,0)}{(2,0)}{(2,1)}{(1,1)}{{$qs$}}{{$qs$}}{{$\ast$}}{{$s'$}}{$t_{qs}$}
\end{tikzpicture}
\caption{Tile types for instruction $\delta_0\colon qs\mapsto q's'L$}
\label{fig:6}
\end{figure}

Let ${T}_{M_0}$ be the set of all tile types defined above; notice that ${T}_{M_0}$ is finite. The set ${T}_{M_0}$ contains a description of $M_0$, but does not contain any tile types that allow us to simulate inputs for~$M_0$. To fix this, define the tile types $t_\Box^{\ast\ast}$, $t_k^\otimes$, and $t_k^{\ast\ast}$, where $k\in\numN$, by
\begin{itemize}
\item 
$\leftsq  t_\Box^{\ast\ast} = \rightsq t_\Box^{\ast\ast} = {\ast}{\ast}$, 
$\upsq    t_\Box^{\ast\ast} = \Box$, 
$\downsq  t_\Box^{\ast\ast} = \otimes$;
\item 
$\leftsq  t_k^\otimes = \numeral{k}$,
$\rightsq t_k^\otimes = \numeral{k+1}$, 
$\upsq    t_k^\otimes = |$, 
$\downsq  t_k^\otimes = \otimes$;
\item 
$\leftsq  t_k^{\ast\ast} = \numeral{k}$,
$\rightsq t_k^{\ast\ast} = {\ast}{\ast}$, 
$\upsq    t_k^{\ast\ast} = \Box$, 
$\downsq  t_k^{\ast\ast} = \otimes$,
\end{itemize}
see Figure~\ref{fig:7}. 

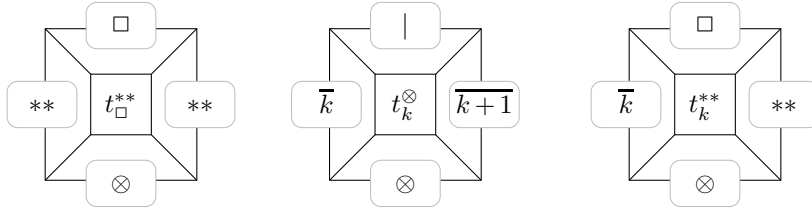
\begin{figure}
\centering
\begin{tikzpicture}[scale=2.1]
\drawtileflattm{(0,0)}{(1,0)}{(1,1)}{(0,1)}{${\ast}{\ast}$}{$\otimes$}{${\ast}{\ast}$}{$\Box$}{$t_\Box^{\ast\ast}$}
\end{tikzpicture}
\hspace{1em}
\begin{tikzpicture}[scale=2.1]
\drawtileflattm{(0,0)}{(1,0)}{(1,1)}{(0,1)}{$\numeral{k}$}{$\otimes$}{$\numeral{k\,{+}\,1}$}{$|$}{$t_k^\otimes$}
\end{tikzpicture}
\hspace{1em}
\begin{tikzpicture}[scale=2.1]
\drawtileflattm{(0,0)}{(1,0)}{(1,1)}{(0,1)}{$\numeral{k}$}{$\otimes$}{${\ast}{\ast}$}{$\Box$}{$t_k^{\ast\ast}$}
\end{tikzpicture}
\caption{Tile types $t_\Box^{\ast\ast}$, $t_k^\otimes$, and $t_k^{\ast\ast}$}
\label{fig:7}
\end{figure}

Now, we can simulate any initial configuration of~$M_0$. Indeed, to simulate a configuration $C = q_0\#\numeral{k}\Box\Box\Box\ldots{}$, take the row of tiles whose types are
$$
t_0^{\phantom{i}}, t_0^\otimes, \ldots, t_{k-1}^\otimes, t_k^{\ast\ast}, 
t_\Box^{\ast\ast}, t_\Box^{\ast\ast}, t_\Box^{\ast\ast}, \ldots{},
$$
and then
$$
\begin{array}{lcl}
C & = & 
\upsq t_0^{\phantom{i}}\ \upsq t_0^\otimes\ \ldots\ \upsq t_{k-1}^\otimes\ \upsq t_k^{\ast\ast}\ 
\upsq t_\Box^{\ast\ast}\ \upsq t_\Box^{\ast\ast}\ \upsq t_\Box^{\ast\ast}\ \ldots{},
\end{array}
$$
see Figure~\ref{fig:8}. 

\begin{figure}
\centering
\begin{tikzpicture}[scale=2.1]
\drawtileflattm{(0,0)}{(1,0)}{(1,1)}{(0,1)}{${\otimes}$}{${\otimes}$}{$\numeral{0}$}{$q_0\#$}{$t_0$}
\drawtileflattm{(1,0)}{(2,0)}{(2,1)}{(1,1)}{${\numeral{0}}$}{${\otimes}$}{$\numeral{1}$}{$|$}{$t_0^\otimes$}
\drawtileflattm{(3-.2,0)}{(4-.2,0)}{(4-.2,1)}{(3-.2,1)}{${\numeral{k\,{-}\,1}}$}{${\otimes}$}{$\numeral{k}$}{$|$}{$t_{k-1}^\otimes$}
\drawtileflattm{(4-.2,0)}{(5-.2,0)}{(5-.2,1)}{(4-.2,1)}{${\numeral{k}}$}{${\otimes}$}{${\ast}{\ast}$}{$\Box$}{$t_{k}^{\ast\ast}$}
\drawtileflattm{(5-.2,0)}{(6-.2,0)}{(6-.2,1)}{(5-.2,1)}{${\ast}{\ast}$}{${\otimes}$}{${\ast}{\ast}$}{$\Box$}{$t_{\Box}^{\ast\ast}$}
\node [] at (2.42,0.5) {$\cdots$};
\node [] at (6.2,0.5) {$\cdots$};
\end{tikzpicture}
\caption{Simulation of the initial configuration with input $\numeral{k}$}
\label{fig:8}
\end{figure}
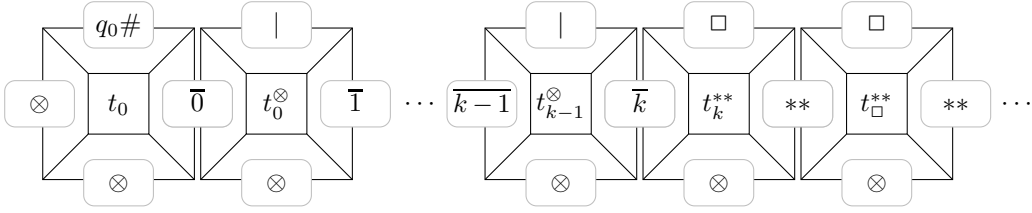

We shall use the new tile types together with the tile types of ${T}_{M_0}$ to simulate the computations of $M_0$ with the inputs we choose. To this end, let
$$
\begin{array}{lcl}
T_n & = & {T}_{M_0}\cup\{t_{k}^\otimes : k<n\}\cup\{t_{n}^{\ast\ast}\}\cup\{t_{\Box}^{\ast\ast}\}.
\end{array}
$$

\begin{proposition}
\label{prop:fn}
There exists a unique $T_n$-tiling $f\colon\numN\times\numN\to T_n$ such that $f(0,0)=t_0$. Moreover, if $C_0,C_1,C_2,\ldots{}$ is the computation of $M_0$ on the input $\numeral{n}$, then
$$
\begin{array}{lcl}
C_m & = & \upsq f(0,m)\ \upsq f(1,m)\ \upsq f(2,m)\ \ldots,
\end{array}
$$
for every $m\in\numN$.
\end{proposition}

\begin{proof}
Induction on~$m$.
\end{proof}

Due to Proposition~\ref{prop:fn}, a $T_n$-tiling $f_n\colon\numN\times\numN\to T_n$ such that $f_n(0,0)=t_0$ is defined uniquely; we call it the \defnotion{special $T_n$-tiling}.
Let, for convenience,
$$
\begin{array}{lcl}
t_{1} = t_{q_{\mathbb{X}}\#} & \mbox{and} & t_{2} = t_{q_{\mathbb{Y}}\#}.
\end{array}
$$
From Proposition~\ref{prop:fn} we obtain that, for every $n\in\numN$,
\begin{equation}
\label{eq:fn}
\begin{array}{lcl}
n\in\mathbb{X} 
  & \iff 
  & \mbox{there exists $m\in\numN$ such that $f_n(0,m) = t_{1}$;} 
  \\
n\in\mathbb{Y} 
  & \iff 
  & \mbox{there exists $m\in\numN$ such that $f_n(0,m) = t_{2}$.} 
  \\
\end{array}
\end{equation}

In the following, it does not matter to us how exactly the tile types of $T_n$ are constructed; we just enumerate them so that we could refer to~$(\ref{eq:fn})$. So, let 
$$
\begin{array}{lcl}
T_n & = & \{t^n_0,\ldots,t^n_{k_n}\} 
\end{array}
$$
and assume that
$$
\begin{array}{lcl}
t^n_0 = t_0^{\phantom{n}}, 
  & t^n_1 = t_1^{\phantom{n}}, 
  & t^n_2 = t_2^{\phantom{n}}; 
\end{array}
$$
the enumeration of all other tile types of $T_n$ is arbitrary. We shall use $T_n$ to describe the properties $(\ref{eq:fn})$ of the special $T_n$-tiling by formulas of different languages. This will allow us to prove that the validity and the refutability problems for certain theories and logics are recursively inseparable, and hence the theories and the logics are undecidable or even not recursively enumerable.

\section{Superclassical logics and theories}
\label{sec:classical}
\setcounter{equation}{0}

\subsection{Syntax and semantics}

We assume that the first-order language $\lang{L}$ contains countably many individual variables, countably many predicate letters of every arity, the constant $\bot$, the binary connectives $\wedge$, $\vee$, $\to$, the quantifier symbols~$\forall$ and~$\exists$. Formulas in $\lang{L}$, or \defnotion{$\lang{L}$-formulas}, as well as the symbols $\neg$ and $\leftrightarrow$, are defined in the usual way; in particular, $\neg\varphi = \varphi\to \bot$, $\top=\neg\bot$, and $\varphi \leftrightarrow \psi = (\varphi\to \psi)\wedge (\psi\to\varphi)$. From now on, we identify the language $\lang{L}$ with the set of $\lang{L}$-formulas. For a formula $\varphi$, let $\sub\varphi$ denote the set of subformulas of~$\varphi$.

A \defnotion{model} is a tuple $\cModel{M} = \langle \mathcal{D},\mathcal{I}\rangle$, where $D$ is a non-empty set of \defnotion{individuals}, or the \defnotion{domain} of $\cModel{M}$, and $\mathcal{I}$ an \defnotion{interpretation of predicate letters} in $\mathcal{D}$, i.e., a function assigning to an \mbox{$n$-ary} predicate letter $P$ an \mbox{$n$-ary} relation $\mathcal{I}(P)$ on~$\mathcal{D}$.
%We write $P^{I}$ rather than~$I(P)$.

An \defnotion{assignment} in a model $\cModel{M} = \langle \mathcal{D},\mathcal{I}\rangle$ is a function $g$ associating with every individual variable $x$ an individual $g(x)\in \mathcal{D}$. As before, $g\stackrel{x}{=}h$ means that the assignment $g$ differs from the assignment $h$ in at most the value of~$x$.

The truth of a formula $\varphi$ in a model $\cModel{M} = \langle \mathcal{D},\mathcal{I}\rangle$ under an assignment $g$ is defined recursively:
%%%%%%%%%%%%%%%%%%%%%%%%%%%%%%%%%%%%%%%%%%%%%%%%%%%%%%%%%%%%%%%%%%%%%%
\settowidth{\templength}{\mbox{$\cModel{M}\models^g\varphi'$ and $\cModel{M}\models^g\varphi''$;}}
\settowidth{\templengtha}{\mbox{$w$}}
\settowidth{\templengthb}{\mbox{$\cModel{M}\models^{h}\varphi'$, for every assignment $h$ such that $h \stackrel{x}{=} g$}}
\settowidth{\templengthc}{\mbox{$\cModel{M}\models^g P(x_1,\ldots,x_n)$}}
%%%%%%%%%%%%%%%%%%%%%%%%%%%%%%%%%%%%%%%%%%%%%%%%%%%%%%%%%%%%%%%%%%%%%%
$$
\begin{array}{lcl}
\cModel{M}\models^g P(x_1,\ldots,x_n)
  & \leftrightharpoons
  & \parbox{\templengthb}{$\langle g(x_1),\ldots,g(x_n)\rangle \in \mathcal{I}(P)$,} \\
\end{array}
$$
\mbox{where $P$ is an $n$-ary predicate letter;}
%%%%%%%%%%%%%%%%%%%%%%%%%%%%%%%%%%%%%%%%%%%%%%%%%%%%%%%%%%%%%%%%%%%%%%
\settowidth{\templength}{\mbox{$\cModel{M}\models^g\varphi'$ and $\cModel{M}\models^g\varphi''$;}}
\settowidth{\templengtha}{\mbox{$w$}}
\settowidth{\templengthb}{\mbox{$\cModel{M}\models^{g}\varphi'\to\varphi''$}}
\settowidth{\templengthc}{\mbox{$\cModel{M}\models^g P(x_1,\ldots,x_n)$}}
%%%%%%%%%%%%%%%%%%%%%%%%%%%%%%%%%%%%%%%%%%%%%%%%%%%%%%%%%%%%%%%%%%%%%%
$$
\begin{array}{lcl}
\parbox{\templengthc}{{}\hfill\parbox{\templengthb}{$\cModel{M} \not\models^g \bot;$}}
  \\
\parbox{\templengthc}{{}\hfill\parbox{\templengthb}{$\cModel{M}\models^g\varphi' \wedge \varphi''$}}
  & \leftrightharpoons
  & \parbox[t]{\templength}{$\cModel{M}\models^g\varphi'$ and $\cModel{M}\models^g\varphi''$;}
  \\
\parbox{\templengthc}{{}\hfill\parbox{\templengthb}{$\cModel{M}\models^g\varphi' \vee \varphi''$}}
  & \leftrightharpoons
  & \parbox[t]{\templength}{$\cModel{M}\models^g\varphi'$\hfill or\hfill $\cModel{M}\models^g\varphi''$;}
  \\
\parbox{\templengthc}{{}\hfill\parbox{\templengthb}{$\cModel{M}\models^g\varphi' \to \varphi''$}}
  & \leftrightharpoons
  & \parbox[t]{\templength}{$\cModel{M}\not\models^g\varphi'$\hfill or\hfill $\cModel{M}\models^g\varphi''$;}
  \\
\parbox{\templengthc}{{}\hfill\parbox{\templengthb}{$\cModel{M}\models^g\forall x\,\varphi'$}}
  & \leftrightharpoons
  & \mbox{$\cModel{M}\models^{h}\varphi'$, for every assignment $h$ such that $h \stackrel{x}{=} g$;}
  \\
\parbox{\templengthc}{{}\hfill\parbox{\templengthb}{$\cModel{M}\models^g\exists x\,\varphi'$}}
  & \leftrightharpoons
  & \mbox{$\cModel{M}\models^{h}\varphi'$, for some assignment $h$ such that $h \stackrel{x}{=} g$.}
\end{array}
$$

For a formula $\varphi(x_1,\ldots,x_n)$ with its free individual variables in the list $x_1, \ldots, x_n$ and individuals $a_1, \ldots, a_n$ of $\cModel{M}$, we write $\cModel{M} \models \varphi (a_1, \ldots, a_n)$ if $\cModel{M} \models^g \varphi (x_1, \ldots, x_n)$, for an assignment~$g$ such that $g(x_k) = a_k$, for every $k\in\{1,\ldots,n\}$.  This notation is unambiguous since the languages we consider lack constants and the truth value of $\varphi(x_1, \ldots, x_n)$ does not depend on the values of variables different from $x_1, \ldots, x_n$.

For a model $\cModel{M}$, a class $\scls{C}$ of models, a formula $\varphi$, and a set of formulas~$X$, define
%%%%%%%%%%%%%%%%%%%%%%%%%%%%%%%%%%%%%%%%%%%%%%%%%%%%%%%%%%%%%%%%%%%%%%
\settowidth{\templengtha}{\mbox{$\cModel{M}\models^{h}\varphi'$, for every $h$ such that $h \stackrel{x}{=} g$}}
\settowidth{\templengthd}{\mbox{$\cModel{M}$}}
\settowidth{\templengthb}{\mbox{$\cModel{M}\models X$}}
%%%%%%%%%%%%%%%%%%%%%%%%%%%%%%%%%%%%%%%%%%%%%%%%%%%%%%%%%%%%%%%%%%%%%%
$$
\begin{array}{lcl}
\parbox{\templengthc}{{}\hfill\parbox{\templengthb}{$\cModel{M}\models \varphi$}}
  & \leftrightharpoons
  & \parbox[t]{\templengtha}{$\cModel{M}\models^g\varphi$, for every assignment~$g$;}
  \\
\parbox{\templengthc}{{}\hfill\parbox{\templengthb}{$\cModel{M}\models X$}}
  & \leftrightharpoons
  & \parbox[t]{\templengtha}{$\cModel{M}\models\varphi$, for every $\varphi\in X$;}
  \\
\parbox{\templengthc}{{}\hfill\parbox{\templengthb}{$\parbox{\templengthd}{\hfill$\scls{C}$}\models \varphi$}}
  & \leftrightharpoons
  & \parbox[t]{\templengtha}{$\cModel{M}\models\varphi$, for every $\cModel{M}\in \scls{C}$.}
  \\
\end{array}
$$

If $\cModel{M}\models \varphi$, we say that $\varphi$ is \defnotion{true} in $\cModel{M}$; otherwise $\varphi$ is \defnotion{refuted} in $\cModel{M}$.

We say that a model $\cModel{M} = \langle \mathcal{D},\mathcal{I}\rangle$ is \defnotion{finite} if $\mathcal{D}$ is finite.

For a class $\scls{C}$ of models, define $\mathit{Th}(\scls{C}) = \{\varphi : \scls{C}\models\varphi\}$. We call $\mathit{Th}(\scls{C})$ the \defnotion{theory} of the class~$\scls{C}$.
For a theory $\Gamma$, let $\Gamma_{\mathit{fin}}$ denote the theory of all finite models of~$\Gamma$. 
A model $\cModel{M}$ is a \defnotion{$\Gamma$-model} if $\cModel{M}\models\Gamma$.
For a formula $\varphi$, let 
$$
\begin{array}{lcl}
\Gamma\models\varphi & \bydef & \mbox{$\cModel{M}\models\varphi$, for every $\Gamma$-model $\cModel{M}$;}
\smallskip\\
\Gamma\uplus\varphi & = & \{\psi\in\lang{L} : \Gamma\cup\{\varphi\}\models\psi\}.
\end{array}
$$
Below we do not differ $\Gamma$ and $\Gamma\uplus\top$ (i.e., $\{\psi\in\lang{L} : \Gamma\models\psi\}$).

A~closed $\lang{L}$-formula $\varphi$ is called 
\begin{itemize}
\item 
\defnotion{$\Gamma$-valid} if $\Gamma\models\varphi$; 
\item
\defnotion{$\Gamma$-satisfiable}, or \defnotion{$\Gamma$-consistent}, if $\Gamma\not\models\neg\varphi$;
\item
\defnotion{$\Gamma$-refutable} if $\Gamma\not\models\varphi$.
\end{itemize}

Define logics $\logic{QCl}$ and $\logic{QCl}_{\mathit{fin}}$ as, respectively, the theory of the class of all models and the theory of the class of all finite models. 
%Both $\logic{QCl}$ and $\logic{QCl}_{\mathit{fin}}$ are logics, in the sense, that they are closed under predicate substitution.

\subsection{Trakhtenbrot theorem for three variables}

The Trakhtenbrot theorem~\cite{Trakhtenbrot50,Trakhtenbrot53} states that $\logic{QCl}$ and $\logic{QCl}_{\mathit{fin}}$ are recursively inseparable (or equivalently, $\logic{QCl}$-validity and $\logic{QCl}_{\mathit{fin}}$-refutability problems are). 
We give a very short proof of the theorem and then obtain corollaries that relate theories of a binary predicate. The proof is based on the simulation of the special $T_n$-tiling, for every $n\in\numN$, by $\lang{L}$-formulas. To this end, we will use a binary predicate letter~$P$ to describe some connections between the elements of the models (the elements are viewed as tile holders) and unary letters $P_0,P_1,P_2,\ldots{}$ to say, tiles of which types are held by the elements.

\begin{figure}
\centering
\begin{tikzpicture}[scale=2.1, color=black!32]
\coordinate (g00) at (0,0);
\coordinate (g10) at (1,0);
\coordinate (g20) at (2,0);
\coordinate (g30) at (3,0);
\coordinate (g40) at (4,0);
\coordinate (g50) at (5,0);
\coordinate (g01) at (0,1);
\coordinate (g11) at (1,1);
\coordinate (g21) at (2,1);
\coordinate (g31) at (3,1);
\coordinate (g41) at (4,1);
\coordinate (g51) at (5,1);
\coordinate (g02) at (0,2);
\coordinate (g12) at (1,2);
\coordinate (g22) at (2,2);
\coordinate (g32) at (3,2);
\coordinate (g42) at (4,2);
\coordinate (g52) at (5,2);
\coordinate (g03) at (0,3);
\coordinate (g13) at (1,3);
\coordinate (g23) at (2,3);
\coordinate (g33) at (3,3);
\coordinate (g43) at (4,3);
\coordinate (g53) at (5,3);
\coordinate (g04) at (0,4);
\coordinate (g14) at (1,4);
\coordinate (g24) at (2,4);
\coordinate (g34) at (3,4);
\coordinate (g44) at (4,4);
\coordinate (g54) at (5,4);
\coordinate (g05) at (0,5);
\coordinate (g15) at (1,5);
\coordinate (g25) at (2,5);
\coordinate (g35) at (3,5);
\coordinate (g45) at (4,5);
\coordinate (g55) at (5,5);

\coordinate (c1) at ($(g00)-(0.5,0.5)$);
\coordinate (c2) at ($(c1)+(1,0)$);
\coordinate (c3) at ($(c2)+(0,1)$);
\coordinate (c4) at ($(c3)-(1,0)$);
\drawtileflattm{(c1)}{(c2)}{(c3)}{(c4)}{$\otimes$}{$\otimes$}{$\numeral{0}$}{$q_0\#$}{}
\coordinate (c1) at ($(g20)-(0.5,0.5)$);
\coordinate (c2) at ($(c1)+(1,0)$);
\coordinate (c3) at ($(c2)+(0,1)$);
\coordinate (c4) at ($(c3)-(1,0)$);
\drawtileflattm{(c1)}{(c2)}{(c3)}{(c4)}{$\numeral{n}$}{$\otimes$}{${\ast}{\ast}$}{$\Box$}{}
\coordinate (c1) at ($(g40)-(0.5,0.5)$);
\coordinate (c2) at ($(c1)+(1,0)$);
\coordinate (c3) at ($(c2)+(0,1)$);
\coordinate (c4) at ($(c3)-(1,0)$);
\drawtileflattm{(c1)}{(c2)}{(c3)}{(c4)}{${\ast}{\ast}$}{$\otimes$}{${\ast}{\ast}$}{$\Box$}{}
\coordinate (c1) at ($(g50)-(0.5,0.5)$);
\coordinate (c2) at ($(c1)+(1,0)$);
\coordinate (c3) at ($(c2)+(0,1)$);
\coordinate (c4) at ($(c3)-(1,0)$);
\drawtileflattm{(c1)}{(c2)}{(c3)}{(c4)}{${\ast}{\ast}$}{$\otimes$}{${\ast}{\ast}$}{$\Box$}{}

\coordinate (c1) at ($(g01)-(0.5,0.5)$);
\coordinate (c2) at ($(c1)+(1,0)$);
\coordinate (c3) at ($(c2)+(0,1)$);
\coordinate (c4) at ($(c3)-(1,0)$);
\drawtileflattm{(c1)}{(c2)}{(c3)}{(c4)}{$\otimes$}{$q_0\#$}{$q_0\#$}{$\#$}{}
\coordinate (c1) at ($(g21)-(0.5,0.5)$);
\coordinate (c2) at ($(c1)+(1,0)$);
\coordinate (c3) at ($(c2)+(0,1)$);
\coordinate (c4) at ($(c3)-(1,0)$);
\drawtileflattm{(c1)}{(c2)}{(c3)}{(c4)}{$\ast$}{$\Box$}{${\ast}$}{$\Box$}{}
\coordinate (c1) at ($(g41)-(0.5,0.5)$);
\coordinate (c2) at ($(c1)+(1,0)$);
\coordinate (c3) at ($(c2)+(0,1)$);
\coordinate (c4) at ($(c3)-(1,0)$);
\drawtileflattm{(c1)}{(c2)}{(c3)}{(c4)}{${\ast}$}{$\Box$}{${\ast}$}{$\Box$}{}
\coordinate (c1) at ($(g51)-(0.5,0.5)$);
\coordinate (c2) at ($(c1)+(1,0)$);
\coordinate (c3) at ($(c2)+(0,1)$);
\coordinate (c4) at ($(c3)-(1,0)$);
\drawtileflattm{(c1)}{(c2)}{(c3)}{(c4)}{${\ast}$}{$\Box$}{${\ast}$}{$\Box$}{}

\coordinate (c1) at ($(g03)-(0.5,0.5)$);
\coordinate (c2) at ($(c1)+(1,0)$);
\coordinate (c3) at ($(c2)+(0,1)$);
\coordinate (c4) at ($(c3)-(1,0)$);
\drawtileflattm{(c1)}{(c2)}{(c3)}{(c4)}{$\otimes$}{$q_{\mathbb{Y}}\#$}{$\ast$}{$q_{\mathbb{Y}}\#$}{}
\coordinate (c1) at ($(g23)-(0.5,0.5)$);
\coordinate (c2) at ($(c1)+(1,0)$);
\coordinate (c3) at ($(c2)+(0,1)$);
\coordinate (c4) at ($(c3)-(1,0)$);
\drawtileflattm{(c1)}{(c2)}{(c3)}{(c4)}{$\ast$}{$s$}{${\ast}$}{$s$}{}
\coordinate (c1) at ($(g43)-(0.5,0.5)$);
\coordinate (c2) at ($(c1)+(1,0)$);
\coordinate (c3) at ($(c2)+(0,1)$);
\coordinate (c4) at ($(c3)-(1,0)$);
\drawtileflattm{(c1)}{(c2)}{(c3)}{(c4)}{${\ast}$}{$\Box$}{${\ast}$}{$\Box$}{}
\coordinate (c1) at ($(g53)-(0.5,0.5)$);
\coordinate (c2) at ($(c1)+(1,0)$);
\coordinate (c3) at ($(c2)+(0,1)$);
\coordinate (c4) at ($(c3)-(1,0)$);
\drawtileflattm{(c1)}{(c2)}{(c3)}{(c4)}{${\ast}$}{$\Box$}{${\ast}$}{$\Box$}{}

\coordinate (c1) at ($(g04)-(0.5,0.5)$);
\coordinate (c2) at ($(c1)+(1,0)$);
\coordinate (c3) at ($(c2)+(0,1)$);
\coordinate (c4) at ($(c3)-(1,0)$);
\drawtileflattm{(c1)}{(c2)}{(c3)}{(c4)}{$\otimes$}{$q_{\mathbb{Y}}\#$}{$\ast$}{$q_{\mathbb{Y}}\#$}{}
\coordinate (c1) at ($(g24)-(0.5,0.5)$);
\coordinate (c2) at ($(c1)+(1,0)$);
\coordinate (c3) at ($(c2)+(0,1)$);
\coordinate (c4) at ($(c3)-(1,0)$);
\drawtileflattm{(c1)}{(c2)}{(c3)}{(c4)}{$\ast$}{$s$}{${\ast}$}{$s$}{}
\coordinate (c1) at ($(g44)-(0.5,0.5)$);
\coordinate (c2) at ($(c1)+(1,0)$);
\coordinate (c3) at ($(c2)+(0,1)$);
\coordinate (c4) at ($(c3)-(1,0)$);
\drawtileflattm{(c1)}{(c2)}{(c3)}{(c4)}{${\ast}$}{$\Box$}{${\ast}$}{$\Box$}{}
\coordinate (c1) at ($(g54)-(0.5,0.5)$);
\coordinate (c2) at ($(c1)+(1,0)$);
\coordinate (c3) at ($(c2)+(0,1)$);
\coordinate (c4) at ($(c3)-(1,0)$);
\drawtileflattm{(c1)}{(c2)}{(c3)}{(c4)}{${\ast}$}{$\Box$}{${\ast}$}{$\Box$}{}

\begin{scope}[color=black, thick]
\filldraw [] (g00) circle [radius=1.75pt];
\filldraw [] (g20) circle [radius=1.75pt];
\filldraw [] (g40) circle [radius=1.75pt];
\filldraw     [] (g50) circle [radius=1.75pt];
\filldraw [] (g01) circle [radius=1.75pt];
\filldraw [] (g21) circle [radius=1.75pt];
\filldraw [] (g41) circle [radius=1.75pt];
\filldraw     [] (g51) circle [radius=1.75pt];
\filldraw [] (g03) circle [radius=1.75pt];
\filldraw [] (g23) circle [radius=1.75pt];
\filldraw [] (g43) circle [radius=1.75pt];
\filldraw     [] (g53) circle [radius=1.75pt];
\filldraw     [] (g04) circle [radius=1.75pt];
\filldraw     [] (g24) circle [radius=1.75pt];
\filldraw     [] (g44) circle [radius=1.75pt];
\filldraw     [] (g54) circle [radius=1.75pt];

\node [] at ($(g10)+(0.02,0)$) {$\cdots$};
\node [] at ($(g30)+(0.02,0)$) {$\cdots$};
\node [] at ($(g11)+(0.02,0)$) {$\cdots$};
\node [] at ($(g31)+(0.02,0)$) {$\cdots$};
\node [] at ($(g13)+(0.02,0)$) {$\cdots$};
\node [] at ($(g33)+(0.02,0)$) {$\cdots$};
\node [] at ($(g14)+(0.02,0)$) {$\cdots$};
\node [] at ($(g34)+(0.02,0)$) {$\cdots$};
\node [] at ($(g02)+(0,0.04)$) {$\vdots$};
\node [] at ($(g22)+(0,0.04)$) {$\vdots$};
\node [] at ($(g42)+(0,0.04)$) {$\vdots$};
\node [] at ($(g52)+(0,0.04)$) {$\vdots$};

\node [] at ($(g12)+(0.02,0)$) {$\cdots$};
\node [] at ($(g32)+(0.02,0)$) {$\cdots$};
\end{scope}

\begin{scope}[>=latex, ->, shorten >= 4.25pt, shorten <= 4.25pt, color=black]
\draw [] (g00)--(g01);
\draw [] (g20)--(g21);
\draw [] (g40)--(g41);
\draw [] (g50)--(g51);
\draw [<->] (g03)--(g04);
\draw [<->] (g23)--(g24);
\draw [<->] (g43)--(g44);
\draw [<->] (g53)--(g54);
\draw [<->] (g40)--(g50);
\draw [<->] (g41)--(g51);
\draw [<->] (g43)--(g53);
\draw [<->] (g44)--(g54);
\end{scope}

\begin{scope}[>=latex, ->, shorten >= 12pt, shorten <= 4.25pt, color=black]
\draw [] (g00)--(g10);
\draw [] (g01)--(g11);
\draw [] (g03)--(g13);
\draw [] (g04)--(g14);
\draw [] (g20)--(g30);
\draw [] (g21)--(g31);
\draw [] (g23)--(g33);
\draw [] (g24)--(g34);
\draw [] (g01)--(g02);
\draw [] (g21)--(g22);
\draw [] (g41)--(g42);
\draw [] (g51)--(g52);
\end{scope}

\begin{scope}[>=latex, ->, shorten >= 4.25pt, shorten <= 12pt, color=black]
\draw [] (g10)--(g20);
\draw [] (g11)--(g21);
\draw [] (g13)--(g23);
\draw [] (g14)--(g24);
\draw [] (g30)--(g40);
\draw [] (g31)--(g41);
\draw [] (g33)--(g43);
\draw [] (g34)--(g44);
\draw [] (g02)--(g03);
\draw [] (g22)--(g23);
\draw [] (g42)--(g43);
\draw [] (g52)--(g53);
\end{scope}

\begin{scope}[color=black]
\node [below=2pt] at (g00) {$0$};
\node [below=2pt] at (g20) {$n+1$};
\node [below=2pt] at (g40) {$r+1$};
\node [below=2pt] at (g50) {$r+2$};
\node [left =2pt] at (g00) {$0$};
\node [left =2pt] at (g01) {$1$};
\node [left =2pt] at (g03) {$r+1$};
\node [left =2pt] at (g04) {$r+2$};
\end{scope}

\end{tikzpicture}
\caption{Finite model for the special $T_n$-tiling with $n\in\mathbb{Y}$}
\label{fig:9}
\end{figure}
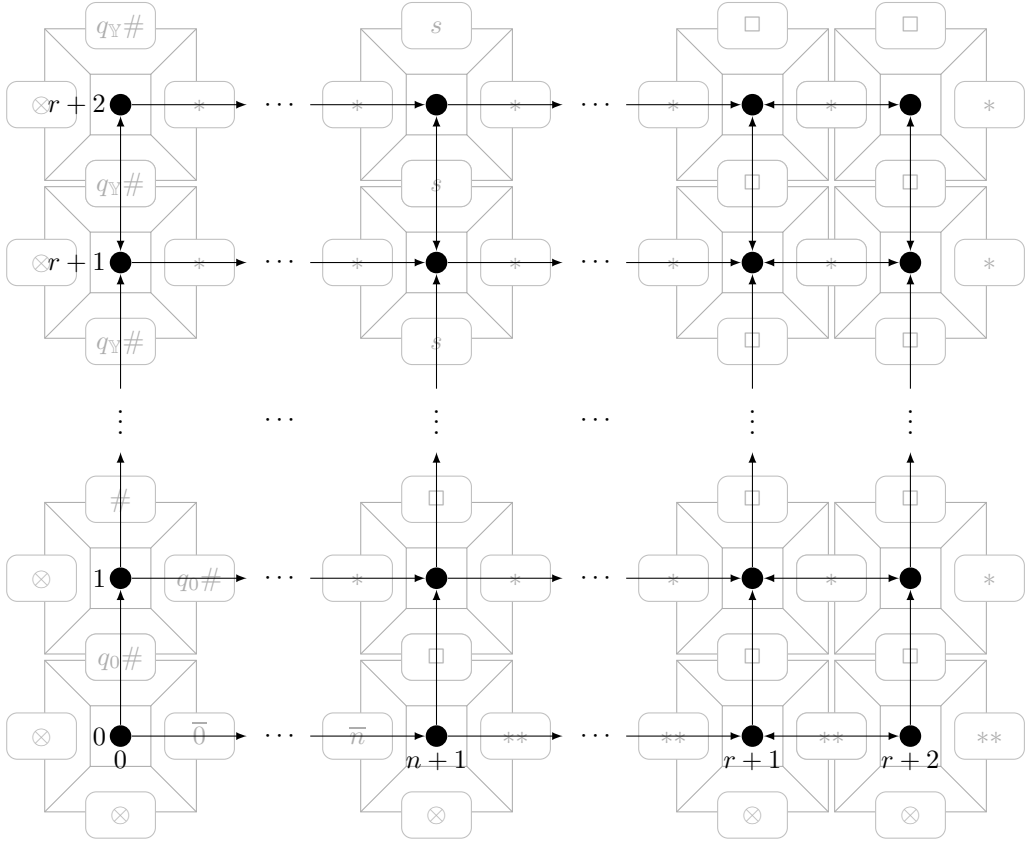

Let us introduce abbreviations:
$$
\begin{array}{rcl}
~~H_n(x,y) 
  & = 
  & \displaystyle
    P(x,y) \wedge
    \bigvee\big\{P_i(x)\wedge P_j(y) : \mbox{$i,j\in\{0,\ldots,k_n\}$ and $\rightsq t^n_i = \leftsq t^n_j$}\big\};
\smallskip\\
~~V_n(x,y) 
  & = 
  & \displaystyle
    P(x,y) \wedge
    \bigvee\big\{P_i(x)\wedge P_j(y) : \mbox{$i,j\in\{0,\ldots,k_n\}$ and $\upsq t^n_i = \downsq t^n_j$}\big\}. 
\end{array}
$$
Using them, define formulas describing conditions for the special $T_n$-tiling:
$$
\begin{array}{lcl}
\mathit{TC}_0 
  & = 
  & \displaystyle
    \forall x\,\bigvee\limits_{i=0}^{k_n} \Big(P_i(x)\wedge\bigwedge\limits_{j\ne i}\neg P_j(x)\Big);
  \smallskip\\
\mathit{TC}_1 
  & = 
  & \forall x\exists y\, H_n(x,y);
  \smallskip\\
\mathit{TC}_2 
  & = 
  & \forall x\exists y\, V_n(x,y);
  \smallskip\\
\mathit{TC}_3 
  & = 
  & \forall x\forall y\,(\exists z\, (H_n(x,z)\wedge V_n(z,y)) \leftrightarrow \exists z\, (V_n(x,z)\wedge H_n(z,y))); 
  \smallskip\\
\mathit{TC}_4 
  & = 
  & \exists x\,P_0(x). 
\end{array}
$$
Then, we can say that we are given the special $T_n$-tiling:  
$$
\begin{array}{lcl}
\mathit{Tiling}_n 
  & = 
  & \mathit{TC}_0 
    \wedge \mathit{TC}_1 
    \wedge \mathit{TC}_2 
    \wedge \mathit{TC}_3 
    \wedge \mathit{TC}_4.
\end{array}
$$
Also, it is easy to say that $n\in\mathbb{X}$ or that $n\in\mathbb{Y}$:
$$
\begin{array}{lcl}
\mathit{Tiling}_n^{\mathbb{X}} 
  & = 
  & \mathit{Tiling}_n \to \exists x\,P_1(x);
  \smallskip\\
\mathit{Tiling}_n^{\mathbb{Y}} 
  & = 
  & \mathit{Tiling}_n \to \exists x\,P_2(x).
\end{array}
$$
In fact, to prove the Trakhtenbrot theorem, we need just one of these formulas. To obtain the proof, it is sufficient to make two observations concerning any (say, the first) of them; we present the observations in two lemmas below.

\begin{lemma}
\label{lem:Trakhtenbrot:lem1}
If $n\in\mathbb{X}$, then $\mathit{Tiling}_n^{\mathbb{X}}\in\logic{QCl}$.
\end{lemma}

\begin{proof}
Let $\cModel{M}=\langle \mathcal{D},\mathcal{I}\rangle$ be a model such that $\cModel{M}\models\mathit{Tiling}_n$. We show that then $\cModel{M}\models \exists x\,P_1(x)$.
To this end, for all $i,j\in\numN$ we pick out an element $a_{i}^{j}\in \mathcal{D}$ so that, for all $i,j\in\numN$,
$$
\begin{array}{lcl}
\cModel{M}\models H_n(a_i^j,a_{i+1}^j) 
  & \mbox{and} 
  & \cModel{M}\models V_n(a_i^j,a_i^{j+1}).
\end{array}
$$

%We introduce an auxiliary function $g\colon\numN\times\numN\to D$, those intending meaning is to point on individuals of $D$ which we take to form an $\numN\times\numN$~grid.

Since $\cModel{M}\models \mathit{TC}_4$, there exists $a_0^0\in D$ such that $\cModel{M}\models P_0(a_0^0)$. 
%Notice that $f_n(0,0)=t^n_0=t_0$.

Let $k$ be a number from $\numN$. Suppose that for all $i,j\in \{0,\ldots,k\}$, the element $a_i^j$ is defined; we have to define, for all $i,j\in \{0,\ldots,k\}$, the elements $a_{k+1}^j$, $a_i^{k+1}$, and $a_{k+1}^{k+1}$.

We start with $a_{k+1}^{k+1}$. Due to $\mathit{TC}_1$ and $\mathit{TC}_2$, there are $b,c\in \mathcal{D}$ such that $\cModel{M}\models H_n(a_k^k,b)\wedge V_n(b,c)$. Then take $a_{k+1}^{k+1}=c$. Also, let $a_{k+1}^{k}=b$.
Due to $\mathit{TC}_3$, there is $b'\in \mathcal{D}$ such that $\cModel{M}\models V_n(a_k^k,b')\wedge H_n(b',c)$; let $a_{k}^{k+1}=b'$. 

Let $a_{k+1}^{j+1}$ be defined so that $\cModel{M}\models V_n(a_k^j,a_k^{j+1})\wedge H_n(a_k^{j+1},a_{k+1}^{j+1})$. Due to $\mathit{TC}_3$, there is $b\in \mathcal{D}$ such that $\cModel{M}\models H_n(a_k^j,b')\wedge V_n(b',a_{k+1}^{j+1})$; let $a_{k+1}^{j}=b$. This allows us to define $a_{k+1}^j$, for every $j\in \{0,\ldots,k\}$.

Let $a_{i+1}^{k+1}$ be defined so that $\cModel{M}\models H_n(a_i^k,a_{i+1}^{k})\wedge V_n(a_{i+1}^{k},a_{i+1}^{k+1})$. Due to $\mathit{TC}_3$, there is $b\in \mathcal{D}$ such that $\cModel{M}\models V_n(a_i^k,b')\wedge H_n(b',a_{i+1}^{k+1})$; let $a_{i}^{k+1}=b$. This allows us to define $a_i^{k+1}$, for every $i\in \{0,\ldots,k\}$.

For any $i,j\in \numN$, due to $\mathit{TC}_0$, there exists a unique $m\in\{0,\ldots,k_n\}$ such that $\cModel{M}\models P_m(a_i^j)$. Then, put $f(i,j)=t^n_m$.
It should be clear that $f\colon\numN\times\numN\to T_n$ is a $T_n$-tiling. Indeed, the conditions $(1)$ and~$(2)$ hold:
$$
\begin{array}{lcl}
\cModel{M}\models H_n(a_i^j,a_{i+1}^j) & \Longrightarrow & \rightsq f(i,j) = \leftsq f(i+1,j);
\smallskip\\
\cModel{M}\models \hfill V_n(a_i^j,a_i^{j+1}) & \Longrightarrow & \upsq f(i,j) = \downsq f(i,j+1).
\end{array}
$$

Notice that $f(0,0)=t^n_0=t_0$ since $\cModel{M}\models P_0(a_0^0)$. Then, Proposition~\ref{prop:fn} guarantees that $f$ is the special $T_n$-tiling~$f_n$. Since $n\in\mathbb{X}$, by $({\ast})$ we obtain that there exists $m\in\numN$ such that $f_n(0,m)=t_1=t^n_1$. This means that $\cModel{M}\models P_1(a_0^m)$, and hence, $\cModel{M}\models \exists x\,P_1(x)$.

Thus, $\mathit{Tiling}_n^{\mathbb{X}}\in\logic{QCl}$.
\end{proof}

\begin{lemma}
\label{lem:Trakhtenbrot:lem2}
If $n\in\mathbb{Y}$, then $\mathit{Tiling}_n^{\mathbb{X}}\not\in\logic{QCl}_{\mathit{fin}}$.
\end{lemma}

\begin{proof}
Let $n\in\mathbb{Y}$. By $({\ast})$, there exists $m\in \numN$ such that $f_n(0,m)=t_2$. Then $M_0$ loops the halting configuration, and, for all $i,j\in\numN$,
$$
\begin{array}{lcl}
f_n(i,m+j) & = & f_n(0,m).
\end{array}
$$
Let $r=\max\{m,n+1\}$. 
Since the head of a Turing machine cannot scan cells numbered grater than the number of steps in the computation (i.e., greater than~$m$) and every cell numbered grater than~$n$ contains $\Box$ at the initial configuration, we conclude that, for all $i,j\in\numN^+$,
$$
\begin{array}{lcl}
f_n(r+i,0) & = & t^{\ast\ast}_\Box; \smallskip\\
f_n(r+i,j) & = & t^{\ast}_\Box.
\end{array}
$$
Thus, we can describe the special $T_n$-tiling by a finite model. Let 
$$
\begin{array}{lcl}
\mathcal{D} & = & \{0,\ldots,r+2\}\times\{0,\ldots,r+2\}. \\
\end{array}
$$
Define a model $\cModel{M} = \otuple{\mathcal{D},\mathcal{I}}$ by 
$$
\begin{array}{lcl}
\cModel{M}\models P(\otuple{i,j},\otuple{i',j'}) 
  & \iff 
  & \mbox{either $i'=i+1$ and $j'=j$,}
  \\
  &&\mbox{or $i'=i$ and $j'=j+1$,}
  \\
  &&\mbox{or $i=r+2$, $i'=r+1$, and $j'=j$,}
  \\
  &&\mbox{or $i'=i$, $j=r+2$, and $j'=r+1$;}
  \smallskip\\
\cModel{M}\models P_k(\otuple{i,j}) 
  & \iff 
  & f_n(i,j) = t_k^n,
\end{array}
$$
see Figure~\ref{fig:9}.

Then, by the definition of $\cModel{M}$, we obtain $\cModel{M}\models\mathit{Tiling}_n$. Since $n\not\in\mathbb{X}$, there is no $i,j\in\numN$ such that $f_n(i,j)=t_1$, therefore, $\cModel{M}\not\models\exists x\,P_1(x)$.

Hence, $\mathit{Tiling}_n^{\mathbb{X}}\not\in\logic{QCl}_{\mathit{fin}}$.
\end{proof}

Then, we readily obtain the Trakhtenbrot theorem.

\begin{theorem}[Trakhtenbrot]
\label{th:Trakhtenbrot}
Logics\/ $\logic{QCl}$ and\/ $\logic{QCl}_{\mathit{fin}}$ are recursively inseparable in a language containing a binary predicate letter, an infinite supply of unary predicate letters, and three individual variables.
\end{theorem}

\begin{proof}
Immediate from Lemmas~\ref{lem:Trakhtenbrot:lem1} and~\ref{lem:Trakhtenbrot:lem2}.
\end{proof}

Below, we refine this statement and expand it on a lot of theories of a binary predicate.

\subsection{Binary relation and three variables}

We show how to eliminate all unary predicate letters by simulating them with formulas containing the binary letter~$P$. First of all, let us make a relativization. Let $G$ be a new unary predicate letter; define the quantifiers $\forall_G$ and $\exists_G$
%, $\forall_{\neg{G}}$, and $\exists_{\neg{G}}$ 
by
$$
\begin{array}{lclclcl}
\forall_Gx\,\varphi & = & \forall x\,(G(x)\to\varphi); 
%&~&
%\forall_{\neg{G}}\,x\varphi & = & \forall x\,(\neg G(x)\to\varphi); 
\smallskip\\
\exists_Gx\,\varphi & = & \exists x\,(G(x)\hfill\wedge\hfill\varphi). 
%&~&
%\exists_{\neg{G}}\,x\varphi & = & \exists x\,(\neg G(x)\hfill\wedge\hfill\varphi). \\
\end{array}
$$
For an $\lang{L}$-formula~$\varphi$, denote by $\varphi_G$ the formula obtained from $\varphi$ by replacing each quantifier $\forall x$ or $\exists x$ with $\forall_G x$ or $\exists_G x$, respectively.

The following lemma is obvious.

\begin{lemma}
\label{lem:relativization:G}
Let $\varphi$ be a closed $\lang{L}$-formula without occurrences of~$G$. Then
$$
\begin{array}{lclclcl}
\varphi \in \logic{QCl} 
  & \iff 
  & \exists x\,G(x) \to \varphi_G \in \logic{QCl}; 
  \smallskip\\
\varphi \in \logic{QCl}_{\mathit{fin}} 
  & \iff 
  & \exists x\,G(x) \to \varphi_G \in \logic{QCl}_{\mathit{fin}}. 
\end{array}
$$
\end{lemma}

\begin{proof}
We give a sketch of the proof, leaving the details to the reader.

Let $\cModel{M}\not\models\varphi$, for some model~$\cModel{M}$. Modify $\cModel{M}$ to $\cModel{M}'$ so that $\cModel{M}'\models\forall x\,G(x)$. Then, clearly, $\cModel{M}'\not\models \exists x\,G(x) \to \varphi_G$.

Let $\cModel{M}\not\models \exists x\,G(x) \to \varphi_G$, for some model~$\cModel{M}$. Take the submodel $\cModel{M}'$ of $\cModel{M}$ formed by the elements of $\cModel{M}$ on which $G$ is true. Such elements exist since $\cModel{M}\models \exists x\,G(x)$. Then, clearly, $\cModel{M}'\not\models\varphi$. 
Notice that if $\cModel{M}$ is finite, then $\cModel{M}'$ is finite, too.
\end{proof}

Let us define formulas we shall use to simulate the unary predicate letters:
$$
\begin{array}{rcl}
\varepsilon^y(x)   & = & \neg\exists y\,P(x,y); \smallskip\\
\tau_0^y(x)        & = & \exists y\,(\neg G(y)\wedge P(x,y)\wedge \varepsilon^x(y)); \smallskip\\
\tau_{k+1}^y(x)    & = & \exists y\,(\neg G(y)\wedge P(x,y)\wedge \tau_k^x(y)); \smallskip\\
\mathit{tile}_k(x) & = & \tau_k^y(x); \smallskip\\
\mathit{tile}_k(y) & = & \tau_k^x(y); \smallskip\\
\mathit{tile}_k(z) & = & \tau_k^x(z), \smallskip\\
\end{array}
$$
where $k\in\numN$, and
$$
\begin{array}{rcl}
\gamma^y(x)        & = & \neg P(x,x)\wedge \exists y\,(P(x,y)\wedge P(y,y)); \smallskip\\
\mathit{grid}(x)   & = & \gamma^y(x); \smallskip\\
\mathit{grid}(y)   & = & \gamma^x(y); \smallskip\\
\hspace{2em}
\mathit{grid}(z)   & = & \gamma^x(z). \smallskip\\
\end{array}
$$

%\begin{remark}
%\label{rem:tile:x&y:1}
%Formulas $\mathit{tile}_k(x)$ and $\mathit{tile}_k(y)$ contain no variables except $x$ and~$y$; also they contain no predicate letters except $P$ and~$G$.
%\end{remark}

%\begin{remark}
%\label{rem:tile:x&y:1a}
%Formulas $\mathit{grid}(x)$ and $\mathit{grid}(y)$ contain no variables except $x$ and~$y$; also they contain no predicate letters except~$P$.
%\end{remark}

Finally, let us define a function~$S_0$ associating with each formula $\mathit{Tiling}_n^{\mathbb{X}}$ a special formula that contains no predicate letters except the binary letter~$P$. Define $S_0\mathit{Tiling}_n^{\mathbb{X}}$ to be obtained from $\exists x\,G(x)\to (\mathit{Tiling}_n^{\mathbb{X}})_G^{\phantom{i}}$ by replacing each occurrence of $P_m(x)$, $P_m(y)$, and $P_m(z)$ with $\mathit{tile}_m(x)$, $\mathit{tile}_m(y)$, and $\mathit{tile}_m(z)$, respectively, where $m\in\{0,\ldots,k_n\}$, and then each occurrence of $G(x)$, $G(y)$, and $G(z)$ with $\mathit{grid}(x)$, $\mathit{grid}(y)$, and $\mathit{grid}(z)$, respectively. Notice that these replacements are formula substitutions\footnote{Such substitutions are known as \defnotion{indirect} substitutions: substituting formulas we use formulas that are congruent to those we substitute.\label{footnote:1}} and contain three individual variables.

\begin{lemma}
\label{lem:Trakhtenbrot:lem1:binP}
If $n\in\mathbb{X}$, then $S_0\mathit{Tiling}_n^{\mathbb{X}}\in\logic{QCl}$.
\end{lemma}

\begin{proof}
Follows from Lemmas~\ref{lem:Trakhtenbrot:lem1} and~\ref{lem:relativization:G} taking into account that $\logic{QCl}$ is closed by Substitution.
\end{proof}

\begin{lemma}
\label{lem:Trakhtenbrot:lem2:binP}
If $n\in\mathbb{Y}$, then $S_0\mathit{Tiling}_n^{\mathbb{X}}\not\in\logic{QCl}_{\mathit{fin}}$.
\end{lemma}

\begin{proof}
Let $n\in\mathbb{Y}$. By Lemma~\ref{lem:Trakhtenbrot:lem2}, $\mathit{Tiling}_n^{\mathbb{X}}\not\in\logic{QCl}_{\mathit{fin}}$, and by Lemma~\ref{lem:relativization:G},
$$
\begin{array}{lcl}
\exists x\,G(x)\to (\mathit{Tiling}_n^{\mathbb{X}})_G^{\phantom{i}}
  & \not\in
  & \logic{QCl}_{\mathit{fin}},
\end{array}
$$  
hence, there exists a finite model $\cModel{M}=\otuple{\mathcal{D},\mathcal{I}}$ such that 
$$
\begin{array}{c}
\cModel{M} ~\not\models~ \exists x\,G(x)\to (\mathit{Tiling}_n^{\mathbb{X}})_G^{\phantom{i}}.
\end{array}
$$
Without a loss of generality we may assume that $\cModel{M}\models\forall x\,G(x)$.
We extend $\cModel{M}$ to a bigger finite model refuting the formula $S_0\mathit{Tiling}_n^{\mathbb{X}}$.

\begin{figure}
\centering
\begin{tikzpicture}[scale=3.2, rectnode/.style={rectangle, thick, draw=black!60, dashed, rounded corners = 2pt}]

\coordinate (c1) at (-1.4,0.30); %(-1.4, 0.30);
\coordinate (c2) at (-0.9,-0.3); %(-1.0,-0.3);
\coordinate (c3) at (0.2,-0.05); %( 0.0,-0.05);
\coordinate (c4) at (-0.4,0.5); %(-0.5, 0.50);

\coordinate (c12) at ($0.5*(c1)+0.5*(c2)$);
\coordinate (c23) at ($0.5*(c2)+0.5*(c3)$);
\coordinate (c34) at ($0.5*(c3)+0.5*(c4)$);
\coordinate (c41) at ($0.5*(c4)+0.5*(c1)$);

\coordinate (vecX1) at ($(c3)-(c2)$);
\coordinate (vecX2) at ($(c4)-(c1)$);
\coordinate (vecY1) at ($(c1)-(c2)$);
\coordinate (vecY2) at ($(c4)-(c3)$);

\coordinate (c1r) at (c4);
\coordinate (c2r) at (c3);
\coordinate (c3r) at ($(c3)+0.96*(vecX1)$);
\coordinate (c4r) at ($(c4)+0.96*(vecX2)$);

\coordinate (c1u) at ($(c1)+0.84*(vecY1)$);
\coordinate (c2u) at (c1);
\coordinate (c3u) at (c4);
\coordinate (c4u) at ($(c4)+0.84*(vecY2)$);

\coordinate (c12r) at ($0.5*(c1r)+0.5*(c2r)$);
\coordinate (c23r) at ($0.5*(c2r)+0.5*(c3r)$);
\coordinate (c34r) at ($0.5*(c3r)+0.5*(c4r)$);
\coordinate (c41r) at ($0.5*(c4r)+0.5*(c1r)$);

\coordinate (c12u) at ($0.5*(c1u)+0.5*(c2u)$);
\coordinate (c23u) at ($0.5*(c2u)+0.5*(c3u)$);
\coordinate (c34u) at ($0.5*(c3u)+0.5*(c4u)$);
\coordinate (c41u) at ($0.5*(c4u)+0.5*(c1u)$);

\draw [white, opacity = 0, name path = dg 1] (c1)--(c3);
\draw [white, opacity = 0, name path = dg 2] (c2)--(c4);
\draw [name intersections = {of = dg 1 and dg 2, by = {c}}];
\draw [white, opacity = 0, name path = dg 3] (c12)--(c34);
\draw [white, opacity = 0, name path = dg 4] (c23)--(c41);
\draw [name intersections = {of = dg 3 and dg 4, by = {d}}];

\draw [white, opacity = 0, name path = dgr 1] (c1r)--(c3r);
\draw [white, opacity = 0, name path = dgr 2] (c2r)--(c4r);
\draw [name intersections = {of = dgr 1 and dgr 2, by = {cr}}];
\draw [white, opacity = 0, name path = dgr 3] (c12r)--(c34r);
\draw [white, opacity = 0, name path = dgr 4] (c23r)--(c41r);
\draw [name intersections = {of = dgr 3 and dgr 4, by = {dr}}];

\draw [white, opacity = 0, name path = dgu 1] (c1u)--(c3u);
\draw [white, opacity = 0, name path = dgu 2] (c2u)--(c4u);
\draw [name intersections = {of = dgu 1 and dgu 2, by = {cu}}];
\draw [white, opacity = 0, name path = dgu 3] (c12u)--(c34u);
\draw [white, opacity = 0, name path = dgu 4] (c23u)--(c41u);
\draw [name intersections = {of = dgu 3 and dgu 4, by = {du}}];

\coordinate (diffR) at ($(cr)-(c)$);
\coordinate (diffU) at ($(cu)-(c)$);
\coordinate (corR) at (0,0.012);
\coordinate (corU) at (0,0.032);

\coordinate (em) at ($(c)+1*(0,0.25)$);
\coordinate (e2) at ($(c)+2*(0,0.25)$);
\coordinate (e1) at ($(c)+3*(0,0.25)$);
\coordinate (e0) at ($(c)+4*(0,0.25)$);

\coordinate (ekn) at ($(c)-1*(0,0.25)$);
\coordinate (ek2) at ($(c)-2*(0,0.25)$);
\coordinate (ek1) at ($(c)-3*(0,0.25)$);
\coordinate (ek0) at ($(c)-4*(0,0.25)$);

\shade [ball color=white] (ekn) circle [radius = 1.5pt];
\draw  [>=latex, ->, shorten >= 3pt, shorten <= 2.5pt, color=black] (c)--(ekn);

\begin{scope}[color=black!42]
\shade [ball color=white] ($(ekn)+(diffR)+(corR)$) circle [radius = 0.96*1.5pt];
\filldraw [color=white, opacity = 0.5] ($(ekn)+(diffR)+(corR)$) circle [radius = 0.96*1.55pt];
\draw  [>=latex, ->, shorten >= 3pt, shorten <= 4pt] (cr)--($(ekn)+(diffR)+(corR)$);
\shade [ball color=white] ($(ekn)+(diffU)+(corU)$) circle [radius = 0.90*1.5pt];
\filldraw [color=white, opacity = 0.5] ($(ekn)+(diffU)+(corU)$) circle [radius = 0.90*1.55pt];
\draw  [>=latex, ->, shorten >= 3pt, shorten <= 4pt] (cu)--($(ekn)+(diffU)+(corU)$);
\end{scope}

\begin{scope}[color=black!42]
\drawtileflattmslanted{(c1)}{(c2)}{(c3)}{(c4)}
\end{scope}

\begin{scope}[color=black!25]
\drawtileflattmslanted{(c1r)}{(c2r)}{(c3r)}{(c4r)}
\drawtileflattmslanted{(c1u)}{(c2u)}{(c3u)}{(c4u)}
\end{scope}

\begin{scope}[>=latex, ->, shorten >= -7.5pt, shorten <= 4pt, color=black!32]
\draw [] (cr)--($(c34r)+(cr)-(dr)$);
\draw [] (cr)--($(c41r)+(cr)-(dr)$);
\draw [] (cu)--($(c34u)+(cu)-(du)$);
\draw [] (cu)--($(c41u)+(cu)-(du)$);
\end{scope}

\shade [ball color=black!64] (cr) circle [radius = 0.96*1.5pt];
\filldraw [color=white, opacity = 0.5] (cr) circle [radius = 0.96*1.55pt];
\shade [ball color=black!64] (cu) circle [radius = 0.90*1.5pt];
\filldraw [color=white, opacity = 0.5] (cu) circle [radius = 0.90*1.55pt];

\begin{scope}[>=latex, ->, shorten >= 3pt, shorten <= -7.5pt, color=black!32]
\draw [] ($(c12u)+(cu)-(du)$)--(cu);
\draw [] ($(c23r)+(cr)-(dr)$)--(cr);
\end{scope}

\begin{scope}[>=latex, ->, shorten >= 3pt, shorten <= 2pt, color=black!84]
\draw [] (c)--(cr); %($(c34)+(c)-(d)$);
\draw [] (c)--(cu); %($(c41)+(c)-(d)$);
\end{scope}

\shade [ball color=black!64] (c) circle [radius = 1.5pt];

\draw  [>=latex, ->, shorten >= 3pt, shorten <= 2.5pt, color=black] (c)--(em);
\shade [ball color=black] (em) circle [radius = 1.5pt];
\draw  [>=latex, ->, shorten >= 7pt, shorten <= 2.5pt, color=black] (em)--(e2);
\node  [] at ($(e2)+(0,0.03)$) {$\vdots$};
\draw  [>=latex, ->, shorten >= 3pt, shorten <= 6.5pt, color=black] (e2)--(e1);
\shade [ball color=black] (e1) circle [radius = 1.5pt];
\draw  [>=latex, ->, shorten >= 3pt, shorten <= 2.5pt, color=black] (e1)--(e0);
\shade [ball color=black] (e0) circle [radius = 1.5pt];

\begin{scope}[color=black!42]
\filldraw  [color=white] ($(em)+(diffR)-(corR)$) circle [radius = 0.96*1.5pt];
\draw  [>=latex, ->, shorten >= 3pt, shorten <= 2.5pt] (cr)--($(em)+(diffR)-(corR)$);
\shade [ball color=black!84] ($(em)+(diffR)-(corR)$) circle [radius = 0.96*1.5pt];
\filldraw [color=white, opacity = 0.5] ($(em)+(diffR)-(corR)$) circle [radius = 0.96*1.55pt];
\draw  [>=latex, ->, shorten >= 7pt, shorten <= 2.5pt] ($(em)+(diffR)-(corR)$)--($(e2)+(diffR)-(corR)-(corR)$);
\node  [] at ($(e2)+(diffR)+(0,0.03)-(corR)-(corR)$) {$\vdots$};
\draw  [>=latex, ->, shorten >= 3pt, shorten <= 7.5pt] ($(e2)+(diffR)-(corR)-(corR)$)--($(e1)+(diffR)-(corR)-(corR)-(corR)$);
\shade [ball color=black!84] ($(e1)+(diffR)-(corR)-(corR)-(corR)$) circle [radius = 0.96*1.5pt];
\filldraw [color=white, opacity = 0.5] ($(e1)+(diffR)-(corR)-(corR)-(corR)$) circle [radius = 0.96*1.55pt];
\draw  [>=latex, ->, shorten >= 3pt, shorten <= 2.5pt] ($(e1)+(diffR)-(corR)-(corR)-(corR)$)--($(e0)+(diffR)-(corR)-(corR)-(corR)-(corR)$);
\shade [ball color=black!84] ($(e0)+(diffR)-(corR)-(corR)-(corR)-(corR)$) circle [radius = 0.96*1.5pt];
\filldraw [color=white, opacity = 0.5] ($(e0)+(diffR)-(corR)-(corR)-(corR)-(corR)$) circle [radius = 0.96*1.55pt];
\filldraw  [color=white] ($(em)+(diffU)-(corU)$) circle [radius = 0.90*1.5pt];
\draw  [>=latex, ->, shorten >= 3pt, shorten <= 2.5pt] (cu)--($(em)+(diffU)-(corU)$);
\shade [ball color=black!84] ($(em)+(diffU)-(corU)$) circle [radius = 0.90*1.5pt];
\filldraw [color=white, opacity = 0.5] ($(em)+(diffU)-(corU)$) circle [radius = 0.90*1.55pt];
\draw  [>=latex, ->, shorten >= 7pt, shorten <= 2.5pt] ($(em)+(diffU)-(corU)$)--($(e2)+(diffU)-(corU)-(corU)$);
\node  [] at ($(e2)+(diffU)+(0,0.03)-(corU)-(corU)$) {$\vdots$};
\draw  [>=latex, ->, shorten >= 3pt, shorten <= 7.5pt] ($(e2)+(diffU)-(corU)-(corU)$)--($(e1)+(diffU)-(corU)-(corU)-(corU)$);
\shade [ball color=black!84] ($(e1)+(diffU)-(corU)-(corU)-(corU)$) circle [radius = 0.90*1.5pt];
\filldraw [color=white, opacity = 0.5] ($(e1)+(diffU)-(corU)-(corU)-(corU)$) circle [radius = 0.90*1.55pt];
\draw  [>=latex, ->, shorten >= 3pt, shorten <= 2.5pt] ($(e1)+(diffU)-(corU)-(corU)-(corU)$)--($(e0)+(diffU)-(corU)-(corU)-(corU)-(corU)$);
\shade [ball color=black!84] ($(e0)+(diffU)-(corU)-(corU)-(corU)-(corU)$) circle [radius = 0.90*1.5pt];
\filldraw [color=white, opacity = 0.5] ($(e0)+(diffU)-(corU)-(corU)-(corU)-(corU)$) circle [radius = 0.90*1.55pt];
\end{scope}

\begin{scope}[>=latex, ->, shorten >= 3pt, shorten <= -7.5pt, color=black!84]
\draw [] ($(c12)+(c)-(d)$)--(c);
\draw [] ($(c23)+(c)-(d)$)--(c);
\end{scope}

\node [right] at ($(c)+(0.03,-0.03)$) {$a$} ;
\node [right] at ($(ekn)+(0.03,-0.03)$) {$e_\ast^a$} ;
\node [right] at ($(em)+(0.03,-0.03)$) {$e_m^a$} ;
\node [right] at ($(e1)+(0.03,-0.03)$) {$e_1^a$} ;
\node [right] at ($(e0)+(0.03,-0.03)$) {$e_0^a$} ;

%\draw [rounded corners = 2pt] ($(em)+(0.075,-0.075)$) 
%-- ($(em)+(0.075,-0.075)+(0.032,0)$) 
%-- ($(em)+(0.075,-0.075)+(0.032,0)+(0.0,0.9)$) 
%-- ($(em)+(0.075,-0.075)+(0.032,0)+(0.0,0.9)+(-0.032,0)$);
%\node [right=7.5pt] at ($0.5*(em)+0.5*(e0)$) {$m+1$};
%\node [rectnode, right] at ($0.5*(em)+0.5*(e0)+(1.6,0)$) {\begin{tabular}{l}$m+1$ elements \\ simulate $P_m(a)$\end{tabular}};

%\draw [rounded corners = 2pt] ($(ek0)+(0.075,-0.075)$) 
%-- ($(ek0)+(0.075,-0.075)+(0.032,0)$) 
%-- ($(ek0)+(0.075,-0.075)+(0.032,0)+(0.0,0.9)$) 
%-- ($(ek0)+(0.075,-0.075)+(0.032,0)+(0.0,0.9)+(-0.032,0)$);
%\node [right=7.7pt] at ($0.5*(ekn)+0.5*(ek0)$) {$k_n+2$};
%\node [rectnode, right] at ($0.5*(ekn)+0.5*(ek0)+(1.6,0)$) {\begin{tabular}{l}$k_n+2$ elements \\ simulate $G(a)$\end{tabular}};

%\draw [>=latex, ->, color=black!60, thick, dashed] ($0.5*(em)+0.5*(e0)+(1.25,0)$) -- ($0.5*(em)+0.5*(e0)+(0.15,0)$);
%\draw [>=latex, ->, color=black!60, thick, dashed] ($0.5*(ekn)+0.5*(ek0)+(1.25,0)$) -- ($0.5*(ekn)+0.5*(ek0)+(0.15,0)$);

\end{tikzpicture}
\caption{Simulation of $P_m(a)$ and $G(a)$}
\label{fig:10}
\end{figure}
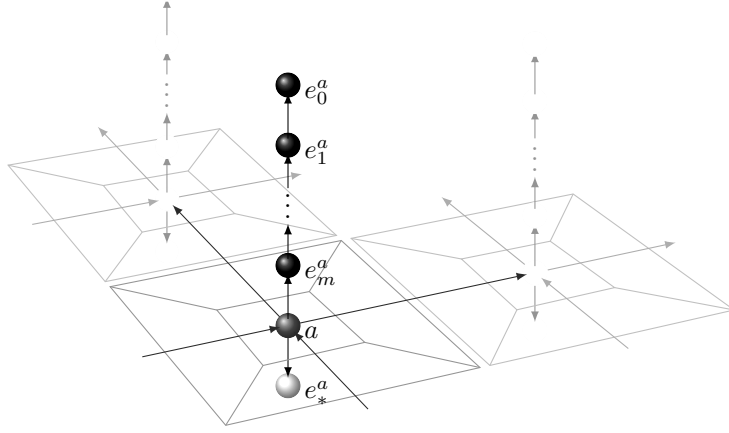

For every $a\in \mathcal{D}$ such that $\cModel{M}\models P_m(a)$, for some $m\in\{0,\ldots,k_n\}$, add to $\mathcal{D}$ new elements $e_\ast^a, e_0^a,\ldots,e_m^a$ and extend $\mathcal{I}(P)$ with $\otuple{a,e_\ast^a},\otuple{e_\ast^a,e_\ast^a},\otuple{a,e_m^a},\otuple{e_m^a,e_{m-1}^a},\ldots,\otuple{e_1^a,e_0^a}$; let $\cModel{M}'=\otuple{\mathcal{D}',\mathcal{I}'}$ be the resulting model; see Figure~\ref{fig:10}. By the definition of $\cModel{M}'$, for every $e\in \mathcal{D}'$,
$$
\begin{array}{rcl}
  \phantom{\cModel{M}\models P_m(a)}
  \cModel{M}'\models \mathit{grid}(e)
    & \iff
    & \mbox{$e\in \mathcal{D}$}
      \phantom{\cModel{M}'\models \mathit{tile}_m(a).}
    \\
\end{array}
$$
and, for every~$m\in\{0,\ldots,k_n\}$ and every $a\in \mathcal{D}$,
$$
\begin{array}{rcl}
  \phantom{\cModel{M}'\models \mathit{grid}(e)}
  \cModel{M}\models P_m(a)
    & \iff
    & \cModel{M}'\models \mathit{tile}_m(a).
      \phantom{\mbox{$e\in \mathcal{D}$}}
\end{array}
$$
Using these, we obtain that $\cModel{M}'\not\models S_0\mathit{Tiling}_n^{\mathbb{X}}$. Notice that model $\cModel{M}'$ is finite. 
Thus, $S_0\mathit{Tiling}_n^{\mathbb{X}}\not\in\logic{QCl}_{\mathit{fin}}$.
\end{proof}

As a result, we obtain the following refinement of the Trakhtenbrot theorem.

\begin{theorem} %[Trakhtenbrot]
\label{th:Trakhtenbrot:bin:P}
Logics\/ $\logic{QCl}$ and\/ $\logic{QCl}_{\mathit{fin}}$ are recursively inseparable in a language containing a binary predicate letter and three individual variables.
\end{theorem}

\begin{proof}
Immediate from Lemmas~\ref{lem:Trakhtenbrot:lem1:binP} and~\ref{lem:Trakhtenbrot:lem2:binP}.
\end{proof}

But our aim is not Theorem~\ref{th:Trakhtenbrot:bin:P}; we shall prove a similar statement for some special theories of a binary predicate 
%using a modification of the 
by adapting
constructions from the proof of the theorem.

\subsection{Expanding Trakhtenbrot theorem on other theories}
\label{subsec:Trakhtenbrot:theories}

In fact, Theorem~\ref{th:Trakhtenbrot:bin:P} says that the graph theory and the theory of finite graphs are recursively inseparable in the language with three variables (and a binary letter corresponding to the edges of graphs). 
%
%Observe that we can add extra conditions to the theories~--- claiming (directed) graphs to be planar, bipartite, connected, etc.~--- and the theorem stays true due to the properties of used finite models refuting 
%
Note that we can impose additional conditions on theories, for example, by requiring graphs to be planar, bipartite, or connected. Despite these additional limitations, the theorem remains true due to the properties of the finite models used to refute 
$S_0\mathit{Tiling}_n^{\mathbb{X}}$ with $n\in\mathbb{Y}$. 
In fact, we can add some other conditions, claiming graphs to be also antisymmetric, serial, transitive or intransitive, etc.; it is possible to use ideas from~\cite{MR:2022:DoklMath}. But we are interested in two theories not covered by the results of~\cite{MR:2022:DoklMath}: the theory $\logic{SIB}$ of a symmetric irreflexive binary relation and the theory $\logic{SRB}$ of a symmetric reflexive binary relation. Both theories are undecidable in the language with three variables~\cite{MR:2022:DoklMath} but the methods of~\cite{MR:2022:DoklMath} do not give anything about the decidability of $\logic{SIB}_{\mathit{fin}}$ and $\logic{SRB}_{\mathit{fin}}$. We shall show that $\logic{SIB}$ and $\logic{SIB}_{\mathit{fin}}$ are recursively inseparable and, as a corollary, $\logic{SIB}_{\mathit{fin}}$ is not recursively enumerable; similarly for $\logic{SRB}$ and $\logic{SRB}_{\mathit{fin}}$. 
%
%This result will be very useful also to solve some problems posed for modal and superintuitionistic predicate logics.
%
This result will also be very useful for solving some problems posed in modal and superintuitionistic predicate logics.

We shall deal with $\logic{SIB}$ and $\logic{SIB}_{\mathit{fin}}$; the results for $\logic{SRB}$ and $\logic{SRB}_{\mathit{fin}}$ will be obtained as a corollary. Notice that $\logic{SIB}$ and $\logic{SIB}_{\mathit{fin}}$ are the theories of simple graphs, all and finite, respectively. Also, $\logic{SIB}$ is a fragment of the theory $\logic{QCl}\uplus\bm{sib}$, there 
$$
\begin{array}{lcl}
\bm{sib} 
  & = 
  & \forall x\forall y\,(P(x,y)\to P(y,x)) 
    \wedge \forall x\,\neg P(x,x)
\end{array}
$$
and the language of $\logic{SIB}$ contains~$P$.

To prove an analogue of Theorem~\ref{th:Trakhtenbrot:bin:P} for $\logic{SIB}$ and $\logic{SIB}_{\mathit{fin}}$, we have to solve the following problems: 
\begin{itemize}
\item 
We have to redefine $\mathit{Tiling}_n^{\mathbb{X}}$, where $n\in\numN$. If $n\in\mathbb{X}$, then $\mathit{Tiling}_n^{\mathbb{X}}\in\logic{QCl}\uplus\bm{sib}$ since $\mathit{Tiling}_n^{\mathbb{X}}\in\logic{QCl}$ by Lemma~\ref{lem:Trakhtenbrot:lem2}, therefore, an analogue of Lemma~\ref{lem:Trakhtenbrot:lem2} for $\logic{QCl}\uplus\bm{sib}$ holds. But if $n\in\mathbb{Y}$, then we cannot guarantee that there exists a finite model of $\logic{QCl}\uplus\bm{sib}$ refuting $\mathit{Tiling}_n^{\mathbb{X}}$. This is because now $H_n(a,b)$ and $H_n(b,a)$, for some elements $a$ and $b$ of a model, may lead to a contradiction with $\mathit{TC}_2\wedge \mathit{TC}_3$. As a result, an analog to Lemma~\ref{lem:Trakhtenbrot:lem2} for $\logic{QCl}_{\mathit{fin}}\uplus\bm{sib}$ fails.
\item
If it works, we still have to redefine $\mathit{tile}_k(x)$, where $k\in\numN$. The formula $\mathit{tile}_k(x)$ says that it is possible to reach a final element in $k+1$ steps; but this is impossible since the final elements of a symmetric relation are singletons. 
%
%Along the way, it is important for us to redefine $\mathit{tile}_k(x)$ using two variables only.
%
Throughout the process, it is crucial for us to redefine $\mathit{tile}_k(x)$ using only two variables.
\end{itemize}

Let us take new unary predicate letters $R_1$, $R_2$, $R_3$, $R_4$ and $U_1$, $U_2$, $U_3$, $U_4$. We shall use $R_1$, $R_2$, $R_3$, $R_4$ to simulate a directed ``horizontal'' $P$-shift (to the right) and $U_1$, $U_2$, $U_3$, $U_4$ to simulate a directed ``vertical'' $P$-shift (upward). We define the following directed $P$-shift conditions:
$$
\begin{array}{lcl}
\mathit{DSR} 
  & = 
  & \displaystyle
    \forall x\,\Big( \bigvee\limits_{\mathclap{i=1}}^4 \big( R_i(x) 
    \wedge \bigwedge\limits_{\mathclap{j\ne i}}\neg R_j(x) \big) \Big); \smallskip\\
\mathit{DSU} 
  & = 
  & \displaystyle
    \forall x\,\Big( \bigvee\limits_{\mathclap{i=1}}^4 \big( U_i(x) 
    \hfill\wedge\hfill \bigwedge\limits_{\mathclap{j\ne i}}\neg U_j(x) \big) \Big). \smallskip\\
\end{array}
$$
Let also
$$
\begin{array}{lcl}
\mathit{RD}(x,y) 
  & = 
  & (R_1(x)\wedge R_2(y))\vee (R_2(x)\wedge R_3(y))\vee (R_3(x)\wedge R_4(y))\vee (R_4(x)\wedge R_1(y)); \smallskip\\
\mathit{UD}(x,y) 
  & = 
  & (U_1(x)\hfill\wedge\hfill U_2(y))\hfill\vee\hfill (U_2(x)\hfill\wedge\hfill U_3(y))\hfill\vee\hfill (U_3(x)\hfill\wedge\hfill U_4(y))\hfill\vee\hfill (U_4(x)\hfill\wedge\hfill U_1(y)). \smallskip\\
\end{array}
$$
Informally, $\mathit{RD}(x,y)$ means that $x$ and $y$ can be viewed as candidates for $P$-moving in the right direction from $x$ to~$y$; similarly, $\mathit{UD}(x,y)$ means that $x$ and $y$ can be viewed as candidates for $P$-moving in the upward direction from $x$ to~$y$.

\begin{figure}
\centering
\begin{tikzpicture}[scale=2.1]

\coordinate (h0) at (-2.55,0.75-0.075);
\coordinate (h1) at ($(h0)+(0.5-0.05,0)$);
\coordinate (h2) at ($(h1)+(0.5-0.05,0)$);
\coordinate (h3) at ($(h2)+(0.5-0.05,0)$);

\coordinate (v0) at (-3.45,0);
\coordinate (v1) at ($(v0)+(0,0.5-0.05)$);
\coordinate (v2) at ($(v1)+(0,0.5-0.05)$);
\coordinate (v3) at ($(v2)+(0,0.5-0.05)$);

\coordinate (c1) at ($(h0)+(-0.2*0.96,-0.2*0.96)$);
\coordinate (c2) at ($(c1) +( 0.4*0.96, 0.0)$);
\coordinate (c3) at ($(c2) +( 0.0, 0.4*0.96)$);
\coordinate (c4) at ($(c3) +(-0.4*0.96, 0.0)$);
\drawhalftileflatsmallh{(c1)}{(c2)}{(c3)}{(c4)}{cg6}{cg0}
\coordinate (c1) at ($(h1)+(-0.2*0.96,-0.2*0.96)$);
\coordinate (c2) at ($(c1) +( 0.4*0.96, 0.0)$);
\coordinate (c3) at ($(c2) +( 0.0, 0.4*0.96)$);
\coordinate (c4) at ($(c3) +(-0.4*0.96, 0.0)$);
\drawhalftileflatsmallh{(c1)}{(c2)}{(c3)}{(c4)}{cg0}{cg2}
\coordinate (c1) at ($(h2)+(-0.2*0.96,-0.2*0.96)$);
\coordinate (c2) at ($(c1) +( 0.4*0.96, 0.0)$);
\coordinate (c3) at ($(c2) +( 0.0, 0.4*0.96)$);
\coordinate (c4) at ($(c3) +(-0.4*0.96, 0.0)$);
\drawhalftileflatsmallh{(c1)}{(c2)}{(c3)}{(c4)}{cg2}{cg4}
\coordinate (c1) at ($(h3)+(-0.2*0.96,-0.2*0.96)$);
\coordinate (c2) at ($(c1) +( 0.4*0.96, 0.0)$);
\coordinate (c3) at ($(c2) +( 0.0, 0.4*0.96)$);
\coordinate (c4) at ($(c3) +(-0.4*0.96, 0.0)$);
\drawhalftileflatsmallh{(c1)}{(c2)}{(c3)}{(c4)}{cg4}{cg6}

\coordinate (c1) at ($(v0)+(-0.2*0.96,-0.2*0.96)$);
\coordinate (c2) at ($(c1) +( 0.4*0.96, 0.0)$);
\coordinate (c3) at ($(c2) +( 0.0, 0.4*0.96)$);
\coordinate (c4) at ($(c3) +(-0.4*0.96, 0.0)$);
\drawhalftileflatsmallv{(c1)}{(c2)}{(c3)}{(c4)}{cg6}{cg0}
\coordinate (c1) at ($(v1)+(-0.2*0.96,-0.2*0.96)$);
\coordinate (c2) at ($(c1) +( 0.4*0.96, 0.0)$);
\coordinate (c3) at ($(c2) +( 0.0, 0.4*0.96)$);
\coordinate (c4) at ($(c3) +(-0.4*0.96, 0.0)$);
\drawhalftileflatsmallv{(c1)}{(c2)}{(c3)}{(c4)}{cg0}{cg2}
\coordinate (c1) at ($(v2)+(-0.2*0.96,-0.2*0.96)$);
\coordinate (c2) at ($(c1) +( 0.4*0.96, 0.0)$);
\coordinate (c3) at ($(c2) +( 0.0, 0.4*0.96)$);
\coordinate (c4) at ($(c3) +(-0.4*0.96, 0.0)$);
\drawhalftileflatsmallv{(c1)}{(c2)}{(c3)}{(c4)}{cg2}{cg4}
\coordinate (c1) at ($(v3)+(-0.2*0.96,-0.2*0.96)$);
\coordinate (c2) at ($(c1) +( 0.4*0.96, 0.0)$);
\coordinate (c3) at ($(c2) +( 0.0, 0.4*0.96)$);
\coordinate (c4) at ($(c3) +(-0.4*0.96, 0.0)$);
\drawhalftileflatsmallv{(c1)}{(c2)}{(c3)}{(c4)}{cg4}{cg6}

\node [below=14pt] at (h0) {$R_1$};
\node [below=14pt] at (h1) {$R_2$};
\node [below=14pt] at (h2) {$R_3$};
\node [below=14pt] at (h3) {$R_4$};

\node [left =14pt] at (v0) {$U_1$};
\node [left =14pt] at (v1) {$U_2$};
\node [left =14pt] at (v2) {$U_3$};
\node [left =14pt] at (v3) {$U_4$};

\node [right=22.5pt] at (h3) {$=$};
\node [right=20.0pt] at ($0.5*(v1)+0.5*(v2)$) {$\times$};

\coordinate (g00) at (0,0);
\coordinate (g10) at (0.5-0.05,0);
\coordinate (g20) at (1-0.10,0);
\coordinate (g30) at (1.5-0.15,0);
\coordinate (g01) at (0,0.5-0.05);
\coordinate (g11) at (0.5-0.05,0.5-0.05);
\coordinate (g21) at (1-0.10,0.5-0.05);
\coordinate (g31) at (1.5-0.15,0.5-0.05);
\coordinate (g02) at (0,1-0.10);
\coordinate (g12) at (0.5-0.05,1-0.10);
\coordinate (g22) at (1-0.10,1-0.10);
\coordinate (g32) at (1.5-0.15,1-0.10);
\coordinate (g03) at (0,1.5-0.15);
\coordinate (g13) at (0.5-0.05,1.5-0.15);
\coordinate (g23) at (1-0.10,1.5-0.15);
\coordinate (g33) at (1.5-0.15,1.5-0.15);

\node [below=14pt] at (g00) {$R_1$};
\node [below=14pt] at (g10) {$R_2$};
\node [below=14pt] at (g20) {$R_3$};
\node [below=14pt] at (g30) {$R_4$};
\node [left =14pt] at (g00) {$U_1$};
\node [left =14pt] at (g01) {$U_2$};
\node [left =14pt] at (g02) {$U_3$};
\node [left =14pt] at (g03) {$U_4$};

\coordinate (c1) at ($(g00)+(-0.2*0.96,-0.2*0.96)$);
\coordinate (c2) at ($(c1) +( 0.4*0.96, 0.0)$);
\coordinate (c3) at ($(c2) +( 0.0, 0.4*0.96)$);
\coordinate (c4) at ($(c3) +(-0.4*0.96, 0.0)$);
\drawtileflatsmall{(c1)}{(c2)}{(c3)}{(c4)}{cg6}{cg6}{cg0}{cg0}
\coordinate (c1) at ($(g10)+(-0.2*0.96,-0.2*0.96)$);
\coordinate (c2) at ($(c1) +( 0.4*0.96, 0.0)$);
\coordinate (c3) at ($(c2) +( 0.0, 0.4*0.96)$);
\coordinate (c4) at ($(c3) +(-0.4*0.96, 0.0)$);
\drawtileflatsmall{(c1)}{(c2)}{(c3)}{(c4)}{cg0}{cg6}{cg2}{cg0}
\coordinate (c1) at ($(g20)+(-0.2*0.96,-0.2*0.96)$);
\coordinate (c2) at ($(c1) +( 0.4*0.96, 0.0)$);
\coordinate (c3) at ($(c2) +( 0.0, 0.4*0.96)$);
\coordinate (c4) at ($(c3) +(-0.4*0.96, 0.0)$);
\drawtileflatsmall{(c1)}{(c2)}{(c3)}{(c4)}{cg2}{cg6}{cg4}{cg0}
\coordinate (c1) at ($(g30)+(-0.2*0.96,-0.2*0.96)$);
\coordinate (c2) at ($(c1) +( 0.4*0.96, 0.0)$);
\coordinate (c3) at ($(c2) +( 0.0, 0.4*0.96)$);
\coordinate (c4) at ($(c3) +(-0.4*0.96, 0.0)$);
\drawtileflatsmall{(c1)}{(c2)}{(c3)}{(c4)}{cg4}{cg6}{cg6}{cg0}

\coordinate (c1) at ($(g01)+(-0.2*0.96,-0.2*0.96)$);
\coordinate (c2) at ($(c1) +( 0.4*0.96, 0.0)$);
\coordinate (c3) at ($(c2) +( 0.0, 0.4*0.96)$);
\coordinate (c4) at ($(c3) +(-0.4*0.96, 0.0)$);
\drawtileflatsmall{(c1)}{(c2)}{(c3)}{(c4)}{cg6}{cg0}{cg0}{cg2}
\coordinate (c1) at ($(g11)+(-0.2*0.96,-0.2*0.96)$);
\coordinate (c2) at ($(c1) +( 0.4*0.96, 0.0)$);
\coordinate (c3) at ($(c2) +( 0.0, 0.4*0.96)$);
\coordinate (c4) at ($(c3) +(-0.4*0.96, 0.0)$);
\drawtileflatsmall{(c1)}{(c2)}{(c3)}{(c4)}{cg0}{cg0}{cg2}{cg2}
\coordinate (c1) at ($(g21)+(-0.2*0.96,-0.2*0.96)$);
\coordinate (c2) at ($(c1) +( 0.4*0.96, 0.0)$);
\coordinate (c3) at ($(c2) +( 0.0, 0.4*0.96)$);
\coordinate (c4) at ($(c3) +(-0.4*0.96, 0.0)$);
\drawtileflatsmall{(c1)}{(c2)}{(c3)}{(c4)}{cg2}{cg0}{cg4}{cg2}
\coordinate (c1) at ($(g31)+(-0.2*0.96,-0.2*0.96)$);
\coordinate (c2) at ($(c1) +( 0.4*0.96, 0.0)$);
\coordinate (c3) at ($(c2) +( 0.0, 0.4*0.96)$);
\coordinate (c4) at ($(c3) +(-0.4*0.96, 0.0)$);
\drawtileflatsmall{(c1)}{(c2)}{(c3)}{(c4)}{cg4}{cg0}{cg6}{cg2}

\coordinate (c1) at ($(g02)+(-0.2*0.96,-0.2*0.96)$);
\coordinate (c2) at ($(c1) +( 0.4*0.96, 0.0)$);
\coordinate (c3) at ($(c2) +( 0.0, 0.4*0.96)$);
\coordinate (c4) at ($(c3) +(-0.4*0.96, 0.0)$);
\drawtileflatsmall{(c1)}{(c2)}{(c3)}{(c4)}{cg6}{cg2}{cg0}{cg4}
\coordinate (c1) at ($(g12)+(-0.2*0.96,-0.2*0.96)$);
\coordinate (c2) at ($(c1) +( 0.4*0.96, 0.0)$);
\coordinate (c3) at ($(c2) +( 0.0, 0.4*0.96)$);
\coordinate (c4) at ($(c3) +(-0.4*0.96, 0.0)$);
\drawtileflatsmall{(c1)}{(c2)}{(c3)}{(c4)}{cg0}{cg2}{cg2}{cg4}
\coordinate (c1) at ($(g22)+(-0.2*0.96,-0.2*0.96)$);
\coordinate (c2) at ($(c1) +( 0.4*0.96, 0.0)$);
\coordinate (c3) at ($(c2) +( 0.0, 0.4*0.96)$);
\coordinate (c4) at ($(c3) +(-0.4*0.96, 0.0)$);
\drawtileflatsmall{(c1)}{(c2)}{(c3)}{(c4)}{cg2}{cg2}{cg4}{cg4}
\coordinate (c1) at ($(g32)+(-0.2*0.96,-0.2*0.96)$);
\coordinate (c2) at ($(c1) +( 0.4*0.96, 0.0)$);
\coordinate (c3) at ($(c2) +( 0.0, 0.4*0.96)$);
\coordinate (c4) at ($(c3) +(-0.4*0.96, 0.0)$);
\drawtileflatsmall{(c1)}{(c2)}{(c3)}{(c4)}{cg4}{cg2}{cg6}{cg4}

\coordinate (c1) at ($(g03)+(-0.2*0.96,-0.2*0.96)$);
\coordinate (c2) at ($(c1) +( 0.4*0.96, 0.0)$);
\coordinate (c3) at ($(c2) +( 0.0, 0.4*0.96)$);
\coordinate (c4) at ($(c3) +(-0.4*0.96, 0.0)$);
\drawtileflatsmall{(c1)}{(c2)}{(c3)}{(c4)}{cg6}{cg4}{cg0}{cg6}
\coordinate (c1) at ($(g13)+(-0.2*0.96,-0.2*0.96)$);
\coordinate (c2) at ($(c1) +( 0.4*0.96, 0.0)$);
\coordinate (c3) at ($(c2) +( 0.0, 0.4*0.96)$);
\coordinate (c4) at ($(c3) +(-0.4*0.96, 0.0)$);
\drawtileflatsmall{(c1)}{(c2)}{(c3)}{(c4)}{cg0}{cg4}{cg2}{cg6}
\coordinate (c1) at ($(g23)+(-0.2*0.96,-0.2*0.96)$);
\coordinate (c2) at ($(c1) +( 0.4*0.96, 0.0)$);
\coordinate (c3) at ($(c2) +( 0.0, 0.4*0.96)$);
\coordinate (c4) at ($(c3) +(-0.4*0.96, 0.0)$);
\drawtileflatsmall{(c1)}{(c2)}{(c3)}{(c4)}{cg2}{cg4}{cg4}{cg6}
\coordinate (c1) at ($(g33)+(-0.2*0.96,-0.2*0.96)$);
\coordinate (c2) at ($(c1) +( 0.4*0.96, 0.0)$);
\coordinate (c3) at ($(c2) +( 0.0, 0.4*0.96)$);
\coordinate (c4) at ($(c3) +(-0.4*0.96, 0.0)$);
\drawtileflatsmall{(c1)}{(c2)}{(c3)}{(c4)}{cg4}{cg4}{cg6}{cg6}
\end{tikzpicture}

\bigskip

\begin{tikzpicture}[scale=2.1]
\drawtileflattm{(2.2,0)}{(3.2,0)}{(3.2,1)}{(2.2,1)}{$\leftsq t$}{$\downsq t$}{$\rightsq t$}{$\upsq t$}{$t$}
\node [] at (3.5+0.1,0.5) {$+$};
\coordinate (c1) at ($(4,0.5)+(-0.2*0.96,-0.2*0.96)$);
\coordinate (c2) at ($(c1) +( 0.4*0.96, 0.0)$);
\coordinate (c3) at ($(c2) +( 0.0, 0.4*0.96)$);
\coordinate (c4) at ($(c3) +(-0.4*0.96, 0.0)$);
\drawtileflatsmall{(c1)}{(c2)}{(c3)}{(c4)}{cg2}{cg0}{cg4}{cg2}
\node [below = 14pt] at (4,0.5) {$(R_3,U_2)$};
\node [] at (4.5-0.1,0.5) {$=$};
\drawtileflattm{(5-0.2,0)}{(6-0.2,0)}{(6-0.2,1)}{(5-0.2,1)}{$\leftsq t$}{$\downsq t$}{$\rightsq t$}{$\upsq t$}{}
\coordinate (c1) at ($(5.5-0.2,0.5)+(-0.2*0.96,-0.2*0.96)$);
\coordinate (c2) at ($(c1) +( 0.4*0.96, 0.0)$);
\coordinate (c3) at ($(c2) +( 0.0, 0.4*0.96)$);
\coordinate (c4) at ($(c3) +(-0.4*0.96, 0.0)$);
\drawtileflatsmall{(c1)}{(c2)}{(c3)}{(c4)}{cg2}{cg0}{cg4}{cg2}

\end{tikzpicture}
\caption{``Hidden'' tiling built-in the special $T_n$-tiling}
\label{fig:12}
\end{figure}
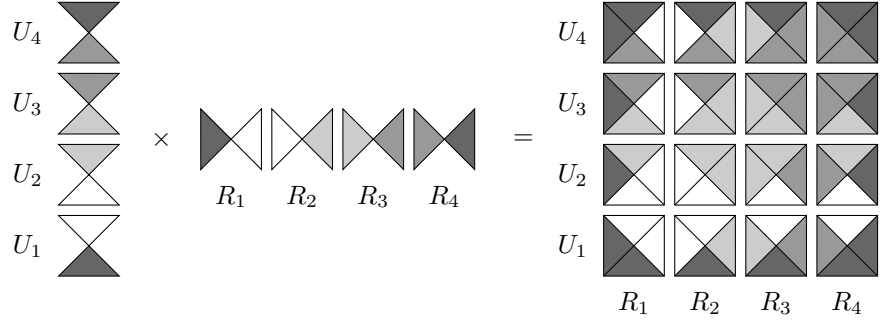

Next, we redefine the formulas $H_n(x,y)$, $V_n(x,y)$, $\mathit{TC}_1$, $\mathit{TC}_2$, and $\mathit{TC}_3$
%, and $\mathit{TC}_4$ 
by
$$
\begin{array}{rcl}
H'_n(x,y) 
  & = 
  & H_n(x,y) \wedge \mathit{RD}(x,y); \smallskip\\
V'_n(x,y) 
  & = 
  & V_n(x,y) \wedge \mathit{UD}(x,y); \smallskip\\
\mathit{TC}'_1 
  & = 
  & \forall x\exists y\, H'_n(x,y);
  \smallskip\\
\mathit{TC}'_2 
  & = 
  & \forall x\exists y\, V'_n(x,y);
  \smallskip\\
\mathit{TC}'_3 
  & = 
  & \forall x\forall y\,(\exists z\, (H'_n(x,z)\wedge V'_n(z,y)) \leftrightarrow \exists z\, (V'_n(x,z)\wedge H'_n(z,y)). 
  \smallskip\\
%\mathit{TC}'_4 
%  & = 
%  & \exists x\,(P_0(x)\wedge R_1(x) \wedge U_1(x)).
%  \smallskip\\
\end{array}
$$

Using the formulas defined above, we can redefine the formula $\mathit{Tiling}_n$ describing the special $T_n$-tiling. Let
$$
\begin{array}{rcl}
\mathit{Tiling}'_n 
  & = 
  & \mathit{TC}_0 
    \wedge \mathit{TC}'_1 
    \wedge \mathit{TC}'_2 
    \wedge \mathit{TC}'_3 
    \wedge \mathit{TC}_4
    \wedge \mathit{DSR}
    \wedge \mathit{DSU}.   
\end{array}
$$

We may think of $\mathit{Tiling}'_n$ as a formula describing two independent tilings: one of them is the special $T_n$-tiling, another is a ``hidden'' tiling generated by sixteen tile types each of which is defined by a pair $(R_i,U_j)$, where $i,j\in\{1,2,3,4\}$, see Figure~\ref{fig:12} (the edges of the tiles are marked by four colors instead of words).

Finally, we define a modification of $\mathit{Tiling}_n^{\mathbb{X}}$ by
$$
\begin{array}{lcl}
\mathit{MTiling}_n^{\mathbb{X}} 
  & = 
  & \mathit{Tiling}'_n \to \exists x\,P_1(x).
  \smallskip\\
%\mathit{Tiling}_n^{\mathbb{Y}} 
%  & = 
%  & \mathit{Tiling}_n \to \exists x\,P_2(x).
\end{array}
$$

\begin{lemma}
\label{lem:Trakhtenbrot:lem1:sib}
If $n\in\mathbb{X}$, then $\mathit{MTiling}_n^{\mathbb{X}}\in\logic{QCl}$.
\end{lemma}

\begin{proof}
Just follow, step by step, to the proof of Lemma~\ref{lem:Trakhtenbrot:lem1}.
\end{proof}

\begin{corollary}
\label{cor:lem:Trakhtenbrot:lem1:sib}
If $n\in\mathbb{X}$, then $\mathit{MTiling}_n^{\mathbb{X}}\in\logic{QCl}\uplus\bm{sib}$.
\end{corollary}

\begin{lemma}
\label{lem:Trakhtenbrot:lem2:sib}
If $n\in\mathbb{Y}$, then $\mathit{MTiling}_n^{\mathbb{X}}\not\in\logic{QCl}_{\mathit{fin}}\uplus\bm{sib}$.
\end{lemma}

\begin{proof}
Let us follow to the proof of Lemma~\ref{lem:Trakhtenbrot:lem2} indicating significant changes.

Let $n\in\mathbb{Y}$, $f_n(0,m)=t_2$, for some $m\in \numN$, and $r=\max\{m,n+1\}$.

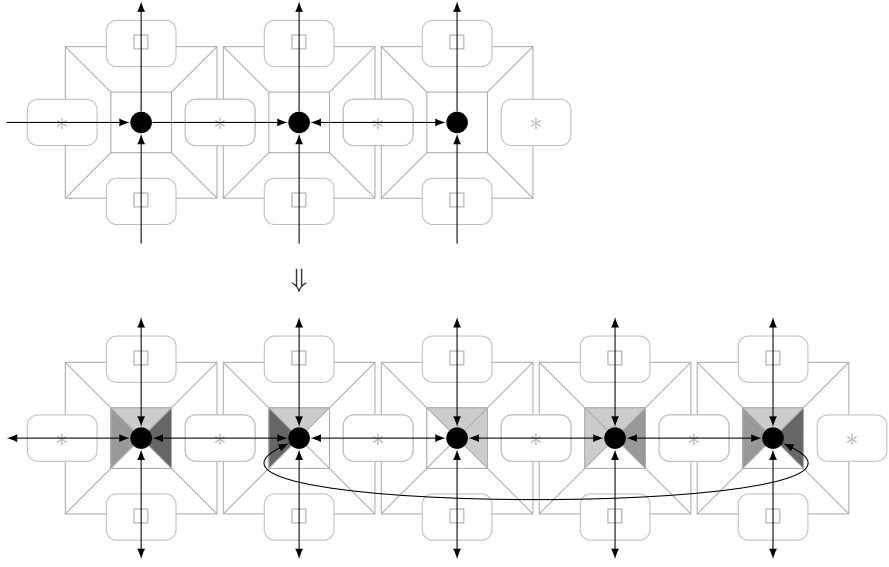
\begin{figure}
\centering
\begin{tikzpicture}[scale=2.1]

\begin{scope}[color=black!32]
\drawtileflattm{(2-0.5,2)}{(3-0.5,2)}{(3-0.5,3)}{(2-0.5,3)}{${\ast}$}{${\Box}$}{$\ast$}{$\Box$}{}
\drawtileflattm{(3-0.5,2)}{(4-0.5,2)}{(4-0.5,3)}{(3-0.5,3)}{${\ast}$}{${\Box}$}{$\ast$}{$\Box$}{}
\drawtileflattm{(4-0.5,2)}{(5-0.5,2)}{(5-0.5,3)}{(4-0.5,3)}{${\ast}$}{${\Box}$}{$\ast$}{$\Box$}{}
\drawtileflattm{(2-0.5,0)}{(3-0.5,0)}{(3-0.5,1)}{(2-0.5,1)}{${\ast}$}{${\Box}$}{$\ast$}{$\Box$}{}
\drawtileflattm{(3-0.5,0)}{(4-0.5,0)}{(4-0.5,1)}{(3-0.5,1)}{${\ast}$}{${\Box}$}{$\ast$}{$\Box$}{}
\drawtileflattm{(4-0.5,0)}{(5-0.5,0)}{(5-0.5,1)}{(4-0.5,1)}{${\ast}$}{${\Box}$}{$\ast$}{$\Box$}{}
\drawtileflattm{(5-0.5,0)}{(6-0.5,0)}{(6-0.5,1)}{(5-0.5,1)}{${\ast}$}{${\Box}$}{$\ast$}{$\Box$}{}
\drawtileflattm{(6-0.5,0)}{(7-0.5,0)}{(7-0.5,1)}{(6-0.5,1)}{${\ast}$}{${\Box}$}{$\ast$}{$\Box$}{}
\coordinate (c1) at ($(2,0.5)+(-0.2*0.96,-0.2*0.96)$);
\coordinate (c2) at ($(c1) +( 0.4*0.96, 0.0)$);
\coordinate (c3) at ($(c2) +( 0.0, 0.4*0.96)$);
\coordinate (c4) at ($(c3) +(-0.4*0.96, 0.0)$);
\drawtileflatsmall{(c1)}{(c2)}{(c3)}{(c4)}{cg4}{cg0}{cg6}{cg2}
\coordinate (c1) at ($(3,0.5)+(-0.2*0.96,-0.2*0.96)$);
\coordinate (c2) at ($(c1) +( 0.4*0.96, 0.0)$);
\coordinate (c3) at ($(c2) +( 0.0, 0.4*0.96)$);
\coordinate (c4) at ($(c3) +(-0.4*0.96, 0.0)$);
\drawtileflatsmall{(c1)}{(c2)}{(c3)}{(c4)}{cg6}{cg0}{cg0}{cg2}
\coordinate (c1) at ($(4,0.5)+(-0.2*0.96,-0.2*0.96)$);
\coordinate (c2) at ($(c1) +( 0.4*0.96, 0.0)$);
\coordinate (c3) at ($(c2) +( 0.0, 0.4*0.96)$);
\coordinate (c4) at ($(c3) +(-0.4*0.96, 0.0)$);
\drawtileflatsmall{(c1)}{(c2)}{(c3)}{(c4)}{cg0}{cg0}{cg2}{cg2}
\coordinate (c1) at ($(5,0.5)+(-0.2*0.96,-0.2*0.96)$);
\coordinate (c2) at ($(c1) +( 0.4*0.96, 0.0)$);
\coordinate (c3) at ($(c2) +( 0.0, 0.4*0.96)$);
\coordinate (c4) at ($(c3) +(-0.4*0.96, 0.0)$);
\drawtileflatsmall{(c1)}{(c2)}{(c3)}{(c4)}{cg2}{cg0}{cg4}{cg2}
\coordinate (c1) at ($(6,0.5)+(-0.2*0.96,-0.2*0.96)$);
\coordinate (c2) at ($(c1) +( 0.4*0.96, 0.0)$);
\coordinate (c3) at ($(c2) +( 0.0, 0.4*0.96)$);
\coordinate (c4) at ($(c3) +(-0.4*0.96, 0.0)$);
\drawtileflatsmall{(c1)}{(c2)}{(c3)}{(c4)}{cg4}{cg0}{cg6}{cg2}
\end{scope}

\node [] at (3.0,1.5) {$\Downarrow$};

\begin{scope}[color=black, thick]
\filldraw     [] (2.5-0.5,2.5) circle [radius=1.75pt];
\filldraw     [] (3.5-0.5,2.5) circle [radius=1.75pt];
\filldraw     [] (4.5-0.5,2.5) circle [radius=1.75pt];
\filldraw [] (2.5-0.5,0.5) circle [radius=1.75pt];
\filldraw [] (3.5-0.5,0.5) circle [radius=1.75pt];
\filldraw [] (4.5-0.5,0.5) circle [radius=1.75pt];
\filldraw [] (5.5-0.5,0.5) circle [radius=1.75pt];
\filldraw [] (6.5-0.5,0.5) circle [radius=1.75pt];
\end{scope}

\begin{scope}[>=latex, <->, shorten >= 4.25pt, shorten <= -16pt, color=black]
\draw [->, shorten <= -21pt] (2.0-0.5,2.5)--(2.5-0.5,2.5);
\draw [->,                 ] (2.5-0.5,2.0)--(2.5-0.5,2.5);
\draw [->,                 ] (3.5-0.5,2.0)--(3.5-0.5,2.5);
\draw [->,                 ] (4.5-0.5,2.0)--(4.5-0.5,2.5);
\draw [->,  shorten <= 4.25pt] (2.5-0.5,2.5)--(3.5-0.5,2.5);
\draw [<->, shorten <= 4.25pt] (3.5-0.5,2.5)--(4.5-0.5,2.5);
\draw [shorten <= -21pt    ] (2.0-0.5,0.5)--(2.5-0.5,0.5);
\draw [                    ] (2.5-0.5,0.0)--(2.5-0.5,0.5);
\draw [                    ] (3.5-0.5,0.0)--(3.5-0.5,0.5);
\draw [                    ] (4.5-0.5,0.0)--(4.5-0.5,0.5);
\draw [                    ] (5.5-0.5,0.0)--(5.5-0.5,0.5);
\draw [                    ] (6.5-0.5,0.0)--(6.5-0.5,0.5);
\draw [<->, shorten <= 4.25pt] (2.5-0.5,0.5)--(3.5-0.5,0.5);
\draw [<->, shorten <= 4.25pt] (3.5-0.5,0.5)--(4.5-0.5,0.5);
\draw [<->, shorten <= 4.25pt] (4.5-0.5,0.5)--(5.5-0.5,0.5);
\draw [<->, shorten <= 4.25pt] (5.5-0.5,0.5)--(6.5-0.5,0.5);
\draw [<->, shorten <= 4.25pt] (3.5-0.5,0.5) .. controls (2,0) and (7,0) .. (6.5-0.5,0.5);
\end{scope}

\begin{scope}[>=latex, <->, shorten >= -16pt, shorten <= 4.25pt, color=black]
\draw [->] (2.5-0.5,2.5)--(2.5-0.5,3.0);
\draw [->] (3.5-0.5,2.5)--(3.5-0.5,3.0);
\draw [->] (4.5-0.5,2.5)--(4.5-0.5,3.0);
\draw [  ] (2.5-0.5,0.5)--(2.5-0.5,1.0);
\draw [  ] (3.5-0.5,0.5)--(3.5-0.5,1.0);
\draw [  ] (4.5-0.5,0.5)--(4.5-0.5,1.0);
\draw [  ] (5.5-0.5,0.5)--(5.5-0.5,1.0);
\draw [  ] (6.5-0.5,0.5)--(6.5-0.5,1.0);
\end{scope}

\end{tikzpicture}
\caption{Replacement of cycles}
\label{fig:13}
\end{figure}

We define a model similar to the model in the proof of Lemma~\ref{lem:Trakhtenbrot:lem2} but with cycles of length~$4$ instead of length~$2$ (see Figure~\ref{fig:13} for ``horizontal'' cycles). Let
$$
\begin{array}{lcl}
D & = & \{0,\ldots,r+4\}\times\{0,\ldots,r+4\}. \\
\end{array}
$$
Define a model $\cModel{M} = \otuple{\mathcal{D},\mathcal{I}}$ by 
$$
\begin{array}{lcl}
\cModel{M}\models P(\otuple{i,j},\otuple{i',j'}) 
  & \iff 
  & \mbox{either $|i'-i|=1$ and $j'=j$,}
  \\
  &&\mbox{or $i'=i$ and $|j'-j|=1$,}
  \\
  &&\mbox{or $\{i',i\}=\{r+1,r+4\}$ and $j'=j$,}
  \\
  &&\mbox{or $i'=i$ and $\{j',j\}=\{r+1,r+4\}$;}
  \smallskip\\
\cModel{M}\models P_k(\otuple{i,j}) 
  & \iff 
  & f_n(i,j) = t_k^n;
  \smallskip\\
\cModel{M}\models R_k(\otuple{i,j}) 
  & \iff 
  & |i-k+1|\mathrel{\vdots} 4;
  \smallskip\\
\cModel{M}\models U_k(\otuple{i,j}) 
  & \iff 
  & |j-k+1|\mathrel{\vdots} 4,
\end{array}
$$
see Figure~\ref{fig:14}.

Then, by the definition of $\cModel{M}$, we obtain $\cModel{M}\models\mathit{Tiling}'_n$. Since $n\not\in\mathbb{X}$, there is no $i,j\in\numN$ such that $f_n(i,j)=t_1$, therefore, $\cModel{M}\not\models\exists x\,P_1(x)$.

Notice that $\cModel{M}\models\bm{sib}$.
Hence, $\mathit{MTiling}_n^{\mathbb{X}}\not\in\logic{QCl}_{\mathit{fin}}\uplus\bm{sib}$.
\end{proof}

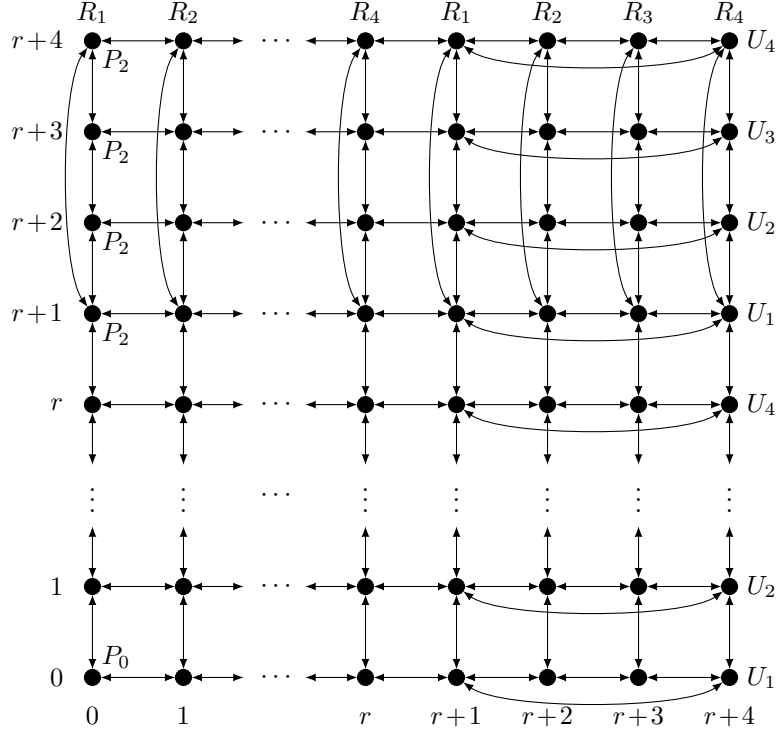
\begin{figure}
\centering
\begin{tikzpicture}[scale=1.21]

\foreach \x in {0,1,3,4,5,6,7}
\foreach \y in {0,1,3,4,5,6,7}
\filldraw [] (\x,\y) circle [radius=2.5pt];

\foreach \x in {0,1,3,4,5,6,7}
\foreach \y in {0,3,4,5,6}
\draw [>=latex, <->, shorten >= 3pt, shorten <= 3pt] (\x,\y)--(\x,\y+1);

\foreach \y in {0,1,3,4,5,6,7}
\foreach \x in {0,3,4,5,6}
\draw [>=latex, <->, shorten >= 3pt, shorten <= 3pt] (\x,\y)--(\x+1,\y);

\foreach \x in {0,1,3,4,5,6,7}   \node [] at (\x,2.04) {$\vdots$};
\foreach \y in {0,1,2,3,4,5,6,7} \node [] at (2.04,\y) {$\cdots$};

\foreach \t in {0,1,3,4,5,6,7} 
{
\draw [>=latex, <->, shorten >= 3pt, shorten <= 3pt] (\t,1)--(\t,1.75);
\draw [>=latex, <->, shorten >= 3pt, shorten <= 3pt] (\t,3)--(\t,2.25);
\draw [>=latex, <->, shorten >= 3pt, shorten <= 3pt] (1,\t)--(1.75,\t);
\draw [>=latex, <->, shorten >= 3pt, shorten <= 3pt] (3,\t)--(2.25,\t);
\draw [>=latex, <->, shorten >= 3pt, shorten <= 3pt] (\t,7) .. controls (\t-0.36,6.5) and (\t-0.36,4.5) .. (\t,4);
\draw [>=latex, <->, shorten >= 3pt, shorten <= 3pt] (7,\t) .. controls (6.5,\t-0.36) and (4.5,\t-0.36) .. (4,\t);
}

\node [below = 7.5pt] at (0,0) {$0$};
\node [below = 7.5pt] at (1,0) {$1$};
\node [below = 7.5pt] at (3,0) {$\phantom{1}r\phantom{1}$};
\node [below = 7.5pt] at (4,0) {$r\hspace{1pt}{+}\hspace{1pt}1$};
\node [below = 7.5pt] at (5,0) {$r\hspace{1pt}{+}\hspace{1pt}2$};
\node [below = 7.5pt] at (6,0) {$r\hspace{1pt}{+}\hspace{1pt}3$};
\node [below = 7.5pt] at (7,0) {$r\hspace{1pt}{+}\hspace{1pt}4$};
\node [left  = 7.5pt] at (0,0) {$0$};
\node [left  = 7.5pt] at (0,1) {$1$};
\node [left  = 7.5pt] at (0,3) {$r$};
\node [left  = 7.5pt] at (0,4) {$r\hspace{1pt}{+}\hspace{1pt}1$};
\node [left  = 7.5pt] at (0,5) {$r\hspace{1pt}{+}\hspace{1pt}2$};
\node [left  = 7.5pt] at (0,6) {$r\hspace{1pt}{+}\hspace{1pt}3$};
\node [left  = 7.5pt] at (0,7) {$r\hspace{1pt}{+}\hspace{1pt}4$};

\node [above right] at (0,0) {$P_0$};
\node [below right] at (0,7) {$P_2$};
\node [below right] at (0,6) {$P_2$};
\node [below right] at (0,5) {$P_2$};
\node [below right] at (0,4) {$P_2$};

\node [above=3pt] at (0,7) {$R_1$};
\node [above=3pt] at (1,7) {$R_2$};
\node [above=3pt] at (3,7) {$R_4$};
\node [above=3pt] at (4,7) {$R_1$};
\node [above=3pt] at (5,7) {$R_2$};
\node [above=3pt] at (6,7) {$R_3$};
\node [above=3pt] at (7,7) {$R_4$};
\node [right=3pt] at (7,0) {$U_1$};
\node [right=3pt] at (7,1) {$U_2$};
\node [right=3pt] at (7,3) {$U_4$};
\node [right=3pt] at (7,4) {$U_1$};
\node [right=3pt] at (7,5) {$U_2$};
\node [right=3pt] at (7,6) {$U_3$};
\node [right=3pt] at (7,7) {$U_4$};

\end{tikzpicture}
\caption{Finite sib-model for the special $T_n$-tiling with $n\in\mathbb{Y}$ (here $r+1\mathrel{\vdots}4$)}
\label{fig:14}
\end{figure}

\begin{remark}
\label{rem:lem:Trakhtenbrot:lem2:sib}
Observe that the length of every $P$-cycle in the model $\cModel{M}$ defined in the proof of Lemma~\ref{lem:Trakhtenbrot:lem2:sib} is even.
\end{remark}

\begin{corollary}
\label{cor:lem:Trakhtenbrot:lem2:sib}
If $n\in\mathbb{Y}$, then $\mathit{MTiling}_n^{\mathbb{X}}\not\in\logic{QCl}_{\mathit{fin}}$.
\end{corollary}

\begin{remark}
Theorem~\ref{th:Trakhtenbrot} follows from Lemma~\ref{lem:Trakhtenbrot:lem1:sib} and Corollary~\ref{cor:lem:Trakhtenbrot:lem2:sib}.
\end{remark}

\begin{corollary}
\label{cor:G:lem:Trakhtenbrot:lem2:sib}
If $n\in\mathbb{Y}$, then $(\mathit{MTiling}_n^{\mathbb{X}})_G^{\phantom{i}}\not\in\logic{QCl}_{\mathit{fin}}\uplus\bm{sib}$.
\end{corollary}

\begin{proof}
It is sufficient to expand the model $\cModel{M}$ defined in the proof of Lemma~\ref{lem:Trakhtenbrot:lem2:sib} so that the formula $\forall x\,G(x)$ becomes true.
\end{proof}

To simulate unary predicate letters $P_m$, $R_i$, and $U_j$, we define formulas similar to $\mathit{tile}_k(x)$. 
Let
$$
\begin{array}{lcl}
\pi_0(x) 
  & = 
  & G(x); %\mathit{grid}'(x);
  \smallskip\\
\pi_1^{y}(x) 
  & = 
  & \exists y\,(P(x,y)\wedge G(y)); %\mathit{grid}'(y));
  \smallskip\\
\pi_{k+1}^{y}(x) 
  & = 
  & \exists y\,(P(x,y)\wedge \pi^x_k(y)),
  \smallskip\\
\end{array}
$$
where $k\in\numNp$. Informally, $\pi_k^y(x)$ says that there is a $P$-path of length $k$ from $x$ to some element of a grid. For every $k\in\numNpp$, define
$$
\begin{array}{lcl}
\mathit{path}_{k}^y(x) 
  & = 
  & \pi_k^y(x) \wedge \neg \pi_{k-1}^y(x) \wedge \neg \pi_{k-2}^y(x) 
    \wedge \forall y\,(P(x,y)\to \pi_{k-1}^x(y)).
\end{array}
$$
Formula $\mathit{path}_{k}(x)$ says that the length of a shortest $P$-path from $x$ to an element of a grid is~$k$ and, moreover, every element $P$-adjacent to $x$ is closer to the grid than~$x$. Next, let
$$
\begin{array}{lcl}
\sigma^y_{k,0}(x) 
  & = 
  & \mathit{path}_{k}^y(x);
  \smallskip\\
\sigma_{k,i+1}^{y}(x) 
  & = 
  & \exists y\,(P(x,y)\wedge \sigma_{k,i}^{x}(y)),
  \smallskip\\
\end{array}
$$
where $k\in\numNpp$ and $i\in\numN$. The formula $\sigma^y_{k,i}(x)$ says that there is a $P$-path of length~$i$ from~$x$ to an element such that a shortest path from it to a grid is of length~$k$ and every element $P$-adjacent to it is closer to the grid. Define
$$
\begin{array}{lcl}
\mathit{tile}'_k(x) 
  & = 
  & \sigma_{k+2,k+2}^y(x);
  \smallskip\\
\mathit{tile}'_k(y) 
  & = 
  & \sigma_{k+2,k+2}^x(y);
  \smallskip\\
\mathit{tile}'_k(z) 
  & = 
  & \sigma_{k+2,k+2}^x(z),
\end{array}
$$
where $k\in\numN$.

\begin{remark}
\label{rem:tile:x&y:2}
Formulas $\mathit{tile}'_k(x)$ and $\mathit{tile}'_k(y)$ contain no variables except $x$ and~$y$; also they contain no predicate letters except $P$ and~$G$.
\end{remark}

Now, let us define formulas simulating $G(u)$ with $u\in\{x,y,z\}$; the formulas are similar to $\mathit{grid}(u)$ but slightly more complicated and contain three variables:
$$
\begin{array}{rcl}
\mathit{triangle}^{yz}(x) 
  & = 
  & \exists y\exists z\,(P(x,y)\wedge P(y,z)\wedge P(z,x));
  \smallskip\\
\mathit{grid}'(x) 
  & = 
  & \neg \mathit{triangle}^{yz}(x) 
    \wedge \exists y\,(P(x,y)\wedge \mathit{triangle}^{xz}(y));
  \smallskip\\
\mathit{grid}'(y) 
  & = 
  & \neg \mathit{triangle}^{xz}(y) 
    \wedge \exists y\,(P(y,x)\wedge \mathit{triangle}^{yz}(x));
  \smallskip\\
\mathit{grid}'(z) 
  & = 
  & \neg \mathit{triangle}^{xy}(z) 
    \wedge \exists y\,(P(z,x)\wedge \mathit{triangle}^{yz}(x)).
\end{array}
$$
%see Figure~\ref{fig:15}. 

Let us define a function~$S_1$, similar to~$S_0$, associating with each formula $\mathit{MTiling}_n^{\mathbb{X}}$ a formula containing no predicate letters except~$P$. Define $S_1\mathit{MTiling}_n^{\mathbb{X}}$ as the formula obtained from $\exists x\,G(x)\to (\mathit{MTiling}_n^{\mathbb{X}})_G^{\phantom{i}}$ by replacing 
\begin{itemize}
\item
each occurrence of $P_m(x)$, $P_m(y)$, and $P_m(z)$ with $\mathit{tile}'_m(x)$, $\mathit{tile}'_m(y)$, and $\mathit{tile}'_m(z)$, respectively, where $m\in\{0,\ldots,k_n\}$;
\item 
then each occurrence of $R_i(x)$, $R_i(y)$, $R_i(z)$ with $\mathit{tile}'_{k_n+i}(x)$, $\mathit{tile}'_{k_n+i}(y)$, $\mathit{tile}'_{k_n+i}(z)$, respectively, where $1\leqslant i\leqslant 4$;
\item
then each occurrence of $U_j(x)$, $U_j(y)$, $U_j(z)$ with $\mathit{tile}'_{k_n+j+4}(x)$, $\mathit{tile}'_{k_n+j+4}(y)$, $\mathit{tile}'_{k_n+j+4}(z)$, respectively, where $1\leqslant j\leqslant 4$;
\item
and then each occurrence of $G(x)$, $G(y)$, and $G(z)$ with 
$\mathit{grid}'(x)$, $\mathit{grid}'(y)$, and $\mathit{grid}'(z)$, respectively. 
\end{itemize}
Notice that these replacements are formula substitutions.\footnote{See footnote~\ref{footnote:1}.}

\begin{lemma}
\label{lem:Trakhtenbrot:lem1:sib:binP}
If $n\in\mathbb{X}$, then $S_1\mathit{MTiling}_n^{\mathbb{X}}\in\logic{QCl}$.
\end{lemma}

\begin{proof}
Follows from Lemmas~\ref{lem:Trakhtenbrot:lem1:sib} and~\ref{lem:relativization:G} taking into account that $\logic{QCl}$ is closed by Substitution.
\end{proof}

%\begin{corollary}
%\label{cor:lem:Trakhtenbrot:lem1:sib:binP}
%If $n\in\mathbb{X}$, then $S_1\mathit{MTiling}_n^{\mathbb{X}}\in\logic{SIB}$.
%\end{corollary}

\begin{lemma}
\label{lem:Trakhtenbrot:lem2:sib:binP}
If $n\in\mathbb{Y}$, then $S_1\mathit{MTiling}_n^{\mathbb{X}}\not\in\logic{SIB}_{\mathit{fin}}$.
\end{lemma}

\begin{proof}
Let $n\in\mathbb{Y}$. Then $\cModel{M}\not\models\mathit{Tiling}_n^{\mathbb{X}}$, where $\cModel{M}=\otuple{\mathcal{D},\mathcal{I}}$ is the model defined in the proof of Lemma~\ref{lem:Trakhtenbrot:lem2:sib}. Without a loss of generality, we may assume that $\cModel{M}\models\forall x\,G(x)$, see the proof of Corollary~\ref{cor:G:lem:Trakhtenbrot:lem2:sib}. Then,
$$
\begin{array}{c}
\cModel{M} ~\not\models~ \exists x\,G(x)\to (\mathit{MTiling}_n^{\mathbb{X}})_G^{\phantom{i}}.
\end{array}
$$  
We extend $\cModel{M}$ to a finite model refuting $S_1\mathit{MTiling}_n^{\mathbb{X}}$.

\begin{figure}
\centering
\begin{tikzpicture}[scale=3.2, rectnode/.style={rectangle, thick, draw=black!60, dashed, rounded corners = 2pt}]

\coordinate (c1) at (-1.4,0.30); %(-1.4, 0.30);
\coordinate (c2) at (-0.9,-0.3); %(-1.0,-0.3);
\coordinate (c3) at (0.2,-0.05); %( 0.0,-0.05);
\coordinate (c4) at (-0.4,0.5); %(-0.5, 0.50);

\coordinate (c12) at ($0.5*(c1)+0.5*(c2)$);
\coordinate (c23) at ($0.5*(c2)+0.5*(c3)$);
\coordinate (c34) at ($0.5*(c3)+0.5*(c4)$);
\coordinate (c41) at ($0.5*(c4)+0.5*(c1)$);

\coordinate (vecX1) at ($(c3)-(c2)$);
\coordinate (vecX2) at ($(c4)-(c1)$);
\coordinate (vecY1) at ($(c1)-(c2)$);
\coordinate (vecY2) at ($(c4)-(c3)$);

\coordinate (c1r) at (c4);
\coordinate (c2r) at (c3);
\coordinate (c3r) at ($(c3)+0.96*(vecX1)$);
\coordinate (c4r) at ($(c4)+0.96*(vecX2)$);

\coordinate (c1u) at ($(c1)+0.84*(vecY1)$);
\coordinate (c2u) at (c1);
\coordinate (c3u) at (c4);
\coordinate (c4u) at ($(c4)+0.84*(vecY2)$);

\coordinate (c12r) at ($0.5*(c1r)+0.5*(c2r)$);
\coordinate (c23r) at ($0.5*(c2r)+0.5*(c3r)$);
\coordinate (c34r) at ($0.5*(c3r)+0.5*(c4r)$);
\coordinate (c41r) at ($0.5*(c4r)+0.5*(c1r)$);

\coordinate (c12u) at ($0.5*(c1u)+0.5*(c2u)$);
\coordinate (c23u) at ($0.5*(c2u)+0.5*(c3u)$);
\coordinate (c34u) at ($0.5*(c3u)+0.5*(c4u)$);
\coordinate (c41u) at ($0.5*(c4u)+0.5*(c1u)$);

\draw [white, opacity = 0, name path = dg 1] (c1)--(c3);
\draw [white, opacity = 0, name path = dg 2] (c2)--(c4);
\draw [name intersections = {of = dg 1 and dg 2, by = {c}}];
\draw [white, opacity = 0, name path = dg 3] (c12)--(c34);
\draw [white, opacity = 0, name path = dg 4] (c23)--(c41);
\draw [name intersections = {of = dg 3 and dg 4, by = {d}}];

\draw [white, opacity = 0, name path = dgr 1] (c1r)--(c3r);
\draw [white, opacity = 0, name path = dgr 2] (c2r)--(c4r);
\draw [name intersections = {of = dgr 1 and dgr 2, by = {cr}}];
\draw [white, opacity = 0, name path = dgr 3] (c12r)--(c34r);
\draw [white, opacity = 0, name path = dgr 4] (c23r)--(c41r);
\draw [name intersections = {of = dgr 3 and dgr 4, by = {dr}}];

\draw [white, opacity = 0, name path = dgu 1] (c1u)--(c3u);
\draw [white, opacity = 0, name path = dgu 2] (c2u)--(c4u);
\draw [name intersections = {of = dgu 1 and dgu 2, by = {cu}}];
\draw [white, opacity = 0, name path = dgu 3] (c12u)--(c34u);
\draw [white, opacity = 0, name path = dgu 4] (c23u)--(c41u);
\draw [name intersections = {of = dgu 3 and dgu 4, by = {du}}];

\coordinate (diffR) at ($(cr)-(c)$);
\coordinate (diffU) at ($(cu)-(c)$);
\coordinate (corR) at (0,0.012);
\coordinate (corU) at (0,0.032);

\coordinate (em)  at ($(c)+1*(0,0.29)$);
\coordinate (e2)  at ($(c)+2*(0,0.29)$);
\coordinate (e1)  at ($(c)+3*(0,0.29)$);
\coordinate (e0)  at ($(c)+4*(0,0.29)$);
\coordinate (e0') at ($(e0)+ (0,0.21)$);

\coordinate (em') at ($(em)+(0.2,0.048)$);
\coordinate (e2') at ($(e2)+(0.2,0.048)$);
\coordinate (e1') at ($(e1)+(0.2,0.048)$);

\coordinate (ekn) at ($(c)-1*(0,0.29)$);
\coordinate (ek2) at ($(c)-2*(-0.28*0.25,0.285*0.85)$);
\coordinate (ek1) at ($(c)-2*(+0.26*0.25,0.325*0.85)$);
%\coordinate (ek0) at ($(c)-4*(0,0.31)$);

\coordinate (rem)  at ($(cr)+1*(0,0.29)-1*(corR)$);
\coordinate (re2)  at ($(cr)+2*(0,0.29)-2*(corR)$);
\coordinate (re1)  at ($(cr)+3*(0,0.29)-3*(corR)$);
\coordinate (re0)  at ($(cr)+4*(0,0.29)-4*(corR)$);
\coordinate (re0') at ($(re0)+ (0,0.21)-1*(corR)$);

\coordinate (rem') at ($(rem)+0.96*(0.2,0.048)$);
\coordinate (re2') at ($(re2)+0.96*(0.2,0.048)$);
\coordinate (re1') at ($(re1)+0.96*(0.2,0.048)$);

\coordinate (rekn) at ($(cr)-1*(0,0.29)+1*(corR)$);
\coordinate (rek2) at ($(cr)-2*(-0.28*0.25,0.285*0.85)+1.7*(corR)$);
\coordinate (rek1) at ($(cr)-2*(+0.26*0.25,0.325*0.85)+1.7*(corR)$);
%\coordinate (rek0) at ($(c)-4*(0,0.31)$);

\coordinate (uem)  at ($(cu)+1*(0,0.29)-1*(corU)$);
\coordinate (ue2)  at ($(cu)+2*(0,0.29)-2*(corU)$);
\coordinate (ue1)  at ($(cu)+3*(0,0.29)-3*(corU)$);
\coordinate (ue0)  at ($(cu)+4*(0,0.29)-4*(corU)$);
\coordinate (ue0') at ($(ue0)+ (0,0.21)-1*(corU)$);

\coordinate (uem') at ($(uem)+0.96*(0.2,0.040)$);
\coordinate (ue2') at ($(ue2)+0.96*(0.2,0.040)$);
\coordinate (ue1') at ($(ue1)+0.96*(0.2,0.040)$);

\coordinate (uekn) at ($(cu)-1*(0,0.29)+1*(corU)$);
\coordinate (uek2) at ($(cu)-2*(-0.28*0.25,0.285*0.85)+1.7*(corU)$);
\coordinate (uek1) at ($(cu)-2*(+0.26*0.25,0.320*0.85)+1.7*(corU)$);
%\coordinate (uek0) at ($(c)-4*(0,0.31)$);

\shade [ball color=black] (ek2) circle [radius = 1.5pt];
\draw  [>=latex, <->, shorten >= 3.5pt, shorten <= 3.5pt, color=black] (ekn)--(ek2);
\shade [ball color=black] (ekn) circle [radius = 1.5pt];
\draw  [>=latex, <->, shorten >= 3pt, shorten <= 3pt, color=black] (c)--(ekn);
\draw  [>=latex, <->, shorten >= 3.5pt, shorten <= 3.5pt, color=black] (ek2)--(ek1);
\draw  [>=latex, <->, shorten >= 3.5pt, shorten <= 3.5pt, color=black] (ekn)--(ek1);
\shade [ball color=black] (ek1) circle [radius = 1.5pt];

\begin{scope}[color=black!42]
\shade [ball color=black!84] (rek2) circle [radius = 0.96*1.5pt];
\filldraw [color=white, opacity = 0.5] (rek2) circle [radius = 0.96*1.55pt];
\draw  [>=latex, <->, shorten >= 3.5pt, shorten <= 3.5pt] (rekn)--(rek2);
\shade [ball color=black!84] (rekn) circle [radius = 0.96*1.5pt];
\filldraw [color=white, opacity = 0.5] (rekn) circle [radius = 0.96*1.55pt];
\draw  [>=latex, <->, shorten >= 3pt, shorten <= 3pt] (cr)--(rekn);
\draw  [>=latex, <->, shorten >= 3.5pt, shorten <= 3.5pt] (rek2)--(rek1);
\draw  [>=latex, <->, shorten >= 3.5pt, shorten <= 3.5pt] (rekn)--(rek1);
\shade [ball color=black!84] (rek1) circle [radius = 0.96*1.5pt];
\filldraw [color=white, opacity = 0.5] (rek1) circle [radius = 0.96*1.55pt];
\shade [ball color=black!84] (uek2) circle [radius = 0.90*1.5pt];
\filldraw [color=white, opacity = 0.5] (uek2) circle [radius = 0.90*1.55pt];
\draw  [>=latex, <->, shorten >= 3.5pt, shorten <= 3.5pt] (uekn)--(uek2);
\shade [ball color=black!84] (uekn) circle [radius = 0.90*1.5pt];
\filldraw [color=white, opacity = 0.5] (uekn) circle [radius = 0.90*1.55pt];
\draw  [>=latex, <->, shorten >= 3pt, shorten <= 3pt] (cu)--(uekn);
\draw  [>=latex, <->, shorten >= 3.5pt, shorten <= 3.5pt] (uek2)--(uek1);
\draw  [>=latex, <->, shorten >= 3.5pt, shorten <= 3.5pt] (uekn)--(uek1);
\shade [ball color=black!84] (uek1) circle [radius = 0.90*1.5pt];
\filldraw [color=white, opacity = 0.5] (uek1) circle [radius = 0.90*1.55pt];
\end{scope}

\begin{scope}[color=black!42]
\drawtileflattmslanted{(c1)}{(c2)}{(c3)}{(c4)}
\end{scope}

\begin{scope}[color=black!25]
\drawtileflattmslanted{(c1r)}{(c2r)}{(c3r)}{(c4r)}
\drawtileflattmslanted{(c1u)}{(c2u)}{(c3u)}{(c4u)}
\end{scope}

\begin{scope}[>=latex, <->, shorten >= -7.5pt, shorten <= 4pt, color=black!32]
\draw [] (cr)--($(c34r)+(cr)-(dr)$);
\draw [] (cr)--($(c41r)+(cr)-(dr)$);
\draw [] (cu)--($(c34u)+(cu)-(du)$);
\draw [] (cu)--($(c41u)+(cu)-(du)$);
\end{scope}

\shade [ball color=black!64] (cr) circle [radius = 0.96*1.5pt];
\filldraw [color=white, opacity = 0.5] (cr) circle [radius = 0.96*1.55pt];
\shade [ball color=black!64] (cu) circle [radius = 0.90*1.5pt];
\filldraw [color=white, opacity = 0.5] (cu) circle [radius = 0.90*1.55pt];

\begin{scope}[>=latex, <->, shorten >= 3pt, shorten <= -7.5pt, color=black!32]
\draw [] ($(c12u)+(cu)-(du)$)--(cu);
\draw [] ($(c23r)+(cr)-(dr)$)--(cr);
\end{scope}

\begin{scope}[>=latex, <->, shorten >= 3pt, shorten <= 2pt, color=black!84]
\draw [] (c)--(cr); %($(c34)+(c)-(d)$);
\draw [] (c)--(cu); %($(c41)+(c)-(d)$);
\end{scope}

\shade [ball color=black!64] (c) circle [radius = 1.5pt];

\draw  [>=latex, <->, shorten >= 3pt, shorten <= 3pt, color=black, densely dashed] (c)--(em);
\shade [ball color=black] (em') circle [radius = 1.5pt];
\draw  [>=latex, <->, shorten >= 3.5pt, shorten <= 3.5pt, color=black] (em)--(em');
\shade [ball color=black] (em) circle [radius = 1.5pt];
\draw  [>=latex, <->, shorten >= 3pt, shorten <= 3pt, color=black, densely dashed] (em)--(e2);
\shade [ball color=black] (e2') circle [radius = 1.5pt];
\draw  [>=latex, <->, shorten >= 3.5pt, shorten <= 3.5pt, color=black] (e2)--(e2');
\shade [ball color=black] (e2) circle [radius = 1.5pt];
\draw  [>=latex, <->, shorten >= 3pt, shorten <= 3pt, color=black, densely dashed] (e2)--(e1);
\shade [ball color=black] (e1') circle [radius = 1.5pt];
\draw  [>=latex, <->, shorten >= 3.5pt, shorten <= 3.5pt, color=black] (e1)--(e1');
\shade [ball color=black] (e1) circle [radius = 1.5pt];
\draw  [>=latex, <->, shorten >= 3pt, shorten <= 3pt, color=black, densely dashed] (e1)--(e0);
\shade [ball color=black] (e0) circle [radius = 1.5pt];
\draw  [>=latex, <->, shorten >= 3pt, shorten <= 3pt, color=black] (e0)--(e0');
\shade [ball color=black] (e0') circle [radius = 1.5pt];

\begin{scope}[color=black!42]
\draw  [>=latex, <->, shorten >= 3pt, shorten <= 3pt, densely dashed] (cr)--(rem);
\shade [ball color=black!84] (rem') circle [radius = 0.96*1.5pt];
\filldraw [color=white, opacity = 0.5] (rem') circle [radius = 0.96*1.55pt];
\draw  [>=latex, <->, shorten >= 3.5pt, shorten <= 3.5pt] (rem)--(rem');
\shade [ball color=black!84] (rem) circle [radius = 0.96*1.5pt];
\filldraw [color=white, opacity = 0.5] (rem) circle [radius = 0.96*1.55pt];
\draw  [>=latex, <->, shorten >= 3pt, shorten <= 3pt, densely dashed] (rem)--(re2);
\shade [ball color=black!84] (re2') circle [radius = 0.96*1.5pt];
\filldraw [color=white, opacity = 0.5] (re2') circle [radius = 0.96*1.55pt];
\draw  [>=latex, <->, shorten >= 3.5pt, shorten <= 3.5pt] (re2)--(re2');
\shade [ball color=black!84] (re2) circle [radius = 0.96*1.5pt];
\filldraw [color=white, opacity = 0.5] (re2) circle [radius = 0.96*1.55pt];
\draw  [>=latex, <->, shorten >= 3pt, shorten <= 3pt, densely dashed] (re2)--(re1);
\shade [ball color=black!84] (re1') circle [radius = 0.96*1.5pt];
\filldraw [color=white, opacity = 0.5] (re1') circle [radius = 0.96*1.55pt];
\draw  [>=latex, <->, shorten >= 3.5pt, shorten <= 3.5pt] (re1)--(re1');
\shade [ball color=black!84] (re1) circle [radius = 0.96*1.5pt];
\filldraw [color=white, opacity = 0.5] (re1) circle [radius = 0.96*1.55pt];
\draw  [>=latex, <->, shorten >= 3pt, shorten <= 3pt, densely dashed] (re1)--(re0);
\shade [ball color=black!84] (re0) circle [radius = 0.96*1.5pt];
\filldraw [color=white, opacity = 0.5] (re0) circle [radius = 0.96*1.55pt];
\draw  [>=latex, <->, shorten >= 3pt, shorten <= 3pt] (re0)--(re0');
\shade [ball color=black!84] (re0') circle [radius = 0.96*1.5pt];
\filldraw [color=white, opacity = 0.5] (re0') circle [radius = 0.96*1.55pt];
\draw  [>=latex, <->, shorten >= 3pt, shorten <= 3pt, densely dashed] (cu)--(uem);
\shade [ball color=black!84] (uem') circle [radius = 0.90*1.5pt];
\filldraw [color=white, opacity = 0.5] (uem') circle [radius = 0.90*1.55pt];
\draw  [>=latex, <->, shorten >= 3.5pt, shorten <= 3.5pt] (uem)--(uem');
\shade [ball color=black!84] (uem) circle [radius = 0.90*1.5pt];
\filldraw [color=white, opacity = 0.5] (uem) circle [radius = 0.90*1.55pt];
\draw  [>=latex, <->, shorten >= 3pt, shorten <= 3pt, densely dashed] (uem)--(ue2);
\shade [ball color=black!84] (ue2') circle [radius = 0.90*1.5pt];
\filldraw [color=white, opacity = 0.5] (ue2') circle [radius = 0.90*1.55pt];
\draw  [>=latex, <->, shorten >= 3.5pt, shorten <= 3.5pt] (ue2)--(ue2');
\shade [ball color=black!84] (ue2) circle [radius = 0.90*1.5pt];
\filldraw [color=white, opacity = 0.5] (ue2) circle [radius = 0.90*1.55pt];
\draw  [>=latex, <->, shorten >= 3pt, shorten <= 3pt, densely dashed] (ue2)--(ue1);
\shade [ball color=black!84] (ue1') circle [radius = 0.90*1.5pt];
\filldraw [color=white, opacity = 0.5] (ue1') circle [radius = 0.90*1.55pt];
\draw  [>=latex, <->, shorten >= 3.5pt, shorten <= 3.5pt] (ue1)--(ue1');
\shade [ball color=black!84] (ue1) circle [radius = 0.90*1.5pt];
\filldraw [color=white, opacity = 0.5] (ue1) circle [radius = 0.90*1.55pt];
\draw  [>=latex, <->, shorten >= 3pt, shorten <= 3pt, densely dashed] (ue1)--(ue0);
\shade [ball color=black!84] (ue0) circle [radius = 0.90*1.5pt];
\filldraw [color=white, opacity = 0.5] (ue0) circle [radius = 0.90*1.55pt];
\draw  [>=latex, <->, shorten >= 3pt, shorten <= 3pt] (ue0)--(ue0');
\shade [ball color=black!84] (ue0') circle [radius = 0.90*1.5pt];
\filldraw [color=white, opacity = 0.5] (ue0') circle [radius = 0.90*1.55pt];
\end{scope}

\begin{scope}[>=latex, <->, shorten >= 3pt, shorten <= -7.5pt, color=black!84]
\draw [] ($(c12)+(c)-(d)$)--(c);
\draw [] ($(c23)+(c)-(d)$)--(c);
\end{scope}

\node [right, color=black] at ($(c)  +( 0.03,-0.03)$) {$a$} ;
\node [right, color=black] at ($(em) +(-0.01,-0.08)$) {$e_m^a$} ;
\node [right, color=black] at ($(e2) +(-0.01,-0.08)$) {$e_{k_n+i}^a$} ;
\node [right, color=black] at ($(e1) +(-0.01,-0.08)$) {$e_{k_n+j+4}^a$} ;
\node [right, color=black] at ($(e0) +(-0.01,+0.06)$) {$e_{k_n+9}^a$} ;
\node [right, color=black] at ($(e0')+(-0.01,+0.06)$) {$e_{k_n+10}^a$} ;
\node [right, color=black] at ($(em') +(0.02,-0.01)$) {$e_P^a$} ;
\node [right, color=black] at ($(e2') +(0.02,-0.01)$) {$e_R^a$} ;
\node [right, color=black] at ($(e1') +(0.02,-0.01)$) {$e_U^a$} ;

\node [right, color=black] at ($(ekn)+(-0.01,+0.06)$) {$e_0^a$} ;
\node [right, color=black] at ($(ek2)+(-0.01,-0.08)$) {$e_2^a$} ;
\node [right, color=black] at ($(ek1)+(-0.01,-0.08)$) {$e_1^a$} ;

%\draw [rounded corners = 2pt] ($(em)+(0.075,-0.075)$) 
%-- ($(em)+(0.075,-0.075)+(0.032,0)$) 
%-- ($(em)+(0.075,-0.075)+(0.032,0)+(0.0,0.9)$) 
%-- ($(em)+(0.075,-0.075)+(0.032,0)+(0.0,0.9)+(-0.032,0)$);
%\node [right=7.5pt] at ($0.5*(em)+0.5*(e0)$) {$m+1$};
%\node [rectnode, right] at ($0.5*(em)+0.5*(e0)+(1.6,0)$) {\begin{tabular}{l}$m+1$ elements \\ simulate $P_m(a)$\end{tabular}};

%\draw [rounded corners = 2pt] ($(ek0)+(0.075,-0.075)$) 
%-- ($(ek0)+(0.075,-0.075)+(0.032,0)$) 
%-- ($(ek0)+(0.075,-0.075)+(0.032,0)+(0.0,0.9)$) 
%-- ($(ek0)+(0.075,-0.075)+(0.032,0)+(0.0,0.9)+(-0.032,0)$);
%\node [right=7.7pt] at ($0.5*(ekn)+0.5*(ek0)$) {$k_n+2$};
%\node [rectnode, right] at ($0.5*(ekn)+0.5*(ek0)+(1.6,0)$) {\begin{tabular}{l}$k_n+2$ elements \\ simulate $G(a)$\end{tabular}};

%\draw [>=latex, ->, color=black!60, thick, dashed] ($0.5*(em)+0.5*(e0)+(1.25,0)$) -- ($0.5*(em)+0.5*(e0)+(0.15,0)$);
%\draw [>=latex, ->, color=black!60, thick, dashed] ($0.5*(ekn)+0.5*(ek0)+(1.25,0)$) -- ($0.5*(ekn)+0.5*(ek0)+(0.15,0)$);

\end{tikzpicture}
\caption{Simulation of $P_m(a)$, $R_i(a)$, $U_j(a)$, and $G(a)$}
\label{fig:15}
\end{figure}
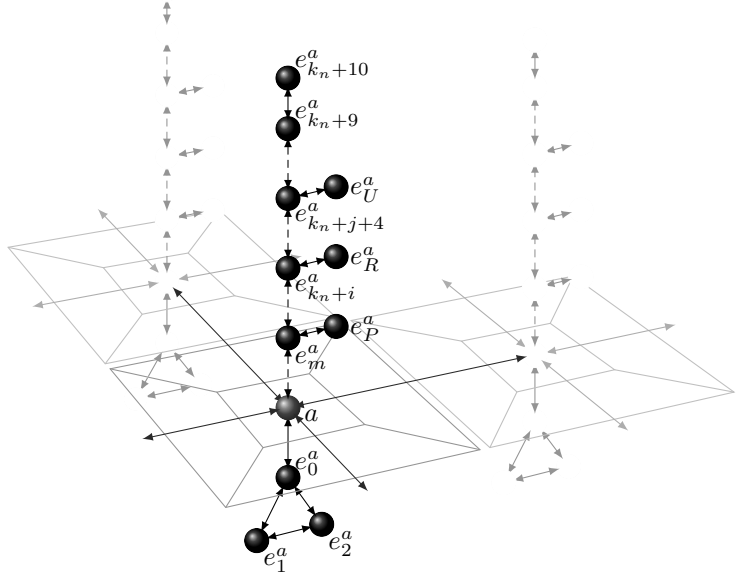

For every individual $a\in \mathcal{D}$, let us add new elements
$c_0^a, c_1^a, c_2^a$; 
$e_0^a,\ldots,e_{k_n+10}^a$;
$e_P^a, e_R^a, e_U^a$
to $\mathcal{D}$,
expand $\mathcal{I}(P)$ with 
\begin{itemize}
\item
$\otuple{a,c_0^a},\otuple{c_0^a,c_1^a},\otuple{c_1^a,c_2^a},\otuple{c_2^a,c_0^a}$;
\item
$\otuple{a,e_0^a},\otuple{e_0^a,e_1^a},\ldots,\otuple{e_{k_n+9}^a,e_{k_n+10}^a}$;
\item
$\otuple{e_m^a,e_P^a}$ if $\cModel{M}\models P_m(a)$ with $m\in\{0,\ldots,k_n\}$;
\item
$\otuple{e_{k_n+i}^a,e_R^a}$ if $\cModel{M}\models R_i(a)$ with $i\in\{1,2,3,4\}$;
\item
$\otuple{e_{k_n+j+4}^a,e_U^a}$ if $\cModel{M}\models U_j(a)$ with $j\in\{1,2,3,4\}$,
\end{itemize}
and then take the symmetric closure of the resulting relation.
Let $\cModel{M}'=\otuple{\mathcal{D}',\mathcal{I}'}$ be the resulting model; see Figure~\ref{fig:15}. 

By the definition of $\cModel{M}'$, for every $a\in \mathcal{D}'$,
$$
\begin{array}{lcl}
  \cModel{M}'\models \mathit{grid}'(a)
    & \iff
    & \mbox{$a\in \mathcal{D}$.}
    \smallskip\\
\end{array}
$$
Indeed, if $a\in \mathcal{D}$, then $\cModel{M}'\models \mathit{grid}'(a)$ since
$$
\cModel{M}'~\models~ P(a,c_0^a)\wedge P(c_0^a,c_1^a)\wedge P(c_1^a,c_2^a)\wedge P(c_2^a,c_0^a).
$$
For the converse implication, assume that $\cModel{M}'\models \mathit{grid}'(a)$. Then there is an element $c\in \mathcal{D}'$ such that $\cModel{M}'\models P(a,c)$ and $c$ is in a cycle of length~$3$. But the length of every $P$-cycle in $\mathcal{D}$ is even (see Remark~\ref{rem:lem:Trakhtenbrot:lem2:sib}), therefore, $c\in \{c_0^b,c_1^b,c_2^b\}$, for some $b\in \mathcal{D}$. Observe that $\cModel{M}'\not\models \mathit{grid}'(c_i^b)$, for every $i\in\{0,1,2\}$, and hence, $a\not\in\{c_0^b,c_1^b,c_2^b\}$. Then, by the construction of $\cModel{M}'$, we conclude that $a=b$ and $c=c_0^b$, so, in particular, $a\in \mathcal{D}$.

Also, it is not hard to check that, by the definition of $\cModel{M}'$, for every $i\in\{0,\ldots,k_n\}$, $j\in\{1,2,3,4\}$, and $a\in \mathcal{D}$,
$$
\begin{array}{lcl}
  \cModel{M}\models \hfill P_i\hfill(a)
    & \iff
    & \cModel{M}'\models \mathit{tile}'_i(a);
    \smallskip\\
  \cModel{M}\models \hfill R_j(a)
    & \iff
    & \cModel{M}'\models \mathit{tile}'_{k_n+j}(a);
    \smallskip\\
  \cModel{M}\models \hfill U_j(a)
    & \iff
    & \cModel{M}'\models \mathit{tile}'_{k_n+j+4}(a).
\end{array}
$$
Using these, we obtain that $\cModel{M}'\not\models S_1\mathit{MTiling}_n^{\mathbb{X}}$. Notice that the model $\cModel{M}'$ is finite. 
Thus, $S_1\mathit{MTiling}_n^{\mathbb{X}}\not\in\logic{SIB}_{\mathit{fin}}$.
\end{proof}

As a result, we obtain an analogue of the Trakhtenbrot theorem for the theory of symmetric irreflexive binary relation when the language contains only three variables (and a single binary predicate letter).

\begin{theorem} %[Trakhtenbrot]
\label{th:Trakhtenbrot:bin:sib:P}
Theories\/ $\logic{QCl}^{\mathit{bin}}$ and\/ $\logic{SIB}_{\mathit{fin}}$ are recursively inseparable in a language containing a binary predicate letter and three individual variables.
\end{theorem}

\begin{proof}
Immediate from Lemmas~\ref{lem:Trakhtenbrot:lem1:sib:binP} and~\ref{lem:Trakhtenbrot:lem2:sib:binP}.
\end{proof}

This theorem provides us with an answer to the question raised by S.\,Speranski: is $\logic{SIB}_{\mathit{fin}}$ decidable in the language with three variables? As we can see, it is undecidable; moreover, it is not recursively enumerable (since $\logic{QCl}$ is recursively enumerable) and even $\Pi^0_1$-complete (we will make a general remark on this below, see Proposition~\ref{prop:Pi01} and Corollary~\ref{cor:sib:srb:prop:Pi01}).

\begin{corollary} %[Trakhtenbrot]
\label{cor:a:th:Trakhtenbrot:bin:sib:P}
Theories\/ $\logic{SIB}$ and\/ $\logic{SIB}_{\mathit{fin}}$ are recursively inseparable in a language containing a binary predicate letter and three individual variables.
\end{corollary}

\begin{proof}
If $\logic{SIB}$ and $\logic{SIB}_{\mathit{fin}}$ are recursively separable, then $\logic{QCl}^{\mathit{bin}}$ and $\logic{SIB}_{\mathit{fin}}$ should be recursively separable, since $\logic{QCl}^{\mathit{bin}}\subseteq \logic{SIB}$, which contradicts Theorem~\ref{th:Trakhtenbrot:bin:sib:P}. 
\end{proof}

Now, let us turn to $\logic{SRB}$ and~$\logic{SRB}_{\mathit{fin}}$. Observe that a binary relation is irreflexive and symmetric if, and only if, its complement is reflexive and symmetric. 
For a formula $\varphi$ in the language containing a binary predicate letter~$P$, define ${\sim}\varphi$ as the formula obtained from $\varphi$ by substituting $\neg P(x_1,x_2)$ instead of $P(x_1,x_2)$, i.e., by placing $\neg$ before every occurrence of $P$ in~$\varphi$. Then, clearly, 
$$
\begin{array}{lcl}
  \varphi\in\logic{SIB}
    & \iff
    & {\sim}\varphi\in\logic{SRB};
    \smallskip\\
  \varphi\in\logic{SIB}_{\mathit{fin}}
    & \iff
    & {\sim}\varphi\in\logic{SRB}_{\mathit{fin}},
    \smallskip\\
\end{array}
$$

This observation yields us to the following corollaries.

\begin{corollary} %[Trakhtenbrot]
\label{cor:srb:th:Trakhtenbrot:bin:sib:P}
Theories\/ $\logic{QCl}^{\mathit{bin}}$ and\/ $\logic{SRB}_{\mathit{fin}}$ are recursively inseparable in a language containing a binary predicate letter and three individual variables.
\end{corollary}

\begin{corollary} %[Trakhtenbrot]
\label{cor:a:srb:th:Trakhtenbrot:bin:sib:P}
Theories\/ $\logic{SRB}$ and\/ $\logic{SRB}_{\mathit{fin}}$ are recursively inseparable in a language containing a binary predicate letter and three individual variables.
\end{corollary}

We can generalize these statements to a lot of other theories.
% in the spirit of the Theorem~\ref{th:Trakhtenbrot:binP:gen}.

\begin{theorem}
\label{th:Trakhtenbrot:binP:gen:sib:srb}
Let\/ $\Gamma$ and\/ $\Gamma'$ be theories of a binary predicate such that\/ $\logic{QCl}^{\mathit{bin}}\subseteq \Gamma\subseteq \Gamma'$ and also\/ $\Gamma'\subseteq \logic{SIB}$ or\/ $\Gamma'\subseteq \logic{SRB}$. Then\/ $\Gamma$ and\/ $\Gamma'_{\mathit{fin}}$ are recursively inseparable in a language with three variables.
\end{theorem}

\begin{proof}
Assume, for the sake of contradiction, that $\Gamma$ and $\Gamma'_{\mathit{fin}}$ are recursively separable in such a language. If $\Gamma'\subseteq \logic{SIB}$, then $\logic{QCl}^{\mathit{bin}}$ and\/ $\logic{SIB}_{\mathit{fin}}$ should be recursively separable; if $\Gamma'\subseteq \logic{SRB}$, then $\logic{QCl}^{\mathit{bin}}$ and\/ $\logic{SRB}_{\mathit{fin}}$ should be recursively separable. In any case, we obtain a contradiction.
\end{proof}

\subsection{Remarks on complexity}
\label{subsec:remarks-on-complexity}

We make remarks on the complexity of the theories $\Gamma$ and $\Gamma'_{\mathit{fin}}$ in the statement of Theorem~\ref{th:Trakhtenbrot:binP:gen:sib:srb}. It is not difficult to show that $\Gamma$ is $\Sigma^0_1$-hard and $\Gamma'_{\mathit{fin}}$ is $\Pi^0_1$-hard. We explain this in terms of the constructions described above. To this end, consider, instead of $\mathbb{X}$ and $\mathbb{Y}$ a $\Sigma^0_1$-complete set $\mathbb{A}$ of natural numbers and the empty set which we denote~$\mathbb{B}$. Notice that these sets, unlike $\mathbb{X}$ and $\mathbb{Y}$, are recursively separable. We are interested in the function $f_{\mathbb{AB}}\colon\numN\to\numN$ defined by
$$
\begin{array}{lcl}
f_{\mathbb{AB}}(x) 
  & = 
  & \left\{
      \begin{array}{rl}
        0 & \mbox{if $x\in \mathbb{A}$;} \\
        1 & \mbox{if $x\in \mathbb{B}$;} \\
        \mbox{undefined} & \mbox{if $x\not\in \mathbb{A}\cup\mathbb{B}$.} \\
      \end{array}
    \right.
\end{array}
$$
This function is computable since $\mathbb{A}$ is recursively enumerable and, taking into account that $\mathbb{B}$ is empty,
$$
\begin{array}{lcl}
f_{\mathbb{AB}}(x) 
  & = 
  & \left\{
      \begin{array}{rl}
        0 & \mbox{if $x\in \mathbb{A}$;} \\
        \mbox{undefined} & \mbox{if $x\not\in \mathbb{A}.\phantom{{}\cup\mathbb{B}}$} \\
      \end{array}
    \right.
\end{array}
$$
So, it is computable by a Turing machine
$M_1=\tuple{\Sigma_1,Q_1,q_0,F_1,\delta_1}$ with halting states $q_{\mathbb{A}},q_{\mathbb{B}}\in F_1$ such that
\begin{itemize}
\item if $m\in\mathbb{A}$, then $q_{\mathbb{A}}!M_1(\numeral{m})$;
\item if $m\in\mathbb{B}$, then $q_{\mathbb{B}}!M_1(\numeral{m})$;
\item if $m\not\in\mathbb{A}\cup\mathbb{B}$, then $\neg !M_1(\numeral{m})$;
\end{itemize}
notice that, since $\mathbb{B}$ is empty, $q_{\mathbb{B}}!M_1(\numeral{m})$ never heppens and $m\not\in\mathbb{A}\cup\mathbb{B}$ means that $m\not\in\mathbb{A}$.
Then, we can repeat the construction with tiles for~$M_1$ and define, for every $n\in\numN$, a set $T'_n$ of tile types and a special $T'_n$-tiling described by formulas $\mathit{MTiling}_n^{\mathbb{A}}$ and $\mathit{MTiling}_n^{\mathbb{B}}$ so that
\begin{equation}
\label{eq:ast:1}
\begin{array}{rcl}
  n\in\mathbb{A}
    & \Longrightarrow
    & S_1\mathit{MTiling}_n^{\mathbb{A}} \in \logic{QCl}^{\mathit{bin}}; \phantom{{\sim}}
    \smallskip \\
%  \phantom{{\logic{R}}}
  n\in\mathbb{A}
    & \Longrightarrow
    & S_1\mathit{MTiling}_n^{\mathbb{B}} \not\in \logic{SIB}_{\mathit{fin}}; \phantom{{\sim}}\phantom{{\logic{R}}} 
    \smallskip \\
\end{array}
\end{equation}
and
\begin{equation}
\label{eq:ast:2}
\begin{array}{rcl}
  n\in\mathbb{A}
    & \Longrightarrow
    & {\sim}S_1\mathit{MTiling}_n^{\mathbb{A}} \in \logic{QCl}^{\mathit{bin}}; %\phantom{{\sim}}
    \smallskip \\
%  \phantom{{\logic{I}}}
  n\in\mathbb{A}
    & \Longrightarrow
    & {\sim}S_1\mathit{MTiling}_n^{\mathbb{B}} \not\in \logic{SRB}_{\mathit{fin}}; \phantom{{\logic{I}}}
    \smallskip \\
\end{array}
\end{equation}
we leave the details to the reader. 

Observe that
\begin{equation}
\label{eq:ast:3}
\begin{array}{rcl}
  n\not\in\mathbb{A}
    & \Longrightarrow
    & \phantom{{\sim}}S_1\mathit{MTiling}_n^{\mathbb{A}} \not\in \logic{SIB}; %\phantom{{\sim}}
    \smallskip \\
%  \phantom{{\logic{I}}}
  n\not\in\mathbb{A}
    & \Longrightarrow
    & {\sim}S_1\mathit{MTiling}_n^{\mathbb{A}} \not\in \logic{SRB}. \phantom{{\logic{I}_{\mathit{fin}}}}
    \smallskip \\
\end{array}
\end{equation}
Indeed, if $n\not\in\mathbb{A}$, then $\neg !M_1(\numeral{n})$, and the special $T'_n$-tiling does not contain a tile of a type $t$ with $\upsq t = \# q_{\mathbb{A}}$, and we can use the tiling to obtain a required countermodels for $S_1\mathit{MTiling}_n^{\mathbb{A}}$ and for ${\sim}S_1\mathit{MTiling}_n^{\mathbb{A}}$.

Then, using the first lines of $(\ref{eq:ast:1})$ and $(\ref{eq:ast:2})$, we obtain that, for every theory $\Gamma$ such that $\logic{QCl}^{\mathit{bin}}\subseteq \Gamma \subseteq \logic{SIB}$,
$$
\begin{array}{rcl}
  n\in\mathbb{A}
    & \iff
    & S_1\mathit{MTiling}_n^{\mathbb{A}} \in \Gamma \phantom{{\sim}.}
\end{array}
$$
and, using the first line of $(\ref{eq:ast:2})$ and the second line of $(\ref{eq:ast:3})$, we obtain that, for every theory $\Gamma$ such that $\logic{QCl}^{\mathit{bin}}\subseteq \Gamma \subseteq \logic{SRB}$,
$$
\begin{array}{rcl}
  n\in\mathbb{A}
    & \iff
    & {\sim}S_1\mathit{MTiling}_n^{\mathbb{A}} \in \Gamma.
\end{array}
$$

These observations yield the following proposition.

\begin{proposition}
\label{prop:Sigma01}
If\/ $\logic{QCl}^{\mathit{bin}}\subseteq \Gamma$ and also\/ $\Gamma\subseteq \logic{SIB}$ or\/ $\Gamma\subseteq \logic{SRB}$, then\/ $\Gamma$ is\/ $\Sigma^0_1$-hard in a language with a binary predicate letter and three individual variables.
\end{proposition}

\begin{remark}
\label{rem:prop:Sigma01}
If, additionally, $\Gamma$ is recursively enumerable, then\/ $\Gamma$ is\/ $\Sigma^0_1$-complete.
\end{remark}

\begin{corollary}
\label{cor:sib:srb:prop:Sigma01}
Theories\/ $\logic{QCl}^{\mathit{bin}}$, $\logic{SIB}$, and\/ $\logic{SRB}$ are\/ $\Sigma^0_1$-complete in a language with a binary predicate letter and three individual variables.
\end{corollary}

To prove $\Pi^0_1$-hardness of $\Gamma'$, we proceed in a similar way. 

%Before, notice that there is a dummy difficulty that we need to overcome. 
%
Before proceeding, note that there is a minor difficulty that we need to overcome.
By the second line of~(\ref{eq:ast:1}), $n\in\mathbb{A}$ implies $S_1\mathit{MTiling}_n^{\mathbb{B}} \not\in \logic{SIB}_{\mathit{fin}}$. In addition, we want $n\not\in\mathbb{A}$ to imply $\mathit{MTiling}_n^{\mathbb{B}} \in \logic{QCl}_{\mathit{fin}}$ (and then $S_1\mathit{MTiling}_n^{\mathbb{B}} \in \logic{QCl}^{\mathit{bin}}_{\mathit{fin}}$). But the last implication can be wrong. Indeed, imagine that $M_1$, starting with input $\numeral{n}$, turns out to be in the same configuration twice (and hence, we obtain a computation with an infinite cycle). Then, as we have seen above, there exists a finite model for the special $T'_n$-tiling, and it refutes $\mathit{MTiling}_n^{\mathbb{B}}$. To avoid this, we have to ``remove'' all cycles in computations of~$M_1$ (except artificial ones involving~$q_{\mathbb{A}}$). Fortunately, this is possible.
To show this, consider a Turing machine~$M$ executing $M_1$ step by step and saving all configurations of the computation of~$M_1$ (with the content of the tape from the leftmost cell to the last non-empty cell or the scanned) into a sequence. For every new configuration of $M_1$, machine $M$ compares it with every one of the sequence; if two configurations are equal and their state components are different from~$q_{\mathbb{A}}$, then $M$ stops the computation of $M_1$ and then starts moving the head to the right on each step. If $M$ sees a new configuration with $q_{\mathbb{A}}$, then $M$ deletes the sequence and then executes $M_1$ on the input (without any controlling actions). Then, clearly, $M$ computes the same function as~$M_1$ and its computations do not contain cycles (except involving~$q_{\mathbb{A}}$). So, without a loss of generality we may assume that $M_1$ already satisfies this property. 

In this case, it is not hard to see that
\begin{equation}
\label{eq:ast:4}
\begin{array}{rcl}
  n\not\in\mathbb{A}
    & \Longrightarrow
    & \phantom{{\sim}}S_1\mathit{MTiling}_n^{\mathbb{B}} \in \logic{QCl}^{\mathit{bin}}_{\mathit{fin}}; %\phantom{{\sim}}
    \smallskip \\
  n\not\in\mathbb{A}
    & \Longrightarrow
    & {\sim}S_1\mathit{MTiling}_n^{\mathbb{B}} \in \logic{QCl}^{\mathit{bin}}_{\mathit{fin}}. %\phantom{{\sim}}
    \smallskip \\
\end{array}
\end{equation}
Indeed, let $n\not\in\mathbb{A}$. Then $\neg !M_1(\numeral{n})$, therefore, the computation of $M_1$ on the input $\numeral{n}$ is infinite and acyclic. This means that every model validating the formula describing the special $T'_n$\nobreakdash-tiling, i.e., the premise of $\mathit{MTiling}_n^{\mathbb{B}}$, is infinite. Thus, the premise of formula $\mathit{MTiling}_n^{\mathbb{B}}$ is refuted in every finite model; therefore, $\mathit{MTiling}_n^{\mathbb{B}}$ is true in every finite model. Then, $\mathit{MTiling}_n^{\mathbb{B}}\in \logic{QCl}^{\mathit{bin}}_{\mathit{fin}}$. Hence, $S_1\mathit{MTiling}_n^{\mathbb{B}}\in \logic{QCl}^{\mathit{bin}}_{\mathit{fin}}$ and ${\sim}S_1\mathit{MTiling}_n^{\mathbb{B}}\in \logic{QCl}^{\mathit{bin}}_{\mathit{fin}}$. 

Using $(\ref{eq:ast:4})$ and the second lines of $(\ref{eq:ast:1})$ and $(\ref{eq:ast:2})$, we obtain that, for every theory $\Gamma$ such that $\logic{QCl}^{\mathit{bin}}_{\mathit{fin}}\subseteq \Gamma \subseteq \logic{SIB}_{\mathit{fin}}$,
$$
\begin{array}{rcl}
  n\not\in\mathbb{A}
    & \iff
    & S_1\mathit{MTiling}_n^{\mathbb{B}} \in \Gamma \phantom{{\sim}.}
    \smallskip \\
\end{array}
$$
and, for every theory $\Gamma$ such that $\logic{QCl}^{\mathit{bin}}_{\mathit{fin}}\subseteq \Gamma \subseteq \logic{SRB}_{\mathit{fin}}$,
$$
\begin{array}{rcl}
  n\not\in\mathbb{A}
    & \iff
    & {\sim}S_1\mathit{MTiling}_n^{\mathbb{B}} \in \Gamma. %\phantom{{\sim}}
    \smallskip \\
\end{array}
$$

These observations lead to another proposition.

\begin{proposition}
\label{prop:Pi01}
If\/ $\logic{QCl}^{\mathit{bin}}\subseteq \Gamma$ and also\/ $\Gamma\subseteq \logic{SIB}$ or\/ $\Gamma\subseteq \logic{SRB}$, then\/ $\Gamma_{\mathit{fin}}$ is\/ $\Pi^0_1$-hard in a language with a binary predicate letter and three individual variables.
\end{proposition}

\begin{remark}
\label{rem:prop:Pi01}
If, additionally, the class of finite\/ $\Gamma$-models is recursively enumerable, then\/ $\Gamma_{\mathit{fin}}$ is\/ $\Pi^0_1$\nobreakdash-complete.
\end{remark}

\begin{corollary}
\label{cor:sib:srb:prop:Pi01}
Theories\/ $\logic{QCl}^{\mathit{bin}}_{\mathit{fin}}$, $\logic{SIB}_{\mathit{fin}}$, and\/ $\logic{SRB}_{\mathit{fin}}$ are\/ $\Pi^0_1$-complete in a language with a binary predicate letter and three individual variables.
\end{corollary}

\subsection{Positive formulas}
\label{sec:3:subsec:positive}

All the results above can be obtained using positive formulas only; we make a short remark on this (in fact, we will need positive formulas only below when dealing with superintuitionistic logics).
We call an $\lang{L}$-formula \defnotion{positive} if it does not contain occurrences of~$\bot$. Note that positive formulas do not contain occurrences of $\neg$ as well, since it is an abbreviation defined with the use of~$\bot$.

We just simulate $\bot$ by the formula 
$$
\begin{array}{rcl}
\mathit{false} 
  & = 
  & \forall x\forall y\,P(x,y).
\end{array}
$$
For an $\lang{L}$-formula $\varphi$, let $\varphi^+$ be the formula obtained from $\varphi$ by replacing every occurrence of~$\bot$ with $\mathit{false}$. 

\pagebreak[3]

\begin{samepage}
\begin{lemma}
\label{lem:false}
Let $\cModel{M}\not\models\mathit{false}$. Then, for every formula~$\varphi$,
$$
\begin{array}{rcl}
\cModel{M}\models\varphi 
  & \iff 
  & \cModel{M}\models\varphi^+.
\end{array}
$$
\end{lemma}

\nopagebreak[3]

\begin{proof}
Obvious since $\cModel{M}\models\mathit{false}\lra\bot$.
\end{proof}
\end{samepage}

\pagebreak[3]

\begin{corollary}
\label{cor1:lem:false}
Let $\cModel{M}\not\models S_1\mathit{MTiling}_n^{\mathbb{X}}$. Then, $\cModel{M}\not\models (S_1\mathit{MTiling}_n^{\mathbb{X}})^+$.
\end{corollary}

\begin{proof}
Observe that $\cModel{M}$ satisfies the condition of Lemma~\ref{lem:false}.
\end{proof}

\begin{corollary}
\label{cor2:lem:false}
Let $\cModel{M}\not\models {\sim}S_1\mathit{MTiling}_n^{\mathbb{X}}$. Then, $\cModel{M}\not\models ({\sim}S_1\mathit{MTiling}_n^{\mathbb{X}})^+$.
\end{corollary}

\begin{proof}
The same.
\end{proof}

\begin{corollary}
\label{cor3:lem:false}
Let $\cModel{M}$ be a model and $\varphi$ be a formula containing no predicate letters except~$P$ such that $\cModel{M}\not\models\varphi^+$. Then $\cModel{M}\not\models\varphi$.
\end{corollary}

\begin{proof}
Observe that $\cModel{M}$ satisfies the condition of Lemma~\ref{lem:false}, since otherwise every positive formula containing $P$ as the unique predicate letter should be true in~$\cModel{M}$. Then, by Lemma~\ref{lem:false}, $\cModel{M}\not\models \varphi$.
\end{proof}

Corollaries~\ref{cor1:lem:false}--\ref{cor3:lem:false} ensure that we can replace $S_1\mathit{MTiling}_n^{\mathbb{X}}$ and ${\sim}S_1\mathit{MTiling}_n^{\mathbb{X}}$ with $(S_1\mathit{MTiling}_n^{\mathbb{X}})^+$ and $({\sim}S_1\mathit{MTiling}_n^{\mathbb{X}})^+$, respectively, in all constructions above. Let us formulate statements immediately following from this observation.

\begin{theorem}
\label{th:Trakhtenbrot:binP:gen:sib:srb:pos}
Let\/ $\Gamma$ and\/ $\Gamma'$ be theories of a binary predicate such that\/ $\logic{QCl}^{\mathit{bin}}\subseteq \Gamma\subseteq \Gamma'$ and also\/ $\Gamma'\subseteq \logic{SIB}$ or\/ $\Gamma'\subseteq \logic{SRB}$. Then the positive fragments of\/ $\Gamma$ and\/ $\Gamma'_{\mathit{fin}}$ are recursively inseparable in a language with three variables.
\end{theorem}

%For a set $X$ of formulas let $X^+$ be a positive fragment of~$X$.

\begin{proposition}
\label{prop:Sigma01:pos}
If\/ $\logic{QCl}^{\mathit{bin}}\subseteq \Gamma$ and also\/ $\Gamma\subseteq \logic{SIB}$ or\/ $\Gamma\subseteq \logic{SRB}$, then the positive fragment of\/ $\Gamma$ is\/ $\Sigma^0_1$-hard in a language with a binary predicate letter and three individual variables.
\end{proposition}

\begin{corollary}
\label{cor:sib:srb:prop:Sigma01:pos}
The positive fragments of\/ $\logic{QCl}^{\mathit{bin}}$, $\logic{SIB}$, and\/ $\logic{SRB}$ are\/ $\Sigma^0_1$-complete in a language with a binary predicate letter and three individual variables.
\end{corollary}

\begin{proposition}
\label{prop:Pi01:pos}
If\/ $\logic{QCl}^{\mathit{bin}}\subseteq \Gamma$ and also\/ $\Gamma\subseteq \logic{SIB}$ or\/ $\Gamma\subseteq \logic{SRB}$, then the positive fragment of\/ $\Gamma_{\mathit{fin}}$ is\/ $\Pi^0_1$-hard in a language with a binary predicate letter and three individual variables.
\end{proposition}

\begin{corollary}
\label{cor:sib:srb:prop:Pi01:pos}
The positive fragments of\/ $\logic{QCl}^{\mathit{bin}}_{\mathit{fin}}$, $\logic{SIB}_{\mathit{fin}}$, and\/ $\logic{SRB}_{\mathit{fin}}$ are\/ $\Pi^0_1$-complete in a language with a binary predicate letter and three individual variables.
\end{corollary}

\section{Modal predicate logics}
\label{sec:modal}
\setcounter{equation}{0}

\subsection{Syntax and semantics}

The modal predicate language $\lang{ML}$ extends $\lang{L}$ by adding a modality $\Box$ (``necessity'') allowing, for a formula~$\varphi$,  to construct the formula~$\Box\varphi$. Formulas in $\lang{ML}$ are called \defnotion{modal predicate formulas}, or \defnotion{$\lang{ML}$-formulas}. Define $\Diamond$ (``possibility'') by $\Diamond\varphi=\neg\Box\neg\varphi$. 
From now on, we identify the language $\lang{ML}$ with the set of $\lang{ML}$-formulas.

\Rem{%%%%%%%%%%%%%%%%%%%%%%%%%%%%%%%%%%%%%%%%%%
Define the \defnotion{modal depth} $\md\varphi$ of an $\lang{ML}$\nobreakdash-formula~$\varphi$ as the largest number of nested modalities in~$\varphi$, i.e., 
$$
\begin{array}{lcll}
\md\varphi & = & 0 & \mbox{if $\varphi$ is atomic;}\\
\md\varphi & = & \max\{\md\varphi_1,\md\varphi_2\} & \mbox{if $\varphi=\varphi_1\wedge\varphi_2$, $\varphi=\varphi_1\vee\varphi_2$ or $\varphi=\varphi_1\to\varphi_2$;}\\
\md\varphi & = & \md\varphi_1 & \mbox{if $\varphi=\forall x\,\varphi_1$ or $\varphi=\exists x\,\varphi_1$;}\\
\md\varphi & = & \md\varphi_1 + 1 & \mbox{if $\varphi=\Box\varphi_1$.}
\end{array}
$$
}%%%%%%%%%%%%%%%%%%%%%%%%%%%%%%%%%%%%%%%%%%%%%%

We call a set $L$ of $\lang{ML}$-formulas \defnotion{modal predicate logic} if $L$ is closed under Substitution.\footnote{For predicate substitution see~\cite{GShS}.} In fact, below we shall consider only modal predicate logics of some special classes; in particular, they contain $\logic{QCl}$ and are closed, additionally, under Modus Ponens and Generalization. 

To study modal predicate logics we shall use Kripke semantics. Here, we describe the expanding domains semantics only; it is enough for a wide class of logics~\cite{GShS}. 
%Later (Section~\ref{subsec:varsem}) we consider the varying domain semantics, which generates a different kind of logics~\cite{FM98}.

A \defnotion{Kripke frame} is a pair $\kframe{F} = \langle W,R \rangle$, where $W$ is a non-empty set of \defnotion{possible worlds} and $R$ is a binary \defnotion{accessibility relation} on~$W$. Speaking of Kripke frames, we use the standard notation $R(w) = \{ w' \in W: w R w'\}$, so $w' \in R(w)$ means the same as $w R w'$. We say that $\kframe{F}$ is finite if $W$ is finite; we say that $\kframe{F}$ is reflexive, symmetric, anti-symmetric, transitive, serial, etc.\ if $R$~is. 

A \defnotion{Kripke frame with domains}, or, for short, an \defnotion{augmented frame}, is a pair $\kFrame{F} = \langle \kframe{F}, D \rangle$, where $\kframe{F}$ is a Kripke frame and $D$ is a \defnotion{domain function} $D\colon W\to \Pow{\mathcal{D}}$ associating with every world $w\in W$ a non-empty subset of a non-empty set $\mathcal{D}$ of \defnotion{individuals}. The set $D(w)$, also denoted by $D_w$, is called \defnotion{the domain of the world\/~$w$}. Sets of the form $D_w$ are also called \defnotion{local domains} of $\kFrame{F}$ and $\mathcal{D}$ is called \defnotion{the global domain} of $\kFrame{F}$; we assume that
$$
\begin{array}{rcl}
\mathcal{D} & = & \displaystyle\bigcup\limits_{\mathclap{w\in W}}D_w.
\end{array}
$$
The augmented frame $\kFrame{F} = \langle \kframe{F}, D \rangle$ is also denoted by~$\kframe{F}_D$. We say that $\kframe{F}_D$ is \defnotion{based on $\kframe{F}$}. 
%or \defnotion{defined over~$\kframe{F}$}. 
The global domain of $\kframe{F}_D$ is denoted by~$D^+$~\cite{GShS}.

We say that an augmented frame $\kframe{F}_D$ based on a Kripke frame $\kframe{F}=\otuple{W,R}$ satisfies the \defnotion{expanding domain condition} if, for all $u,w \in W$,
$$
\begin{array}{lcl}
  uRw & \Longrightarrow & D_u \subseteq D_{w};
\end{array}
\eqno{(\mathit{ED})}
$$
then we call $\kframe{F}_D$ \defnotion{augmented frame with expanding domains} or, for short, \defnotion{e\nobreakdash-augmented frame}. We say that  $\kframe{F}_D$ satisfies the \defnotion{locally constant domain condition} if, for all $u,w \in W$,
$$
\begin{array}{lcl}
  uRw & \Longrightarrow & D_u = D_{w},
\end{array}
\eqno{(\mathit{LCD})}
$$
and the \defnotion{globally constant domain condition} if, for all $u,w \in W$,
$$
\begin{array}{lcl}
D_u = D_{w}.
\end{array}
\eqno{(\mathit{GCD})}
$$
For our purposes, ${(\mathit{LCD})}$ is sufficient; nevertheless, the frames satisfying ${(\mathit{LCD})}$ we shall consider below, also satisfy ${(\mathit{GCD})}$. If $\kframe{F}_D$ satisfies ${(\mathit{LCD})}$, then we call it \defnotion{augmented frame with constant domains} or, for short, \defnotion{c\nobreakdash-augmented frame}. If $\kframe{F}_D$ satisfies ${(\mathit{GCD})}$ and $\mathcal{D}$ is the global domain of $\kframe{F}_D$, then, following~\cite{GShS}, we also denote it by~$\kframe{F}\odot\mathcal{D}$. In general, when we do not claim the domain function to satisfy any additional conditions, we obtain the class of all augmented frames; they are known as \defnotion{augmented frames with varying domains}, and, to emphasize this, let us also call them \defnotion{v\nobreakdash-augmented frames}.
%; but we are going to use only e\nobreakdash-augmented and c\nobreakdash-augmented frames.
We will focus on e\nobreakdash-augmented and c\nobreakdash-augmented frames only.\footnote{For more background on the varying domains semantics, we refer the reader to~\cite{FM23}.} For convenience, sometimes we write $\otuple{W,R,D}$ for $\otuple{\kframe{F},D}$ with $\kframe{F}=\otuple{W,R}$.

In this section, we assume that all augmented frames satisfy~$(\mathit{ED})$. 

%A \textit{subframe} of an augmented frame $\langle W, R, D \rangle$ is an augmented frame $\langle W', R', D' \rangle$ where $W'$ is a non-empty subset of $W$, and $R' = R \upharpoonright W'$ and $D' = D \upharpoonright W'$.

A \defnotion{predicate Kripke model with expanding domains}, or simply a \defnotion{Kripke model}, is a tuple $\kModel{M} = \langle \kframe{F}_D, I\rangle$, where $\kframe{F}_D = \langle W, R, D \rangle$ is an e\nobreakdash-augmented frame and $I$ is a map, called an \defnotion{interpretation of predicate letters}, assigning to a world $w\in W$ and an $n$-ary predicate letter $P$ an $n$-ary relation $I(w,P)$ on $D_w$; we also write $P^{I,w}$ for $I(w,P)$ and $\langle W, R, D, I\rangle$ for $\langle \kframe{F}_D, I\rangle$.  We note that, if a predicate letter $P$ is nullary (i.e., $P$ is a proposition letter), then $P^{I,w} \subseteq D^0_w = \{\otuple{}\}$, for every $w \in W$.  Conceptually, $P^{I,w} = \varnothing$ corresponds to assigning the truth value ``false'', and $P^{I,w} = \{\otuple{}\}$ the truth value ``true'', to~$P$ at~$w$. For a Kripke model $\kModel{M} = \langle \kframe{F}_D, I\rangle$, we say that $\kModel{M}$ is \defnotion{based on\/ $\kframe{F}_D$} and is \defnotion{based on\/ $\kframe{F}$}.

An \defnotion{assignment} in a model $\kModel{M} = \otuple{W, R, D, I}$ is a map $g$ associating with every variable $x$ an element $g(x)$ of the global domain of the augmented frame $\otuple{W, R, D}$. If $g$ and $h$ are assignments such that $g(y) = h(y)$ whenever $y \ne x$, we write $g \stackrel{x}{=} h$.

The truth of an $\lang{ML}$-formula $\varphi$ at a world $w$ of a model $\kModel{M} = \otuple{W,R,D,I}$ under an assignment $g$ is defined recursively:
%%%%%%%%%%%%%%%%%%%%%%%%%%%%%%%%%%%%%%%%%%%%%%%%%%%%%%%%%%%%%%%%%%%%%%
\settowidth{\templength}{\mbox{$\kModel{M},w\models^g\varphi'$ and $\kModel{M},w\models^g\varphi''$;}}
\settowidth{\templengtha}{\mbox{$w$}}
\settowidth{\templengthb}{\mbox{$\kModel{M},w\models^{h}\varphi'$, for every assignment $h$ such that}}
\settowidth{\templengthc}{\mbox{$\kModel{M},w\models^g P(x_1,\ldots,x_n)$}}
%%%%%%%%%%%%%%%%%%%%%%%%%%%%%%%%%%%%%%%%%%%%%%%%%%%%%%%%%%%%%%%%%%%%%%
$$
\begin{array}{lcl}
\kModel{M},w\models^g P(x_1,\ldots,x_n)
  & \leftrightharpoons
  & \parbox{\templengthb}{$\langle g(x_1),\ldots,g(x_n)\rangle \in P^{I, w}$,} \\
\end{array}
$$
\mbox{where $P$ is an $n$-ary predicate letter;}
%%%%%%%%%%%%%%%%%%%%%%%%%%%%%%%%%%%%%%%%%%%%%%%%%%%%%%%%%%%%%%%%%%%%%%
\settowidth{\templength}{\mbox{$\kModel{M},w\models^g\varphi'$ and $\kModel{M},w\models^g\varphi''$;}}
\settowidth{\templengtha}{\mbox{$w$}}
\settowidth{\templengthb}{\mbox{$\kModel{M},w\models^{g}\varphi'\to\varphi''$}}
\settowidth{\templengthc}{\mbox{$\kModel{M},w\models^g P(x_1,\ldots,x_n)$}}
%%%%%%%%%%%%%%%%%%%%%%%%%%%%%%%%%%%%%%%%%%%%%%%%%%%%%%%%%%%%%%%%%%%%%%
$$
\begin{array}{lcl}
\parbox{\templengthc}{{}\hfill\parbox{\templengthb}{$\kModel{M},w \not\models^g \bot;$}}
  \\
\parbox{\templengthc}{{}\hfill\parbox{\templengthb}{$\kModel{M},w\models^g\varphi' \wedge \varphi''$}}
  & \leftrightharpoons
  & \parbox[t]{\templength}{$\kModel{M},w\models^g\varphi'$ and $\kModel{M},w\models^g\varphi''$;}
  \\
\parbox{\templengthc}{{}\hfill\parbox{\templengthb}{$\kModel{M},w\models^g\varphi' \vee \varphi''$}}
  & \leftrightharpoons
  & \parbox[t]{\templength}{$\kModel{M},w\models^g\varphi'$\hfill or\hfill $\kModel{M},w\models^g\varphi''$;}
  \\
\parbox{\templengthc}{{}\hfill\parbox{\templengthb}{$\kModel{M},w\models^g\varphi' \to \varphi''$}}
  & \leftrightharpoons
  & \parbox[t]{\templength}{$\kModel{M},w\not\models^g\varphi'$\hfill or\hfill $\kModel{M},w\models^g\varphi''$;}
  \\
\parbox{\templengthc}{{}\hfill\parbox{\templengthb}{$\kModel{M},w\models^g\Box\varphi'$}}
  & \leftrightharpoons
  & \mbox{$\kModel{M},\parbox{\templengtha}{$v$}\models^g\varphi'$, for every $v\in R(w)$;}
  \\
\parbox{\templengthc}{{}\hfill\parbox{\templengthb}{$\kModel{M},w\models^g\forall x\,\varphi'$}}
  & \leftrightharpoons
  & \mbox{$\kModel{M},w\models^{h}\varphi'$, for every assignment $h$ such that}
  \\
  &
  & \mbox{\phantom{$\kModel{M},w\models^{g'}\varphi'$, }$h \stackrel{x}{=} g$ and $h(x)\in D_w$;}
  \\
\parbox{\templengthc}{{}\hfill\parbox{\templengthb}{$\kModel{M},w\models^g\exists x\,\varphi'$}}
  & \leftrightharpoons
  & \mbox{$\kModel{M},w\models^{h}\varphi'$, for some assignment $h$ such that}
  \\
  &
  & \mbox{\phantom{$\kModel{M},w\models^{g'}\varphi'$, }$h \stackrel{x}{=} g$ and $h(x)\in D_w$.}
\end{array}
$$

Let $\kModel{M}$, $\kframe{F}_D$, $\kframe{F}$, and $\Scls{C}$ be a Kripke model, an e-augmented frame, a Kripke frame, and a class of e-augmented frames, respectively, $w$ a world of $\kModel{M}$, and $\varphi$ a formula with free variables $x_1,\ldots,x_n$; then define
%%%%%%%%%%%%%%%%%%%%%%%%%%%%%%%%%%%%%%%%%%%%%%%%%%%%%%%%%%%%%%%%%%%%%%%%%%%%%%%%%%
\settowidth{\templength}{\mbox{$\kModel{M},w\models^g P(x_1,\ldots,x_n)$}}
\settowidth{\templengtha}{\mbox{$\kModel{M},w\models^{h}\varphi'$, for every assignment $h$ such that}}
\settowidth{\templengthb}{\mbox{$w$}}
\settowidth{\templengthc}{\mbox{$\kframe{F}_D$}}
%%%%%%%%%%%%%%%%%%%%%%%%%%%%%%%%%%%%%%%%%%%%%%%%%%%%%%%%%%%%%%%%%%%%%%%%%%%%%%%%%%
$$
\begin{array}{rcl}
\parbox{\templength}{{}\hfill$\kModel{M},w\models \varphi$}
  & \leftrightharpoons
  & \parbox[t]{\templengtha}{$\kModel{M},w\models^g \varphi$, for every assignment $g$ such that}
  \\
  &
  & \mbox{\phantom{$\kModel{M},w\models^g \varphi$, }$g(x_1),\ldots,g(x_n)\in D_w$;}
  \\
\parbox{\templength}{{}\hfill$\kModel{M}\models \varphi$}
  & \leftrightharpoons
  & \parbox[t]{\templengtha}{$\kModel{M},\parbox{\templengthb}{$v$}\models^{\phantom{g}} \varphi$, for every world $v$ of $\kModel{M}$;}
  \\
\parbox{\templength}{{}\hfill$\kframe{F}_D\models \varphi$}
  & \leftrightharpoons
  & \parbox[t]{\templengtha}{$\parbox{\templengthc}{$\kModel{M}$}\models \varphi$, for every $\kModel{M}$ based on $\kframe{F}_D$;}
  \\
\parbox{\templength}{{}\hfill$\kframe{F}\models \varphi$}
  & \leftrightharpoons
  & \parbox[t]{\templengtha}{$\parbox{\templengthc}{$\kModel{M}$}\models \varphi$, for every $\kModel{M}$ based on $\kframe{F}$;}
  \\
\parbox{\templength}{{}\hfill$\Scls{C}\models \varphi$}
  & \leftrightharpoons
  & \parbox[t]{\templengtha}{$\parbox{\templengthc}{$\kframe{F}_D$}\models \varphi$, for every $\kframe{F}_D\in\Scls{C}$.}
  \\
\end{array}
$$
If $\mathfrak{S}\models\varphi$, for a structure $\mathfrak{S}$ (a~world, a~model, a~frame, etc.), we say that the formula $\varphi$ is \defnotion{true}, or \defnotion{valid},\footnote{We use ``valid'' only for frames and classes of frames.} in (on, at)~$\mathfrak{S}$; otherwise, $\varphi$ is \defnotion{refuted} in (on, at)~$\mathfrak{S}$.
These notions, and the corresponding notations, can be extended to sets of formulas in a natural way: for a set of formulas $X$, define $\mathfrak{S}\models X$ as $\mathfrak{S}\models\varphi$, for every $\varphi\in X$.

Observe that, given a Kripke model $\kModel{M} = \langle W,R,D,I\rangle$ and $w \in W$, we can define the interpretation $I_w$ by $I_w(P) = I(w,P)$, for every predicate letter~$P$. Then, the tuple $\kModel{M}_w = \langle D_w, I_w \rangle$ is a classical model. So, we can see on the Kripke model $\kModel{M}$ as the set $\{\kModel{M}_w : w\in W\}$ of classical models structured by~$R$.

Let $\kModel{M} = \langle W,R,D,I\rangle$ be a Kripke model, $w \in W$, and $a_1, \ldots, a_n \in D_w$; let also $\varphi(x_1, \ldots, x_n)$ be a formula whose free variables are among $x_1, \ldots, x_n$.  We write $\kModel{M}, w \models \varphi (a_1, \ldots, a_n)$ if $\kModel{M}, w \models^g \varphi (x_1, \ldots, x_n)$, where $g(x_1) = a_1, \ldots, g(x_n) = a_n$.  This notation is unambiguous since the languages we consider lack constants and the truth value of $\varphi(x_1, \ldots, x_n)$ does not depend on the values of variables different from $x_1, \ldots, x_n$.
Given a non-atomic formula $\varphi(x_1, \ldots, x_n)$ with $x_1, \ldots, x_n$ being the list of its free variables,\footnote{We assume that the variables in the list are in a certain order.} define, for every $w\in W$, the $n$-ary relation $\varphi^{I,w}$:
$$
\begin{array}{lcl}
\varphi^{I,w} & = & \{(a_1,\ldots,a_n)\in D_w^n : \kModel{M},w\models \varphi(a_1,\ldots,a_n)\}.
\end{array}
$$
In fact, this definition expands $I$ to the set of all formulas. We write $\varphi^{I,w}(a_1,\ldots,a_n)$ rather than $(a_1,\ldots,a_n)\in\varphi^{I,w}$.

Let $\Scls{C}$ be a class of e-augmented frames. Define the \defnotion{modal predicate logic\/ $\QML \Scls{C}$ of the class\/~$\Scls{C}$} by
$$
\begin{array}{lcl}
\QML \Scls{C} & = & \{\varphi\in\lang{ML} : \Scls{C}\models\varphi\}.
\end{array} 
$$
Let $\scls{C}$ be a class of Kripke frames and $\alpha\in\{\mathit{c},\mathit{e}\}$; define 
\begin{itemize}
\item
$\fin\scls{C}$ be the class of finite Kripke frames of~$\scls{C}$;
\item
$\aug{\alpha}{all}{\scls{C}}$ be the classes of $\alpha$-augmented frames based on Kripke frames of~$\scls{C}$;
\item
$\aug{\alpha}{dfin}{\scls{C}}$ be the classes of $\alpha$-augmented frames based on Kripke frames of~$\scls{C}$, whose local domains are finite;
\item
$\aug{\alpha}{wfin}{\scls{C}}$ be the classes of $\alpha$-augmented frames based on frames of~$\fin\scls{C}$.
\end{itemize} 
In addition, let $\beta\in\{\mathit{all},\mathit{dfin},\mathit{wfin}\}$. Define the logic $\QMLext{\alpha}{\beta}\scls{C}$~by
$$
\begin{array}{lcl}
\QMLext{\alpha}{\beta} \scls{C} & = & \QML \aug{\alpha}{\beta} \scls{C}.
\end{array} 
$$
For a Kripke frame $\kframe{F}$, we write $\QMLext{\alpha}{\beta} \kframe{F}$ rather than $\QMLext{\alpha}{\beta} \{\kframe{F}\}$.
Notice that
%, for $\alpha\in\{\mathit{c},\mathit{e}\}$ and $\beta\in\{\mathit{all},\mathit{dfin},\mathit{wfin}\}$,
$$
\begin{array}{lclcl}
\QMLext{e}{\beta} \scls{C} \subseteq \QMLext{c}{\beta} \scls{C}, 
  && \QMLext{\alpha}{\mathit{all}} \scls{C} \subseteq \QMLext{\alpha}{\mathit{dfin}} \scls{C},
  && \QMLext{\alpha}{\mathit{all}} \scls{C} \subseteq \QMLext{\alpha}{\mathit{wfin}} \scls{C}. 
\end{array} 
$$

For a modal predicate logic $L$, denote by $\ckf L$ the class of Kripke frames validating~$L$. Define 
$$
\begin{array}{lcl}
L_{\mathit{dfin}} & = & \QMLed \ckf L; \\
%L\logic{.bf}_{\mathit{dfin}} & = & \QMLcd \ckf L; \\
L_{\mathit{wfin}} & = & \QMLew \ckf L. \\
%L\logic{.bf}_{\mathit{wfin}} & = & \QMLcw \ckf L. \\
\end{array}
$$
For sets $X$ and $Y$ of $\lang{ML}$-formulas, define 
\begin{itemize}
\item
$X\oplus Y$ to be the smallest modal predicate logic, containing $X\cup Y$ and closed under Substitution, Modus Ponens, Generalization, and Necessitation; 
\item
$X+Y$ to be the smallest modal predicate logic, containing $X\cup Y$ and closed under Substitution, Modus Ponens, and Generalization. 
\end{itemize}
We write $X\oplus\varphi$ and $X+\varphi$ rather than $X\oplus\{\varphi\}$ and $X+\{\varphi\}$, respectively, if $\varphi$ is an $\lang{ML}$-formula.

Define the Barcan formula $\bm{bf}$ by $\bm{bf} = \forall x\,\Box Q(x)\to\Box\forall x\,Q(x)$, where $Q$ is a unary predicate letter; it is well known and easy to check that $\bm{bf}$ is valid on an e-augmented frame $\kframe{F}_D$ if, and only if, $\kframe{F}_D$ is a c-augmented frame.

For further definitions we need refer to propositional logics. We assume that the reader is familiar with modal propositional logics; for details, consult~\cite{ChZ}. 
Also, we assume the modal propositional language be the propositional fragment of $\lang{M}$, where propositional variables are proposition letters.\footnote{Formally, this is not the case: if $p$ is a propositional variable, then $p$ is a propositional formula; if $p$ is a proposition letter, then $p$ is not an $\lang{ML}$-formula but $p()$ is. So, we identify $p()$ with~$p$.}
We call\footnote{The claim of closure under Necessitation is sometimes replaced with the claim that Necessitation is postulated as an inference rule; this approach is more subtle but we do not need it, even without a loss of generality.} a propositional modal logic~$L$ \defnotion{normal} if it contains $\logic{K}$ and closed under Substitution, Modes Ponens, and Necessitation; we will deal mostly with predicate counterparts of normal modal propositional logics.
For a normal modal propositional logic $L$, define $\logic{Q}L$ and $\logic{Q}L\logic{.bf}$ by
$$
\begin{array}{lcl}
\logic{Q}L & = & L \oplus \logic{QCl}; \\
\logic{Q}L{\logic{.bf}} & = & \logic{Q}L \oplus \bm{bf}.
\end{array}
$$
For convenience, we recall definitions for some special modal predicate logics we will refer~to:
$$
\begin{array}{lclcl}
\logic{QK}     & = & \logic{QCl} & \!\!\oplus\!\! & \Box(p \to q) \to (\Box p\to \Box q); \\
\logic{QT}     & = & \logic{QK}  & \!\!\oplus\!\! & \Box p \to p; \\
\logic{QD}     & = & \logic{QK}  & \!\!\oplus\!\! & \Diamond\top; \\
\logic{QKB}    & = & \logic{QK}  & \!\!\oplus\!\! & p \to \Box\Diamond p; \\
\logic{QKTB}   & = & \logic{QT}  & \!\!\oplus\!\! & \logic{QKB}; \\
\logic{QK4}    & = & \logic{QK}  & \!\!\oplus\!\! & \Box p \to \Box\Box p; \\
\logic{QK4B}   & = & \logic{QK4} & \!\!\oplus\!\! & \logic{QKB}; \\
\logic{QS4}    & = & \logic{QT}  & \!\!\oplus\!\! & \logic{QK4}; \\
\logic{QK5}    & = & \logic{QK}  & \!\!\oplus\!\! & \Diamond\Box p\to\Box p; \\
\logic{QS5}    & = & \logic{QS4} & \!\!\oplus\!\! & \logic{QK5}; \\
\logic{QK45}   & = & \logic{QK4} & \!\!\oplus\!\! & \logic{QK5}; \\
\logic{QKD45}  & = & \logic{QD}  & \!\!\oplus\!\! & \logic{QK45}; \\
\logic{QGL}    & = & \logic{QK4} & \!\!\oplus\!\! & \Box(\Box p\to p)\to \Box p; \\
\logic{QGrz}   & = & \logic{QS4} & \!\!\oplus\!\! & \Box(\Box(p\to \Box p)\to p)\to p; \\
\logic{QwGrz}  & = & \logic{QK4} & \!\!\oplus\!\! & \Box^+(\Box(p\to \Box p)\to p)\to p; \\
\logic{QTriv}  & = & \logic{QK}  & \!\!\oplus\!\! & p\lra\Box p; \\
\logic{QVer}   & = & \logic{QK}  & \!\!\oplus\!\! & \Box\bot, \\
\end{array}
$$
where $\Box^+$ is defined by $\Box^+\varphi = \varphi\wedge \Box\varphi$.
Notice that some of the logics are Kripke complete,\footnote{A modal predicate logic $L$ is \defnotion{Kripke complete} if there exists a class $\Scls{C}$ of e\nobreakdash-augmented frames such that $L=\QML \Scls{C}$.} some are not~\cite{Montagna84,TO01,GShS,Shehtman23}, some of them contain the Barcan formula (for example, $\logic{QS5}$, $\logic{QTriv}$, $\logic{QVer}$), some do not; for details see~\cite{GShS}.

Let us make a remark. Note that, generally speaking, it is possible that, for a modal predicate logic~$L$, the class $\ckf L$ is empty or degenerate but the class $\caf^{\mathit{e}} L$ of e\nobreakdash-augmented frames validating~$L$ is not empty and not degenerate. For example, Kripke frames for $\logic{QK.bf}$ may contain only singletons and clusters not seeing each other, while the class $\caf^{\mathit{e}} \logic{QK.bf}$ consists of all c\nobreakdash-augmented frames; something similar is true for $\logic{QK}\oplus\Box\bm{bf}$, $\logic{QK}\oplus\Box\Box\bm{bf}$, etc.; the logic $\logic{QK}\oplus\forall x\forall y\,(P(x)\leftrightarrow P(y))$ is not validated by a Kripke frame, but is validated by every e\nobreakdash-augmented frame whose local domains are singletons. Therefore, it seems reasonable to deal with~$\caf^{\mathit{e}} L$ rather than~$\ckf L$; in some sense it is really the case. 
%
%At the same time, the gain from such approach is not so great, since for the most interesting logics, replacing the class of Kripke frames with the class of e\nobreakdash-augmented frames does not give a different logic. 
%
At the same time, the benefit of such an approach is not significant, as for the most interesting logics, replacing the class of Kripke frames with the class of e\nobreakdash-augmented frames does not result in a different logic.
Nevertheless, we will be interested in logics of c\nobreakdash-augmented frames; but we can describe the locally constant domain condition using the Barcan formula, and the difficulty does not arise. Thus, for such logics, the use of Kripke frames instead of e\nobreakdash-augmented frames will not lead to a violation of generality.

\subsection{Recursive inseparability of dyadic fragments}

Here, we make some straightforward observations concerning recursive inseparability of modal logics when their language contains a binary predicate letter.

We start with recursive inseparability of a modal predicate logic and the logic of its e-augmented frames with finite local domains.

\begin{lemma}
\label{lem:ml:insep:bin}
Let $L$ and $L'$ be modal logics such that $L\subseteq L'$, $L\cap\lang{L}=\logic{QCl}$, and $L'\cap\lang{L}=\logic{QCl}_{\mathit{fin}}$. Then $L$ and $L'\oplus\bm{bf}$ are recursively inseparable in a language containing a binary predicate letter and three individual variables.
\end{lemma}

\begin{proof}
Immediate from Theorem~\ref{th:Trakhtenbrot:bin:P}.
\end{proof}

\begin{corollary}
\label{cor:lem:ml:insep:bin}
Let $L$ and $L'$ be modal logics such that $L\subseteq L'$, $L\cap\lang{L}=\logic{QCl}$, and $L'\cap\lang{L}=\logic{QCl}_{\mathit{fin}}$. Then $L$ and $L'$ are recursively inseparable in a language containing a binary predicate letter and three individual variables.
\end{corollary}

Also, as a corollaries, we obtain the following propositions.

\begin{proposition}
\label{prop:1:ml:insep:bin}
Let $L$ and $L'$ be modal logics such that $\logic{QCl}\subseteq L\subseteq L'$ and $\ckf L'\ne\varnothing$. Then $L$ and $L'_{\mathit{dfin}}\oplus\bm{bf}$ are recursively inseparable in a language containing a binary predicate letter and three individual variables.
\end{proposition}

\begin{proof}
Observe that then $\ckf L\ne\varnothing$, and therefore $L\cap\lang{L}=\logic{QCl}$. Also, observe that $L'_{dfin}\cap\lang{L}=\logic{QCl}_{\mathit{fin}}$. Thus, we can apply Lemma~\ref{lem:ml:insep:bin} to $L$ and~$L'_{\mathit{dfin}}$.
\end{proof}

\begin{proposition}
\label{prop:2:ml:insep:bin}
Let $L$ and $L'$ be modal logics such that $\logic{QK}\subseteq L\subseteq L'$ and also either $L'\subseteq \logic{QTriv}$ or $L'\subseteq \logic{QVer}$. Then $L$ and $L'_{\mathit{dfin}}\oplus\bm{bf}$ are recursively inseparable in a language containing a binary predicate letter and three individual variables.
\end{proposition}

\begin{proof}
Observe that $\ckf\logic{QTriv}\ne \varnothing$ and $\ckf\logic{QVer}\ne \varnothing$. Also, $\logic{QCl}\subseteq\logic{QK}$. Thus, we can apply Proposition~\ref{prop:1:ml:insep:bin} to $L$ and~$L'$.
\end{proof}

\begin{proposition}
\label{prop:3:ml:insep:bin}
Let $L$ be one of\/ $\logic{QK}$, $\logic{QT}$, $\logic{QD}$, $\logic{QKB}$, $\logic{QKTB}$, $\logic{QK4}$, $\logic{QS4}$, $\logic{QK5}$, $\logic{QS5}$, $\logic{QK45}$, $\logic{QKD45}$, $\logic{QK4B}$, $\logic{QGL}$, $\logic{QGrz}$. Then $L$ and $L_{\mathit{dfin}}\oplus\bm{bf}$ are recursively inseparable in a language containing a binary predicate letter and three individual variables.
\end{proposition}

\begin{proof}
Observe that $\logic{QK}\subseteq L$ and also $L\subseteq \logic{QTriv}$ or $L\subseteq \logic{QVer}$; then apply Proposition~\ref{prop:2:ml:insep:bin} taking $L'=L$.
\end{proof}

Now, let us turn to the situation when a modal predicate logic and the logic of its finite Kripke frames are considered; notice that the local domains of finite frames can be infinite.

\begin{proposition}
\label{prop:lem:2:ml:insep:bin}
Let $L$ and $L'$ be modal logics such that $L\subseteq L'$, $L\cap\lang{L}=\logic{QCl}$, and $\fin\ckf L' \ne \varnothing$. Then $L$ and $L'_{\mathit{wfin}}\oplus\bm{bf}$ are recursively inseparable in a language containing a binary predicate letter and three individual variables.
\end{proposition}

\begin{proof}
Follows from the Church theorem~\cite{Church36,TG87} or Theorem~\ref{th:Trakhtenbrot:bin:P}. Indeed, Theorem~\ref{th:Trakhtenbrot:bin:P} gives us that $\logic{QCl}$ is undecidable in a language containing a binary predicate letter and three individual variables. Suppose that $L$ and $L'_{\mathit{wfin}}\oplus\bm{bf}$ are recursively separable. Then $L\cap\lang{L}$ and $L'_{\mathit{wfin}}\oplus\bm{bf}\cap\lang{L}$ should also be recursively separable. But $L\cap\lang{L} = L'_{\mathit{wfin}}\oplus\bm{bf}\cap\lang{L} = \logic{QCl}$ (since $\fin\ckf L' \ne \varnothing$), and then $\logic{QCl}$ should be decidable, that is a contradiction. Thus, $L$ and $L'_{\mathit{wfin}}\oplus\bm{bf}$ are recursively inseparable.
\end{proof}

\begin{corollary}
\label{cor:lem:2:ml:insep:bin}
Let $L$ and $L'$ be modal logics such that $L\subseteq L'$, $L\cap\lang{L}=\logic{QCl}$, and $\fin\ckf L' \ne \varnothing$. Then $L$ and $L'_{\mathit{wfin}}$ are recursively inseparable in a language containing a binary predicate letter and three individual variables.
\end{corollary}

This lemma provides us with some other corollaries.

\pagebreak[3]

\begin{samepage}
\begin{proposition}
\label{prop:4:ml:insep:bin}
Let $L$ and $L'$ be modal logics such that $\logic{QK}\subseteq L\subseteq L'$ and also either $L'\subseteq \logic{QTriv}$ or $L'\subseteq \logic{QVer}$. Then $L$ and $L'_{\mathit{wfin}}\oplus\bm{bf}$ are recursively inseparable in a language containing a binary predicate letter and three individual variables.
\end{proposition}

\nopagebreak[3]

\begin{proof}
Observe that $\fin\ckf\logic{QTriv}\ne \varnothing$ and $\fin\ckf\logic{QVer}\ne \varnothing$. Also, $\logic{QCl}\subseteq\logic{QK}$. Thus, we can apply Proposition~\ref{prop:lem:2:ml:insep:bin} to $L$ and $L'$.
\end{proof}
\end{samepage}

\pagebreak[3]

\begin{proposition}
\label{prop:5:ml:insep:bin}
Let $L$ be one of\/ $\logic{QK}$, $\logic{QT}$, $\logic{QD}$, $\logic{QKB}$, $\logic{QKTB}$, $\logic{QK4}$, $\logic{QS4}$, $\logic{QK5}$, $\logic{QS5}$, $\logic{QK45}$, $\logic{QKD45}$, $\logic{QK4B}$, $\logic{QGL}$, $\logic{QGrz}$. Then $L$ and $L_{\mathit{wfin}}\oplus\bm{bf}$ are recursively inseparable in a language containing a binary predicate letter and three individual variables.
\end{proposition}

\begin{proof}
Observe that $\logic{QK}\subseteq L$ and also $L\subseteq \logic{QTriv}$ or $L\subseteq \logic{QVer}$; then apply Proposition~\ref{prop:4:ml:insep:bin} taking $L'=L$.
\end{proof}

These observations show us that the questions about the recursive inseparability of logics $L$ and $L_{\mathit{dfin}}$ or logics $L$ and $L_{\mathit{wfin}}$, for a ``natural'' modal predicate logic~$L$, is trivial due to the Trakhtenbrot theorem.

\subsection{Monadic fragments and three variables}

The monadic fragments of $\logic{QCl}$ and $\logic{QCl}_{\mathit{wfin}}$ coincide and are decidable~\cite[Chapter~21]{BBJ07}. At the same time, the monadic fragments of many modal predicate logics are undecidable, even in languages with a single unary predicate letter and two-three individual variables~\cite{RSh19SL,RShJLC20a,RShJLC21b}. Here, we improve many of the known results concerning languages with three individual variables.

We say that a Kripke frame $\kframe{F}=\otuple{W,R}$ satisfies the \defnotion{Kripke--Hughes--Cresswell condition} (for short, \defnotion{KHC}) if $R(w)$ is infinite for some $w\in W$. We say that a modal predicate logic $L$ is \defnotion{KHC-friendly} if there exists a Kripke frame $\kframe{F}\in\ckf L$ satisfying KHC. The monadic fragments of KHC-friendly modal predicate logics are known to be undecidable~\cite{Kripke62,HC96}. For our purposes, we need a weaker condition. 
We say that a class $\scls{C}$ of Kripke frames satisfies the \defnotion{weak Kripke--Hughes--Cresswell condition} (for short, \defnotion{wKHC}) if, for every $n\in\numN$, there exists a Kripke frame $\otuple{W,R}\in\scls{C}$ with $w \in W$ such that $|R(w)|\geqslant n$. Clearly, if a class of Kripke frames contains a frame satisfying KHC, then it also satisfies wKHC. However, if $\scls{C}$ contains only finite Kripke frames, then it can not satisfy KHC; at the same time, the class of all finite Kripke frames satisfies wKHC.  
Neither Kripke~\cite{Kripke62} nor Hughes and Cresswell~\cite{HC96} consider modal predicate logics determined by classes of finite Kripke frames. 
%
%Such logics were investigated by Skvortsov in the context of superintuitionistic predicate logics~\cite{Skvortsov05JSL}, and later algorithmic properties of monadic fragments of modal and superintuitionistic predicate logics determined by classes of finite Kripke frames were studied~\cite{RShJLC20a,RShJLC21b}. 
%
Skvortsov investigated such logics in the context of superintuitionistic predicate logics~\cite{Skvortsov05JSL}; later, the algorithmic properties of monadic fragments of modal and superintuitionistic predicate logics, determined by classes of finite Kripke frames, were studied~\cite{RShJLC20a,RShJLC21b}.
We say that a class $\scls{C}$ of Kripke frames is a \defnotion{Skvortsov class} if the class $\fin\scls{C}$ satisfies wKHC.

Let $Q$ be a unary predicate letter and $P$ a binary predicate letter. Let $S_2$ be a formula substitution that substitutes $\Box(Q(x_1)\vee Q(x_2))$ for $P(x_1,x_2)$ in $\lang{L}$-formulas. 

\begin{lemma}
\label{lem:ml:monadic:insep:1}
Let $L$ be a modal predicate logic such that\/ $\logic{QCl}\subseteq L$. Then, for every\/ $\lang{L}$-formula~$\varphi$,
$$
\begin{array}{lcl}
\varphi\in\logic{QCl} 
  & \imply %\iff 
  & S_2\varphi \in L. 
\end{array}
$$
\end{lemma}

\begin{proof}
If $\varphi\in\logic{QCl}$, then $\varphi\in L$ (since $\logic{QCl}\subseteq L$), and $S_2\varphi \in L$ by Substitution.
\end{proof}

\begin{lemma}
\label{lem:ml:monadic:insep:2}
Let $L$ be a modal predicate logic such that\/ $\ckf L$ is a wKHC class. Then, for every $\lang{L}$\nobreakdash-formula~$\varphi$, containing no predicate letters except the binary letter~$P$,
$$
\begin{array}{lcl}
\varphi\not\in\logic{SIB}_{\mathit{fin}} %\uplus\bm{sib} 
  & \imply %\iff 
  & S_2\varphi \not\in L_{\mathit{dfin}}\oplus\bm{bf}. 
\end{array}
$$
\end{lemma}

\begin{proof}
Assume that $\varphi\not\in\logic{SIB}_{\mathit{fin}}$.
%, i.e., $\bm{sib}\to\varphi\not\in\logic{QCl}_{\mathit{fin}}$. 
Then, there exists a finite classical model $\cModel{M}=\otuple{\mathcal{D},\mathcal{I}}$ such that $\cModel{M}\models\bm{sib}$ and $\cModel{M}\not\models\varphi$. 
%Let $n=|\mathcal{D}|$. 
%Without a loss of generality we may assume the $\mathcal{D}=\{1,\ldots,n\}$. 
Let $\kframe{F}=\otuple{W,R}$ be a Kripke frame and $w_0$ a world such that $|R(w_0)\setminus\{w_0\}|\geqslant |\mathcal{D}|^2$; such a frame exists, since $\ckf L$ is a wKHC class. Then $R(w_0)\setminus\{w_0\}$ contains a subset $W'=\{w_{ij} : i,j\in\mathcal{D}\}$, where $w_{ij}=w_{kl}$ if, and only if, $i=k$ and $j=l$.
Let $\kModel{M}=\otuple{\kframe{F}\odot\mathcal{D},I}$ be a Kripke model such that, for every $w\in W$ and every $a\in\mathcal{D}$,
$$
\begin{array}{lcll}
\kModel{M},w\models Q(a) 
  & \iff 
  & \mbox{there is no $i,j\in\mathcal{D}$ such that $w=w_{ij}$, 
    $a\in\{i,j\}$, and $\cModel{M}\not\models P(i,j)$.}
\end{array}
$$
Notice that such a model exists, since $\cModel{M}\models \bm{sib}$.
Then, for all $a,b\in\mathcal{D}$,
$$
\begin{array}{lcll}
\cModel{M}\models P(a,b) 
  & \iff 
  & \kModel{M},w_0\models \Box(Q(a)\vee Q(b)).
\end{array}
$$
As a result, $\kModel{M},w_0\not\models S_2\varphi$, hence $S_2\varphi \not\in L_{\mathit{dfin}}\oplus\bm{bf}$.
\end{proof}

\begin{theorem}
\label{th:ml:monadic:insep:1}
Let $L$ and $L'$ be modal predicate logics such that\/ $\logic{QCl}\subseteq L\subseteq L'$ and $\ckf L'$ is a wKHC class. Then $L$ and $L'_{\mathit{dfin}}\oplus\bm{bf}$ are recursively inseparable in the language with a single unary predicate letter and three individual variables. 
\end{theorem}

\begin{proof}
Follows from Theorem~\ref{th:Trakhtenbrot:bin:sib:P}, Lemma~\ref{lem:ml:monadic:insep:1}, and Lemma~\ref{lem:ml:monadic:insep:2}.

Indeed, if $\varphi\in\logic{QCl}^{\mathit{bin}}$, then, by Lemma~\ref{lem:ml:monadic:insep:1}, $S_2\varphi\in L$; if $\varphi\not\in\logic{SIB}_{\mathit{fin}}$, then, by Lemma~\ref{lem:ml:monadic:insep:1}, $S_2\varphi\in L'_{\mathit{dfin}}\oplus\bm{bf}$. Therefore, assuming that $L$ and $L'_{\mathit{dfin}}\oplus\bm{bf}$ are recursively separable, we obtain that $\logic{QCl}^{\mathit{bin}}$ and $\logic{SIB}_{\mathit{fin}}$ should be recursively separable, which contradicts Theorem~\ref{th:Trakhtenbrot:bin:sib:P}.
\end{proof}

\begin{corollary}
\label{cor1:th:ml:monadic:insep:1}
Let $L$ and $L'$ be modal predicate logics such that\/ $\logic{QCl}\subseteq L\subseteq L'$ and $\ckf L'$ is a wKHC class. Then $L$ and $L'_{\mathit{dfin}}$ are recursively inseparable in the language with a single unary predicate letter and three individual variables. 
\end{corollary}

\begin{proof}
If $L$ and $L'_{\mathit{dfin}}$ are recursively separable, then $L$ and $L'_{\mathit{dfin}}\oplus\bm{bf}$ should be recursively separable, since $L'_{\mathit{dfin}}\subseteq L'_{\mathit{dfin}}\oplus\bm{bf}$, which contradicts Theorem~\ref{th:ml:monadic:insep:1}.
\end{proof}

\begin{lemma}
\label{lem:ml:monadic:insep:3}
Let $L$ be a modal predicate logic such that\/ $\ckf L$ is a Skvortsov class. Then, for every $\lang{L}$\nobreakdash-formula~$\varphi$, containing no predicate letters except the binary letter~$P$,
$$
\begin{array}{lcl}
\varphi\not\in\logic{SIB}_{\mathit{fin}} %\uplus\bm{sib} 
  & \imply %\iff 
  & S_2\varphi \not\in L_{\mathit{wfin}}\oplus\bm{bf}. 
\end{array}
$$
\end{lemma}

\begin{proof}
Similar to the proof of Lemma~\ref{lem:ml:monadic:insep:2} with the difference that the corresponding Kripke frame should be finite; such a frame exists since $\ckf L$ is a Skvortsov class.
\end{proof}

\begin{theorem}
\label{th:ml:monadic:insep:2}
Let $L$ and $L'$ be modal predicate logics such that\/ $\logic{QCl}\subseteq L\subseteq L'$ and $\ckf L'$ is a Skvortsov class. Then $L$ and $L'_{\mathit{wfin}}\oplus\bm{bf}$ are recursively inseparable in the language with a single unary predicate letter and three individual variables. 
\end{theorem}

\begin{proof}
Follows from Theorem~\ref{th:Trakhtenbrot:bin:sib:P}, Lemma~\ref{lem:ml:monadic:insep:1}, and Lemma~\ref{lem:ml:monadic:insep:3}.
\end{proof}

\begin{corollary}
\label{cor1:th:ml:monadic:insep:2}
Let $L$ and $L'$ be modal predicate logics such that\/ $\logic{QCl}\subseteq L\subseteq L'$ and $\ckf L'$ is a Skvortsov class. Then $L$ and $L'_{\mathit{wfin}}$ are recursively inseparable in the language with a single unary predicate letter and three individual variables. 
\end{corollary}

\begin{proof}
Similar to the proof of Corollary~\ref{cor1:th:ml:monadic:insep:1}.
\end{proof}

%Theorem~\ref{th:ml:monadic:insep:2} covers all the results presented in~\cite{RShJLC20a}, moreover, it captures also such logics as $\logic{QK5}$, $\logic{QS5}$, $\logic{QK45}$, $\logic{QKD45}$, $\logic{QK4B}$, and logics containing formulas bounding the depth of Kripke frames or formulas bounding the width of Kripke frames; the technique of~\cite{RShJLC20a} is not applicable to all of the logics mentioned. At the same time, the proof of the theorem is short and do not use any complicated technique. 
%
Theorem~\ref{th:ml:monadic:insep:2} encompasses all the results outlined in~\cite{RShJLC20a}. Additionally, it extends its scope to include logics such as $\logic{QK5}$, $\logic{QS5}$, $\logic{QK45}$, $\logic{QKD45}$, $\logic{QK4B}$, as well as logics containing formulas that limit the depth or width of Kripke frames. It is worth noting that the methodology employed in Theorem~\ref{th:ml:monadic:insep:2} is not universally applicable to all the mentioned logics.

Let $\bm{bd}_n$ and $\bm{bw}_n$ be the formulas bounding, respectively, the depth and the width of Kripke frames by~$n$, see~\cite[Propositions~3.42--3.44]{ChZ}.

\begin{corollary}
\label{cor3:th:ml:monadic:insep:1&2}
Let $L$ be one of\/ 
$\logic{QK}$, $\logic{QT}$, $\logic{QD}$, $\logic{QKB}$, $\logic{QKTB}$, $\logic{QK4}$, $\logic{QS4}$, $\logic{QK5}$, $\logic{QS5}$, $\logic{QK45}$, $\logic{QKD45}$, $\logic{QK4B}$, $\logic{QGL}$, $\logic{QGrz}$, $\logic{QK4}\oplus\bm{bd}_n$, $\logic{QS4}\oplus\bm{bd}_n$, $\logic{QK4}\oplus\bm{bw}_m$, $\logic{QS4}\oplus\bm{bw}_m$, 
$\logic{QGL}\oplus\bm{bd}_n$, $\logic{QGrz}\oplus\bm{bd}_n$, $\logic{QGL}\oplus\bm{bw}_m$, $\logic{QGrz}\oplus\bm{bw}_m$,
where $n\in\numNpp$ and $m\in\numNp$. 
Then $L$ and $L_{\mathit{wfin}}$ are recursively inseparable in the language with a single unary predicate letter and three individual variables, and the same for $L$ and $L_{\mathit{dfin}}$. 
\end{corollary}

\begin{remark}
\label{rem:modal:bot-free}
Notice that Theorems~\ref{th:ml:monadic:insep:1} and~\ref{th:ml:monadic:insep:2}, as well as Corollaries~\ref{cor1:th:ml:monadic:insep:1}, \ref{cor1:th:ml:monadic:insep:2}, and~\ref{cor3:th:ml:monadic:insep:1&2}, remain true if we additionally require that the language does not contain~$\bot$ (and hence, $\neg$ and~$\Diamond$ as well).
\end{remark}

\begin{proof}
Use Theorem~\ref{th:Trakhtenbrot:binP:gen:sib:srb:pos} and the definition of~$S_2$.
\end{proof}

\subsection{Examples}

We give some examples of logics with certain algorithmic properties to draw attention to a number of points that may not be obvious or may not directly follow from the stated results.

Let us start with 
%from 
a separating example for logics $L_{\mathit{dfin}}$ and $L_{\mathit{wfin}}$, for some modal predicate logic~$L$; moreover, we shall show that even the monadic fragments with a single unary predicate letter and a single individual variable can be different. To this end, let us take $L=\logic{QS4}$ and use a formula from~\cite{WZ01}. Define
$$
\begin{array}{lcl}
\mathit{InfW}
  & = &
%  \Box\exists x\,Q(x) \wedge
  \Box\exists x\,\neg Q(x) \wedge
  \Box\forall x\,(\neg Q(x)\to\Diamond Q(x)) \wedge
  \Box\forall x\,(Q(x)\to \Box Q(x)).
\end{array}
$$
%Observe that $\neg\mathit{InfW}$ can not be refuted on finite Kripke frame validating~$\logic{QS4}$. 
We can not refute $\neg\mathit{InfW}$ on a finite Kripke frame validating~$\logic{QS4}$. 
Indeed, let $\kModel{M}=\otuple{W,R,D,I}$ be a Kripke model with a reflexive transitive accessibility relation~$R$ such that $\kModel{M},w_0\models\mathit{InfW}$, for some $w_0\in W$. Then there exist $w_1\in R(w_0)$ and $a_1\in D_{w_1}$ such that $\kModel{M},w_1\models \neg Q(a_1)$. Then $\kModel{M},w_1\models \neg Q(a_1)\to \Diamond Q(a_1)$, and hence, $\kModel{M},w_1\models \Diamond Q(a_1)$. Then there exist $w_2\in R(w_1)$ and $a_2\in D_{w_2}$ such that $\kModel{M},w_2\models \Box Q(a_1)\wedge \neg Q(a_2)$. Since $\kModel{M},w_2\models \neg Q(a_2)$, repeating the same steps for $a_2$, we conclude that there exist $w_3\in R(w_2)$ and $a_3\in D_{w_3}$ such that $\kModel{M},w_3\models \Box Q(a_2)\wedge \neg Q(a_3)$. Reasoning in the same way, we come to the conclusion that there are worlds $w_1,w_2,w_3,\ldots{}$ and individuals $a_1,a_2,a_3,\ldots{}$ such that, for every $k\in\numNp$, 
$$
\begin{array}{lcl}
w_kRw_{k+1} 
  & \mbox{and}
  & \kModel{M},w_{k+1}\models \Box Q(a_k)\wedge \neg Q(a_{k+1}).
\end{array}
$$
It should be clear that $w_i\ne w_j$ if $i\ne j$, and hence, $W$ is infinite. Thus, $\neg\mathit{InfW}\in\logic{QS4}_{\mathit{wfin}}$. To show that $\neg\mathit{InfW}\not\in\logic{QS4}_{\mathit{dfin}}$, consider a Kripke model $\kModel{M} = \otuple{\numN,\leqslant,D,I}$ such that $D_k = \{0,\ldots,k\}$, for every $k\in\numN$, and 
$$
\begin{array}{lcl}
\kModel{M},k\models Q(m)
  & \iff
  & m < k,
\end{array}
$$
for every $k\in\numN$ and every $m\in D_k$.
Then, $\kModel{M},0\models \mathit{InfW}$. Since $\kModel{M}$ is based on a reflexive transitive Kripke frame and all its local domains are finite,\footnote{Observe that the global domain of $\kModel{M}$ is infinite.} $\kModel{M}\models\logic{QS4}_{\mathit{dfin}}$. Thus, $\neg\mathit{InfW}\not\in\logic{QS4}_{\mathit{dfin}}$. 

Observe that it is not hard to expand this example to other modal predicate logics. For example, to expand it on $\logic{QK4}$, replace $\Box$ with $\Box^+$, where $\Box^+\varphi = \varphi\wedge \Box\varphi$;
we leave the details to the reader.

Let us give separating examples for Proposition~\ref{prop:1:ml:insep:bin} and Theorem~\ref{th:ml:monadic:insep:1} as well as Proposition~\ref{prop:lem:2:ml:insep:bin} and Theorem~\ref{th:ml:monadic:insep:2}. Dummy examples are logics $\logic{QTriv}$ and $\logic{QVer}$. Since both logics satisfy the conditions of Propositions~\ref{prop:1:ml:insep:bin} and~\ref{prop:lem:2:ml:insep:bin}, the fragments with a single binary predicate letter and three variables of 
$\logic{QTriv}$ and $\logic{QTriv}_{\mathit{dfin}}$, 
$\logic{QTriv}$ and $\logic{QTriv}_{\mathit{wfin}}$, 
$\logic{QVer}$ and $\logic{QVer}_{\mathit{dfin}}$, 
$\logic{QVer}$ and $\logic{QVer}_{\mathit{wfin}}$ are recursively inseparable. Theorems~\ref{th:ml:monadic:insep:1} and~\ref{th:ml:monadic:insep:2} say nothing about the recursive separability of the monadic fragments of the same pairs of logics.
%
%Observe that the monadic fragments of all these logics are decidable (they are equivalent, for each of the logics, to the monadic fragment of~$\logic{QCl}$), therefore, clearly, such fragments of the same pairs of logics are recursively separable.
%
Note that the monadic fragments of all these logics are decidable (they are equivalent, for each of the logics, to the monadic fragment of~$\logic{QCl}$). Therefore, it is clear that such fragments of logics from the same pairs are recursively separable.
So, let us give less trivial examples.

To this end, let us make an obvious observation, cf.~\cite{MR:2017:LI,RShsubmitted,RShsubmitted2}.

\begin{proposition}
\label{prop:finite:frame}
Let $\kframe{F}$ be a finite Kripke frame. Then the monadic fragments of the logics\/ $\QMLe\kframe{F}$ and\/ $\QMLc\kframe{F}$ are decidable.
\end{proposition}

\begin{proof}
We give a sketch of a proof; for more details, see~\cite{MR:2017:LI,RShsubmitted,RShsubmitted2}.

Let $\kframe{F} = \otuple{W,R}$. Let $\varphi$ be an $\lang{ML}$-formula containing no predicate letters except unary letters $P_1,\ldots,P_n$ such that $\kModel{M},w_0\not\models\varphi$, for some Kripke model $\kModel{M}=\otuple{W,R,D,I}$ and world $w_0\in W$. Define the equivalence relation $\sim$ on~$D^+$ by 
$$
\begin{array}{lcl}
a\sim b 
  & \bydef 
  & \left\{
    \begin{array}{rcl}
    \{w\in W : a\in D_w\} & \!\!=\!\! & \{w\in W : b\in D_w\}; 
    \smallskip\\
    \{w\in W : \kModel{M},w\models P(a)\} & \!\!=\!\! & \{w\in W : \kModel{M},w\models P(b)\},
    \end{array}
    \right.
\end{array}
$$
and then define
$$
\begin{array}{rcl}
\bar{a}
  & =
  & \{b\in D^+ : b\sim a\};
  \smallskip\\
D'_w
  & = 
  & \{\bar{b} : b\in D_w\};
  \smallskip\\
I'(w,P_k)
  & = 
  & \{\otuple{\bar{b}} : \otuple{b}\in I(w,P_k)\};
  \smallskip\\
\kModel{M}'
  & = 
  & \otuple{W,R,D',I'},
\end{array}
$$
where $a\in D^+$, $w\in W$, and $k\in\{1,\ldots,n\}$. 
Then it is not hard to prove that $\kModel{M}',w_0\not\models\varphi$.

Observe that the domain of $\kModel{M}'$ contains at most $2^{|W|\cdot(n+1)}$ elements. Thus, to check whether $\varphi\in \QMLe\kframe{F}$, it is sufficient to check whether $\varphi$ is true in all such models; it is similar for~$\QMLc\kframe{F}$.
\end{proof}

So, due to Proposition~\ref{prop:finite:frame}, we may replace the logics $\logic{QTriv}$ and $\logic{QTriv}$ above with $\QMLe\kframe{F}$ or $\QMLc\kframe{F}$, where $\kframe{F}$ is a finite Kripke frame.

Let us turn to logics not definable by a finite Kripke frame.

Consider, for each $n\in\numN$, the logics $\logic{QAlt}_n$ and $\logic{QTAlt}_n$ defined by
$$
\begin{array}{rcl}
\logic{QAlt}_n
  & = 
  & \logic{QK}\oplus\bm{alt}_n; \\
\logic{QTAlt}_n
  & = 
  & \logic{QT}\hfill\oplus\bm{alt}_n,
\end{array}
$$
where
$$
\begin{array}{lcl}
\bm{alt}_n
  & = 
  & \displaystyle
    \neg\bigwedge\limits_{\mathclap{i=0}}^n \Diamond\Big(p_i\wedge\bigwedge\limits_{\mathclap{j\ne i}}\neg p_j\Big).
\end{array}
$$
The class $\ckf{\logic{QAlt}_n}$ consists of \defnotion{$n$-alternative} Kripke frames, i.e., where each world in a frame sees at most $n$ worlds; the same for $\ckf{\logic{QTAlt}_n}$, where the frames are both $n$-alternative and reflexive. 
It is known that $\logic{QAlt}_n$ and $\logic{QTAlt}_n$ are Kripke complete~\cite{ShehtmanShkkatov20,ShShFOMTL}, therefore, we obtain the following statement from Proposition~\ref{prop:finite:frame}, cf.~\cite{MR:2017:LI,RShsubmitted,RShsubmitted2}.

\begin{proposition}
\label{prop:QAlt-n:decidability}
The monadic fragments of\/ $\logic{QAlt}_n$, $\logic{QTAlt}_n$, $\logic{QAlt}_n\logic{.bf}$, and\/ $\logic{QTAlt}_n\logic{.bf}$ are decidable.
\end{proposition}

\begin{proof}
We give just a sketch of the proof for $\logic{QAlt}_n$; for more details, see~\cite{MR:2017:LI,RShsubmitted,RShsubmitted2}.

Indeed, let $\varphi$ be an $\lang{ML}$-formula containing no predicate letters except unary letters $P_1,\ldots,P_n$ such that $\kModel{M},w_0\not\models\varphi$, for some Kripke model $\kModel{M}=\otuple{W,R,D,I}$ based on an $n$-alternative Kripke frame $\kframe{F}=\otuple{W,R}$ and some world $w_0\in W$. 
Let $m$ be the modal depth\footnote{Largest number of nested modalities in the formula.} of~$\varphi$. 
Delete all worlds from $\kModel{M}$ except
$$
\{w\in W : \mbox{$w_0R^k w$, for some $k\leqslant m$}\},
$$
and denote the resulting Kripke model by~$\kModel{M}'$. It is not hard to see that $\kModel{M}'$ is based on an $n$-alternative Kripke frame of depth at most $m+1$, therefore, the number of worlds in $\kModel{M}'$ does not exceed the number
$$
1+n+n^2+\ldots+n^{m+1},
$$
and $\kModel{M},w_0\not\models\varphi$. Thus, to check whether $\varphi\in \logic{QAlt}_n$, it is sufficient to check whether $\varphi$ is valid on the $n$-alternative Kripke frame of depth $m+1$, which is possible by Proposition~\ref{prop:finite:frame}. 
\end{proof}

Notice that, by the proof, 
$$
\begin{array}{lclcl}
\logic{QAlt}_n & = & \QMLew\ckf{\logic{QAlt}_n}
               & = & \QMLe \ckf{\logic{QAlt}_n}; 
\smallskip\\
\logic{QAlt}_n\logic{.bf} & = & \QMLcw\ckf{\logic{QAlt}_n}
                          & = & \QMLc \ckf{\logic{QAlt}_n},
\end{array}
$$
and, in fact, we proved the dicidability of the monadic fragments of $\QMLe\ckf{\logic{QAlt}_n}$ and $\QMLc\ckf{\logic{QAlt}_n}$; the decidability of the corresponding fragments of $\logic{QAlt}_n$ and $\logic{QAlt}_n\logic{.bf}$ then follows from the Kripke completeness of the logics~\cite{ShehtmanShkkatov20,ShShFOMTL}. Without Kripke completeness, we could talk about the recursive separability of $\logic{QAlt}_n$ and $\QMLc\ckf{\logic{QAlt}_n}$. Let us generalize this observation.

A modal predicate logic $L$ is called \defnotion{subframe} if the class $\ckf L$ is closed under taking subframes\footnote{A Kripke frame $\otuple{W',R'}$ is a \defnotion{subframe} of a Kripke frame $\otuple{W,R}$ if $W'\subseteq W$ and $R' = R\upharpoonright W'$.} of its Kripke frames. For example, $\logic{QK}$, $\logic{QT}$, $\logic{QK4}$, $\logic{QS4}$, $\logic{QS5}$ are subframe logics and $\logic{QD}$, $\logic{QKD4}$ are not.  

\begin{proposition}
Let $L$ be a subframe modal predicate logic such that\/ 
$\logic{QK}\subseteq L$,
$L\subseteq \logic{QTriv}$ or\/ $L\subseteq \logic{QVer}$,
and the class\/ $\fin\ckf{L}$ is recursive.
Then the monadic fragments of\/ $\QMLe\ckf(L\oplus\bm{alt}_n)$ and\/ $\QMLc\ckf(L\oplus\bm{alt}_n)$ are decidable.
\end{proposition}

\begin{proof}
Similar to Proposition~\ref{prop:QAlt-n:decidability}. 
%
%The only difference is that we have also to check that the submodel corresponding to $\kModel{M}'$ in the proof of Proposition~\ref{prop:QAlt-n:decidability} is also validates~$L$; it is possible since $L$ is a subframe logic and the class $\fin\ckf{L}$ is recursive.
%
The only difference is that we also have to check that the submodel corresponding to $\kModel{M}'$ in the proof of Proposition~\ref{prop:QAlt-n:decidability} also validates~$L$. This is possible since~$L$ is a subframe logic and $\fin\ckf{L}$ is a recursive class.
\end{proof}

\begin{corollary}
Let $L$ be a subframe modal predicate logic such that\/ 
$\logic{QK}\subseteq L$,
$L\subseteq \logic{QTriv}$ or\/ $L\subseteq \logic{QVer}$,
and the class\/ $\fin\ckf{L}$ is recursive.
Then the monadic fragments of\/ $L\oplus\bm{alt}_n$ and\/ $\QMLc\ckf(L\oplus\bm{alt}_n)$ are recursively separable.
\end{corollary}

\begin{corollary}
Let $L$ be a one of\/ $\logic{QK}$, $\logic{QT}$, $\logic{QKB}$, $\logic{QKTB}$, $\logic{QK4}$, $\logic{QS4}$, $\logic{QK5}$, $\logic{QS5}$, $\logic{QK45}$, $\logic{QK4B}$, $\logic{QGL}$, $\logic{QGrz}$, $\logic{QwGrz}$, $\logic{QK4}\oplus\bm{bd}_m$, $\logic{QS4}\oplus\bm{bd}_m$, $\logic{QGL}\oplus\bm{bd}_m$, $\logic{QGrz}\oplus\bm{bd}_m$, $\logic{QwGrz}\oplus\bm{bd}_m$, where $m\in\numN^+$.
Then the monadic fragments of\/ $L\oplus\bm{alt}_n$ and\/ $\QMLc\ckf(L\oplus\bm{alt}_n)$ are recursively separable.
\end{corollary}

Now, let us pay attention to the fact that the logics under consideration are not necessarily in~$\Sigma^0_1$ or~$\Pi^0_1$. 
Of course, by Propositions~\ref{prop:Sigma01} and~\ref{prop:Pi01}, each of two recursively inseparable logics (or its fragment) with which we deal is hard in some of the classes; however, this does not necessarily imply membership within the class. We give some examples.
By Theorem~\ref{th:ml:monadic:insep:1}, the monadic fragments of $\QMLe\logic{QGL}$ and $\QMLed\logic{QGL}$ are recursively inseparable; and, by Theorem~\ref{th:ml:monadic:insep:2}, the same for the monadic fragments of $\QMLe\logic{QGL}$ and $\QMLew\logic{QGL}$. The monadic fragments of $\QMLew\logic{QGL}$ and $\QMLcw\logic{QGL}$ are $\Pi^0_1$-complete by~\cite[Corollary~4.3]{RShsubmitted}. But the monadic fragments of $\QMLe\logic{QGL}$ and $\QMLc\logic{QGL}$ are $\Pi^1_1$-hard~\cite{RybIGPL22} (in particular, $\logic{QGL}\ne \QMLe\logic{QGL}$). 
The same applies to $\logic{QGrz}$, $\logic{QwGrz}$, and some closely related logics considered in~\cite{RybIGPL22}.

Let us give examples where the conditions of Theorems~\ref{th:ml:monadic:insep:1} and~\ref{th:ml:monadic:insep:2} are not satisfied for logics $L$ and $L'$ but the monadic fragments of $L$ and $L'_{\mathit{dfin}}$ as well as the monadic fragments of $L$ and $L'_{\mathit{wfin}}$ are recursively inseparable. In~\cite{MNPERSON.1146721}, for any degree of unsolvability $A$, a linearly approximable propositional logic $\logic{L}_A$ with the variable-free fragment belonging to the degree is constructed. By construction, the logic contains $\bm{alt}_2$ (and $\Box\bm{alt}_1$); therefore, $\logic{QL}_A$ does not satisfy the conditions of Theorems~\ref{th:ml:monadic:insep:1} and~\ref{th:ml:monadic:insep:2} (if we take $L=L'=\logic{L}_A$). However, the monadic fragments of $\logic{QL}_A$ and $\QMLcd\logic{QL}_A$ as well as the monadic fragments of $\logic{QL}_A$ and $\QMLcw\logic{QL}_A$ are not recursively separable, since even the propositional variable-free fragments of the logics in the pairs are the same and undecidable.

Finally, we pay attention to very poor logical systems. As an example let us consider the modal predicate logic $\logic{QCl}^{\Box}$ defined by
$$
\begin{array}{lcl}
\logic{QCl}^{\Box} & = & \logic{QCl} + \Box\top\to\Box\top,
\end{array}
$$
i.e., $\logic{QCl}^{\Box}$ is, in fact, the classical logic whose language is enriched with~$\Box$ without any condition on~$\Box$. Clearly, $\logic{QCl}^\Box$ is contained in every modal predicate logic containing~$\logic{QCl}$. At first glance, it seems inappropriate to talk about the logics $\logic{QCl}^\Box_{\mathit{dfin}}$ and $\logic{QCl}^\Box_{\mathit{wfin}}$, since, by the definitions, they coincide with $\logic{QK}_{\mathit{dfin}}$ and $\logic{QK}_{\mathit{wfin}}$, respectively; but it seems appropriate to consider $\logic{QCl}^\Box_{\mathit{fin}}$ defined by
$$
\begin{array}{lcl}
\logic{QCl}^\Box_{\mathit{fin}} & = & \logic{QCl}_{\mathit{fin}} + \Box\top\to\Box\top.
\end{array}
$$
%Observe 
Note that, at least, $\logic{QCl}^\Box_{\mathit{fin}}\subseteq \logic{QK}_{\mathit{dfin}}\oplus\bm{bf}$, since $\logic{QK}_{\mathit{dfin}}\oplus\bm{bf}\cap \lang{L} = \logic{QCl}_{\mathit{fin}}$.\footnote{The logics are different. For example, $\Box\top\in\logic{QK}_{\mathit{dfin}}\oplus\bm{bf}$ but $\Box\top\not\in\logic{QCl}^\Box_{\mathit{fin}}$. We leave the details to the reader.} Then as a corollary we obtain the following proposition. 

\begin{proposition}
Logics\/ $\logic{QCl}^\Box$ and\/ $\logic{QCl}^\Box_{\mathit{fin}}$ are recursively inseparable in a language containing a single unary predicate letter and three individual variables.
\end{proposition}

\begin{proof}
Let $\varphi$ be an $\lang{L}$-formula containing no predicate letters except a binary letter~$P$ and containing just three individual variables.

If $\varphi\in\logic{QCl}^{\mathit{bin}}$, then $S_2\varphi\in \logic{QCl}^\Box$ by Substitution.

Suppose $\varphi\not\in\logic{SIB}_{\mathit{fin}}$. We are to prove that $S_2\varphi\not\in\logic{QCl}^\Box_{\mathit{fin}}$. Assume, for the sake of contradiction, that $S_2\varphi\in\logic{QCl}^\Box_{\mathit{fin}}$. Since $\logic{QCl}^\Box_{\mathit{fin}}\subseteq \logic{QK}_{\mathit{dfin}}$, we conclude that $S_2\varphi\in\logic{QK}_{\mathit{dfin}}$, which contradicts Lemma~\ref{lem:ml:monadic:insep:2}. Thus, $S_2\varphi\not\in\logic{QCl}^\Box_{\mathit{fin}}$.

Then, the statement follows by Theorem~\ref{th:Trakhtenbrot:bin:sib:P}.
\end{proof}

Notice that a similar statement can be proved for the logics $\logic{QCl} + \Box\top$ and $\logic{QCl}_{\mathit{fin}} + \Box\top$ as well as $\logic{QCl}\oplus \top$ and $\logic{QCl}_{\mathit{fin}}\oplus \top$, and many others. We leave the details to the reader.

\subsection{Monadic fragments and two variables}

It is known that the monadic fragments of modal logics can be undecidable even if the language contains only two individual variables~\cite{KKZ05} and a single unary predicate letter~\cite{RSh19SL}. For the logics of classes of finite frames, a similar result is proved with the use of three individual variables~\cite{RShJLC20a}. Using the ideas presented in~\cite{KKZ05} we shall show that, for many modal predicate logics, the fragment of a logic with a single unary predicate letter and two individual variables and the one of the logic of its finite Kripke frames are recursively inseparable. To this end, we can not use the results on theories like $\logic{SIB}$ and $\logic{SIB}_{\mathit{fin}}$ directly, since they are decidable in languages with fewer than three individual variables; therefore, we are going to describe the special $T_n$-tiling, for each $n\in\numN$, by $\lang{ML}$-formulas containing a single unary predicate letter and two individual variables. To simulate the special $T_n$-tiling in an appropriate way, let us slightly modify the constructions proposed in~\cite{KKZ05}.

Below, we shall use the formulas and denotations introduced in Section~\ref{subsec:Trakhtenbrot:theories}. We define a formula similar to $\mathit{Tiling}'_n$ using $\Box$ and just two individual variables. Let $C$ be a new unary predicate letter. Define (cf.~\cite{KKZ05}) the following $\lang{ML}$-formulas:
$$
\begin{array}{lcl}
\mathit{TC}^{\Box}_3
  & =
  & \Box\forall x\forall y\,(V'_n(x,y) \wedge \exists x\,(C(x)\wedge H'_n(y,x)) \to 
%    \phantom{\forall x\,(C(x)\to V'_n(y,x)));}
%  \\
%  &
%  & %\phantom{\forall x\forall y\,(V'_n(x,y) \wedge \exists x\,(C(x))}
%    \hfill
    \forall y\,(H'_n(x,y) \to \forall x\,(C(x)\to V'_n(y,x))));
  \smallskip\\
\mathit{TC}^{\Box}_5
  & =
  & \forall x\,\Diamond C(x);
  \smallskip\\
\mathit{TC}^{\Box}_6
  & =
  & \forall x\forall y\,(V'_n(x,y)\to \Box V'_n(x,y));
  \smallskip\\
\mathit{TC}^{\Box}_7
  & =
  & \forall x\forall y\,(H'_n(x,y)\to \Box H'_n(x,y));
  \smallskip\\
\mathit{TC}^{\Box}_8
  & =
  & \forall x\forall y\,(\Diamond V'_n(x,y)\to V'_n(x,y)).
\end{array}
$$
Letter $C$ allows us to replace $\mathit{TC}'_3$, which contains three individual variables, with $\mathit{TC}^{\Box}_3$ that contains two individual variables; we need $\mathit{TC}^{\Box}_5$--$\mathit{TC}^{\Box}_8$ to ensure that $\mathit{TC}^{\Box}_3$ works properly. Let, for convenience, 
$$
\begin{array}{cccc}
\mathit{TC}^{\Box}_0 = \mathit{TC}_0,
  & \mathit{TC}^{\Box}_1 = \mathit{TC}'_1,
  & \mathit{TC}^{\Box}_2 = \mathit{TC}'_2,
  & \mathit{TC}^{\Box}_4 = \mathit{TC}_4.
\end{array}
$$
Then, define 
$$
\begin{array}{rcl}
\mathit{Tiling}^\Box_n 
  & = 
  & \displaystyle
    \bigwedge\limits_{\mathclap{i=0}}^{8}\mathit{TC}^\Box_i 
    \wedge \Box^+\mathit{DSR}
    \wedge \Box^+\mathit{DSU}.   
\end{array}
$$
Finally, we define a new, modal, modification of $\mathit{Tiling}_n^{\mathbb{X}}$ by
$$
\begin{array}{lcl}
\mathit{M}^\Box\mathit{Tiling}_n^{\mathbb{X}} 
  & = 
  & \mathit{Tiling}^\Box_n \to \exists x\,P_1(x).
  \smallskip\\
\end{array}
$$

\begin{lemma}
\label{lem:1:tiling:QK}
If $n\in\mathbb{X}$, then\/ $\mathit{M}^\Box\mathit{Tiling}_n^{\mathbb{X}}\in\logic{QK}$.
\end{lemma}

\begin{proof}
Let $\kModel{M},w_0\models\mathit{Tiling}_n$ for a Kripke model $\kModel{M}=\langle W,R,D,I\rangle$ and a world $w_0\in W$. We show that then $\kModel{M},w_0\models \exists x\,P_1(x)$.
To this end, for all $i,j\in\numN$ we pick out an element $a_{i}^{j}\in D_{w_0}$ so that, for all $i,j\in\numN$,
$$
\begin{array}{lcl}
\kModel{M},w_0\models H'_n(a_i^j,a_{i+1}^j) 
  & \mbox{and} 
  & \kModel{M},w_0\models V'_n(a_i^j,a_i^{j+1}).
\end{array}
$$

Since $\kModel{M},w_0\models \mathit{TC}^\Box_4$, there exists $a_0^0\in D_{w_0}$ such that $\kModel{M},w_0\models P_0(a_0^0)$. 

Let $k$ be a number from $\numN$. Suppose that for all $i,j\in \{0,\ldots,k\}$, the element $a_i^j$ is defined; we have to define, for all $i,j\in \{0,\ldots,k\}$, the elements $a_{k+1}^j$, $a_i^{k+1}$, and $a_{k+1}^{k+1}$.

We start with $a_{0}^{k+1}$. Due to $\mathit{TC}^\Box_2$, there exists $b\in D_{w_0}$ such that $\kModel{M},w_0\models V'_n(a_0^k,b)$. Then take $a_{0}^{k+1}=b$. To define other elements, we first prove an auxiliary statement, cf.~\cite{KKZ05}.

\begin{sublemma}
\label{sublem:kkz:modal}
Let\/ $\kModel{M},w_0\models H'_n(a,c)\wedge V'_n(a,e)\wedge H'_n(e,b)$, for some $a,c,e,b\in D_{w_0}$. Then\/ $\kModel{M},w_0\models V'_n(c,b)$.
\end{sublemma}

\begin{proof}
Using $\mathit{TC}^\Box_5$, we obtain that $\kModel{M},w_0\models \Diamond C(b)$, and hence, there exists a world $w\in R(w_0)$ such that,
\begin{equation}
\label{eq:star:1}
\kModel{M},w\models C(b).
\end{equation} 
Since  $\kModel{M},w_0\models H'_n(a,c)\wedge V'_n(a,e)\wedge H'_n(e,b)$ and $w_0Rw$, we obtain, by $\mathit{TC}^\Box_6$ and $\mathit{TC}^\Box_7$, that
\begin{equation}
\label{eq:star:2}
\kModel{M},w\models H'_n(a,c)\wedge V'_n(a,e)\wedge H'_n(e,b).
\end{equation} 
Now, let us use the formula $\mathit{TC}^\Box_3$. It follows from $(\ref{eq:star:1})$ and $(\ref{eq:star:2})$ that $\kModel{M},w\models C(b)\wedge H'_n(e,b)$, and thus, $\kModel{M},w\models \exists x\,(C(x)\wedge H'_n(e,x))$; then, applying $(\ref{eq:star:2})$ again, we conclude that
\begin{equation}
\label{eq:star:3}
\kModel{M},w\models V'_n(a,e)\wedge \exists x\,(C(x)\wedge H'_n(e,x)).
\end{equation} 
By $\mathit{TC}^\Box_3$, from $(\ref{eq:star:3})$ we obtain that
$$
\kModel{M},w\models \forall y\,(H'_n(a,y)\to \forall x\,(C(x)\wedge V'_n(y,x)),
$$
and hence,
$$
\kModel{M},w\models H'_n(a,c)\to \forall x\,(C(x)\wedge V'_n(c,x)),
$$
that, by $(\ref{eq:star:2})$, gives us
$$
\kModel{M},w\models \forall x\,(C(x)\wedge V'_n(c,x)),
$$
and, by $(\ref{eq:star:1})$, 
\begin{equation}
\label{eq:star:4}
\kModel{M},w\models V'_n(c,b).
\end{equation} 
Since $w_0Rw$, we obtain $\kModel{M},w_0\models \Diamond V'_n(c,b)$ from $(\ref{eq:star:4})$. Then, by $\mathit{TC}^\Box_8$, we can conclude that $\kModel{M},w_0\models V'_n(c,b)$.
\end{proof}

Suppose that $a_0^{k+1},\ldots,a_i^{k+1}$ are already defined for some $i\in\{0,\ldots,k\}$; we have to define $a_{i+1}^{k+1}$. Due to $\mathit{TC}^\Box_1$, there exists $b\in D_{w_0}$ such that $\kModel{M},w_0\models H'_n(a_i^{k+1},b)$. Since also $\kModel{M},w_0\models H'_n(a_i^k,a_{i+1}^k)$ and $\kModel{M},w_0\models V'_n(a_i^k,a_{i}^{k+1})$, we obtain, by Sublemma~\ref{sublem:kkz:modal}, that $\kModel{M},w_0\models H'_n(a_{i+1}^{k+1},b)$, and we may take $a_{i+1}^{k+1}=b$.

Next step, let us define $a_{k+1}^{k+1}$. By $\mathit{TC}^\Box_1$, there exists $b\in D_{w_0}$ such that $\kModel{M},w_0\models H'_n(a_k^{k+1},b)$; take $a_{k+1}^{k+1}=b$. 

Suppose that $a_{k+1}^{k+1},\ldots,a_{k+1}^{j+1}$ are already defined for some $j\in\{0,\ldots,k\}$; we have to define $a_{k+1}^{j}$. By $\mathit{TC}^\Box_1$, there exists $c\in D_{w_0}$ such that $\kModel{M},w_0\models H'_n(a_k^{j},c)$. Then, 
$$
\kModel{M},w_0\models H'_n(a_k^j,c)\wedge V'_n(a_k^j,a_k^{j+1})\wedge H'_n(a_k^{j+1},a_{k+1}^{j+1}),
$$
and by Sublemma~\ref{sublem:kkz:modal}, $\kModel{M},w_0\models V'_n(c,a_{k+1}^{j+1})$. Hence, we may take $a_{k+1}^{k+1}=c$. 

For all $i,j\in \numN$, due to $\mathit{TC}^\Box_0$, there exists a unique $m\in\{0,\ldots,k_n\}$ such that $\kModel{M},w_0\models P_m(a_i^j)$. Then, put $f(i,j)=t^n_m$.
Clearly, $f\colon\numN\times\numN\to T_n$ is a $T_n$-tiling, since
$$
\begin{array}{lcl}
\kModel{M},w_0\models H_n(a_i^j,a_{i+1}^j) & \Longrightarrow & \rightsq f(i,j) = \leftsq f(i+1,j);
\smallskip\\
\kModel{M},w_0\models \hfill V_n(a_i^j,a_i^{j+1}) & \Longrightarrow & \upsq f(i,j) = \downsq f(i,j+1).
\end{array}
$$

Notice that $f(0,0)=t^n_0=t_0$, since $\kModel{M},w_0\models P_0(a_0^0)$. Then, by Proposition~\ref{prop:fn}, $f$ is the special $T_n$\nobreakdash-tiling~$f_n$. Since $n\in\mathbb{X}$, by $(\ref{eq:fn})$ we obtain that there exists $m\in\numN$ such that $f_n(0,m)=t_1=t^n_1$. This means that $\kModel{M},w_0\models P_1(a_0^m)$, and hence, $\kModel{M},w_0\models \exists x\,P_1(x)$.

Thus, $\mathit{M}^\Box\mathit{Tiling}_n^{\mathbb{X}}\in\logic{QK}$.
\end{proof}

\Rem{
\begin{lemma}
\label{lem:2:tiling:QK-wKHC}
Let $L$ be a modal predicate logic such that\/ $\ckf L$ is a wKHC class.
Then, $n\in\mathbb{Y}$ implies\/ $\mathit{M}^\Box\mathit{Tiling}_n^{\mathbb{X}}\not\in \QMLcd \ckf L$.
\end{lemma}

\begin{proof}
Let $n\in\mathbb{Y}$.
Then, by Lemma~\ref{lem:Trakhtenbrot:lem2:sib}, $\mathit{MTiling}_n^{\mathbb{X}}\not\in\logic{QCl}_{\mathit{fin}}\uplus\bm{sib}$. Hence,
there exists a finite classical model $\cModel{M}=\otuple{\mathcal{D},\mathcal{I}}$ such that $\cModel{M}\models\bm{sib}$ and $\cModel{M}\not\models\mathit{MTiling}_n^{\mathbb{X}}$; we may assume that $\cModel{M}$ is the model defined in the proof of Lemma~\ref{lem:Trakhtenbrot:lem2:sib}, in particular,
$$
\begin{array}{lcl}
\mathcal{D} & = & \{0,\ldots,r+4\}\times\{0,\ldots,r+4\}, \\
\end{array}
$$
for a suitable~$r\in\numN$.

Let $\kframe{F}=\otuple{W,R}$ be a Kripke frame and $w_0$ a world such that $|R(w_0)\setminus\{w_0\}|\geqslant |\mathcal{D}|$; such a frame exists, since $\ckf L$ is a wKHC class. Then $R(w_0)\setminus\{w_0\}$ contains a subset $W'=\{w_{a} : a\in\mathcal{D}\}$, where $w_{a}=w_{b}$ if, and only if, $a=b$.
Let $\kModel{M}=\otuple{\kframe{F}\odot\mathcal{D},I}$ be a Kripke model such that, for every $w\in W$ and every predicate letter~$E$,
$$
\begin{array}{lcl}
I(w,E) 
  & = 
  & \left\{ 
    \begin{array}{rl}
      \mathcal{I}(E) & \mbox{if $E\ne C$;} \\
      \{\otuple{a}\} & \mbox{if $E=C$ and $w=w_a$.} \\
    \end{array}
    \right.
\end{array}
$$
Then it is not hard to show by routine verification that $\kModel{M},w_0\not\models \mathit{M}^\Box\mathit{Tiling}_n^{\mathbb{X}}$.
\end{proof}

\begin{lemma}
\label{lem:2:tiling:QK-S}
Let $L$ be a modal predicate logic such that\/ $\ckf L$ is a Skvortsov class.
Then, $n\in\mathbb{Y}$ implies\/ $\mathit{M}^\Box\mathit{Tiling}_n^{\mathbb{X}}\not\in \QMLcw \ckf L$.
\end{lemma}

\begin{proof}
Similarly to the proof of Lemma~\ref{lem:2:tiling:QK-wKHC} with the only difference that the corresponding Kripke frame should be finite; such a frame exists, since $\ckf L$ is a Skvortsov class.
A~Skvortsov class is also a wKHC class.
\end{proof}
}

% STOP %

The next aim is to eliminate almost all unary predicate letters by simulating them with formulas containing the binary letter~$P$ and two individual variables. Let us make a relativization. Let $G$ be a new unary predicate letter and, as before,
$$
\begin{array}{lclclcl}
\forall_Gx\,\varphi & = & \forall x\,(G(x)\to\varphi); 
\smallskip\\
\exists_Gx\,\varphi & = & \exists x\,(G(x)\hfill\wedge\hfill\varphi). 
\end{array}
$$
Also, for an $\lang{ML}$-formula~$\varphi$, denote by $\varphi_G$ the formula obtained from~$\varphi$ by replacing each quantifier $\forall x$ or $\exists x$ with $\forall_G x$ or $\exists_G x$, respectively. 
%Let us define $\Box^+$ in a usual way: $\Box^+\varphi = \varphi\wedge\Box\varphi$.

Define
$$
\begin{array}{lclclcl}
\mathit{M}^\Box_G\mathit{Tiling}_n^{\mathbb{X}} 
  & = 
  & \exists x\,G(x)\wedge \forall x\,(G(x)\to \Box G(x)) \to (\mathit{M}^\Box\mathit{Tiling}_n^{\mathbb{X}})_G. 
\end{array}
$$

\begin{lemma}
\label{lem:relativization:G:K}
If $n\in\mathbb{X}$, then $\mathit{M}^\Box_G\mathit{Tiling}_n^{\mathbb{X}} \in \logic{QK}$. 
\end{lemma}

\begin{proof}
The statement should be clear; we just give a sketch of the proof.

Let $n\in\mathbb{X}$. Suppose that $\mathit{M}^\Box_G\mathit{Tiling}_n^{\mathbb{X}} \not\in \logic{QK}$. Consequently, there exist a model $\kmodel{M}=\otuple{W,R,D,I}$ and a world $w\in W$ such that $\kmodel{M},w\not\models \mathit{M}^\Box_G\mathit{Tiling}_n^{\mathbb{X}}$. Consider the model $\kmodel{M}'=\otuple{W',R',D',I'}$ defined by
$$
\begin{array}{rcll}
W' & = & \{w\}\cup R(w); \\
R' & = & R \upharpoonright W'; \\
D'_v & = & \{a\in D_v : \kmodel{M},v\models G(a)\}, & \mbox{where $v\in W'$;} \\
I'(v,E) & = & I(v,E) \upharpoonright D'_v, & \mbox{where $v\in W'$ and $E$ is a predicate letter.}
\end{array}
$$
Then it is not hard to check that $\kmodel{M}',w\not\models \mathit{M}^\Box\mathit{Tiling}_n^{\mathbb{X}}$; we leave the details to the reader. This contradicts Lemma~\ref{lem:1:tiling:QK}, and therefore, $\mathit{M}^\Box_G\mathit{Tiling}_n^{\mathbb{X}} \in \logic{QK}$.
\end{proof}

\Rem{
\begin{lemma}
\label{lem:relativization:G:wKHC}
Let $L$ be a modal predicate logic such that\/ $\ckf L$ is a wKHC class. Then, $n\in\mathbb{Y}$ implies
$\mathit{M}^\Box_G\mathit{Tiling}_n^{\mathbb{X}} \not\in \QMLcd \ckf L$. 
\end{lemma}

\begin{proof}
Follows from Lemma~\ref{lem:2:tiling:QK-wKHC}. In particular, it is sufficient to expand model $\kmodel{M}=\otuple{\kframe{F}\odot\mathcal{D},I}$ defined in the proof of Lemma~\ref{lem:2:tiling:QK-wKHC} with the condition that $\kmodel{M},w\models G(a)$, for every its world $w$ and every $a\in\mathcal{D}$; then it should be clear that $\mathit{M}^\Box_G\mathit{Tiling}_n^{\mathbb{X}}$ is refuted in the resulting model at~$w_0$.
\end{proof}

\begin{lemma}
\label{lem:relativization:G:S}
Let $L$ be a modal predicate logic such that\/ $\ckf L$ is a Skvortsov class. Then, $n\in\mathbb{Y}$ implies
$\mathit{M}^\Box_G\mathit{Tiling}_n^{\mathbb{X}} \not\in \QMLcw \ckf L$. 
\end{lemma}

\begin{proof}
Follows similarly from Lemma~\ref{lem:2:tiling:QK-S}.
\end{proof}
}

Let us define a function~$S_3$, which is similar to~$S_1$. Define $S_3\mathit{M}^\Box_G\mathit{Tiling}_n^{\mathbb{X}}$ as the formula obtained from $\mathit{M}^\Box_G\mathit{Tiling}_n^{\mathbb{X}}$ by replacing 
\begin{itemize}
\item
each occurrence of $P_m(x)$ and $P_m(y)$ with $\mathit{tile}'_m(x)$ and $\mathit{tile}'_m(y)$, respectively, where $m\in\{0,\ldots,k_n\}$;
\item 
then each occurrence of $R_i(x)$ and $R_i(y)$ with $\mathit{tile}'_{k_n+i}(x)$ and $\mathit{tile}'_{k_n+i}(y)$, respectively, where $1\leqslant i\leqslant 4$;
\item
and then each occurrence of $U_j(x)$ and $U_j(y)$ with $\mathit{tile}'_{k_n+j+4}(x)$ and $\mathit{tile}'_{k_n+j+4}(y)$, respectively, where $1\leqslant j\leqslant 4$. 
\end{itemize}
Notice that these replacements are formula substitutions.\footnote{See footnote~\ref{footnote:1}.}

Observe that, for every $n\in\numN$, formula $S_3\mathit{M}^\Box_G\mathit{Tiling}_n^{\mathbb{X}}$ contains two individual variables and does not contain predicate letters except $P$, $C$, and~$G$.

\begin{lemma}
\label{lem:relativization:G:K:S3}
If $n\in\mathbb{X}$, then $S_3\mathit{M}^\Box_G\mathit{Tiling}_n^{\mathbb{X}} \in \logic{QK}$. 
\end{lemma}

\begin{proof}
Immediately follows from Lemma~\ref{lem:relativization:G:K}, since $\logic{QK}$ is closed under Substitution.
\end{proof}

\begin{lemma}
\label{lem:relativization:G:wKHC:S3}
Let $L$ be a modal predicate logic such that\/ $\ckf L$ is a wKHC class. Then, $n\in\mathbb{Y}$ implies 
$S_3\mathit{M}^\Box_G\mathit{Tiling}_n^{\mathbb{X}} \not\in \QMLcd \ckf L$. 
\end{lemma}

\begin{proof}
We are going to apply an argumentation similar to that used in the proof of Lemma~\ref{lem:Trakhtenbrot:lem2:sib:binP}. In fact, the difference is that we do not need to simulate $G$.

Let $n\in\mathbb{Y}$.
Then, by Lemma~\ref{lem:Trakhtenbrot:lem2:sib}, $\mathit{MTiling}_n^{\mathbb{X}}\not\in\logic{QCl}_{\mathit{fin}}\uplus\bm{sib}$. Hence,
there exists a finite classical model $\cModel{M}=\otuple{\mathcal{D},\mathcal{I}}$ such that $\cModel{M}\models\bm{sib}$ and $\cModel{M}\not\models\mathit{MTiling}_n^{\mathbb{X}}$; we may assume that $\cModel{M}$ is the model defined in the proof of Lemma~\ref{lem:Trakhtenbrot:lem2:sib}, in particular,
$$
\begin{array}{lcl}
\mathcal{D} & = & \{0,\ldots,r+4\}\times\{0,\ldots,r+4\}, \\
\end{array}
$$
for a suitable~$r\in\numN$.

For every $a\in \mathcal{D}$, take the elements
$e_0^a,\ldots,e_{k_n+10}^a$ and $e_P^a, e_R^a, e_U^a$; let
$$
\begin{array}{lcl}
\mathcal{D}' 
  & = 
  & \mathcal{D} \cup \{e_0^a,\ldots,e_{k_n+10}^a, e_P^a, e_R^a, e_U^a : a\in\mathcal{D}\}.
\end{array}
$$

Let $\kframe{F}=\otuple{W,R}$ be a Kripke frame from $\ckf L$ with a world $w_0$ such that $|R(w_0)\setminus\{w_0\}|\geqslant |\mathcal{D}|$; such a frame exists, since $\ckf L$ is a wKHC class. Then $R(w_0)\setminus\{w_0\}$ contains a subset $W'=\{w_{a} : a\in\mathcal{D}\}$, where $w_{a}=w_{b}$ if, and only if, $a=b$.
Let $\kModel{M}=\otuple{\kframe{F}\odot\mathcal{D}',I}$ be a Kripke model such that, for every $w\in W$ and every $a\in\mathcal{D}'$,
$$
\begin{array}{lcl}
\kmodel{M},w\models G(a) & \iff & a\in \mathcal{D}; \\
\kmodel{M},w\models C(a) & \iff & w = w_a,
\end{array}
$$
and, for every $w\in W$, $I(w,P)$ is the symmetric closure of the relation
$$
\begin{array}{l}
\mathcal{I}(P) \cup 
  \{\otuple{a,e_0^a},\otuple{e_0^a,e_1^a},\ldots,\otuple{e_{k_n+9}^a,e_{k_n+10}^a} : a\in\mathcal{D}\}
  \\
  \phantom{\mathcal{I}(P)} \cup 
  \{\otuple{e_m^a,e_P^a} : \mbox{$m\in\{0,\ldots,k_n\}$ and $\cModel{M}\models P_m(a)$} \}
  \\
  \phantom{\mathcal{I}(P)} \cup
  \{\otuple{e_{k_n+i}^a,e_R^a} : \mbox{$i\in\{1,2,3,4\}$ and $\cModel{M}\models R_i(a)$} \}
  \\
  \phantom{\mathcal{I}(P)} \cup 
  \{\otuple{e_{k_n+j+4}^a,e_U^a} : \mbox{$j\in\{1,2,3,4\}$ and $\cModel{M}\models U_j(a)$} \}.
\end{array}
$$
see Figure~\ref{fig:16}.
Then it should be clear that $\kModel{M},w_0\not\models S_3\mathit{M}^\Box_G\mathit{Tiling}_n^{\mathbb{X}}$; we leave the details to the reader.

Thus, $S_3\mathit{M}^\Box_G\mathit{Tiling}_n^{\mathbb{X}} \not\in \QMLcd \ckf L$. 
\end{proof}

\begin{figure}
\centering
\begin{tikzpicture}[scale=3.2, rectnode/.style={rectangle, thick, draw=black!60, dashed, rounded corners = 2pt}]

\coordinate (c1) at (-1.4,0.30); %(-1.4, 0.30);
\coordinate (c2) at (-0.9,-0.3); %(-1.0,-0.3);
\coordinate (c3) at (0.2,-0.05); %( 0.0,-0.05);
\coordinate (c4) at (-0.4,0.5); %(-0.5, 0.50);

\coordinate (c12) at ($0.5*(c1)+0.5*(c2)$);
\coordinate (c23) at ($0.5*(c2)+0.5*(c3)$);
\coordinate (c34) at ($0.5*(c3)+0.5*(c4)$);
\coordinate (c41) at ($0.5*(c4)+0.5*(c1)$);

\coordinate (vecX1) at ($(c3)-(c2)$);
\coordinate (vecX2) at ($(c4)-(c1)$);
\coordinate (vecY1) at ($(c1)-(c2)$);
\coordinate (vecY2) at ($(c4)-(c3)$);

\coordinate (c1r) at (c4);
\coordinate (c2r) at (c3);
\coordinate (c3r) at ($(c3)+0.96*(vecX1)$);
\coordinate (c4r) at ($(c4)+0.96*(vecX2)$);

\coordinate (c1u) at ($(c1)+0.84*(vecY1)$);
\coordinate (c2u) at (c1);
\coordinate (c3u) at (c4);
\coordinate (c4u) at ($(c4)+0.84*(vecY2)$);

\coordinate (c12r) at ($0.5*(c1r)+0.5*(c2r)$);
\coordinate (c23r) at ($0.5*(c2r)+0.5*(c3r)$);
\coordinate (c34r) at ($0.5*(c3r)+0.5*(c4r)$);
\coordinate (c41r) at ($0.5*(c4r)+0.5*(c1r)$);

\coordinate (c12u) at ($0.5*(c1u)+0.5*(c2u)$);
\coordinate (c23u) at ($0.5*(c2u)+0.5*(c3u)$);
\coordinate (c34u) at ($0.5*(c3u)+0.5*(c4u)$);
\coordinate (c41u) at ($0.5*(c4u)+0.5*(c1u)$);

\draw [white, opacity = 0, name path = dg 1] (c1)--(c3);
\draw [white, opacity = 0, name path = dg 2] (c2)--(c4);
\draw [name intersections = {of = dg 1 and dg 2, by = {c}}];
\draw [white, opacity = 0, name path = dg 3] (c12)--(c34);
\draw [white, opacity = 0, name path = dg 4] (c23)--(c41);
\draw [name intersections = {of = dg 3 and dg 4, by = {d}}];

\draw [white, opacity = 0, name path = dgr 1] (c1r)--(c3r);
\draw [white, opacity = 0, name path = dgr 2] (c2r)--(c4r);
\draw [name intersections = {of = dgr 1 and dgr 2, by = {cr}}];
\draw [white, opacity = 0, name path = dgr 3] (c12r)--(c34r);
\draw [white, opacity = 0, name path = dgr 4] (c23r)--(c41r);
\draw [name intersections = {of = dgr 3 and dgr 4, by = {dr}}];

\draw [white, opacity = 0, name path = dgu 1] (c1u)--(c3u);
\draw [white, opacity = 0, name path = dgu 2] (c2u)--(c4u);
\draw [name intersections = {of = dgu 1 and dgu 2, by = {cu}}];
\draw [white, opacity = 0, name path = dgu 3] (c12u)--(c34u);
\draw [white, opacity = 0, name path = dgu 4] (c23u)--(c41u);
\draw [name intersections = {of = dgu 3 and dgu 4, by = {du}}];

\coordinate (diffR) at ($(cr)-(c)$);
\coordinate (diffU) at ($(cu)-(c)$);
\coordinate (corR) at (0,0.012);
\coordinate (corU) at (0,0.032);

\coordinate (em)  at ($(c)+1*(0,0.29)$);
\coordinate (e2)  at ($(c)+2*(0,0.29)$);
\coordinate (e1)  at ($(c)+3*(0,0.29)$);
\coordinate (e0)  at ($(c)+4*(0,0.29)$);
\coordinate (e0') at ($(e0)+ (0,0.21)$);
\coordinate (e0'') at ($(e0')+ (0,0.21)$);

\coordinate (em') at ($(em)+(0.2,0.048)$);
\coordinate (e2') at ($(e2)+(0.2,0.048)$);
\coordinate (e1') at ($(e1)+(0.2,0.048)$);
\coordinate (e3') at ($(e0)+(0.2,0.048)$);

\coordinate (ekn) at ($(c)-1*(0,0.29)$);
\coordinate (ek2) at ($(c)-2*(-0.28*0.25,0.285*0.85)$);
\coordinate (ek1) at ($(c)-2*(+0.26*0.25,0.325*0.85)$);
%\coordinate (ek0) at ($(c)-4*(0,0.31)$);

\coordinate (rem)  at ($(cr)+1*(0,0.29)-1*(corR)$);
\coordinate (re2)  at ($(cr)+2*(0,0.29)-2*(corR)$);
\coordinate (re1)  at ($(cr)+3*(0,0.29)-3*(corR)$);
\coordinate (re0)  at ($(cr)+4*(0,0.29)-4*(corR)$);
\coordinate (re0') at ($(re0)+ (0,0.21)-1*(corR)$);
\coordinate (re0'') at ($(re0')+ (0,0.21)-1*(corR)$);

\coordinate (rem') at ($(rem)+0.96*(0.2,0.048)$);
\coordinate (re2') at ($(re2)+0.96*(0.2,0.048)$);
\coordinate (re1') at ($(re1)+0.96*(0.2,0.048)$);
\coordinate (re3') at ($(re0)+0.96*(0.2,0.048)$);

\coordinate (rekn) at ($(cr)-1*(0,0.29)+1*(corR)$);
\coordinate (rek2) at ($(cr)-2*(-0.28*0.25,0.285*0.85)+1.7*(corR)$);
\coordinate (rek1) at ($(cr)-2*(+0.26*0.25,0.325*0.85)+1.7*(corR)$);
%\coordinate (rek0) at ($(c)-4*(0,0.31)$);

\coordinate (uem)  at ($(cu)+1*(0,0.29)-1*(corU)$);
\coordinate (ue2)  at ($(cu)+2*(0,0.29)-2*(corU)$);
\coordinate (ue1)  at ($(cu)+3*(0,0.29)-3*(corU)$);
\coordinate (ue0)  at ($(cu)+4*(0,0.29)-4*(corU)$);
\coordinate (ue0') at ($(ue0)+ (0,0.21)-1*(corU)$);
\coordinate (ue0'') at ($(ue0')+ (0,0.21)-1*(corU)$);

\coordinate (uem') at ($(uem)+0.96*(0.2,0.040)$);
\coordinate (ue2') at ($(ue2)+0.96*(0.2,0.040)$);
\coordinate (ue1') at ($(ue1)+0.96*(0.2,0.040)$);
\coordinate (ue3') at ($(ue0)+0.96*(0.2,0.040)$);

\coordinate (uekn) at ($(cu)-1*(0,0.29)+1*(corU)$);
\coordinate (uek2) at ($(cu)-2*(-0.28*0.25,0.285*0.85)+1.7*(corU)$);
\coordinate (uek1) at ($(cu)-2*(+0.26*0.25,0.320*0.85)+1.7*(corU)$);
%\coordinate (uek0) at ($(c)-4*(0,0.31)$);

\Rem{
\shade [ball color=black] (ek2) circle [radius = 1.5pt];
\draw  [>=latex, <->, shorten >= 3.5pt, shorten <= 3.5pt, color=black] (ekn)--(ek2);
\shade [ball color=black] (ekn) circle [radius = 1.5pt];
\draw  [>=latex, <->, shorten >= 3pt, shorten <= 3pt, color=black] (c)--(ekn);
\draw  [>=latex, <->, shorten >= 3.5pt, shorten <= 3.5pt, color=black] (ek2)--(ek1);
\draw  [>=latex, <->, shorten >= 3.5pt, shorten <= 3.5pt, color=black] (ekn)--(ek1);
\shade [ball color=black] (ek1) circle [radius = 1.5pt];

\begin{scope}[color=black!42]
\shade [ball color=black!84] (rek2) circle [radius = 0.96*1.5pt];
\filldraw [color=white, opacity = 0.5] (rek2) circle [radius = 0.96*1.55pt];
\draw  [>=latex, <->, shorten >= 3.5pt, shorten <= 3.5pt] (rekn)--(rek2);
\shade [ball color=black!84] (rekn) circle [radius = 0.96*1.5pt];
\filldraw [color=white, opacity = 0.5] (rekn) circle [radius = 0.96*1.55pt];
\draw  [>=latex, <->, shorten >= 3pt, shorten <= 3pt] (cr)--(rekn);
\draw  [>=latex, <->, shorten >= 3.5pt, shorten <= 3.5pt] (rek2)--(rek1);
\draw  [>=latex, <->, shorten >= 3.5pt, shorten <= 3.5pt] (rekn)--(rek1);
\shade [ball color=black!84] (rek1) circle [radius = 0.96*1.5pt];
\filldraw [color=white, opacity = 0.5] (rek1) circle [radius = 0.96*1.55pt];
\shade [ball color=black!84] (uek2) circle [radius = 0.90*1.5pt];
\filldraw [color=white, opacity = 0.5] (uek2) circle [radius = 0.90*1.55pt];
\draw  [>=latex, <->, shorten >= 3.5pt, shorten <= 3.5pt] (uekn)--(uek2);
\shade [ball color=black!84] (uekn) circle [radius = 0.90*1.5pt];
\filldraw [color=white, opacity = 0.5] (uekn) circle [radius = 0.90*1.55pt];
\draw  [>=latex, <->, shorten >= 3pt, shorten <= 3pt] (cu)--(uekn);
\draw  [>=latex, <->, shorten >= 3.5pt, shorten <= 3.5pt] (uek2)--(uek1);
\draw  [>=latex, <->, shorten >= 3.5pt, shorten <= 3.5pt] (uekn)--(uek1);
\shade [ball color=black!84] (uek1) circle [radius = 0.90*1.5pt];
\filldraw [color=white, opacity = 0.5] (uek1) circle [radius = 0.90*1.55pt];
\end{scope}
}

\begin{scope}[color=black!42]
\drawtileflattmslanted{(c1)}{(c2)}{(c3)}{(c4)}
\end{scope}

\begin{scope}[color=black!25]
\drawtileflattmslanted{(c1r)}{(c2r)}{(c3r)}{(c4r)}
\drawtileflattmslanted{(c1u)}{(c2u)}{(c3u)}{(c4u)}
\end{scope}

\begin{scope}[>=latex, <->, shorten >= -7.5pt, shorten <= 4pt, color=black!32]
\draw [] (cr)--($(c34r)+(cr)-(dr)$);
\draw [] (cr)--($(c41r)+(cr)-(dr)$);
\draw [] (cu)--($(c34u)+(cu)-(du)$);
\draw [] (cu)--($(c41u)+(cu)-(du)$);
\end{scope}

\shade [ball color=black!64] (cr) circle [radius = 0.96*1.5pt];
\filldraw [color=white, opacity = 0.5] (cr) circle [radius = 0.96*1.55pt];
\shade [ball color=black!64] (cu) circle [radius = 0.90*1.5pt];
\filldraw [color=white, opacity = 0.5] (cu) circle [radius = 0.90*1.55pt];

\begin{scope}[>=latex, <->, shorten >= 3pt, shorten <= -7.5pt, color=black!32]
\draw [] ($(c12u)+(cu)-(du)$)--(cu);
\draw [] ($(c23r)+(cr)-(dr)$)--(cr);
\end{scope}

\begin{scope}[>=latex, <->, shorten >= 3pt, shorten <= 2pt, color=black!84]
\draw [] (c)--(cr); %($(c34)+(c)-(d)$);
\draw [] (c)--(cu); %($(c41)+(c)-(d)$);
\end{scope}

\shade [ball color=black!64] (c) circle [radius = 1.5pt];

\draw  [>=latex, <->, shorten >= 3pt, shorten <= 3pt, color=black, densely dashed] (c)--(em);
\shade [ball color=black] (em') circle [radius = 1.5pt];
\draw  [>=latex, <->, shorten >= 3.5pt, shorten <= 3.5pt, color=black] (em)--(em');
\shade [ball color=black] (em) circle [radius = 1.5pt];
\draw  [>=latex, <->, shorten >= 3pt, shorten <= 3pt, color=black, densely dashed] (em)--(e2);
\shade [ball color=black] (e2') circle [radius = 1.5pt];
\draw  [>=latex, <->, shorten >= 3.5pt, shorten <= 3.5pt, color=black] (e2)--(e2');
\shade [ball color=black] (e2) circle [radius = 1.5pt];
\draw  [>=latex, <->, shorten >= 3pt, shorten <= 3pt, color=black, densely dashed] (e2)--(e1);
\shade [ball color=black] (e1') circle [radius = 1.5pt];
\draw  [>=latex, <->, shorten >= 3.5pt, shorten <= 3.5pt, color=black] (e1)--(e1');
\shade [ball color=black] (e1) circle [radius = 1.5pt];
\draw  [>=latex, <->, shorten >= 3pt, shorten <= 3pt, color=black, densely dashed] (e1)--(e0);
%\shade [ball color=black] (e3') circle [radius = 1.5pt];
%\draw  [>=latex, <->, shorten >= 3.5pt, shorten <= 3.5pt, color=black] (e0)--(e3');
\shade [ball color=black] (e0) circle [radius = 1.5pt];
\draw  [>=latex, <->, shorten >= 3pt, shorten <= 3pt, color=black] (e0)--(e0');
\shade [ball color=black] (e0') circle [radius = 1.5pt];
%\draw  [>=latex, <->, shorten >= 3pt, shorten <= 3pt, color=black] (e0')--(e0'');
%\shade [ball color=black] (e0'') circle [radius = 1.5pt];

\begin{scope}[color=black!42]
\draw  [>=latex, <->, shorten >= 3pt, shorten <= 3pt, densely dashed] (cr)--(rem);
\shade [ball color=black!84] (rem') circle [radius = 0.96*1.5pt];
\filldraw [color=white, opacity = 0.5] (rem') circle [radius = 0.96*1.55pt];
\draw  [>=latex, <->, shorten >= 3.5pt, shorten <= 3.5pt] (rem)--(rem');
\shade [ball color=black!84] (rem) circle [radius = 0.96*1.5pt];
\filldraw [color=white, opacity = 0.5] (rem) circle [radius = 0.96*1.55pt];
\draw  [>=latex, <->, shorten >= 3pt, shorten <= 3pt, densely dashed] (rem)--(re2);
\shade [ball color=black!84] (re2') circle [radius = 0.96*1.5pt];
\filldraw [color=white, opacity = 0.5] (re2') circle [radius = 0.96*1.55pt];
\draw  [>=latex, <->, shorten >= 3.5pt, shorten <= 3.5pt] (re2)--(re2');
\shade [ball color=black!84] (re2) circle [radius = 0.96*1.5pt];
\filldraw [color=white, opacity = 0.5] (re2) circle [radius = 0.96*1.55pt];
\draw  [>=latex, <->, shorten >= 3pt, shorten <= 3pt, densely dashed] (re2)--(re1);
\shade [ball color=black!84] (re1') circle [radius = 0.96*1.5pt];
\filldraw [color=white, opacity = 0.5] (re1') circle [radius = 0.96*1.55pt];
\draw  [>=latex, <->, shorten >= 3.5pt, shorten <= 3.5pt] (re1)--(re1');
\shade [ball color=black!84] (re1) circle [radius = 0.96*1.5pt];
\filldraw [color=white, opacity = 0.5] (re1) circle [radius = 0.96*1.55pt];
\draw  [>=latex, <->, shorten >= 3pt, shorten <= 3pt, densely dashed] (re1)--(re0);
\shade [ball color=black!84] (re0) circle [radius = 0.96*1.5pt];
\filldraw [color=white, opacity = 0.5] (re0) circle [radius = 0.96*1.55pt];
\draw  [>=latex, <->, shorten >= 3pt, shorten <= 3pt] (re0)--(re0');
\shade [ball color=black!84] (re0') circle [radius = 0.96*1.5pt];
\filldraw [color=white, opacity = 0.5] (re0') circle [radius = 0.96*1.55pt];
%\draw  [>=latex, <->, shorten >= 3pt, shorten <= 3pt] (re0')--(re0'');
%\shade [ball color=black!84] (re0'') circle [radius = 0.96*1.5pt];
%\filldraw [color=white, opacity = 0.5] (re0'') circle [radius = 0.96*1.55pt];
%\shade [ball color=black!84] (re3') circle [radius = 0.96*1.5pt];
%\filldraw [color=white, opacity = 0.5] (re3') circle [radius = 0.96*1.55pt];
%
\draw  [>=latex, <->, shorten >= 3pt, shorten <= 3pt, densely dashed] (cu)--(uem);
\shade [ball color=black!84] (uem') circle [radius = 0.90*1.5pt];
\filldraw [color=white, opacity = 0.5] (uem') circle [radius = 0.90*1.55pt];
\draw  [>=latex, <->, shorten >= 3.5pt, shorten <= 3.5pt] (uem)--(uem');
\shade [ball color=black!84] (uem) circle [radius = 0.90*1.5pt];
\filldraw [color=white, opacity = 0.5] (uem) circle [radius = 0.90*1.55pt];
\draw  [>=latex, <->, shorten >= 3pt, shorten <= 3pt, densely dashed] (uem)--(ue2);
\shade [ball color=black!84] (ue2') circle [radius = 0.90*1.5pt];
\filldraw [color=white, opacity = 0.5] (ue2') circle [radius = 0.90*1.55pt];
\draw  [>=latex, <->, shorten >= 3.5pt, shorten <= 3.5pt] (ue2)--(ue2');
\shade [ball color=black!84] (ue2) circle [radius = 0.90*1.5pt];
\filldraw [color=white, opacity = 0.5] (ue2) circle [radius = 0.90*1.55pt];
\draw  [>=latex, <->, shorten >= 3pt, shorten <= 3pt, densely dashed] (ue2)--(ue1);
\shade [ball color=black!84] (ue1') circle [radius = 0.90*1.5pt];
\filldraw [color=white, opacity = 0.5] (ue1') circle [radius = 0.90*1.55pt];
\draw  [>=latex, <->, shorten >= 3.5pt, shorten <= 3.5pt] (ue1)--(ue1');
\shade [ball color=black!84] (ue1) circle [radius = 0.90*1.5pt];
\filldraw [color=white, opacity = 0.5] (ue1) circle [radius = 0.90*1.55pt];
\draw  [>=latex, <->, shorten >= 3pt, shorten <= 3pt, densely dashed] (ue1)--(ue0);
\shade [ball color=black!84] (ue0) circle [radius = 0.90*1.5pt];
\filldraw [color=white, opacity = 0.5] (ue0) circle [radius = 0.90*1.55pt];
\draw  [>=latex, <->, shorten >= 3pt, shorten <= 3pt] (ue0)--(ue0');
\shade [ball color=black!84] (ue0') circle [radius = 0.90*1.5pt];
\filldraw [color=white, opacity = 0.5] (ue0') circle [radius = 0.90*1.55pt];
%\draw  [>=latex, <->, shorten >= 3pt, shorten <= 3pt] (ue0')--(ue0'');
%\shade [ball color=black!84] (ue0'') circle [radius = 0.90*1.5pt];
%\filldraw [color=white, opacity = 0.5] (ue0'') circle [radius = 0.90*1.55pt];
%\shade [ball color=black!84] (ue3') circle [radius = 0.90*1.5pt];
%\filldraw [color=white, opacity = 0.5] (ue3') circle [radius = 0.90*1.55pt];

\end{scope}

\begin{scope}[>=latex, <->, shorten >= 3pt, shorten <= -7.5pt, color=black!84]
\draw [] ($(c12)+(c)-(d)$)--(c);
\draw [] ($(c23)+(c)-(d)$)--(c);
\end{scope}

\node [right, color=black] at ($(c)  +( 0.05,-0.03)$) {$a$} ;
\node [right, color=black] at ($(em) +(-0.01,-0.08)$) {$e_m^a$} ;
\node [right, color=black] at ($(e2) +(-0.01,-0.08)$) {$e_{k_n+i}^a$} ;
\node [right, color=black] at ($(e1) +(-0.01,-0.08)$) {$e_{k_n+j+4}^a$} ;
\node [right, color=black] at ($(e0) +(-0.01,+0.06)$) {$e_{k_n+9}^a$} ;
\node [right, color=black] at ($(e0')+(-0.01,+0.06)$) {$e_{k_n+10}^a$} ;
%\node [right, color=black] at ($(e0'')+(-0.01,+0.06)$) {$e_{k_n+11}^a$} ;
\node [right, color=black] at ($(em') +(0.02,-0.01)$) {$e_P^a$} ;
\node [right, color=black] at ($(e2') +(0.02,-0.01)$) {$e_R^a$} ;
\node [right, color=black] at ($(e1') +(0.02,-0.01)$) {$e_U^a$} ;
%\node [right, color=black] at ($(e3') +(0.02,-0.01)$) {$e_C^a$} ;

%\node [right, color=black] at ($(cu)  +( 0.05,-0.03)$) {$b$} ;
%\node [right, color=black] at ($(cr)  +( 0.05,-0.03)$) {$c$} ;
%\node [right, color=black] at ($(ue0'')+(-0.01,+0.06)$) {$e_{k_n+11}^b$} ;
%\node [right, color=black] at ($(re0'')+(-0.01,+0.06)$) {$e_{k_n+11}^c$} ;
%\node [right, color=black] at ($(ue3') +(-0.01,+0.06)$) {$e_C^b$} ;
%\node [right, color=black] at ($(re3') +(-0.01,+0.06)$) {$e_C^c$} ;

%\node [right, color=black] at ($(ekn)+(-0.01,+0.06)$) {$e_0^a$} ;
%\node [right, color=black] at ($(ek2)+(-0.01,-0.08)$) {$e_2^a$} ;
%\node [right, color=black] at ($(ek1)+(-0.01,-0.08)$) {$e_1^a$} ;

\end{tikzpicture}
\caption{Simulation of $P_m(a)$, $R_i(a)$, and $U_j(a)$}
\label{fig:16}
\end{figure}
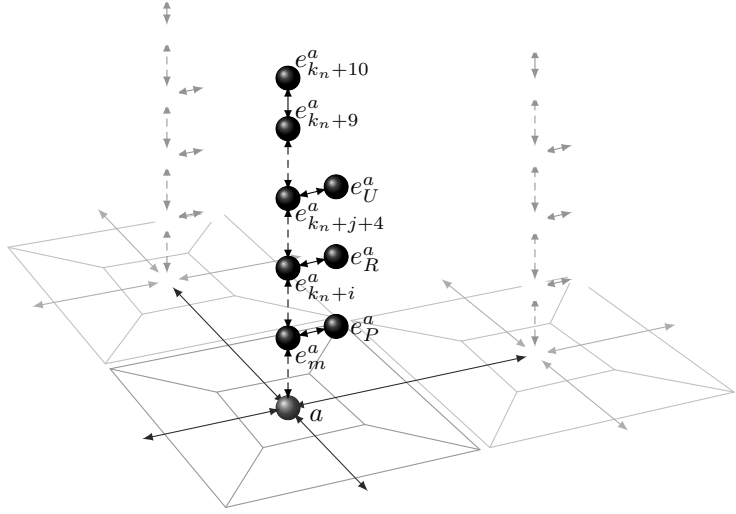

\begin{lemma}
\label{lem:relativization:G:S:S3}
Let $L$ be a modal predicate logic such that\/ $\ckf L$ is a Skvortsov class. Then $n\in\mathbb{Y}$ implies $S_3\mathit{M}^\Box_G\mathit{Tiling}_n^{\mathbb{X}} \not\in \QMLcw \ckf L$. 
\end{lemma}

\begin{proof}
Similar to the proof of Lemma~\ref{lem:relativization:G:wKHC:S3} with the difference that the corresponding Kripke frame should be finite; such a frame exists since $\ckf L$ is a Skvortsov class.
%Skvortsov class is also a wKHC class.
\end{proof}

As corollaries of the lemmas, we obtain the following theorems.

\begin{theorem}
\label{th:ml:insep:PCG:dfin}
Let $L$ be a modal predicate logic such that\/ $\ckf L$ is a wKHC class. Then $\logic{QK}$ and $L_{\mathit{dfin}}\oplus\bm{bf}$ are recursively inseparable in the language with a single binary predicate letter, two unary predicate letters, and two individual variables. 
\end{theorem}

\begin{proof}
Follows from Lemmas~\ref{lem:relativization:G:K:S3} and~\ref{lem:relativization:G:wKHC:S3}.
\end{proof}

\begin{theorem}
\label{th:ml:insep:PCG:wfin}
Let $L$ be a modal predicate logic such that\/ $\ckf L$ is a Skvortsov class. Then $\logic{QK}$ and $L_{\mathit{wfin}}\oplus\bm{bf}$ are recursively inseparable in the language with a single binary predicate letter, two unary predicate letters, and two individual variables. 
\end{theorem}

\begin{proof}
Follows from Lemmas~\ref{lem:relativization:G:K:S3} and~\ref{lem:relativization:G:S:S3}.
\end{proof}

To eliminate the binary letter $P$, we need some additional condition, similar to the condition proposed in~\cite[Theorem~3]{KKZ05}. We say that a Kripke frame $\kframe{F}=\otuple{W,R}$ satisfies the \defnotion{Kontchakov--Kurucz--Zakharyaschev condition} (for short, \defnotion{KKZ}) if there are $w\in W$ and two infinite disjoint subsets, $V$ and $U$, of $W$ such that $V\subseteq R(w)$ and $U\subseteq R(v)$, for every $v \in V$.
We say that a modal predicate logic $L$ is \defnotion{KKZ-friendly} if there exists a Kripke frame $\kframe{F}\in\ckf L$ satisfying KKZ. It is known that the monadic fragments~--- with an infinite supply of unary predicate letters~--- of KKZ-friendly modal predicate logics are undecidable in the language with two individual variables~\cite{KKZ05}. For our purposes, we need a weaker condition. We say that a class $\scls{C}$ of Kripke frames satisfies the \defnotion{weak Kontchakov--Kurucz--Zakharyaschev condition} (for short, \defnotion{wKKZ}) if, for every $n\in\numN$, there exist a Kripke frame $\otuple{W,R}\in\scls{C}$, a world $w \in W$, and a subset~$V$ of~$R(w)$ such that $|V|\geqslant n$ and, for every $v\in V$, $|R(v)\setminus V|\geqslant n$.
Clearly, if a class $\scls{C}$ of Kripke frames contains a frame that satisfies KKZ, then it also satisfies KHC, wKHC, and wKKZ. But if $\scls{C}$ contains only finite Kripke frames, then it can not satisfy KKZ; at the same time, the class of all finite Kripke frames satisfies wKKZ. If $\scls{C}$ satisfies wKKZ, then it also satisfies wKHC. Nevertheless, the class of all frames of depth\footnote{The largest number of worlds in chains in a (transitive) Kripke frame.} two satisfies KHC and wKHC but it does not satisfy both KKZ and wKKZ. 

Let us expand $S_2$ on $\lang{ML}$-formulas, i.e., let $S_2$ be a formula substitution that substitutes $\Box(Q(x_1)\vee Q(x_2))$ for $P(x_1,x_2)$ in $\lang{ML}$-formulas.

\begin{lemma}
\label{lem:relativization:G:K:S3:S2}
If $n\in\mathbb{X}$, then $S_2S_3\mathit{M}^\Box_G\mathit{Tiling}_n^{\mathbb{X}} \in \logic{QK}$. 
\end{lemma}

\begin{proof}
Immediately follows from Lemma~\ref{lem:relativization:G:K:S3}, since $\logic{QK}$ is closed under Substitution.
\end{proof}

\begin{lemma}
\label{lem:relativization:G:wKHC:S3:S2}
Let $L$ be a modal predicate logic such that\/ $\ckf L$ is a wKKZ class. Then, $n\in\mathbb{Y}$ implies 
$S_2S_3\mathit{M}^\Box_G\mathit{Tiling}_n^{\mathbb{X}} \not\in \QMLcd \ckf L$. 
\end{lemma}

\begin{proof}
Let $\kModel{M}=\otuple{\kframe{F}\odot\mathcal{D}',I}$ be the model constructed in the proof of Lemma~\ref{lem:relativization:G:wKHC:S3}. Since $\ckf L$ is a wKKZ class, it contains a Kripke frame $\bar{\kframe{F}}=\otuple{\bar{W},\bar{R}}$ such that there exist $\bar{w}_0 \in W$ and a subset~$V$ of~$\bar{R}(\bar{w}_0)$ such that $|V|\geqslant 2|\mathcal{D}'|^3$ and, for every $v\in V$, $|\bar{R}(v)\setminus V|\geqslant 2|\mathcal{D}'|^3$. 

Recall that $\mathcal{D}'$ contains $\mathcal{D}$ and $\mathcal{D}=\{0,\ldots,r+4\}\times\{0,\ldots,r+4\}$.
Let $V'=\{\bar{w}_{a}\in V : a\in\mathcal{D}\}$, where $\bar{w}_{a}=\bar{w}_{b}$ if, and only if, $a=b$. For every $v\in\{\bar{w}_0\}\cup V$, choose a set
$$
\begin{array}{lcll}
U(v) & = & \{u^{v}_{bc}\in \bar{R}(v) : b,c\in\mathcal{D}'\},
\end{array}
$$
so that, for all $v,v'\in\{\bar{w}_0\}\cup V$ and all $b,c,d,e\in\mathcal{D}'$,
$$
\begin{array}{lcl}
  u^v_{bc} = u^{v'}_{de} & \imply & \{b,c\} = \{d,e\};
\end{array}
$$  
notice that it is possible, since the sets $V$ and $\bar{R}(v)\setminus V$, for every $v\in V$, are quite big.
Next, define a model $\bar{\kModel{M}} = \otuple{\bar{\kframe{F}}\odot\mathcal{D}',\bar{I}}$ so that, for every $w\in\bar{W}$ and every $b\in \mathcal{D}'$,
$$
\begin{array}{lcl}
\bar{\kModel{M}},w\models G(b) 
  & \iff 
  & b\in\mathcal{D}; 
  \smallskip \\
\bar{\kModel{M}},w\models C(b) 
  & \iff 
  & \mbox{for some $a\in\mathcal{D}$, both $w=\bar{w}_a$ and $b=a$}; 
  \smallskip \\
\bar{\kModel{M}},w\not\models Q(b) 
  & \iff 
  & \mbox{for some $v\in\{\bar{w}_0\}\cup V$ and some $c\in\mathcal{D}'$,}
  \\
  &
  & \mbox{both $w\in\{u^v_{bc},u^v_{cb}\}$ and $\kModel{M},w_0\not\models P(b,c)$,} 
\end{array}
$$
where $\kModel{M}$ and $w_0$ are defined in the proof of Lemma~\ref{lem:relativization:G:wKHC:S3}.

Let us compare $\bar{\kModel{M}}$ and $\kModel{M}$. Observe that, for every $a\in\mathcal{D}\cup\{0\}$,
$$
\begin{array}{lcl}
\{{w}\in    {W} :     {\kModel{M}},{w}\models C(a)\} & = & \{    {w}_a\}; \\
\{{w}\in\bar{W} : \bar{\kModel{M}},{w}\models C(a)\} & = & \{\bar{w}_a\};
\end{array}
$$
also, for every $w\in \{w_0\}\cup R(w_0)$, every $\bar{w}\in \{\bar{w}_0\}\cup R(\bar{w}_0)$, and all $b,c\in\mathcal{D}'$,
$$
\begin{array}{lcl}
\kModel{M},{w}\models G(b) 
  & \iff 
  & \bar{\kModel{M}},\bar{w}\models G(b); 
  \smallskip \\
\kModel{M},{w}\models P(b,c) 
  & \iff 
  & \bar{\kModel{M}},\bar{w}\models \Box(Q(b)\vee Q(c)). 
\end{array}
$$
Since $\kModel{M},{w}_0\not\models S_3\mathit{M}^\Box_G\mathit{Tiling}_n^{\mathbb{X}}$, we readily obtain that $\bar{\kModel{M}},\bar{w}_0\not\models S_2S_3\mathit{M}^\Box_G\mathit{Tiling}_n^{\mathbb{X}}$. 

Thus, $S_2S_3\mathit{M}^\Box_G\mathit{Tiling}_n^{\mathbb{X}} \not\in \QMLcd \ckf L$. 
\end{proof}

\begin{lemma}
\label{lem:relativization:G:S:S3:S2}
Let $L$ be a modal predicate logic such that\/ $\fin \ckf L$ is a wKKZ class. Then $n\in\mathbb{Y}$ implies $S_2S_3\mathit{M}^\Box_G\mathit{Tiling}_n^{\mathbb{X}} \not\in \QMLcw \ckf L$. 
\end{lemma}

\begin{proof}
Similar to the proof of Lemma~\ref{lem:relativization:G:wKHC:S3:S2} with the difference that the corresponding Kripke frame should be finite; such a frame exists since $\fin \ckf L$ is a wKKZ class.
\end{proof}

As corollaries of the lemmas, we obtain the following theorems.

\begin{theorem}
\label{th:ml:insep:QCG:dfin}
Let $L$ be a modal predicate logic such that\/ $\ckf L$ is a wKKZ class. Then\/ $\logic{QK}$ and $L_{\mathit{dfin}}\oplus\bm{bf}$ are recursively inseparable in the language with three unary predicate letters and two individual variables. 
\end{theorem}

\begin{proof}
Follows from Lemmas~\ref{lem:relativization:G:K:S3:S2} and~\ref{lem:relativization:G:wKHC:S3:S2}.
\end{proof}

\begin{theorem}
\label{th:ml:insep:QCG:wfin}
Let $L$ be a modal predicate logic such that\/ $\fin\ckf L$ is a wKKZ class. Then $\logic{QK}$ and $L_{\mathit{wfin}}\oplus\bm{bf}$ are recursively inseparable in the language with three unary predicate letters and two individual variables. 
\end{theorem}

\begin{proof}
Follows from Lemmas~\ref{lem:relativization:G:K:S3:S2} and~\ref{lem:relativization:G:S:S3:S2}.
\end{proof}

Theorems~\ref{th:ml:insep:QCG:dfin} and~\ref{th:ml:insep:QCG:wfin} provide us with the following corollary on particular logics.

\begin{corollary}
\label{cor:th:ml:insep:QCG:dfin+wfin}
Let $L$ be one of\/ $\logic{QK}$, $\logic{QT}$, $\logic{QKB}$, $\logic{QKTB}$, $\logic{QK4}$, $\logic{QS4}$, $\logic{QK5}$, $\logic{QS5}$, $\logic{QK45}$, $\logic{QK4B}$, $\logic{QGL}$, $\logic{QGrz}$, $\logic{QwGrz}$, $\logic{QK4}\oplus\bm{bd}_m$, $\logic{QS4}\oplus\bm{bd}_m$, $\logic{QGL}\oplus\bm{bd}_m$, $\logic{QGrz}\oplus\bm{bd}_m$, $\logic{QwGrz}\oplus\bm{bd}_m$, where $m\geqslant 3$. Then $L$ and $L_{\mathit{dfin}}$ as well as $L$ and $L_{\mathit{wfin}}$ are recursively inseparable in the language with three unary predicate letters and two individual variables. 
\end{corollary}

Now, we are going to eliminate all predicate letters except for a single unary one. To this end, we propose some other relativizations and substitutions. Also, we need classes of Kripke frames to satisfy a slightly stronger condition than wKKZ. 
%
%We will not formulate this condition explicitly, limiting ourselves to describing the logics whose classes of Kripke frames satisfy~it; nevertheless, the description will be given in such a way that the reader will have a possibility to vary the resulting construction, including both formulas and additional requirements for classes of frames. Also, in some of technical statements below, we will deal with subframe modal predicate logics; observe that the class of Kripke frames of such a logic satisfies wKHC condition if, and only if, it is a Skvortsov class.
%
We will not explicitly formulate this condition. Instead, we will describe the logics whose classes of Kripke frames satisfy it. However, our description will allow readers to adjust the resulting construction, including both formulas and additional requirements for classes of frames. Additionally, some of the technical statements below will involve subframe modal predicate logics. Note that the class of Kripke frames of such logic satisfies wKHC if, and only if, it is a Skvortsov class.

First, let us define a modification $S'_3$ of $S_3$. Let $S'_3\mathit{M}^\Box_G\mathit{Tiling}_n^{\mathbb{X}}$ be a formula obtained from $\mathit{M}^\Box_G\mathit{Tiling}_n^{\mathbb{X}}$ by applying~$S_3$ to it and then replacing every occurrence of $C(x)$ and $C(y)$ with $\mathit{tile}'_{k_n+9}(x)$ and $\mathit{tile}'_{k_n+9}(y)$, respectively.

Let $p$ be a proposition letter. Define $\Box_p$ by 
$$
\begin{array}{lcl}
\Box_p\varphi & = & \Box(p\to\varphi).
\end{array}
$$ 
For an $\lang{ML}$\nobreakdash-formula $\varphi$, let $\varphi_p$ be the formula obtained from $\varphi$ by replacing every occurrence of $\Box$ with~$\Box_p$.

\begin{lemma}
\label{lem:relativization:p:K}
If $n\in\mathbb{X}$, then $p\to (S'_3\mathit{M}^\Box_G\mathit{Tiling}_n^{\mathbb{X}})_p \in \logic{QK}$. 
\end{lemma}

\begin{proof}
We give just a sketch of the proof.

Let $n\in\mathbb{X}$. Then, by Lemma~\ref{lem:relativization:G:K}, $\mathit{M}^\Box_G\mathit{Tiling}_n^{\mathbb{X}} \in \logic{QK}$. Then, $S'_3\mathit{M}^\Box_G\mathit{Tiling}_n^{\mathbb{X}} \in \logic{QK}$, since $\logic{QK}$ is closed under Substitution. 

Suppose that $p\to (S'_3\mathit{M}^\Box_G\mathit{Tiling}_n^{\mathbb{X}})_p \not\in \logic{QK}$. Then there exists a model that refutes the formula. 
If we remove all the worlds at which $p$ is not true from the model, then we obtain a model refuting $S'_3\mathit{M}^\Box_G\mathit{Tiling}_n^{\mathbb{X}}$, which gives us a contradiction. 
%
%%We obtain a model refuting $S'_3\mathit{M}^\Box_G\mathit{Tiling}_n^{\mathbb{X}}$ if we remove from the model all the worlds where $p$ is false, that gives us a contradiction.
%
Thus, $p\to (S'_3\mathit{M}^\Box_G\mathit{Tiling}_n^{\mathbb{X}})_p \in \logic{QK}$. 
\end{proof}

\begin{lemma}
\label{lem:relativization:p:wKHC}
Let $L$ be a modal predicate logic such that\/ $\ckf L$ is a wKHC class. Then $n\in\mathbb{Y}$ implies $p\to (S'_3\mathit{M}^\Box_G\mathit{Tiling}_n^{\mathbb{X}})_p \not\in \QMLcd \ckf L$. 
\end{lemma}

\begin{proof}
We are going to apply an argumentation similar to that used in the proof of Lemma~\ref{lem:relativization:G:wKHC:S3}; we do not need to simulate $G$ again, but we do need to simulate~$C$ as well.

Let $n\in\mathbb{Y}$.
Then, by Lemma~\ref{lem:Trakhtenbrot:lem2:sib}, $\mathit{MTiling}_n^{\mathbb{X}}\not\in\logic{QCl}_{\mathit{fin}}\uplus\bm{sib}$. Hence,
there exists a finite classical model $\cModel{M}=\otuple{\mathcal{D},\mathcal{I}}$ such that $\cModel{M}\models\bm{sib}$ and $\cModel{M}\not\models\mathit{MTiling}_n^{\mathbb{X}}$; we may assume that $\cModel{M}$ is the model defined in the proof of Lemma~\ref{lem:Trakhtenbrot:lem2:sib}, in particular,
$$
\begin{array}{lcl}
\mathcal{D} & = & \{0,\ldots,r+4\}\times\{0,\ldots,r+4\}, \\
\end{array}
$$
for a suitable~$r\in\numN$.

For every $a\in \mathcal{D}$, take the elements
$e_0^a,\ldots,e_{k_n+11}^a$ and $e_P^a, e_R^a, e_U^a, e_C^a$; let
$$
\begin{array}{lcl}
\mathcal{D}' 
  & = 
  & \mathcal{D} \cup \{e_0^a,\ldots,e_{k_n+11}^a, e_P^a, e_R^a, e_U^a, e_C^a : a\in\mathcal{D}\}.
\end{array}
$$

Let $\kframe{F}=\otuple{W,R}$ be a Kripke frame of $\ckf L$ and $w_0$ a world of $W$ such that $|R(w_0)\setminus\{w_0\}|\geqslant |\mathcal{D}|$; such frame exists since $\ckf L$ is a wKHC class. Then $R(w_0)\setminus\{w_0\}$ contains a subset $W'=\{w_{a} : a\in\mathcal{D}\}$, where $w_{a}=w_{b}$ if, and only if, $a=b$.
Let $\kModel{M}=\otuple{\kframe{F}\odot\mathcal{D}',I}$ be a Kripke model such that
\begin{itemize}
\item
for every $w\in W$ and every $a\in\mathcal{D}'$,
$$
\begin{array}{lcl}
\kmodel{M},w\models p    & \iff & w\in\{w_0\}\cup W'; \\
\kmodel{M},w\models G(a) & \iff & \mbox{$a\in \mathcal{D}$ and $\kmodel{M},w\models p$;} \\
\end{array}
$$
\item 
$I(w_0,P)$ is the symmetric closure of the relation
$$
\begin{array}{l}
\mathcal{I}(P) \cup 
  \{\otuple{a,e_0^a},\otuple{e_0^a,e_1^a},\ldots,\otuple{e_{k_n+10}^a,e_{k_n+11}^a} : a\in\mathcal{D}\}
  \\
  \phantom{\mathcal{I}(P)} \cup 
  \{\otuple{e_m^a,e_P^a} : \mbox{$m\in\{0,\ldots,k_n\}$ and $\cModel{M}\models P_m(a)$} \}
  \\
  \phantom{\mathcal{I}(P)} \cup
  \{\otuple{e_{k_n+i}^a,e_R^a} : \mbox{$i\in\{1,2,3,4\}$ and $\cModel{M}\models R_i(a)$} \}
  \\
  \phantom{\mathcal{I}(P)} \cup 
  \{\otuple{e_{k_n+j+4}^a,e_U^a} : \mbox{$j\in\{1,2,3,4\}$ and $\cModel{M}\models U_j(a)$} \};
\end{array}
$$
\item 
for every $a\in\mathcal{D}$,
$$
\begin{array}{lcl}
I(w_a,P) & = & I(w_0,P) \cup \{\otuple{e_{k_n+9}^a,e_C^a},\otuple{e_C^a,e_{k_n+9}^a}\};
\end{array}
$$
\item 
for every $w\in W\setminus(W'\cup\{w_0\})$,
$$
\begin{array}{lcl}
I(w,P) & = & \varnothing; %I(w_0,P);
\end{array}
$$
\end{itemize}
see Figure~\ref{fig:17} for a part of $I(w_a,P)$.
Then it is not hard to show by routine verification that $\kModel{M},w_0\not\models p\to (S'_3\mathit{M}^\Box_G\mathit{Tiling}_n^{\mathbb{X}})_p$; we leave the details to the reader.

Thus, $p\to (S'_3\mathit{M}^\Box_G\mathit{Tiling}_n^{\mathbb{X}})_p \not\in \QMLcd \ckf L$. 
\end{proof}

\begin{figure}
\centering
\begin{tikzpicture}[scale=3.2, rectnode/.style={rectangle, thick, draw=black!60, dashed, rounded corners = 2pt}]

\coordinate (c1) at (-1.4,0.30); %(-1.4, 0.30);
\coordinate (c2) at (-0.9,-0.3); %(-1.0,-0.3);
\coordinate (c3) at (0.2,-0.05); %( 0.0,-0.05);
\coordinate (c4) at (-0.4,0.5); %(-0.5, 0.50);

\coordinate (c12) at ($0.5*(c1)+0.5*(c2)$);
\coordinate (c23) at ($0.5*(c2)+0.5*(c3)$);
\coordinate (c34) at ($0.5*(c3)+0.5*(c4)$);
\coordinate (c41) at ($0.5*(c4)+0.5*(c1)$);

\coordinate (vecX1) at ($(c3)-(c2)$);
\coordinate (vecX2) at ($(c4)-(c1)$);
\coordinate (vecY1) at ($(c1)-(c2)$);
\coordinate (vecY2) at ($(c4)-(c3)$);

\coordinate (c1r) at (c4);
\coordinate (c2r) at (c3);
\coordinate (c3r) at ($(c3)+0.96*(vecX1)$);
\coordinate (c4r) at ($(c4)+0.96*(vecX2)$);

\coordinate (c1u) at ($(c1)+0.84*(vecY1)$);
\coordinate (c2u) at (c1);
\coordinate (c3u) at (c4);
\coordinate (c4u) at ($(c4)+0.84*(vecY2)$);

\coordinate (c12r) at ($0.5*(c1r)+0.5*(c2r)$);
\coordinate (c23r) at ($0.5*(c2r)+0.5*(c3r)$);
\coordinate (c34r) at ($0.5*(c3r)+0.5*(c4r)$);
\coordinate (c41r) at ($0.5*(c4r)+0.5*(c1r)$);

\coordinate (c12u) at ($0.5*(c1u)+0.5*(c2u)$);
\coordinate (c23u) at ($0.5*(c2u)+0.5*(c3u)$);
\coordinate (c34u) at ($0.5*(c3u)+0.5*(c4u)$);
\coordinate (c41u) at ($0.5*(c4u)+0.5*(c1u)$);

\draw [white, opacity = 0, name path = dg 1] (c1)--(c3);
\draw [white, opacity = 0, name path = dg 2] (c2)--(c4);
\draw [name intersections = {of = dg 1 and dg 2, by = {c}}];
\draw [white, opacity = 0, name path = dg 3] (c12)--(c34);
\draw [white, opacity = 0, name path = dg 4] (c23)--(c41);
\draw [name intersections = {of = dg 3 and dg 4, by = {d}}];

\draw [white, opacity = 0, name path = dgr 1] (c1r)--(c3r);
\draw [white, opacity = 0, name path = dgr 2] (c2r)--(c4r);
\draw [name intersections = {of = dgr 1 and dgr 2, by = {cr}}];
\draw [white, opacity = 0, name path = dgr 3] (c12r)--(c34r);
\draw [white, opacity = 0, name path = dgr 4] (c23r)--(c41r);
\draw [name intersections = {of = dgr 3 and dgr 4, by = {dr}}];

\draw [white, opacity = 0, name path = dgu 1] (c1u)--(c3u);
\draw [white, opacity = 0, name path = dgu 2] (c2u)--(c4u);
\draw [name intersections = {of = dgu 1 and dgu 2, by = {cu}}];
\draw [white, opacity = 0, name path = dgu 3] (c12u)--(c34u);
\draw [white, opacity = 0, name path = dgu 4] (c23u)--(c41u);
\draw [name intersections = {of = dgu 3 and dgu 4, by = {du}}];

\coordinate (diffR) at ($(cr)-(c)$);
\coordinate (diffU) at ($(cu)-(c)$);
\coordinate (corR) at (0,0.012);
\coordinate (corU) at (0,0.032);

\coordinate (em)  at ($(c)+1*(0,0.29)$);
\coordinate (e2)  at ($(c)+2*(0,0.29)$);
\coordinate (e1)  at ($(c)+3*(0,0.29)$);
\coordinate (e0)  at ($(c)+4*(0,0.29)$);
\coordinate (e0') at ($(e0)+ (0,0.21)$);
\coordinate (e0'') at ($(e0')+ (0,0.21)$);

\coordinate (em') at ($(em)+(0.2,0.048)$);
\coordinate (e2') at ($(e2)+(0.2,0.048)$);
\coordinate (e1') at ($(e1)+(0.2,0.048)$);
\coordinate (e3') at ($(e0)+(0.2,0.048)$);

\coordinate (ekn) at ($(c)-1*(0,0.29)$);
\coordinate (ek2) at ($(c)-2*(-0.28*0.25,0.285*0.85)$);
\coordinate (ek1) at ($(c)-2*(+0.26*0.25,0.325*0.85)$);
%\coordinate (ek0) at ($(c)-4*(0,0.31)$);

\coordinate (rem)  at ($(cr)+1*(0,0.29)-1*(corR)$);
\coordinate (re2)  at ($(cr)+2*(0,0.29)-2*(corR)$);
\coordinate (re1)  at ($(cr)+3*(0,0.29)-3*(corR)$);
\coordinate (re0)  at ($(cr)+4*(0,0.29)-4*(corR)$);
\coordinate (re0') at ($(re0)+ (0,0.21)-1*(corR)$);
\coordinate (re0'') at ($(re0')+ (0,0.21)-1*(corR)$);

\coordinate (rem') at ($(rem)+0.96*(0.2,0.048)$);
\coordinate (re2') at ($(re2)+0.96*(0.2,0.048)$);
\coordinate (re1') at ($(re1)+0.96*(0.2,0.048)$);
\coordinate (re3') at ($(re0)+0.96*(0.2,0.048)$);

\coordinate (rekn) at ($(cr)-1*(0,0.29)+1*(corR)$);
\coordinate (rek2) at ($(cr)-2*(-0.28*0.25,0.285*0.85)+1.7*(corR)$);
\coordinate (rek1) at ($(cr)-2*(+0.26*0.25,0.325*0.85)+1.7*(corR)$);
%\coordinate (rek0) at ($(c)-4*(0,0.31)$);

\coordinate (uem)  at ($(cu)+1*(0,0.29)-1*(corU)$);
\coordinate (ue2)  at ($(cu)+2*(0,0.29)-2*(corU)$);
\coordinate (ue1)  at ($(cu)+3*(0,0.29)-3*(corU)$);
\coordinate (ue0)  at ($(cu)+4*(0,0.29)-4*(corU)$);
\coordinate (ue0') at ($(ue0)+ (0,0.21)-1*(corU)$);
\coordinate (ue0'') at ($(ue0')+ (0,0.21)-1*(corU)$);

\coordinate (uem') at ($(uem)+0.96*(0.2,0.040)$);
\coordinate (ue2') at ($(ue2)+0.96*(0.2,0.040)$);
\coordinate (ue1') at ($(ue1)+0.96*(0.2,0.040)$);
\coordinate (ue3') at ($(ue0)+0.96*(0.2,0.040)$);

\coordinate (uekn) at ($(cu)-1*(0,0.29)+1*(corU)$);
\coordinate (uek2) at ($(cu)-2*(-0.28*0.25,0.285*0.85)+1.7*(corU)$);
\coordinate (uek1) at ($(cu)-2*(+0.26*0.25,0.320*0.85)+1.7*(corU)$);
%\coordinate (uek0) at ($(c)-4*(0,0.31)$);

\Rem{
\shade [ball color=black] (ek2) circle [radius = 1.5pt];
\draw  [>=latex, <->, shorten >= 3.5pt, shorten <= 3.5pt, color=black] (ekn)--(ek2);
\shade [ball color=black] (ekn) circle [radius = 1.5pt];
\draw  [>=latex, <->, shorten >= 3pt, shorten <= 3pt, color=black] (c)--(ekn);
\draw  [>=latex, <->, shorten >= 3.5pt, shorten <= 3.5pt, color=black] (ek2)--(ek1);
\draw  [>=latex, <->, shorten >= 3.5pt, shorten <= 3.5pt, color=black] (ekn)--(ek1);
\shade [ball color=black] (ek1) circle [radius = 1.5pt];

\begin{scope}[color=black!42]
\shade [ball color=black!84] (rek2) circle [radius = 0.96*1.5pt];
\filldraw [color=white, opacity = 0.5] (rek2) circle [radius = 0.96*1.55pt];
\draw  [>=latex, <->, shorten >= 3.5pt, shorten <= 3.5pt] (rekn)--(rek2);
\shade [ball color=black!84] (rekn) circle [radius = 0.96*1.5pt];
\filldraw [color=white, opacity = 0.5] (rekn) circle [radius = 0.96*1.55pt];
\draw  [>=latex, <->, shorten >= 3pt, shorten <= 3pt] (cr)--(rekn);
\draw  [>=latex, <->, shorten >= 3.5pt, shorten <= 3.5pt] (rek2)--(rek1);
\draw  [>=latex, <->, shorten >= 3.5pt, shorten <= 3.5pt] (rekn)--(rek1);
\shade [ball color=black!84] (rek1) circle [radius = 0.96*1.5pt];
\filldraw [color=white, opacity = 0.5] (rek1) circle [radius = 0.96*1.55pt];
\shade [ball color=black!84] (uek2) circle [radius = 0.90*1.5pt];
\filldraw [color=white, opacity = 0.5] (uek2) circle [radius = 0.90*1.55pt];
\draw  [>=latex, <->, shorten >= 3.5pt, shorten <= 3.5pt] (uekn)--(uek2);
\shade [ball color=black!84] (uekn) circle [radius = 0.90*1.5pt];
\filldraw [color=white, opacity = 0.5] (uekn) circle [radius = 0.90*1.55pt];
\draw  [>=latex, <->, shorten >= 3pt, shorten <= 3pt] (cu)--(uekn);
\draw  [>=latex, <->, shorten >= 3.5pt, shorten <= 3.5pt] (uek2)--(uek1);
\draw  [>=latex, <->, shorten >= 3.5pt, shorten <= 3.5pt] (uekn)--(uek1);
\shade [ball color=black!84] (uek1) circle [radius = 0.90*1.5pt];
\filldraw [color=white, opacity = 0.5] (uek1) circle [radius = 0.90*1.55pt];
\end{scope}
}

\begin{scope}[color=black!42]
\drawtileflattmslanted{(c1)}{(c2)}{(c3)}{(c4)}
\end{scope}

\begin{scope}[color=black!25]
\drawtileflattmslanted{(c1r)}{(c2r)}{(c3r)}{(c4r)}
\drawtileflattmslanted{(c1u)}{(c2u)}{(c3u)}{(c4u)}
\end{scope}

\begin{scope}[>=latex, <->, shorten >= -7.5pt, shorten <= 4pt, color=black!32]
\draw [] (cr)--($(c34r)+(cr)-(dr)$);
\draw [] (cr)--($(c41r)+(cr)-(dr)$);
\draw [] (cu)--($(c34u)+(cu)-(du)$);
\draw [] (cu)--($(c41u)+(cu)-(du)$);
\end{scope}

\shade [ball color=black!64] (cr) circle [radius = 0.96*1.5pt];
\filldraw [color=white, opacity = 0.5] (cr) circle [radius = 0.96*1.55pt];
\shade [ball color=black!64] (cu) circle [radius = 0.90*1.5pt];
\filldraw [color=white, opacity = 0.5] (cu) circle [radius = 0.90*1.55pt];

\begin{scope}[>=latex, <->, shorten >= 3pt, shorten <= -7.5pt, color=black!32]
\draw [] ($(c12u)+(cu)-(du)$)--(cu);
\draw [] ($(c23r)+(cr)-(dr)$)--(cr);
\end{scope}

\begin{scope}[>=latex, <->, shorten >= 3pt, shorten <= 2pt, color=black!84]
\draw [] (c)--(cr); %($(c34)+(c)-(d)$);
\draw [] (c)--(cu); %($(c41)+(c)-(d)$);
\end{scope}

\shade [ball color=black!64] (c) circle [radius = 1.5pt];

\draw  [>=latex, <->, shorten >= 3pt, shorten <= 3pt, color=black, densely dashed] (c)--(em);
\shade [ball color=black] (em') circle [radius = 1.5pt];
\draw  [>=latex, <->, shorten >= 3.5pt, shorten <= 3.5pt, color=black] (em)--(em');
\shade [ball color=black] (em) circle [radius = 1.5pt];
\draw  [>=latex, <->, shorten >= 3pt, shorten <= 3pt, color=black, densely dashed] (em)--(e2);
\shade [ball color=black] (e2') circle [radius = 1.5pt];
\draw  [>=latex, <->, shorten >= 3.5pt, shorten <= 3.5pt, color=black] (e2)--(e2');
\shade [ball color=black] (e2) circle [radius = 1.5pt];
\draw  [>=latex, <->, shorten >= 3pt, shorten <= 3pt, color=black, densely dashed] (e2)--(e1);
\shade [ball color=black] (e1') circle [radius = 1.5pt];
\draw  [>=latex, <->, shorten >= 3.5pt, shorten <= 3.5pt, color=black] (e1)--(e1');
\shade [ball color=black] (e1) circle [radius = 1.5pt];
\draw  [>=latex, <->, shorten >= 3pt, shorten <= 3pt, color=black, densely dashed] (e1)--(e0);
\shade [ball color=black] (e3') circle [radius = 1.5pt];
\draw  [>=latex, <->, shorten >= 3.5pt, shorten <= 3.5pt, color=black] (e0)--(e3');
\shade [ball color=black] (e0) circle [radius = 1.5pt];
\draw  [>=latex, <->, shorten >= 3pt, shorten <= 3pt, color=black] (e0)--(e0');
\shade [ball color=black] (e0') circle [radius = 1.5pt];
\draw  [>=latex, <->, shorten >= 3pt, shorten <= 3pt, color=black] (e0')--(e0'');
\shade [ball color=black] (e0'') circle [radius = 1.5pt];

\begin{scope}[color=black!42]
\draw  [>=latex, <->, shorten >= 3pt, shorten <= 3pt, densely dashed] (cr)--(rem);
\shade [ball color=black!84] (rem') circle [radius = 0.96*1.5pt];
\filldraw [color=white, opacity = 0.5] (rem') circle [radius = 0.96*1.55pt];
\draw  [>=latex, <->, shorten >= 3.5pt, shorten <= 3.5pt] (rem)--(rem');
\shade [ball color=black!84] (rem) circle [radius = 0.96*1.5pt];
\filldraw [color=white, opacity = 0.5] (rem) circle [radius = 0.96*1.55pt];
\draw  [>=latex, <->, shorten >= 3pt, shorten <= 3pt, densely dashed] (rem)--(re2);
\shade [ball color=black!84] (re2') circle [radius = 0.96*1.5pt];
\filldraw [color=white, opacity = 0.5] (re2') circle [radius = 0.96*1.55pt];
\draw  [>=latex, <->, shorten >= 3.5pt, shorten <= 3.5pt] (re2)--(re2');
\shade [ball color=black!84] (re2) circle [radius = 0.96*1.5pt];
\filldraw [color=white, opacity = 0.5] (re2) circle [radius = 0.96*1.55pt];
\draw  [>=latex, <->, shorten >= 3pt, shorten <= 3pt, densely dashed] (re2)--(re1);
\shade [ball color=black!84] (re1') circle [radius = 0.96*1.5pt];
\filldraw [color=white, opacity = 0.5] (re1') circle [radius = 0.96*1.55pt];
\draw  [>=latex, <->, shorten >= 3.5pt, shorten <= 3.5pt] (re1)--(re1');
\shade [ball color=black!84] (re1) circle [radius = 0.96*1.5pt];
\filldraw [color=white, opacity = 0.5] (re1) circle [radius = 0.96*1.55pt];
\draw  [>=latex, <->, shorten >= 3pt, shorten <= 3pt, densely dashed] (re1)--(re0);
\shade [ball color=black!84] (re0) circle [radius = 0.96*1.5pt];
\filldraw [color=white, opacity = 0.5] (re0) circle [radius = 0.96*1.55pt];
\draw  [>=latex, <->, shorten >= 3pt, shorten <= 3pt] (re0)--(re0');
\shade [ball color=black!84] (re0') circle [radius = 0.96*1.5pt];
\filldraw [color=white, opacity = 0.5] (re0') circle [radius = 0.96*1.55pt];
\draw  [>=latex, <->, shorten >= 3pt, shorten <= 3pt] (re0')--(re0'');
\shade [ball color=black!84] (re0'') circle [radius = 0.96*1.5pt];
\filldraw [color=white, opacity = 0.5] (re0'') circle [radius = 0.96*1.55pt];
\shade [ball color=black!84] (re3') circle [radius = 0.96*1.5pt];
\filldraw [color=white, opacity = 0.5] (re3') circle [radius = 0.96*1.55pt];

\draw  [>=latex, <->, shorten >= 3pt, shorten <= 3pt, densely dashed] (cu)--(uem);
\shade [ball color=black!84] (uem') circle [radius = 0.90*1.5pt];
\filldraw [color=white, opacity = 0.5] (uem') circle [radius = 0.90*1.55pt];
\draw  [>=latex, <->, shorten >= 3.5pt, shorten <= 3.5pt] (uem)--(uem');
\shade [ball color=black!84] (uem) circle [radius = 0.90*1.5pt];
\filldraw [color=white, opacity = 0.5] (uem) circle [radius = 0.90*1.55pt];
\draw  [>=latex, <->, shorten >= 3pt, shorten <= 3pt, densely dashed] (uem)--(ue2);
\shade [ball color=black!84] (ue2') circle [radius = 0.90*1.5pt];
\filldraw [color=white, opacity = 0.5] (ue2') circle [radius = 0.90*1.55pt];
\draw  [>=latex, <->, shorten >= 3.5pt, shorten <= 3.5pt] (ue2)--(ue2');
\shade [ball color=black!84] (ue2) circle [radius = 0.90*1.5pt];
\filldraw [color=white, opacity = 0.5] (ue2) circle [radius = 0.90*1.55pt];
\draw  [>=latex, <->, shorten >= 3pt, shorten <= 3pt, densely dashed] (ue2)--(ue1);
\shade [ball color=black!84] (ue1') circle [radius = 0.90*1.5pt];
\filldraw [color=white, opacity = 0.5] (ue1') circle [radius = 0.90*1.55pt];
\draw  [>=latex, <->, shorten >= 3.5pt, shorten <= 3.5pt] (ue1)--(ue1');
\shade [ball color=black!84] (ue1) circle [radius = 0.90*1.5pt];
\filldraw [color=white, opacity = 0.5] (ue1) circle [radius = 0.90*1.55pt];
\draw  [>=latex, <->, shorten >= 3pt, shorten <= 3pt, densely dashed] (ue1)--(ue0);
\shade [ball color=black!84] (ue0) circle [radius = 0.90*1.5pt];
\filldraw [color=white, opacity = 0.5] (ue0) circle [radius = 0.90*1.55pt];
\draw  [>=latex, <->, shorten >= 3pt, shorten <= 3pt] (ue0)--(ue0');
\shade [ball color=black!84] (ue0') circle [radius = 0.90*1.5pt];
\filldraw [color=white, opacity = 0.5] (ue0') circle [radius = 0.90*1.55pt];
\draw  [>=latex, <->, shorten >= 3pt, shorten <= 3pt] (ue0')--(ue0'');
\shade [ball color=black!84] (ue0'') circle [radius = 0.90*1.5pt];
\filldraw [color=white, opacity = 0.5] (ue0'') circle [radius = 0.90*1.55pt];
\shade [ball color=black!84] (ue3') circle [radius = 0.90*1.5pt];
\filldraw [color=white, opacity = 0.5] (ue3') circle [radius = 0.90*1.55pt];

\end{scope}

\begin{scope}[>=latex, <->, shorten >= 3pt, shorten <= -7.5pt, color=black!84]
\draw [] ($(c12)+(c)-(d)$)--(c);
\draw [] ($(c23)+(c)-(d)$)--(c);
\end{scope}

\node [right, color=black] at ($(c)  +( 0.05,-0.03)$) {$a$} ;
\node [right, color=black] at ($(em) +(-0.01,-0.08)$) {$e_m^a$} ;
\node [right, color=black] at ($(e2) +(-0.01,-0.08)$) {$e_{k_n+i}^a$} ;
\node [right, color=black] at ($(e1) +(-0.01,-0.08)$) {$e_{k_n+j+4}^a$} ;
\node [right, color=black] at ($(e0) +(-0.01,-0.08)$) {$e_{k_n+9}^a$} ;
\node [right, color=black] at ($(e0')+(-0.01,+0.06)$) {$e_{k_n+10}^a$} ;
\node [right, color=black] at ($(e0'')+(-0.01,+0.06)$) {$e_{k_n+11}^a$} ;
\node [right, color=black] at ($(em') +(0.02,-0.01)$) {$e_P^a$} ;
\node [right, color=black] at ($(e2') +(0.02,-0.01)$) {$e_R^a$} ;
\node [right, color=black] at ($(e1') +(0.02,-0.01)$) {$e_U^a$} ;
\node [right, color=black] at ($(e3') +(0.02,-0.01)$) {$e_C^a$} ;

\node [right, color=black] at ($(cu)  +( 0.05,-0.03)$) {$b$} ;
\node [right, color=black] at ($(cr)  +( 0.05,-0.03)$) {$c$} ;
\node [right, color=black] at ($(ue0'')+(-0.01,+0.06)$) {$e_{k_n+11}^b$} ;
\node [right, color=black] at ($(re0'')+(-0.01,+0.06)$) {$e_{k_n+11}^c$} ;
\node [right, color=black] at ($(ue3') +(-0.01,+0.06)$) {$e_C^b$} ;
\node [right, color=black] at ($(re3') +(-0.01,+0.06)$) {$e_C^c$} ;

%\node [right, color=black] at ($(ekn)+(-0.01,+0.06)$) {$e_0^a$} ;
%\node [right, color=black] at ($(ek2)+(-0.01,-0.08)$) {$e_2^a$} ;
%\node [right, color=black] at ($(ek1)+(-0.01,-0.08)$) {$e_1^a$} ;

\end{tikzpicture}
\caption{Simulation of $P_m(a)$, $R_i(a)$, $U_j(a)$, $C(a)$, $\neg C(b)$, and $\neg C(c)$}
\label{fig:17}
\end{figure}
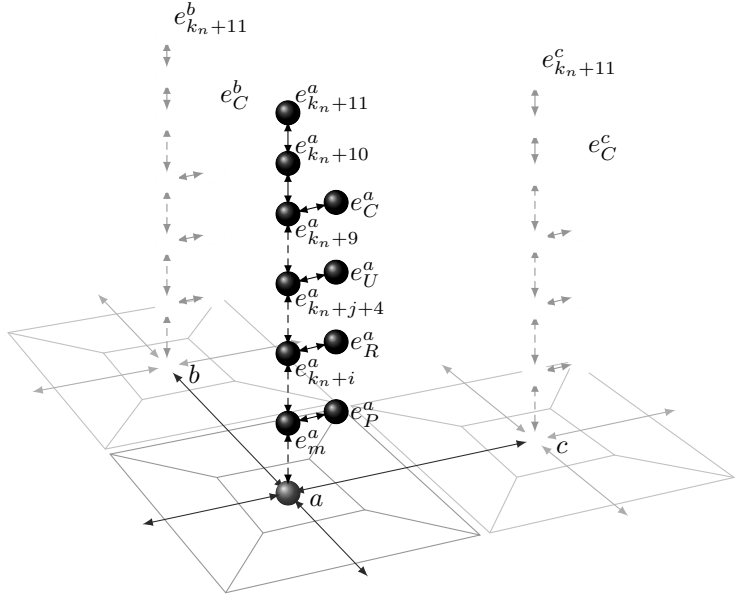

%\begin{lemma}
%\label{lem:relativization:p:S}
%Let $L$ be a subframe modal predicate logic such that\/ $\ckf L$ is a Skvortsov class. Then, $n\in\mathbb{Y}$ implies $p\to (S'_3\mathit{M}^\Box_G\mathit{Tiling}_n^{\mathbb{X}})_p \not\in \QMLcw \ckf L$. 
%\end{lemma}
%
%\begin{proof}
%~
%\end{proof}

Let $q$ and $r$ be two proposition letters different from~$p$.
Let $S_4\mathit{M}^\Box_G\mathit{Tiling}_n^{\mathbb{X}}$ be a formula obtained from $p\to (S'_3\mathit{M}^\Box_G\mathit{Tiling}_n^{\mathbb{X}})_p$ by substituting $\Box(q\to Q(x_1)\vee Q(x_2))$ and $\Box(r\to Q(x_1))$ for $P(x_1,x_2)$ and $G(x_1)$, respectively.

\begin{lemma}
\label{lem:relativization:S4:K}
If $n\in\mathbb{X}$, then $S_4\mathit{M}^\Box_G\mathit{Tiling}_n^{\mathbb{X}} \in \logic{QK}$. 
\end{lemma}

\begin{proof}
Follows from Lemma~\ref{lem:relativization:p:wKHC}, since $\logic{QK}$ is closed under Substitution.
\end{proof}

\begin{lemma}
\label{lem:relativization:S4:wKHC}
If $n\in\mathbb{Y}$, then
$S_4\mathit{M}^\Box_G\mathit{Tiling}_n^{\mathbb{X}} \not\in L_{\mathit{wfin}}\oplus\bm{bd}_3\oplus\bm{bf}$ for $L\in\{\logic{QGL},\logic{QGrz}\}$. 
\end{lemma}

\begin{proof}
Let $L=\logic{QGL}$.
Let $n\in\mathbb{Y}$. 
Let $\kModel{M}=\otuple{\kframe{F}\odot\mathcal{D}',I}$ be the model constructed in the proof of Lemma~\ref{lem:relativization:p:wKHC}. 
Recall that $\mathcal{D}'$ contains $\mathcal{D}$ and $\mathcal{D}=\{0,\ldots,r+4\}\times\{0,\ldots,r+4\}$.
Define
$$
\begin{array}{lcl}
W_0     & = & \{\bar{w}_0\}; \smallskip\\
W_1     & = & \{\bar{w}_a : a\in \mathcal{D}\}; \smallskip\\
W_2^P   & = & \{u^a_{bc} : \mbox{$a\in \mathcal{D}$, $b,c\in\mathcal{D}'$}\}; \smallskip\\
W_2^G   & = & \{u_G\}; \smallskip\\
\bar{W} & = & W_0 \cup W_1 \cup W_2^P \cup W_2^G.
\end{array}
$$
Let $\bar{R}$ be the transitive closure of the following relation on~$W$:
$$
\begin{array}{c}
\displaystyle
(W_0 \times W_1) 
  \cup \{\otuple{\bar{w}_a,u^a_{bc}} : \mbox{$a\in\mathcal{D}$, $b,c\in\mathcal{D}'$}\}
  \cup (W_1\times W_2^G).
\end{array}
$$
Let $\bar{\kModel{M}}=\otuple{\bar{\kframe{F}}\odot\mathcal{D}',\bar{I}}$ be a model such that, for every $w\in \bar{W}$, every $a\in\mathcal{D}$, and all $b,c,d\in\mathcal{D}'$,
$$
\begin{array}{lcl}
\bar{\kModel{M}},w\models p      & \iff & w\in W_0\cup W_1; \smallskip\\
\bar{\kModel{M}},w\models q      & \iff & w\in W_2^P;       \smallskip\\
\bar{\kModel{M}},w\models r      & \iff & w=u_G;            \smallskip\\
\bar{\kModel{M}},u_G\models Q(d) & \iff & d\in\mathcal{D};  \smallskip\\
\bar{\kModel{M}},u^a_{bc}\not\models Q(d) 
  & \iff 
  & \mbox{$d\in\{b,c\}$ and $\kModel{M},w_a\not\models P(b,c)$.}
  \smallskip\\
\end{array}
$$

Let us compare $\bar{\kModel{M}}$ and $\kModel{M}$. 
Observe that
%, for every $a\in\mathcal{D}\cup\{0\}$,
$$
\begin{array}{lcl}
\{{w}\in    {W} :     {\kModel{M}},{w}\models p\} & = & \{w_0\}\cup \{{w}_a : a \in\mathcal{D}\}; \\
\{{w}\in\bar{W} : \bar{\kModel{M}},{w}\models p\} & = & \{\bar{w}_0\}\cup \{\bar{w}_a : a \in\mathcal{D}\};
\end{array}
$$
also, for every $a\in \mathcal{D}\cup\{0\}$ and all $b,c\in\mathcal{D}'$,
$$
\begin{array}{lcl}
\kModel{M},{w}_a\models G(b) 
  & \iff 
  & \bar{\kModel{M}},\bar{w}_a\models \Box (r\to Q(b)); 
  \smallskip \\
\kModel{M},{w}_a\models P(b,c) 
  & \iff 
  & \bar{\kModel{M}},\bar{w}_a\models \Box(q\to Q(b)\vee Q(c)). 
\end{array}
$$
Since $\kModel{M},{w}_0\not\models p\to (S'_3\mathit{M}^\Box_G\mathit{Tiling}_n^{\mathbb{X}})_p$, we readily obtain that $\bar{\kModel{M}},\bar{w}_0\not\models S_4\mathit{M}^\Box_G\mathit{Tiling}_n^{\mathbb{X}}$. 
Thus, $S_4\mathit{M}^\Box_G\mathit{Tiling}_n^{\mathbb{X}} \not\in \logic{QGL}_{\mathit{wfin}}\oplus\bm{bd}_3\oplus\bm{bf}$. 

Let $L=\logic{QGrz}$.
Then just take the model obtained from $\bar{\kModel{M}}$ by replacing $\bar{R}$ with its reflexive closure; we leave the details to the reader. 
\end{proof}

%\begin{lemma}
%\label{lem:relativization:S4:S}
%Let $L$ be a subframe modal predicate logic such that\/ $\ckf L$ is a Skvortsov class. Then, $n\in\mathbb{Y}$ implies $S_4\mathit{M}^\Box_G\mathit{Tiling}_n^{\mathbb{X}} \not\in \QMLcw \ckf L$. 
%\end{lemma}
%
%\begin{proof}
%~
%\end{proof}

Let us eliminate $p$, $q$, and $r$. There are many ways to do this with formulas that do not contain predicate letters different from~$Q$. We show two of them, leading to slightly different results.

Let $S_5$ be a formula substitution defined by 
$$
\begin{array}{lcl}
S_5(p) 
  & = 
  & \Diamond \forall x\, Q(x)\wedge 
    \Diamond \forall x\, \neg Q(x); \\
S_5(q) 
  & = 
  & \exists x\,Q(x)\wedge 
    \Diamond \forall x\,\neg Q(x)\wedge 
    \neg \Diamond \forall x\, Q(x); \\ 
S_5(r) 
  & = 
  & \exists x\,\neg Q(x)\wedge 
    \Diamond \forall x\, Q(x)\wedge 
    \neg \Diamond \forall x\,\neg Q(x).
\end{array}
$$

\begin{lemma}
\label{lem:relativization:S4:S5:K}
If $n\in\mathbb{X}$, then $S_5 S_4\mathit{M}^\Box_G\mathit{Tiling}_n^{\mathbb{X}} \in \logic{QK}$. 
\end{lemma}

\begin{proof}
Follows from Lemma~\ref{lem:relativization:S4:K}, since $\logic{QK}$ is closed under Substitution.
\end{proof}

\begin{lemma}
\label{lem:relativization:S4:S5:QGL:QGrz}
If $n\in\mathbb{Y}$, then
$S_5 S_4\mathit{M}^\Box_G\mathit{Tiling}_n^{\mathbb{X}} \not\in L_{\mathit{wfin}}\oplus\bm{bd}_4\oplus\bm{bf}$ for $L\in\{\logic{QGL},\logic{QGrz}\}$. 
\end{lemma}

\begin{proof}
Let $n\in\mathbb{Y}$ and let $L=\logic{QGL}$. We modify the model $\bar{\kModel{M}}=\otuple{\bar{\kframe{F}}\odot\mathcal{D}',\bar{I}}$ defined in the proof of Lemma~\ref{lem:relativization:S4:wKHC} and based on the Kripke frame $\bar{\kframe{F}} = \otuple{\bar{W},\bar{R}}$. Let
$$
\begin{array}{lcl}
W_3 & = & \{v^-,v^+\}; \smallskip\\
\tilde{W} & = & \bar{W} \cup W_3; \smallskip\\
\end{array}
$$
let also $\tilde{R}$ be the transitive closure of the relation
$$
\begin{array}{c}
\bar{R} \cup (W_2^P \times \{v^-\}) \cup (W_2^G \times \{v^+\}).
\end{array}
$$
Let $\tilde{\kframe{F}}=\otuple{\tilde{W},\tilde{R}}$ and let $\tilde{\kModel{M}} = \otuple{\tilde{\kframe{F}}\odot\mathcal{D}',\tilde{I}}$ be a model such that, for every $w\in \bar{W}$ and every $d\in\mathcal{D}'$,
$$
\begin{array}{lcl}
\tilde{\kModel{M}},w\models Q(d) & \iff & \bar{M},w\models Q(d) \\
%\tilde{\kModel{M}},v^-\not\models Q(d); \\
%\tilde{\kModel{M}},v^+\models Q(d). \\
\end{array}
$$
and also
$$
\begin{array}{lcl}
\tilde{\kModel{M}},v^-\models \forall x\,\neg Q(x) 
  & \mbox{and} 
  & \tilde{\kModel{M}},v^+\models \forall x\,Q(x). 
\end{array}
$$
Then it is not hard to see that
$$
\begin{array}{lcl}
\{w \in \tilde{W} : \tilde{\kModel{M}},w \models S_5(p)\} 
  & = 
  &  \{w \in \bar{W} : \bar{\kModel{M}},w \models p\}; \\
\{w \in \tilde{W} : \tilde{\kModel{M}},w \models S_5(q)\} 
  & = 
  &  \{w \in \bar{W} : \bar{\kModel{M}},w \models q\}; \\
\{w \in \tilde{W} : \tilde{\kModel{M}},w \models S_5(r)\} 
  & = 
  &  \{w \in \bar{W} : \bar{\kModel{M}},w \models r\}. \\
\end{array}
$$
Since $\bar{\kModel{M}},\bar{w}_0\not\models S_4\mathit{M}^\Box_G\mathit{Tiling}_n^{\mathbb{X}}$, we readily obtain that $\tilde{\kModel{M}},\bar{w}_0\not\models S_5 S_4\mathit{M}^\Box_G\mathit{Tiling}_n^{\mathbb{X}}$. 
Thus, $S_5 S_4\mathit{M}^\Box_G\mathit{Tiling}_n^{\mathbb{X}} \not\in \logic{QGL}_{\mathit{wfin}}\oplus\bm{bd}_4\oplus\bm{bf}$. 

Let $L=\logic{QGrz}$.
Then just take the model obtained from $\tilde{\kModel{M}}$ by replacing $\tilde{R}$ with its reflexive closure; we leave the details to the reader. 
\end{proof}

\begin{lemma}
\label{lem:relativization:S4:S5:QGL:QGrz:dfin}
If $n\in\mathbb{Y}$, then
$S_5 S_4\mathit{M}^\Box_G\mathit{Tiling}_n^{\mathbb{X}} \not\in L_{\mathit{dfin}}\oplus\bm{bd}_4\oplus\bm{bf}$ for $L\in\{\logic{QGL},\logic{QGrz}\}$. 
\end{lemma}

\begin{proof}
Both $\logic{QGL}$ and $\logic{QGrz}$ are subframe logics; therefore, the statement follows from Lemma~\ref{lem:relativization:S4:S5:QGL:QGrz}.
\end{proof}

\begin{theorem}
\label{th:ml:insep:QCG:wfin:S5}
Logics\/ $\logic{QK}$ and\/ $\logic{QGL}_{\mathit{wfin}}\oplus\bm{bd}_4\oplus\bm{bf}$ as well as\/ $\logic{QK}$ and\/ $\logic{QGrz}_{\mathit{wfin}}\oplus\bm{bd}_4\oplus\bm{bf}$ are recursively inseparable in the language with a single unary predicate letter and two individual variables.
\end{theorem}

\begin{proof}
Follows from Lemmas~\ref{lem:relativization:S4:S5:K} and~\ref{lem:relativization:S4:S5:QGL:QGrz}.
\end{proof}

\begin{theorem}
\label{th:ml:insep:QCG:wfin:S5:dfin}
Logics\/ $\logic{QK}$ and\/ $\logic{QGL}_{\mathit{dfin}}\oplus\bm{bd}_4\oplus\bm{bf}$ as well as\/ $\logic{QK}$ and\/ $\logic{QGrz}_{\mathit{dfin}}\oplus\bm{bd}_4\oplus\bm{bf}$ are recursively inseparable in the language with a single unary predicate letter and two individual variables.
\end{theorem}

\begin{proof}
Follows from Lemmas~\ref{lem:relativization:S4:S5:K} and~\ref{lem:relativization:S4:S5:QGL:QGrz:dfin}.
\end{proof}

\begin{corollary}
\label{cor:1:th:ml:insep:QCG:wfin:S5}
Let\/ $\logic{QK}\subseteq L\subseteq L'$ and also either $L'\subseteq \logic{QGL}\oplus\bm{bd}_4\oplus\bm{bf}$ or\/ $L'\subseteq \logic{QGrz}\oplus\bm{bd}_4\oplus\bm{bf}$. Then $L$ and $L'_{\mathit{wfin}}$ as well as $L$ and $L'_{\mathit{dfin}}$ are recursively inseparable in the language with a single unary predicate letter and two individual variables. 
\end{corollary}

\begin{corollary}
\label{cor:2:th:ml:insep:QCG:wfin:S5}
Let $L$ be one of $\logic{QK}$, $\logic{QT}$, $\logic{QD}$, $\logic{QK4}$, $\logic{QD4}$, $\logic{QS4}$, $\logic{QGL}$, $\logic{QGrz}$, $\logic{QK4}\oplus\bm{bd}_m$, $\logic{QS4}\oplus\bm{bd}_m$, $\logic{QGL}\oplus\bm{bd}_m$, $\logic{QGrz}\oplus\bm{bd}_m$, where $m\geqslant 4$. Then $L$ and $L_{\mathit{wfin}}$ as well as $L$ and $L_{\mathit{dfin}}$ are recursively inseparable in the language with a single unary predicate letter and two individual variables. 
\end{corollary}

Let $S'_5$ be a formula substitution defined by 
$$
\begin{array}{lcl}
S'_5(p) 
  & = 
  & \Diamond\Diamond\top \wedge \Diamond\Box\bot; \\
S'_5(q) 
  & = 
  & \Box\bot; \\ 
S'_5(r) 
  & = 
  & \Diamond \top \wedge \neg \Diamond\Box\bot.
\end{array}
$$

\begin{lemma}
\label{lem:relativization:S4:S'5:K}
If $n\in\mathbb{X}$, then $S'_5 S_4\mathit{M}^\Box_G\mathit{Tiling}_n^{\mathbb{X}} \in \logic{QK}$. 
\end{lemma}

\begin{proof}
Follows from Lemma~\ref{lem:relativization:S4:K}, since $\logic{QK}$ is closed under Substitution.
\end{proof}

\begin{lemma}
\label{lem:relativization:S4:S'5:QGL:QGrz}
If $n\in\mathbb{Y}$, then
$S'_5 S_4\mathit{M}^\Box_G\mathit{Tiling}_n^{\mathbb{X}} \not\in \logic{QwGrz}_{\mathit{wfin}}\oplus\bm{bd}_3\oplus\bm{bf}$. 
\end{lemma}

\begin{proof}
Let $n\in\mathbb{Y}$. Just modify the model $\bar{\kModel{M}}=\otuple{\bar{\kframe{F}}\odot\mathcal{D}',\bar{I}}$ defined in the proof of Lemma~\ref{lem:relativization:S4:wKHC} and based on the Kripke frame $\bar{\kframe{F}} = \otuple{\bar{W},\bar{R}}$ by replacing $\bar{R}$ with
$$
\begin{array}{c}
\bar{R}\cup \{\otuple{u_G,u_G}\}.
\end{array}
$$
Let $\tilde{\kModel{M}}$ be the resulting model.
Then it should be clear that
$$
\begin{array}{lcl}
\{w \in \bar{W} : \tilde{\kModel{M}},w \models S'_5(p)\} 
  & = 
  &  \{w \in \bar{W} : \bar{\kModel{M}},w \models p\}; \\
\{w \in \bar{W} : \tilde{\kModel{M}},w \models S'_5(q)\} 
  & = 
  &  \{w \in \bar{W} : \bar{\kModel{M}},w \models q\}; \\
\{w \in \bar{W} : \tilde{\kModel{M}},w \models S'_5(r)\} 
  & = 
  &  \{w \in \bar{W} : \bar{\kModel{M}},w \models r\}. \\
\end{array}
$$
Since $\bar{\kModel{M}},\bar{w}_0\not\models S_4\mathit{M}^\Box_G\mathit{Tiling}_n^{\mathbb{X}}$, we readily obtain that $\tilde{\kModel{M}},\bar{w}_0\not\models S'_5 S_4\mathit{M}^\Box_G\mathit{Tiling}_n^{\mathbb{X}}$. 
Thus, $S'_5 S_4\mathit{M}^\Box_G\mathit{Tiling}_n^{\mathbb{X}} \not\in \logic{QwGrz}_{\mathit{wfin}}\oplus\bm{bd}_3\oplus\bm{bf}$. 
\end{proof}

\begin{lemma}
\label{lem:relativization:S4:S'5:QGL:QGrz:dfin}
If $n\in\mathbb{Y}$, then
$S'_5 S_4\mathit{M}^\Box_G\mathit{Tiling}_n^{\mathbb{X}} \not\in \logic{QwGrz}_{\mathit{dfin}}\oplus\bm{bd}_3\oplus\bm{bf}$. 
\end{lemma}

\begin{proof}
The statement follows from Lemma~\ref{lem:relativization:S4:S'5:QGL:QGrz}, since $\logic{QwGrz}$ is a subframe logic.
\end{proof}

\begin{theorem}
\label{th:ml:insep:QCG:wfin:S'5}
Logics\/ $\logic{QK}$ and\/ $\logic{QwGrz}_{\mathit{wfin}}\oplus\bm{bd}_3\oplus\bm{bf}$ are recursively inseparable in the language with a single unary predicate letter and two individual variables.
\end{theorem}

\begin{proof}
Follows from Lemmas~\ref{lem:relativization:S4:S'5:K} and~\ref{lem:relativization:S4:S'5:QGL:QGrz}.
\end{proof}

\begin{theorem}
\label{th:ml:insep:QCG:wfin:S'5:dfin}
Logics\/ $\logic{QK}$ and\/ $\logic{QwGrz}_{\mathit{dfin}}\oplus\bm{bd}_3\oplus\bm{bf}$ are recursively inseparable in the language with a single unary predicate letter and two individual variables.
\end{theorem}

\begin{proof}
Follows from Lemmas~\ref{lem:relativization:S4:S'5:K} and~\ref{lem:relativization:S4:S'5:QGL:QGrz:dfin}.
\end{proof}

\begin{corollary}
\label{cor:1:th:ml:insep:QCG:wfin:S'5}
Let\/ $\logic{QK}\subseteq L\subseteq L'\subseteq \logic{QwGrz}\oplus\bm{bd}_3\oplus\bm{bf}$. Then $L$ and $L'_{\mathit{wfin}}$ as well as $L$ and $L'_{\mathit{dfin}}$ are recursively inseparable in the language with a single unary predicate letter and two individual variables. 
\end{corollary}

\begin{corollary}
\label{cor:2:th:ml:insep:QCG:wfin:S'5}
Let $L$ be one of $\logic{QK}$, $\logic{QK4}$, $\logic{QwGrz}$, $\logic{QK4}\oplus\bm{bd}_m$, $\logic{QwGrz}\oplus\bm{bd}_m$, where $m\geqslant 3$. Then $L$ and $L_{\mathit{wfin}}$ as well as $L$ and $L_{\mathit{dfin}}$ are recursively inseparable in the language with a single unary predicate letter and two individual variables. 
\end{corollary}

% пїЅпїЅпїЅпїЅпїЅпїЅпїЅпїЅпїЅпїЅпїЅпїЅпїЅ пїЅпїЅпїЅ QGL +bd2 -- \Diamond\top 
% пїЅпїЅпїЅпїЅпїЅпїЅпїЅпїЅпїЅпїЅпїЅпїЅпїЅ пїЅпїЅпїЅ QGrz+bw2 -- \exists x\,(\Diamond P(x)\wedge \Diamond\neg P(x))

%\subsection{Varying domain semantics}
%\label{subsec:varsem}
%
%Ax [](Ey (D(x)&D(y) --> (P(x)<-->P(y)))).
%Ax [](Ay (D(x)&D(y) --> (P(x)<-->P(y)))).

\section{Superintuitionistic predicate logics}
\label{sec:int}
\setcounter{equation}{0}

\subsection{Syntax and semantics}

%The intuitionistic predicate language is the same language~$\lang{L}$ defined above for the classical and superclassical logics and theories.
%
The intuitionistic predicate language is identical to the language~$\lang{L}$ defined above for both classical and superclassical logics and theories.

An \defnotion{intuitionistic Kripke frame} is a Kripke frame $\kframe{F} = \otuple{W,R}$ where $R$ is a partial order~--- i.e., a reflexive, transitive, and antisymmetric binary relation~--- on~$W$.  An \defnotion{intuitionistic augmented frame} is an e-augmented frame $\kframe{F}_D = \otuple{\kframe{F}, D}$ such that $\kframe{F}$ is an intuitionistic Kripke frame.  An \defnotion{intuitionistic predicate Kripke model} is a predicate Kripke model $\mathfrak{M} = \otuple{W,R,D,I}$ where $\otuple{W,R,D}$ is an
intuitionistic augmented frame and the interpretation~$I$ satisfies the \defnotion{heredity condition}: for all $w, w' \in W$ and every predicate letter~$P$,
$$
%\begin{equation}
\begin{array}{lcl}
  wRw' & \Longrightarrow & I(w,P) \subseteq I(w',P).
\end{array}
%\eqno{(4)}
%\label{eq:4}
%\end{equation}
$$
%An \textit{assignment} is a map $g \colon \mathit{Var} \to D^+$. 
The truth of an $\lang{L}$-formula $\varphi$ at a world $w$ of an intuitionistic model $\kModel{M} = \otuple{W,R,D,I}$ under an assignment $g$ is defined recursively:
%%%%%%%%%%%%%%%%%%%%%%%%%%%%%%%%%%%%%%%%%%%%%%%%%%%%%%%%%%%%%%%%%%%%%%
\settowidth{\templength}{\mbox{$\kModel{M},w\imodels^g\varphi'$ and $\kModel{M},w\imodels^g\varphi''$;}}
\settowidth{\templengtha}{\mbox{$w$}}
\settowidth{\templengthb}{\mbox{$\kModel{M},w\imodels^{h}\varphi'$, for every assignment $h$ such that}}
\settowidth{\templengthc}{\mbox{$\kModel{M},w\imodels^g P(x_1,\ldots,x_n)$}}
\settowidth{\templengthd}{\mbox{$\kModel{M},\parbox{\templengtha}{$v$}\imodels^{h}\varphi'$, for every $v\in R(w)$ and every assignment;;;}}
%%%%%%%%%%%%%%%%%%%%%%%%%%%%%%%%%%%%%%%%%%%%%%%%%%%%%%%%%%%%%%%%%%%%%%
$$
\begin{array}{lcl}
\kModel{M},w\imodels^g P(x_1,\ldots,x_n)
  & \leftrightharpoons
  & \parbox{\templengthd}{$\langle g(x_1),\ldots,g(x_n)\rangle \in P^{I, w}$,} \\
\end{array}
$$
\mbox{where $P$ is an $n$-ary predicate letter;}
%%%%%%%%%%%%%%%%%%%%%%%%%%%%%%%%%%%%%%%%%%%%%%%%%%%%%%%%%%%%%%%%%%%%%%
\settowidth{\templength}{\mbox{$\kModel{M},w\imodels^g\varphi'$ and $\kModel{M},w\imodels^g\varphi''$;}}
\settowidth{\templengtha}{\mbox{$w$}}
\settowidth{\templengthb}{\mbox{$\kModel{M},w\imodels^{g}\varphi'\to\varphi''$}}
\settowidth{\templengthc}{\mbox{$\kModel{M},w\imodels^g P(x_1,\ldots,x_n)$}}
%%%%%%%%%%%%%%%%%%%%%%%%%%%%%%%%%%%%%%%%%%%%%%%%%%%%%%%%%%%%%%%%%%%%%%
$$
\begin{array}{lcl}
\parbox{\templengthc}{{}\hfill\parbox{\templengthb}{$\kModel{M},w \not\imodels^g \bot;$}}
  \\
\parbox{\templengthc}{{}\hfill\parbox{\templengthb}{$\kModel{M},w\imodels^g\varphi' \wedge \varphi''$}}
  & \leftrightharpoons
  & \parbox[t]{\templength}{$\kModel{M},w\imodels^g\varphi'$ and $\kModel{M},w\imodels^g\varphi''$;}
  \\
\parbox{\templengthc}{{}\hfill\parbox{\templengthb}{$\kModel{M},w\imodels^g\varphi' \vee \varphi''$}}
  & \leftrightharpoons
  & \parbox[t]{\templength}{$\kModel{M},w\imodels^g\varphi'$\hfill or\hfill $\kModel{M},w\imodels^g\varphi''$;}
  \\
\parbox{\templengthc}{{}\hfill\parbox{\templengthb}{$\kModel{M},w\imodels^g\varphi' \to \varphi''$}}
  & \leftrightharpoons
  & \parbox[t]{\templengthd}{$\kModel{M},\parbox{\templengtha}{$v$}\imodels^g\varphi'$ implies $\kModel{M},v\imodels^g\varphi''$, for every $v\in R(w)$;}
  \\
\parbox{\templengthc}{{}\hfill\parbox{\templengthb}{$\kModel{M},w\imodels^g\forall x\,\varphi'$}}
  & \leftrightharpoons
  & \parbox{\templengthd}{$\kModel{M},\parbox{\templengtha}{$v$}\imodels^{h}\varphi'$, for every $v\in R(w)$ and every assignment}
  \\
  &
  & \mbox{\phantom{$\kModel{M},w\imodels^{g'}\varphi'$, }$h$ such that $h \stackrel{x}{=} g$ and $h(x)\in D_w$;}
  \\
\parbox{\templengthc}{{}\hfill\parbox{\templengthb}{$\kModel{M},w\imodels^g\exists x\,\varphi'$}}
  & \leftrightharpoons
  & \parbox{\templengthd}{$\kModel{M},w\imodels^{h}\varphi'$, for some assignment $h$ such that $h \stackrel{x}{=} g$}
  \\
  &
  & \mbox{\phantom{$\kModel{M},w\imodels^{g'}\varphi'$, }and $h(x)\in D_w$.}
\end{array}
$$

Let $\kModel{M}$, $\kframe{F}_D$, $\kframe{F}$, and $\Scls{C}$ be an intuitionistic Kripke model, an intuitionistic augmented frame, an intuitionistic Kripke frame, and a class of intuitionistic augmented frames, respectively, $w$ a world of $\kModel{M}$, and $\varphi$ a formula with free variables $x_1,\ldots,x_n$; then define
%%%%%%%%%%%%%%%%%%%%%%%%%%%%%%%%%%%%%%%%%%%%%%%%%%%%%%%%%%%%%%%%%%%%%%%%%%%%%%%%%%
\settowidth{\templength}{\mbox{$\kModel{M},w\imodels^g P(x_1,\ldots,x_n)$}}
\settowidth{\templengtha}{\mbox{$\kModel{M},w\imodels^{h}\varphi'$, for every assignment $h$ such that}}
\settowidth{\templengthb}{\mbox{$w$}}
\settowidth{\templengthc}{\mbox{$\kframe{F}_D$}}
%%%%%%%%%%%%%%%%%%%%%%%%%%%%%%%%%%%%%%%%%%%%%%%%%%%%%%%%%%%%%%%%%%%%%%%%%%%%%%%%%%
$$
\begin{array}{rcl}
\Rem{
\parbox{\templength}{{}\hfill$\kModel{M},w\imodels \varphi$}
  & \leftrightharpoons
  & \parbox[t]{\templengthd}{$\kModel{M},w\imodels^g \varphi$, for every assignment $g$ such that}
  \\
  &
  & \mbox{\phantom{$\kModel{M},w\imodels^g \varphi$, }$g(x_1),\ldots,g(x_n)\in D_w$;}
  \\
} % for \Rem{...}  
\parbox{\templength}{{}\hfill$\kModel{M}\imodels \varphi$}
  & \leftrightharpoons
  & \parbox[t]{\templengthd}{$\kModel{M},w\imodels^g \varphi$, for every world $w$ of $\kModel{M}$ and every~$g$}
  \\
  &
  & \mbox{\phantom{$\kModel{M},w\imodels^g \varphi$, }such that $g(x_1),\ldots,g(x_n)\in D_w$;}
  \\
\parbox{\templength}{{}\hfill$\kframe{F}_D\imodels \varphi$}
  & \leftrightharpoons
  & \parbox[t]{\templengthd}{$\parbox{\templengthc}{$\kModel{M}$}\imodels \varphi$, for every intuitionistic model $\kModel{M}$ based on $\kframe{F}_D$;}
  \\
\parbox{\templength}{{}\hfill$\kframe{F}\imodels \varphi$}
  & \leftrightharpoons
  & \parbox[t]{\templengthd}{$\parbox{\templengthc}{$\kModel{M}$}\imodels \varphi$, for every intuitionistic model $\kModel{M}$ based on $\kframe{F}$;}
  \\
\parbox{\templength}{{}\hfill$\Scls{C}\imodels \varphi$}
  & \leftrightharpoons
  & \parbox[t]{\templengtha}{$\parbox{\templengthc}{$\kframe{F}_D$}\imodels \varphi$, for every $\kframe{F}_D\in\Scls{C}$.}
  \\
\end{array}
$$
If $\mathfrak{S}\imodels\varphi$, for a structure $\mathfrak{S}$ (a~model, a~frame, etc.), we say that the formula $\varphi$ is \defnotion{true}, or \defnotion{valid}, in (on, at)~$\mathfrak{S}$; otherwise, $\varphi$ is \defnotion{refuted} in (on, at)~$\mathfrak{S}$.
These notions, and the corresponding notations, can be extended to sets of formulas in a natural way: for a set of formulas $X$, define $\mathfrak{S}\imodels X$ as $\mathfrak{S}\imodels\varphi$, for every $\varphi\in X$.

%\Rem{
%Observe that, if $\mathfrak{F} = \langle W,R\rangle$ is an intuitionistic Kripke frame where $W$ is a singleton and $\varphi$ is an \mbox{$\lang$-formula}, then $\mathfrak{F} \Vdash \vp$ if, and only if, $\vp \in \mathbf{QCl}$.

The intuitionistic predicate logic $\logic{QInt}$ is the set of $\lang{L}$-formulas valid on every intuitionistic Kripke frame; notice that $\logic{QInt}$ can also be defined through a Hilbert-style calculus with a finite set of axioms~\cite{GShS,vanDalen}. A \defnotion{superintuitionistic predicate logic} is a set of $\lang{L}$-formulas that includes $\logic{QInt}$ and is closed under Modus Ponens, Substitution, and Generalization. If $L$ is a superintuitionistic predicate logic and $\Gamma$ is a set of $\lang{L}$-formulas, then $L + \Gamma$ denotes the smallest superintuitionistic logic containing $L \cup \Gamma$. If $L$ is a propositional superintuitionistic logic, then define $\logic{Q}L$ by $\logic{Q}L = \logic{QInt} + L$. 
%For a superintuitionistic logic~$L$, let $L^+$ to denote the \defnotion{positive fragment} of~$L$, i.e., the subset of~$L$ consisting of formulas without occurrences of~$\bot$.

Let $\Scls{C}$ be a class of intuitionistic augmented frames. Define the \defnotion{superintuitionistic predicate logic\/ $\QSIL \Scls{C}$ of the class\/~$\Scls{C}$} by
$$
\begin{array}{lcl}
\QSIL \Scls{C} & = & \{\varphi\in\lang{L} : \Scls{C}\imodels\varphi\}.
\end{array} 
$$
%A superintuitionistic predicate logic coinciding with the set of all formulas valid on a class $\mathscr{C}$ of augmented frames is said to be \textit{determined by $\mathscr{C}$}. A logic determined by some class of augmented frames is said to be \textit{Kripke complete}.
For $\alpha\in\{\mathit{c},\mathit{e}\}$ and $\beta\in\{\mathit{all},\mathit{dfin},\mathit{wfin}\}$, define the logic $\QSILext{\alpha}{\beta}\scls{C}$~by
$$
\begin{array}{lcl}
\QSILext{\alpha}{\beta} \scls{C} & = & \QSIL \aug{\alpha}{\beta} \scls{C}.
\end{array} 
$$
For an intuitionistic Kripke frame $\kframe{F}$, we write $\QSILext{\alpha}{\beta} \kframe{F}$ rather than $\QSILext{\alpha}{\beta} \{\kframe{F}\}$.
It should be clear that
$$
\begin{array}{lclcl}
\QSILext{e}{\beta} \scls{C} \subseteq \QSILext{c}{\beta} \scls{C}, 
  && \QSILext{\alpha}{\mathit{all}} \scls{C} \subseteq \QSILext{\alpha}{\mathit{dfin}} \scls{C},
  && \QSILext{\alpha}{\mathit{all}} \scls{C} \subseteq \QSILext{\alpha}{\mathit{wfin}} \scls{C}. 
\end{array} 
$$

For a superintuitionistic predicate logic $L$, denote by $\ckf L$ the class of intuitionistic Kripke frames validating~$L$.
%; also, for a class $\scls{C}$ of intuitionistic Kripke frames, denote by $\fin\scls{C}$ the class of finite Kripke frames of~$\scls{C}$. 
Define 
$$
\begin{array}{lcl}
L_{\mathit{dfin}} & = & \QSILed \ckf L; \\
L_{\mathit{wfin}} & = & \QSILew \ckf L. \\
\end{array}
$$

Let $Q$ be a unary predicate letter and $q$ a proposition letter; the formula $\bm{cd} = \forall x\, (Q(x) \dis q) \imp \forall x\, Q(x) \dis q$ is valid on an intuitionistic augmented frame $\kframe{F}_D$ if, and only if, $\kframe{F}_D$ satisfies $(\mathit{LCD})$. 
%If $L$ is a superintuitionistic predicate logic, then $L\logic{.cd}$ denotes the logic $L + \bm{cd}$. 
Recall also that $\logic{QKC} = \logic{QInt} + \neg p\vee \neg\neg p$.

%For every $n\in\numNp$, let $\bm{bd}_n$ be the intuitionistic formula bounding the depth of intuitionistic Kripke frames by~$n$, see~\cite[Proposition~2.38]{ChZ}. Let us make an observation.

%\begin{proposition}
%\label{prop:QInt:QKC:bd-n}
%For every $n\in\numNp$, the following inclusions hold:
%$$
%\begin{array}{rcl}
%(\logic{QKC}+\bm{bd}_{n+1})^+
%  & \!\!\subseteq\!\! 
%  & (\logic{QInt}+\bm{bd}_{n})^+
%  \\
%(\logic{QKC}+\bm{bd}_{n+1}+\bm)^+
%  & \!\!\subseteq\!\! 
%  & (\logic{QInt}+\bm{bd}_{n})^+
%\end{array}
%$$
%\end{proposition}

\subsection{Positive monadic fragments and three variables}

Let us construct an embedding of the classical predicate logic and some classical theories into the positive fragments of $\logic{QInt}$ and some extensions of $\logic{QInt}$. In view of Theorem~\ref{th:Trakhtenbrot:binP:gen:sib:srb:pos}, to define the embedding, we shall consider the positive fragment of~$\lang{L}$ with a single binary predicate letter~$P$ and three individual variables $x$, $y$, and~$z$. Denote by $K$ the Kolmogorov translation~\cite{Kolmogorov:1925}\footnote{For our purposes, an another translation can be chosen; see~\cite{FerreiraOliva:2010}.} defined for formulas of the fragment as follows (we do not need the first clause for positive formulas):
$$
\begin{array}{lcll}
K(\bot)
  & = 
  & \bot; 
  \\
K(P(x_1,x_2)) 
  & = 
  & \neg\neg P(x_1,x_2), & \mbox{where $x_1,x_2\in\{x,y,z\}$}; 
  \\
K(\varphi'\wedge\varphi'')
  & = 
  & \neg\neg(K(\varphi')\wedge K(\varphi'')); 
  \\
K(\varphi'\vee\varphi'')
  & = 
  & \neg\neg(K(\varphi')\vee K(\varphi'')); 
  \\
K(\varphi'\to\varphi'')
  & = 
  & \neg\neg(K(\varphi')\to K(\varphi'')); 
  \\
K(\forall x\,\varphi')
  & = 
  & \neg\neg \forall x\,K(\varphi'); 
  \\
K(\exists x\,\varphi')
  & = 
  & \neg\neg \exists x\,K(\varphi'). 
  \\
\end{array}
$$
It is follows from~\cite{Kolmogorov:1925} that, for every formula~$\varphi$ of the fragment,
\begin{equation}
\label{eq:kolmogorov:1}
\begin{array}{rcl}
\varphi \in \logic{QCl} & \iff & K(\varphi) \in \logic{QInt}. %;
\end{array}
\end{equation}
%notice that also, for every~$\varphi$,
%\begin{equation}
%\label{eq:kolmogorov:2}
%\begin{array}{rcl}
%\varphi \lra K(\varphi) & \!\!\in\!\! & \logic{QCl}.
%\end{array}
%\end{equation}
Next, let us eliminate $\bot$ by replacing it with a proposition letter~$p$. To this end, for an $\lang{L}$\nobreakdash-formula $\psi$, define $\psi^p$ to be the formula obtained from $\psi$ by replacing every occurrence of $\bot$ with~$p$; for an $\lang{L}$\nobreakdash-formula~$\varphi$ with $x$, $y$, and $z$ as the only its variables, let
$$
\begin{array}{lcl}
%F_1 
%  & = 
%  & \displaystyle
%    \forall x\forall y\forall z\,\bigwedge\limits_{\mathclap{\psi\in\sub\varphi}}(\psi^p\vee(\psi^p\to p));
%    \smallskip\\
%F_2 
%  & = 
%  & \displaystyle
%    \forall x\forall y\forall z\,\bigwedge\limits_{\mathclap{\psi\in\sub\varphi}}(p\to\psi^p);    
%    \smallskip\\
\varphi^+_p 
  & = 
  & \displaystyle
    \forall x\forall y\forall z\,\bigwedge\limits_{\mathclap{\psi\in\sub\varphi}}(p\to\psi^p) 
    \to \varphi^p.    
\end{array}
$$
It follows from~\cite[Proposition~10.1]{RShsubmitted2} that if~$\varphi$ does not contain~$p$, then
\begin{equation}
\label{eq:positivization:1}
\begin{array}{rcl}
\varphi \in \logic{QInt} & \iff & \varphi^+_p \in \logic{QInt}.
\end{array}
\end{equation}

Let $S_6$ be a formula substitution defined by $S_6(P(x_1,x_2))=(Q(x_1)\wedge Q(x_2)\to p)\vee q$, where $Q$ is a unary predicate letter and $q$ a proposition letter different from~$p$.

\begin{lemma}
\label{lem:1:QInt:positive:3var}
If $n\in\mathbb{X}$, then $S_6((K((S_1 \textit{MTiling}^{\mathbb{X}}_n)^+))^+_p) \in \logic{QInt}$.
\end{lemma}

\begin{proof}
Observe that
\settowidth{\templength}{\mbox{$\logic{QInt}$}}
$$
\begin{array}{lclll}
n\in\mathbb{X}
  & \imply
  & S_1 \textit{MTiling}^{\mathbb{X}}_n \hfill\in \parbox{\templength}{$\logic{QCl}$}
  && \mbox{(by Lemma~\ref{lem:Trakhtenbrot:lem1:sib:binP})}
  \smallskip\\
  & \imply
  & (S_1 \textit{MTiling}^{\mathbb{X}}_n)^+ \hfill\in \parbox{\templength}{$\logic{QCl}$}
  && \mbox{(by Corollary~\ref{cor3:lem:false})}
  \smallskip\\
  & \imply
  & K((S_1 \textit{MTiling}^{\mathbb{X}}_n)^+) \hfill\in \logic{QInt}
  && \mbox{(by~(\ref{eq:kolmogorov:1}))}
  \smallskip\\
  & \imply
  & (K((S_1 \textit{MTiling}^{\mathbb{X}}_n)^+))^+_p \hfill\in \logic{QInt}
  && \mbox{(by~(\ref{eq:positivization:1}))}
  \smallskip\\
  & \imply
  & S_6((K((S_1 \textit{MTiling}^{\mathbb{X}}_n)^+))^+_p) \hfill\in \logic{QInt}
  && \mbox{(since $S_6$ is a substitution).}
  \smallskip\\
\end{array}
$$
\end{proof}

We say that a class $\scls{C}$ of intuitionistic Kripke frames satisfies the \defnotion{special weak Kripke--Hughes--Cresswell condition} (for short, \defnotion{swKHC}) if, for every $n\in\numN$, there exists a Kripke frame $\otuple{W,R}\in\scls{C}$ with $w \in W$ such that $R(w)$ contains an $R$-antichain with $n$~elements. Obviously, if a class of Kripke frames satisfies swKHC, then it also satisfies wKHC. However, if, for some $k\in\numNp$, a class $\scls{C}$ contains only Kripke frames of width\footnote{The greatest cardinality of antichains in the frame.} at most~$k$, then it can not satisfy swKHC; at the same time, even the class of all linear Kripke frames (i.e., of width~$1$) satisfies wKHC.  

\begin{lemma}
\label{lem:2:QInt:positive:3var}
Let $L$ be a superintuitionistic predicate logic such that\/ $\ckf L$ is an swKHC class. Then, $n\in\mathbb{Y}$ implies 
$S_6((K((S_1 \textit{MTiling}^{\mathbb{X}}_n)^+))^+_p) \not\in \QSILcd \ckf L$. 
\end{lemma}

\begin{proof}
Let $n\in\mathbb{Y}$. By Lemma~\ref{lem:Trakhtenbrot:lem2:sib:binP}, there exists a finite model $\cModel{M}=\otuple{\mathcal{D},\mathcal{I}}$ such that $\cModel{M}\models\bm{sib}$ and $\cModel{M}\not\models S_1 \textit{MTiling}^{\mathbb{X}}_n$. Then, by Corollary~\ref{cor1:lem:false}, $\cModel{M}\not\models (S_1 \textit{MTiling}^{\mathbb{X}}_n)^+$.
%, and by~(\ref{eq:kolmogorov:2}), $\cModel{M}\not\models K((S_1 \textit{MTiling}^{\mathbb{X}}_n)^+)$.

Since $\ckf L$ is an swKHC class, it contains an intuitionistic Kripke frame $\kframe{F}=\otuple{W,R}$ with a world $w_0$ such that $R(w_0)$ contains an $R$-antichain with $|\mathcal{D}|^2$ elements. Let 
$$
\begin{array}{lcl}
W' & = & \{w_{ab} : a,b \in \mathcal{D}\}
\end{array}
$$
be a subset of the antichain with $w_{ab}\ne w_{cd}$ whenever $\{a,b\}\ne\{c,d\}$. Let $\kModel{M} = \otuple{\kframe{F}\odot\mathcal{D},I}$ be an intuitionistic model such that
$$
\begin{array}{lcl}
\kModel{M},w\imodels p & \iff & R(w)\cap W' = \varnothing; \smallskip\\ 
\kModel{M},w\imodels q & \iff & w_0\not\in R(w); \smallskip\\ 
\kModel{M},w\imodels Q(a) 
  & \iff 
  & \mbox{either $\kModel{M},w\imodels p$} 
  \\ 
  &
  & \mbox{or, for some $b\in\mathcal{D}$, both $w\in\{w_{ab},w_{ba}\}$ and $\cModel{M}\not\models P(a,b)$.}
\end{array}
$$

\begin{sublemma}
\label{sublem:1:lem:2:QInt:positive:3var}
Let $\varphi$ be a positive formula that contains no predicate letters different from~$P$. Then $\kModel{M},w\imodels S_6((K(\varphi))^p)$, for every $w\in W$ such that $w_0\not\in R(w)$.
\end{sublemma}

\begin{proof}
By induction on~$\varphi$. Let $w$ be a world of $\kModel{F}$ such that $w_0\not\in R(w)$. If $\varphi=P(x,y)$, then $S_6((K(\varphi))^p) = (((Q(x)\wedge Q(y)\to p)\vee q)\to p)\to p$. Assume that $\kModel{M},w\not\imodels^g S_6((K(\varphi))^p)$, for some assignment~$g$. Then there exists $w'\in R(w)$ such that
$$
\begin{array}{lcl}
\kModel{M},w'\imodels^g ((Q(x)\wedge Q(y)\to p)\vee q)\to p 
  & \mbox{and} 
  & \kModel{M},w'\not\imodels^g p,
\end{array}
$$
which imply $\kModel{M},w'\not\imodels^g (Q(x)\wedge Q(y)\to p)\vee q$, that is impossible, since $\kModel{M},w'\imodels^g q$ by the definition of~$\kModel{M}$. Thus, $\kModel{M},w\imodels S_6((K(\varphi))^p)$. The induction step is straightforward and is left to the reader.
\end{proof}

\begin{sublemma}
\label{sublem:2:lem:2:QInt:positive:3var}
Let $\varphi$ be a positive formula that contains no predicate letters different from~$P$. Then, for every assignment~$g$,
$$
\begin{array}{lcl}
\kModel{M},w_0\imodels^g S_6((K(\varphi))^p) & \iff & \cModel{M}\models^g \varphi.
\end{array}
$$
\end{sublemma}

\begin{proof}
By induction on~$\varphi$. 

If $\varphi=P(x,y)$, then $S_6((K(\varphi))^p) = (((Q(x)\wedge Q(y)\to p)\vee q)\to p)\to p$. Observe that, by the definition of~$\kModel{M}$, for all $a,b\in\mathcal{D}$,
$$
\begin{array}{lcl}
\kModel{M},w_0\not\imodels (((Q(a)\wedge Q(b)\to p)\vee q)\to p) \to p
  & \iff
  & \cModel{M}\not\models P(a,b),
\end{array}
$$
that provides us with the equivalence required, for $\varphi=P(x,y)$.

The cases $\varphi=\varphi'\wedge\varphi''$, $\varphi=\varphi'\vee\varphi''$, $\varphi=\varphi'\to\varphi''$, $\varphi=\exists x\,\varphi'$, and $\varphi=\forall x\,\varphi'$ are similar to each other; let us consider some of them.

Let $\varphi=\varphi'\to\varphi''$. Then
$$
\begin{array}{lcl}
\cModel{M}\not\models^g \varphi
  & \imply
  & \mbox{$\cModel{M}\models^g \varphi'$ and $\cModel{M}\not\models^g \varphi''$} 
  \\
  & \imply
  & \mbox{$\kModel{M},w_0\imodels^g S_6((K(\varphi'))^p)$ and $\kModel{M},w_0\not\imodels^g S_6((K(\varphi''))^p)$} 
  \\
  & \imply
  & \kModel{M},w_0\not\imodels^g S_6((K(\varphi'))^p)\to S_6((K(\varphi''))^p) 
  \\
  & \imply
  & \kModel{M},w_0\not\imodels^g S_6((K(\varphi'))^p\to (K(\varphi''))^p) 
  \\
  & \imply
  & \kModel{M},w_0\not\imodels^g (S_6((K(\varphi'))^p\to (K(\varphi''))^p)\to p)\to p 
  \\
  & \imply
  & \kModel{M},w_0\not\imodels^g S_6((K(\varphi'\to\varphi''))^p), 
  \\
  & \mbox{i.e.,}
  & \kModel{M},w_0\not\imodels^g S_6((K(\varphi))^p). 
\end{array}
$$
Suppose that $\kModel{M},w_0\not\imodels^g S_6((K(\varphi))^p)$, i.e., $\kModel{M},w_0\not\imodels^g S_6((K(\varphi'\to\varphi''))^p)$. Then there exists $w\in R(w_0)$ such that 
$$
\begin{array}{lcl}
\kModel{M},w\imodels^g S_6((K(\varphi')\to(K(\varphi'')))^p)\to p 
  & \mbox{and}
  & \kModel{M},w\not\imodels^g p, 
\end{array}
$$
which imply $\kModel{M},w\not\imodels^g S_6((K(\varphi')\to(K(\varphi'')))^p)$. But then there exists $w'\in R(w)$ such that 
$$
\begin{array}{lcl}
\kModel{M},w'\imodels^g S_6((K(\varphi'))^p) 
  & \mbox{and}
  & \kModel{M},w'\not\imodels^g S_6((K(\varphi''))^p). 
\end{array}
$$
It then follows from Sublemma~\ref{sublem:1:lem:2:QInt:positive:3var}, that $w_0\in R(w')$. Since $R$ is antisymmetric, from $w_0Rw'$ and $w'Rw_0$ we obtain $w'=w_0$. Thus, by inductive hypothesis, 
$\cModel{M}\models^g \varphi'$ and $\cModel{M}\not\models^g \varphi''$, and hence, $\cModel{M}\not\models^g \varphi'\to\varphi''$, i.e. $\cModel{M}\not\models^g \varphi$. 

Let $\varphi=\forall x\,\varphi'$. Then
$$
\begin{array}{lcl}
\cModel{M}\not\models^g \varphi
  & \imply
  & \mbox{$\cModel{M}\not\models^h \varphi'$, for some $h$ such that $h\stackrel{x}{=}g$} 
  \\
  & \imply
  & \kModel{M},w_0\not\imodels^h S_6((K(\varphi'))^p) 
  \\
  & \imply
  & \kModel{M},w_0\not\imodels^h (S_6((K(\varphi'))^p)\to p)\to p 
  \\
  & \imply
  & \kModel{M},w_0\not\imodels^g S_6((K(\forall x\,\varphi'))^p), 
  \\
  & \mbox{i.e.,}
  & \kModel{M},w_0\not\imodels^g S_6((K(\varphi))^p). 
\end{array}
$$
Suppose that $\kModel{M},w_0\not\imodels^g S_6((K(\varphi))^p)$, i.e., $\kModel{M},w_0\not\imodels^g \kModel{M},w_0\not\imodels^h (S_6(\forall x\,(K(\varphi'))^p)\to p)\to p$. Then there exists $w\in R(w_0)$ such that 
$$
\begin{array}{lcl}
\kModel{M},w\imodels^g S_6(\forall x\,(K(\varphi'))^p)\to p 
  & \mbox{and} 
  & \kModel{M},w0\not\imodels^g p,
\end{array}
$$
which imply $\kModel{M},w\not\imodels^g S_6(\forall x\,(K(\varphi'))^p)$. But then, for some $w'\in R(w)$ and an assignment~$h$ such that $h\stackrel{x}{=}g$,
$$
\begin{array}{c}
\kModel{M},w'\not\imodels^h S_6((K(\varphi'))^p). 
\end{array}
$$
It then follows from Sublemma~\ref{sublem:1:lem:2:QInt:positive:3var}, that $w_0\in R(w')$, and therefore, $w'=w_0$. Thus, by inductive hypothesis, $\cModel{M}\models^h \varphi'$, and hence, $\cModel{M}\not\models^g \forall x\,\varphi'$, i.e. $\cModel{M}\not\models^g \varphi$. 

All other cases are similar (even simpler) and are left to the reader.
\end{proof}

\begin{sublemma}
\label{sublem:3:lem:2:QInt:positive:3var}
If $\psi\in \sub K((S_1 \textit{MTiling}^{\mathbb{X}}_n)^+)$, then $\kModel{M},w_0\imodels p\to S_6(\psi^p)$.
\end{sublemma}

\begin{proof}
It is sufficient to observe that $\kModel{M},w_0\imodels \forall x\, Q(x)$, for every $w\in W$ such that $\kModel{M},w_0\imodels p$.  
\end{proof}

Let us apply the observations obtained:
$$
\begin{array}{lcll}
\cModel{M}\not\models (S_1 \textit{MTiling}^{\mathbb{X}}_n)^+
 & \imply
 & \kModel{M},w_0\not\imodels S_6((K((S_1 \textit{MTiling}^{\mathbb{X}}_n)^+))^p)
 & \mbox{(by Sublemma~\ref{sublem:2:lem:2:QInt:positive:3var})}
 \\
 & \imply
 & \kModel{M},w_0\not\imodels S_6((K((S_1 \textit{MTiling}^{\mathbb{X}}_n)^+))^+_p)
 & \mbox{(by Sublemma~\ref{sublem:3:lem:2:QInt:positive:3var}).} 
\end{array}
$$
Hence, $S_6((K((S_1 \textit{MTiling}^{\mathbb{X}}_n)^+))^+_p) \not\in \QSILcd \ckf L$. 
\end{proof}

\begin{lemma}
\label{lem:3:QInt:positive:3var}
Let $L$ be a superintuitionistic predicate logic such that\/ $\fin\ckf L$ is an swKHC class. Then, $n\in\mathbb{Y}$ implies 
$S_6((K(S_1 \textit{MTiling}^{\mathbb{X}}_n))^+_p) \not\in \QSILcw \ckf L$. 
\end{lemma}

\begin{proof}
Similar to the proof of Lemma~\ref{lem:2:QInt:positive:3var} with the difference that the corresponding intuitionistic Kripke frame must be finite; such a frame exists in $\fin\ckf L$, since it is an swKHC class.
\end{proof}

Let us eliminate $p$ and~$q$. To this end, define a substitution~$S_7$ by
$$
\begin{array}{lcl}
S_7(p) & = & \forall x\,Q(x); \\
S_7(q) & = & \forall x\forall y\,((Q(x)\to Q(y))\vee (Q(y)\to Q(x))). \\
\end{array}
$$

\begin{lemma}
\label{lem:1:QInt:positive:3var:S7}
If $n\in\mathbb{X}$, then $S_7 S_6((K(S_1 \textit{MTiling}^{\mathbb{X}}_n))^+_p) \in \logic{QInt}$.
\end{lemma}

\begin{proof}
Follows from Lemma~\ref{lem:1:QInt:positive:3var}, since $S_7$ is a substitution.
\end{proof}

\begin{lemma}
\label{lem:2:QInt:positive:3var:S7}
Let $L$ be a superintuitionistic predicate logic such that\/ $\ckf L$ is an swKHC class. Then, $n\in\mathbb{Y}$ implies 
$S_7 S_6((K(S_1 \textit{MTiling}^{\mathbb{X}}_n))^+_p) \not\in \QSILcd \ckf L$. 
\end{lemma}

\begin{proof}
Follows from the proof of Lemma~\ref{lem:2:QInt:positive:3var}, since, for model $\kModel{M}$ constructed in the proof,
$$
\begin{array}{lcl}
\kModel{M}\imodels p\leftrightarrow S_7(p)
  & \mbox{and}
  & \kModel{M}\imodels q\leftrightarrow S_7(q).
\end{array}
$$
\end{proof}

\begin{lemma}
\label{lem:3:QInt:positive:3var:S7}
Let $L$ be a superintuitionistic predicate logic such that\/ $\fin\ckf L$ is an swKHC class. Then, $n\in\mathbb{Y}$ implies 
$S_7 S_6((K(S_1 \textit{MTiling}^{\mathbb{X}}_n))^+_p) \not\in \QSILcw \ckf L$. 
\end{lemma}

\begin{proof}
Similarly, follows from the proof of Lemma~\ref{lem:3:QInt:positive:3var} (with mentioned modifications of the proof of Lemma~\ref{lem:2:QInt:positive:3var}); we leave the details to the reader.
\end{proof}

For every $n\in\numNp$, let $\bm{bd}_n$ be the intuitionistic formula bounding the depth of intuitionistic Kripke frames by~$n$, see~\cite[Proposition~2.38]{ChZ}. Let us make an observation.

\begin{theorem}
\label{th:QInt:positive:3var}
The positive fragments of logics\/ 
$$
\begin{array}{lcl}
\logic{QInt} & \mbox{and\/} & \logic{QInt}_{\mathit{dfin}}\hfill +\bm{bd}_2+\bm{cd}, \\
\logic{QInt} & \mbox{and\/} & \logic{QInt}_{\mathit{wfin}}\hfill +\bm{bd}_2+\bm{cd}, \\
\logic{QInt} & \mbox{and\/} & \logic{QKC}_{\mathit{dfin}} \hfill +\bm{bd}_3+\bm{cd}, \\
\logic{QInt} & \mbox{and\/} & \logic{QKC}_{\mathit{wfin}} \hfill +\bm{bd}_3+\bm{cd}\phantom{,}
\end{array}
$$ 
are recursively inseparable in the language with a single unary predicate letter and three individual variables.
\end{theorem}

\begin{proof}
Follows from Lemmas~\ref{lem:1:QInt:positive:3var:S7}, \ref{lem:2:QInt:positive:3var:S7}, and~\ref{lem:3:QInt:positive:3var:S7}.
\end{proof}

\begin{corollary}
\label{cor:1:th:QInt:positive:3var}
Let\/ $\logic{QInt}\subseteq L\subseteq L'$ and also either $L'\subseteq \logic{QInt}+\bm{bd}_2+\bm{cd}$ or\/ $L'\subseteq \logic{QKC}+\bm{bd}_3+\bm{cd}$. Then the positive fragments of $L$ and $L'_{\mathit{wfin}}$ as well as $L$ and $L'_{\mathit{dfin}}$ are recursively inseparable in the language with a single unary predicate letter and three individual variables. 
\end{corollary}

\begin{corollary}
\label{cor:2:th:QInt:positive:3var}
Let $L$ be one of\/ $\logic{QInt}$, $\logic{QKP}$, $\logic{QKC}$, $\logic{QInt}+\bm{bd}_n$ with $n\geqslant 2$, $\logic{QKC}+\bm{bd}_n$ with $n\geqslant 3$. Then the positive fragments of $L$ and $L_{\mathit{wfin}}$ as well as $L$ and $L_{\mathit{dfin}}$ are recursively inseparable in the language with a single unary predicate letter and three individual variables. 
\end{corollary}

\subsection{Positive monadic fragments and two variables}

To obtain similar results for the language with two individual variables, we have to modify the formulas describing, for every $n\in \numN$, the special $T_n$-tiling.

We shall use the formulas and denotations introduced in Section~\ref{subsec:Trakhtenbrot:theories}. We define a formula similar to $\mathit{Tiling}'_n$ using just two individual variables. As in the modal case, let us use a new unary predicate letter~$C$. Define (cf.~\cite{KKZ05}) the following $\lang{L}$-formulas:
$$
\begin{array}{lcl}
\mathit{TC}^{\mathit{int}}_3
  & =
  & \forall x\forall y\,(V'_n(x,y) \wedge \exists x\,(C(x)\wedge H'_n(y,x)) \to 
%    \phantom{\forall x\,(C(x)\to V'_n(y,x)));}
%  \\
%  &
%  & %\phantom{\forall x\forall y\,(V'_n(x,y) \wedge \exists x\,(C(x))}
%    \hfill
    \forall y\,(H'_n(x,y) \to \forall x\,(C(x)\to V'_n(y,x))));
  \smallskip\\
\mathit{TC}^{\mathit{int}}_5
  & =
  & \forall x\forall y\,(V'_n(x,y)\vee \neg V'_n(x,y)).
\end{array}
$$
Let, for convenience, 
$$
\begin{array}{cccc}
\mathit{TC}^{\mathit{int}}_0 = \mathit{TC}_0,
  & \mathit{TC}^{\mathit{int}}_1 = \mathit{TC}'_1,
  & \mathit{TC}^{\mathit{int}}_2 = \mathit{TC}'_2,
  & \mathit{TC}^{\mathit{int}}_4 = \mathit{TC}_4.
\end{array}
$$
Then, define 
$$
\begin{array}{rcl}
\mathit{Tiling}^{\mathit{int}}_n 
  & = 
  & \displaystyle
    \bigwedge\limits_{\mathclap{i=0}}^{5}\mathit{TC}^{\mathit{int}}_i 
    \wedge \mathit{DSR}
    \wedge \mathit{DSU}.   
\end{array}
$$
Finally, we define a new, intuitionistic, modification of $\mathit{Tiling}_n^{\mathbb{X}}$ by
$$
\begin{array}{lcl}
\mathit{M}^{\mathit{int}}\mathit{Tiling}_n^{\mathbb{X}} 
  & = 
  & (\mathit{Tiling}^{\mathit{int}}_n \to \exists x\,P_1(x) \vee \exists x\,\neg C(x))^+,
  \smallskip\\
\end{array}
$$
where $\varphi^+$ is the formula obtained from $\varphi$ by replacing $\bot$ with $\mathit{false}=\forall x\forall y\,P(x,y)$, see Section~\ref{sec:3:subsec:positive}.

\begin{lemma}
\label{lem:MIntTiling:1}
If $n\in\mathbb{X}$, then $\mathit{M}^{\mathit{int}}\mathit{Tiling}_n^{\mathbb{X}} \in \logic{QInt}$.
\end{lemma}

\begin{proof}
Similar to the proof of Lemma~\ref{lem:1:tiling:QK}.

Let $n\in\mathbb{X}$. Suppose that $\mathit{M}^{\mathit{int}}\mathit{Tiling}_n^{\mathbb{X}} \not\in \logic{QInt}$. Then there exist an intuitionistic model $\kModel{M} = \otuple{W,R,D,I}$ and a world $w_0\in W$ such that 
\begin{equation}
\begin{array}{llcl}
\kModel{M},w_0 \imodels (\mathit{Tiling}^{\mathit{int}}_n)^+, 
  & \kModel{M},w_0 \not\imodels \exists x\,P_1(x),
  & \mbox{and}
  & \kModel{M},w_0 \not\imodels \exists x\,(C(x)\to\mathit{false}).
  \smallskip\\
\end{array}
\label{eq:1:lem:MIntTiling:1}
\end{equation}

Next, we are going, for all $i,j\in\numN$, to pick out an element $a_{i}^{j}\in D_{w_0}$ so that, for all $i,j\in\numN$,
$$
\begin{array}{lcl}
\kModel{M},w_0\imodels (H'_n(a_i^j,a_{i+1}^j))^+ 
  & \mbox{and} 
  & \kModel{M},w_0\imodels (V'_n(a_i^j,a_i^{j+1}))^+.
\end{array}
$$

Since $\kModel{M},w_0\imodels \mathit{TC}^{\mathit{int}}_4$, there exists $a_0^0\in D_{w_0}$ such that $\kModel{M},w_0\imodels P_0(a_0^0)$. 

Let $k\in\numN$. Suppose, for all $i,j\in \{0,\ldots,k\}$, the element $a_i^j$ is defined; we have to define, for all $i,j\in \{0,\ldots,k\}$, the elements $a_{k+1}^j$, $a_i^{k+1}$, and $a_{k+1}^{k+1}$.

Due to $(\mathit{TC}^{\mathit{int}}_2)^+$, there exists $b\in D_{w_0}$ such that $\kModel{M},w_0\imodels (V'_n(a_0^k,b))^+$. Then take $a_{0}^{k+1}=b$. To define other elements, let us prove an auxiliary statement, cf.~\cite{KKZ05}.

\begin{sublemma}
\label{sublem:kkz:int}
Let\/ $\kModel{M},w_0\imodels (H'_n(a,c))^+\wedge (V'_n(a,e))^+\wedge (H'_n(e,b))^+$, for some $a,c,e,b\in D_{w_0}$. Then\/ $\kModel{M},w_0\imodels (V'_n(c,b))^+$.
\end{sublemma}

\begin{proof}
Similar to the proof of Lemma~\ref{sublem:kkz:modal}.

Assume that $\kModel{M},w_0 \imodels (H'(a,c))^+ \wedge (V'(a,e))^+ \wedge (H'_n(e,b))^+$. 
%Since $\kModel{M},w_0 \imodels \mathit{TC}^{\mathit{int}}_1$, there exists $b\in D_{w_0}$ such that $\kModel{M},w_0 \imodels H'(e,b)$. 
Since $\kModel{M},w_0 \not\imodels \exists x\,(C(x)\to \mathit{false})$, there exists $w\in R(w_0)$ such that $\kModel{M},w_0 \imodels C(b)$ and $\kModel{M},w_0 \not\imodels \mathit{false}$. By the heredity condition, $\kModel{M},w \imodels (H'(a,c))^+ \wedge (V'(a,e))^+ \wedge (H'(e,b))^+$. Then we obtain by $(\mathit{TC}^{\mathit{int}}_3)^+$ that $\kModel{M},w \imodels (V'(c,b))^+)$. Finally, $\kModel{M},w_0 \imodels (V'(c,b))^+$ follows by $(\mathit{TC}^{\mathit{int}}_5)^+$.
\end{proof}

Suppose that $a_0^{k+1},\ldots,a_i^{k+1}$ are already defined, for some $i\in\{0,\ldots,k\}$; we have to define $a_{i+1}^{k+1}$. Due to $(\mathit{TC}^{\mathit{int}}_1)^+$, there exists $b\in D_{w_0}$ such that $\kModel{M},w_0\imodels (H'_n(a_i^{k+1},b))^+$. Since also $\kModel{M},w_0\imodels (H'_n(a_i^k,a_{i+1}^k))^+$ and $\kModel{M},w_0\imodels (V'_n(a_i^k,a_{i}^{k+1}))^+$, we obtain, by Sublemma~\ref{sublem:kkz:int}, that $\kModel{M},w_0\imodels (H'_n(a_{i+1}^{k+1},b))^+$, and we can take $a_{i+1}^{k+1}=b$.

To define $a_{k+1}^{k+1}$, observe that, by $(\mathit{TC}^{\mathit{int}}_1)^+$, there exists $b\in D_{w_0}$ such that $\kModel{M},w_0\imodels H'_n(a_k^{k+1},b)$; take $a_{k+1}^{k+1}=b$. 

Suppose that $a_{k+1}^{k+1},\ldots,a_{k+1}^{j+1}$ are already defined for some $j\in\{0,\ldots,k\}$; we have to define $a_{k+1}^{j}$. Due to $(\mathit{TC}^{\mathit{int}}_1)^+$, there exists $c\in D_{w_0}$ such that $\kModel{M},w_0\imodels (H'_n(a_k^{j},c))^+$. Then, 
$$
\kModel{M},w_0\imodels (H'_n(a_k^j,c))^+\wedge (V'_n(a_k^j,a_k^{j+1}))^+\wedge (H'_n(a_k^{j+1},a_{k+1}^{j+1}))^+,
$$
and, by Sublemma~\ref{sublem:kkz:int}, $\kModel{M},w_0\imodels (V'_n(c,a_{k+1}^{j+1}))^+$. Hence, we may take $a_{k+1}^{k+1}=c$. 

For all $i,j\in \numN$, due to $(\mathit{TC}^{\mathit{int}}_0)^+$, there exists a unique $m\in\{0,\ldots,k_n\}$ such that $\kModel{M},w_0\models P_m(a_i^j)$. Then, put $f(i,j)=t^n_m$.
Then $f\colon\numN\times\numN\to T_n$ is a $T_n$-tiling, since
$$
\begin{array}{lcl}
\kModel{M},w_0\imodels (H_n(a_i^j,a_{i+1}^j))^+ & \Longrightarrow & \rightsq f(i,j) = \leftsq f(i+1,j);
\smallskip\\
\kModel{M},w_0\imodels \hfill (V_n(a_i^j,a_i^{j+1}))^+ & \Longrightarrow & \upsq f(i,j) = \downsq f(i,j+1).
\end{array}
$$

Notice that $f(0,0)=t^n_0=t_0$, since $\kModel{M},w_0\models P_0(a_0^0)$. Then, by Proposition~\ref{prop:fn}, $f$ is the special $T_n$-tiling~$f_n$. Since $n\in\mathbb{X}$, by $(\ref{eq:fn})$ we obtain that there exists $m\in\numN$ such that $f_n(0,m)=t_1=t^n_1$. This means that $\kModel{M},w_0\imodels P_1(a_0^m)$, and hence, $\kModel{M},w_0\imodels \exists x\,P_1(x)$. But $\kModel{M},w_0\not\imodels \exists x\,P_1(x)$ by~(\ref{eq:1:lem:MIntTiling:1}), hence, we obtain a contradiction.

Thus, $\mathit{M}^{\mathit{int}}\mathit{Tiling}_n^{\mathbb{X}}\in\logic{QInt}$.
\end{proof}

\begin{lemma}
\label{lem:MIntTiling:2}
Let $L$ be a superintuitionistic predicate logic such that $\ckf L$ is an swKHC class. If $n\in\mathbb{Y}$, then $\mathit{M}^{\mathit{int}}\mathit{Tiling}_n^{\mathbb{X}} \not\in \QSILcd \ckf L$.
\end{lemma}

\begin{proof}
Let $n\in\mathbb{Y}$. By Lemma~\ref{lem:Trakhtenbrot:lem2:sib}, $\mathit{MTiling}_n^{\mathbb{X}}\not\in\logic{QCl}_{\mathit{fin}}\uplus\bm{sib}$. Let $\cModel{M} = \otuple{\mathcal{D},\mathcal{I}}$ be a finite model such that $\cModel{M}\not\models \mathit{MTiling}_n^{\mathbb{X}}$ and $\cModel{M}\models\bm{sib}$; we may assume that $\cModel{M}$ is the model constructed for $n$ in the proof of Lemma~\ref{lem:Trakhtenbrot:lem2:sib}.

Since $\ckf L$ is an swKHC class, there exists an intuitionistic Kripke frame $\kframe{F}=\otuple{W,R}$ in it such that, for some $w_0\in W$, the set $R(w_0)$ contains an antichain with at least $|\mathcal{D}|$ worlds. Let $W'=\{w_a : a\in\mathcal{D}\}$ be a subset of the antichain with $w_a\ne w_b$ whenever $a\ne b$. Let us consider a model $\kModel{M} = \otuple{\kframe{F}\odot\mathcal{D},I}$ such that, for every $w\in W$, every $m\in\{0,\ldots,k_n\}$, all $i,j\in\{1,2,3,4\}$, and all $a,b\in\mathcal{D}$,
\begin{equation}
\begin{array}{lcl}
\kModel{M},w \imodels P(a,b)
  & \iff
  & \mbox{$R(w)\cap W' = \varnothing$ or $\cModel{M}\models P(a,b)$;}
  \\
\kModel{M},w \imodels P_m(a)
  & \iff
  & \mbox{$R(w)\cap W' = \varnothing$ or $\cModel{M}\models P_m(a)$;}
  \\
\kModel{M},w \imodels R_i(a)
  & \iff
  & \mbox{$R(w)\cap W' = \varnothing$ or $\cModel{M}\models R_i(a)$;}
  \\
\kModel{M},w \imodels U_j(a)
  & \iff
  & \mbox{$R(w)\cap W' = \varnothing$ or $\cModel{M}\models U_j(a)$;}
  \\
\kModel{M},w \imodels C\phantom{{}_{i}}(a)
  & \iff
  & \mbox{$R(w)\cap W' = \varnothing$ or $w = w_a$.}
\end{array}
\label{eq:lem:MIntTiling:2}
\end{equation}
Then, it is not hard to show that $\kModel{M},w_0\not\imodels \mathit{M}^{\mathit{int}}\mathit{Tiling}_n^{\mathbb{X}}$; we leave the details to the reader. 

Hence, $\mathit{M}^{\mathit{int}}\mathit{Tiling}_n^{\mathbb{X}} \not\in \QSILcd \ckf L$.
\end{proof}

\begin{lemma}
\label{lem:MIntTiling:3}
Let $L$ be a superintuitionistic predicate logic such that $\fin \ckf L$ is an swKHC class. If $n\in\mathbb{Y}$, then $\mathit{M}^{\mathit{int}}\mathit{Tiling}_n^{\mathbb{X}} \not\in \QSILcw \fin \ckf L$.
\end{lemma}

\begin{proof}
Similar to the proof of Lemma~\ref{lem:MIntTiling:2} with the difference that the corresponding intuitionistic Kripke frame must be finite; such a frame exists in $\fin\ckf L$, since it is an swKHC class.
\end{proof}

Let us define a function~$S_8$, which is similar to the composition of $S_7$, $S_6$, and~$S'_3$. Define $S_8\mathit{M}^{\mathit{int}}\mathit{Tiling}_n^{\mathbb{X}}$ as the formula obtained from $\exists x\,G(x) \to (\mathit{M}^{\mathit{int}}\mathit{Tiling}_n^{\mathbb{X}})_G$ by replacing 
\begin{itemize}
\item
each occurrence of $P_m(x)$ and $P_m(y)$ with $(\mathit{tile}'_m(x))^+$ and $(\mathit{tile}'_m(y))^+$, respectively, where $m\in\{0,\ldots,k_n\}$;
\item 
then each occurrence of $R_i(x)$ and $R_i(y)$ with $(\mathit{tile}'_{k_n+i}(x))^+$ and $(\mathit{tile}'_{k_n+i}(y))^+$, respectively, where $1\leqslant i\leqslant 4$;
\item
then each occurrence of $U_j(x)$ and $U_j(y)$ with $(\mathit{tile}'_{k_n+j+4}(x))^+$ and $(\mathit{tile}'_{k_n+j+4}(y))^+$, respectively, where $1\leqslant j\leqslant 4$; 
\item
then each occurrence of $C(x)$ and $C(y)$ with $(\mathit{tile}'_{k_n+9}(x))^+$ and $(\mathit{tile}'_{k_n+9}(y))^+$, respectively;
\item
then each occurrence of $P(x_1,x_2)$ with $(Q(x_1)\wedge Q(x_2)\to \forall x\,Q(x))\vee \forall x\,G(x)$, where $x_1,x_2\in \{x,y\}$.
%;
%\item
%then each occurrence of $q$ with $\forall x\,G(x)$;
%\item
%and then each occurrence of $\bot$ with $\forall x\,(Q(x)\wedge G(x))$.
\end{itemize}
Notice that $S_8 \mathit{M}^{\mathit{int}}\mathit{Tiling}_n^{\mathbb{X}}$ is a positive formula that contains just two individual variables ($x$~and~$y$) and just two predicate letters ($Q$~and~$G$) which are unary.

\begin{lemma}
\label{lem:MIntTiling:1:QG:positive}
If $n\in\mathbb{X}$, then $S_8 \mathit{M}^{\mathit{int}}\mathit{Tiling}_n^{\mathbb{X}} \in \logic{QInt}$.
\end{lemma}

\begin{proof}
It should be clear; we give just a sketch of a proof.

Let $n\in\mathbb{X}$. Suppose that $\exists x\,G(x) \to (\mathit{M}^{\mathit{int}}\mathit{Tiling}_n^{\mathbb{X}})_G\not\in \logic{QInt}$. Then there exist an intuitionistic model~$\kModel{M}=\otuple{W,R,D,I}$ and a world $w\in W$ such that $\kModel{M},w\imodels \exists x\,G(x)$ and $\kModel{M},w\not\imodels (\mathit{M}^{\mathit{int}}\mathit{Tiling}_n^{\mathbb{X}})_G$. Let $\kModel{M}'=\otuple{W',R',D',I'}$ be the model defined by
$$
\begin{array}{rcll}
W' & = & R(w); \\
R' & = & R\upharpoonright W' \\
D'_v & = & \{a \in D_v : \kModel{M},v \imodels G(a)\}, & \mbox{where $v\in W'$}; \\
I'(v,E) & = & I(v,E) \upharpoonright D'_v, & \mbox{where $v\in W'$ and $E$ is a predicate letter}. \\
\end{array}
$$
Then it should be not hard to show that $\kModel{M},w\not\imodels\mathit{M}^{\mathit{int}}\mathit{Tiling}_n^{\mathbb{X}}$, that contradictis to Lemma~\ref{lem:MIntTiling:1}. 

Finally, observe that all 
%bar the last 
clauses in the definition of $S_8 \mathit{M}^{\mathit{int}}\mathit{Tiling}_n^{\mathbb{X}}$ define, in fact, formula substitutions; therefore, the resulting formula is in $\logic{QInt}$. 
%The last clause provides us with a formula in $\logic{QInt}$, too, since otherwise the formula is refuted in some intuitionistic model, and we may take its submodel on the set of worlds at which the formula $\forall x\,(Q(x)\wedge G(x))$ is not true, i.e., equivalent to~$\bot$ in the submodel; we leave the details to the reader.
\end{proof}

We say that a class $\scls{C}$ of intuitionistic Kripke frames satisfies the \defnotion{special weak Kontchakov--Kurucz--Zakharyaschev condition} (for short, \defnotion{swKKZ}) if, for every $n\in\numN$, there exist a Kripke frame $\otuple{W,R}\in\scls{C}$, a world $w \in W$ such that $R(w)$ contains an $R$\nobreakdash-antichain with $n$ elements and, for every $v$ from the antichain, $R(v)$ contains an $R$\nobreakdash-antichain with $n$ elements and satisfies the following property: for all different $v'$ and $v''$ from the antichain in $R(w)$, there is no elements $u'$ and $u''$ in the antichains contained in $R(v')$ and $R(v'')$, respectively, such that $u'Ru''$ or $u''Ru'$.
Clearly, if a class $\scls{C}$ of intuitionistic Kripke frames satisfies swKKZ, then it also satisfies both swKHC and wKKZ. Nevertheless, let us consider the class of all intuitionistic Kripke frames of the following form: a root sees an antichain consisting of worlds seeing the same chain of worlds (variation: the same antichain of worlds); it satisfies both swKHC and wKKZ but does not satisfy swKKZ. 

\begin{lemma}
\label{lem:MIntTiling:2:QG:positive}
Let $L$ be a superintuitionistic predicate logic such that $\ckf L$ is an swKKZ class. If $n\in\mathbb{Y}$, then $S_8 \mathit{M}^{\mathit{int}}\mathit{Tiling}_n^{\mathbb{X}} \not\in \QSILcd \ckf L$.
\end{lemma}

\begin{proof}
We are going to apply an argumentation similar to that used in the proofs of Lemmas~\ref{lem:relativization:p:wKHC}, \ref{lem:2:QInt:positive:3var}, and~\ref{lem:2:QInt:positive:3var:S7}.

Let $n\in\mathbb{Y}$.
Then, by Lemma~\ref{lem:Trakhtenbrot:lem2:sib}, $\mathit{MTiling}_n^{\mathbb{X}}\not\in\logic{QCl}_{\mathit{fin}}\uplus\bm{sib}$; let $\cModel{M}=\otuple{\mathcal{D},\mathcal{I}}$ be a classical model such that $\cModel{M}\models\bm{sib}$ and $\cModel{M}\not\models\mathit{MTiling}_n^{\mathbb{X}}$; we may assume that $\cModel{M}$ is the model defined in the proof of Lemma~\ref{lem:Trakhtenbrot:lem2:sib}, in particular,
$$
\begin{array}{lcl}
\mathcal{D} & = & \{0,\ldots,r+4\}\times\{0,\ldots,r+4\}, \\
\end{array}
$$
for a suitable~$r\in\numN$. As in the proof of Lemma~\ref{lem:relativization:p:wKHC}, let
$$
\begin{array}{lcl}
\mathcal{D}' 
  & = 
  & \mathcal{D} \cup \{e_0^a,\ldots,e_{k_n+11}^a, e_P^a, e_R^a, e_U^a, e_C^a : a\in\mathcal{D}\}.
\end{array}
$$

Since $\ckf L$ is an swKKZ class, it contains a Kripke frame $\bar{\kframe{F}}=\otuple{\bar{W},\bar{R}}$ with a world $\bar{w}_0 \in \bar{W}$ such that $\bar{R}(\bar{w}_0)$ contains an $\bar{R}$\nobreakdash-antichain with $|\mathcal{D}|^2$ elements and, for every $v$ from the antichain, $\bar{R}(v)$ contains an $\bar{R}$\nobreakdash-antichain with $|\mathcal{D}|^2$ elements, withal there are no elements $u'$ and $u''$ from different the antichains satisfying $u'Ru''$ or $u''Ru'$. Let $\bar{W}'=\{\bar{w}_{a} : a\in\mathcal{D}\}$ be a subset of the antichain contained in $\bar{R}(\bar{w}_0)$ and, for every $a\in\mathcal{D}$, let $U_a = \{u^a_{bc} : b,c\in\mathcal{D}'\}$ be a subset of the antichain contained in $\bar{R}(\bar{w}_a)$; we assume that $\bar{w}_a\ne \bar{w}_b$ whenever $a\ne b$ and also $u^a_{bc}\ne u^a_{de}$ whenever $\{b,c\}\ne\{d,e\}$. Let also
$$
\begin{array}{lcl}
U & = & \displaystyle\bigcup\limits_{\mathclap{a\in\mathcal{D}}} U_a. \\
\end{array}
$$

Let $\bar{\kModel{M}} = \otuple{\bar{\kframe{F}}\odot\mathcal{D}',\bar{I}}$ be a model defined so that 
%first, 
\begin{itemize}
\item
for every $w\in \bar{W}$ and every $a\in\mathcal{D}'$,
\begin{equation}
\begin{array}{lcl}
\bar{\kmodel{M}},w\imodels G(a) & \iff & \mbox{$\bar{R}(w)\cap \bar{W}' = \varnothing$ or $a\in \mathcal{D}$;} \\
%\bar{\kmodel{M}},w\imodels q & \iff & \bar{\kmodel{M}},w\models \forall x\,G(x); \\
%\bar{\kmodel{M}},w\imodels C(a) & \iff & \mbox{either $w=\bar{w}_a$ or $\bar{R}(w)\cap \bar{W}'=\varnothing$;} \\
\end{array}
\label{eq:0:lem:MIntTiling:2:QG:positive}
\end{equation}
\item
%second, 
for every $w\in \bar{W}\setminus \bar{W}'$ such that $\bar{R}(w)\cap \bar{W}'\ne\varnothing$, the relation $\bar{I}(w,P)$ is the symmetric closure of the relation
$$
\begin{array}{l}
\mathcal{I}(P) \cup 
  \{\otuple{a,e_0^a},\otuple{e_0^a,e_1^a},\ldots,\otuple{e_{k_n+10}^a,e_{k_n+11}^a} : a\in\mathcal{D}\}
  \\
  \phantom{\mathcal{I}(P)} \cup 
  \{\otuple{e_m^a,e_P^a} : \mbox{$m\in\{0,\ldots,k_n\}$ and $\cModel{M}\models P_m(a)$} \}
  \\
  \phantom{\mathcal{I}(P)} \cup
  \{\otuple{e_{k_n+i}^a,e_R^a} : \mbox{$i\in\{1,2,3,4\}$ and $\cModel{M}\models R_i(a)$} \}
  \\
  \phantom{\mathcal{I}(P)} \cup 
  \{\otuple{e_{k_n+j+4}^a,e_U^a} : \mbox{$j\in\{1,2,3,4\}$ and $\cModel{M}\models U_j(a)$} \};
\end{array}
$$
\item
%third, 
for every $a\in\mathcal{D}$,
$$
\begin{array}{lcl}
\bar{I}(\bar{w}_a,P) & = & I(w_0,P) \cup \{\otuple{e_{k_n+9}^a,e_C^a},\otuple{e_C^a,e_{k_n+9}^a}\};
\end{array}
$$
\item
%forth, 
for every $w\in \bar{W}(w)$ such that $\bar{R}(w)\cap \bar{W}' = \varnothing$,
$$
\begin{array}{lcl}
\bar{I}(w,P) & = & \mathcal{D}'\times\mathcal{D}';
\end{array}
$$
\item
for every $a\in \mathcal{D}$ and all $b,c,e\in\mathcal{D}'$,
$$
\begin{array}{lcl}
\bar{\kmodel{M}},u^a_{bc}\imodels Q(e) & \iff & \mbox{$e\in\{b,c\}$ and $\bar{\kmodel{M}},\bar{w}_a\not\imodels P(b,c)$;} \\
\end{array}
$$
\item
for every $w\in \bar{W}\setminus U$ and every $e\in\mathcal{D}'$,
$$
\begin{array}{lcl}
\bar{\kmodel{M}},w\imodels Q(e) & \iff & \mbox{$\bar{R}(w)\cap U = \varnothing$;} \\
\end{array}
$$
\end{itemize}
see Figure~\ref{fig:17} for a part of $\bar{I}(\bar{w}_a,P)$, where $\neg C(b)$ and $\neg C(c)$ now mean that $\kmodel{M},w_a\not\imodels C(b)$ and $\kmodel{M},w_a\not\imodels C(c)$ for $\kmodel{M}$ defined in the proof of Lemma~\ref{lem:MIntTiling:2}.
Notice that, in fact, we are interested in the definitions for $\bar{I}(w,Q)$ and $\bar{I}(w,G)$ only; the definitions for $\bar{I}(w,P)$, where $w\in\bar{W}$, 
%and $\bar{I}(w,C)$
%, and $\bar{I}(w,q)$ 
are auxiliary.
Nevertheless, let us observe that, by the definition of model~$\bar{\kModel{M}}$, for every $w\in\bar{W}$ and all $c,e\in\mathcal{D}'$,
\begin{equation}
\begin{array}{lcl}
\bar{\kmodel{M}},w\imodels P(c,e) 
  & \iff 
  & \bar{\kmodel{M}},w\imodels (Q(c)\wedge Q(e)\to \forall x\,Q(x))\vee \forall x\,G(x). 
  \\
\end{array}
\label{eq:1:lem:MIntTiling:2:QG:positive}
\end{equation}

Let us compare $\bar{\kModel{M}}$ and $\kModel{M}$, where $\kmodel{M}$ is the model defined in the proof of Lemma~\ref{lem:MIntTiling:2}. Observe that, for every $w\in \bar{W}$, every $m\in\{0,\ldots,k_n\}$, all $i,j\in\{1,2,3,4\}$, and all $a,b\in\mathcal{D}$,
\begin{equation}
\begin{array}{lcl}
\bar{\kModel{M}},w \imodels P(a,b)
  & \iff
  & \mbox{$\bar{R}(w)\cap \bar{W}' = \varnothing$ or $\cModel{M}\models P(a,b)$;}
  \\
\bar{\kModel{M}},w \imodels (\mathit{tile}'_m(a))^+
  & \iff
  & \mbox{$\bar{R}(w)\cap \bar{W}' = \varnothing$ or $\cModel{M}\models P_m(a)$;}
  \\
\bar{\kModel{M}},w \imodels (\mathit{tile}'_{k_n+i}(a))^+
  & \iff
  & \mbox{$\bar{R}(w)\cap \bar{W}' = \varnothing$ or $\cModel{M}\models R_i(a)$;}
  \\
\bar{\kModel{M}},w \imodels (\mathit{tile}'_{k_n+j+4}(a))^+
  & \iff
  & \mbox{$\bar{R}(w)\cap \bar{W}' = \varnothing$ or $\cModel{M}\models U_j(a)$;}
  \\
\bar{\kModel{M}},w \imodels (\mathit{tile}'_{k_n+9}(a))^+
  & \iff
  & \mbox{$\bar{R}(w)\cap \bar{W}' = \varnothing$ or $w = \bar{w}_a$,}
\end{array}
%\label{eq:lem:MIntTiling:2}
\label{eq:2:lem:MIntTiling:2:QG:positive}
\end{equation}
where $\cModel{M}$ is the model defined in the proof of Lemma~\ref{lem:Trakhtenbrot:lem2:sib}.

Then, by (\ref{eq:lem:MIntTiling:2})--(\ref{eq:2:lem:MIntTiling:2:QG:positive}), $\kModel{M},w_0\not\imodels \mathit{M}^{\mathit{int}}\mathit{Tiling}_n^{\mathbb{X}}$ implies that 
$\bar{\kModel{M}},\bar{w}_0\not\imodels S_8 \mathit{M}^{\mathit{int}}\mathit{Tiling}_n^{\mathbb{X}}$.

Thus, $S_8 \mathit{M}^{\mathit{int}}\mathit{Tiling}_n^{\mathbb{X}} \not\in \QSILcd \ckf L$. 
\end{proof}

\begin{lemma}
\label{lem:MIntTiling:3:QG:positive}
Let $L$ be a superintuitionistic predicate logic such that $\fin \ckf L$ is an swKKZ class. If $n\in\mathbb{Y}$, then $S_8 \mathit{M}^{\mathit{int}}\mathit{Tiling}_n^{\mathbb{X}} \not\in \QSILcw \fin \ckf L$.
\end{lemma}

\begin{proof}
Similar to the proof of Lemma~\ref{lem:MIntTiling:2:QG:positive} with the difference that the corresponding intuitionistic Kripke frame must be finite; such a frame exists in $\fin\ckf L$, since it is an swKKZ class; also, use the model mentioned in the proof of Lemma~\ref{lem:MIntTiling:3} instead of the model from the proof of Lemma~\ref{lem:MIntTiling:2}.
\end{proof}

\begin{theorem}
\label{th:QInt:positive:2var}
The positive fragments of logics\/ 
$$
\begin{array}{lcl}
\logic{QInt} & \mbox{and\/} & \logic{QInt}_{\mathit{dfin}}\hfill +\bm{bd}_3+\bm{cd}, \\ 
\logic{QInt} & \mbox{and\/} & \logic{QInt}_{\mathit{wfin}}\hfill +\bm{bd}_3+\bm{cd}, \\
\logic{QInt} & \mbox{and\/} & \logic{QKC}_{\mathit{dfin}} \hfill +\bm{bd}_4+\bm{cd}, \\
\logic{QInt} & \mbox{and\/} & \logic{QKC}_{\mathit{wfin}} \hfill +\bm{bd}_4+\bm{cd}\phantom{,}
\end{array}
$$ 
are recursively inseparable in the language with two unary predicate letters and two individual variables.
\end{theorem}

\begin{proof}
Follows from Lemmas~\ref{lem:MIntTiling:1:QG:positive}, \ref{lem:MIntTiling:2:QG:positive}, and~\ref{lem:MIntTiling:3:QG:positive}.
\end{proof}

\begin{corollary}
\label{cor:1:th:QInt:positive:2var}
Let\/ $\logic{QInt}\subseteq L\subseteq L'$ and also either $L'\subseteq \logic{QInt}+\bm{bd}_3+\bm{cd}$ or\/ $L'\subseteq \logic{QKC}+\bm{bd}_4+\bm{cd}$. Then the positive fragments of $L$ and $L'_{\mathit{wfin}}$ as well as $L$ and $L'_{\mathit{dfin}}$ are recursively inseparable in the language with two unary predicate letters and two individual variables. 
\end{corollary}

\begin{corollary}
\label{cor:2:th:QInt:positive:2var}
Let $L$ be one of\/ $\logic{QInt}$, $\logic{QKP}$, $\logic{QKC}$, $\logic{QInt}+\bm{bd}_n$ with $n\geqslant 3$, $\logic{QKC}+\bm{bd}_n$ with $n\geqslant 4$. Then the positive fragments of $L$ and $L_{\mathit{wfin}}$ as well as $L$ and $L_{\mathit{dfin}}$ are recursively inseparable in the language with two unary predicate letters and two individual variables. 
\end{corollary}

Finally, let us eliminate predicate letters $Q$ and $G$ simulating them by formulas with a single unary predicate letter~$P'$. To this end, let us consider the following formulas (cf.~\cite{Rybakov06,Rybakov08,RShJLC21b,RSh:2023:SCAN}):
$$
\begin{array}{lcl}
  D_1      & = & \exists x\, P'(x); \smallskip\\
  D_2(x)   & = & \exists x\, P'(x) \imp P'(x); \smallskip\\
  D_3(x)   & = & P'(x) \imp \forall x\, P'(x); 
  \bigskip\\
  A_1^0(x) & = & D_2(x) \imp D_1 \dis D_3(x); \smallskip\\
  A_2^0(x) & = & D_3(x) \imp D_1 \dis D_2(x); \smallskip\\
  B_1^0(x) & = & D_1 \imp D_2(x) \dis D_3(x); \smallskip\\
  B_2^0(x) & = & A_1^0(x) \con A_2^0(x) \con B_1^0(x) \imp  D_1 \dis D_2(x) \dis D_3(x); 
  \bigskip \\
  A_1^1(x) & = & A_1^0(x) \con A_2^0(x) \imp B_1^0(x) \dis B_2^0(x); \smallskip\\
  A_2^1(x) & = & A_1^0(x) \con B_1^0(x) \imp A_2^0(x) \dis B_2^0(x); \smallskip\\
  A_3^1(x) & = & A_1^0(x) \con B_2^0(x) \imp A_2^0(x) \dis B_1^0(x); \smallskip\\
  B_1^1(x) & = & A_2^0(x) \con B_1^0(x) \imp A_1^0(x) \dis B_2^0(x); \smallskip\\
  B_2^1(x) & = & A_2^0(x) \con B_2^0(x) \imp A_1^0(x) \dis B_1^0(x); \smallskip\\
  B_3^1(x) & = & B_1^0(x) \con B_2^0(x) \imp A_1^0(x) \dis A_2^0(x); 
  \bigskip\\
  A_1^2(x) & = & A_1^1(x) \imp B_1^1(x) \dis A_2^1(x) \dis B_2^1(x); \smallskip\\
  A_2^2(x) & = & A_1^1(x) \imp B_1^1(x) \dis A_2^1(x) \dis B_3^1(x); \smallskip\\
%  A_3^2(x) & = & A_1^1(x) \imp B_1^1(x) \dis A_3^1(x) \dis B_2^1(x); \smallskip\\
  B_1^2(x) & = & B_1^1(x) \imp A_1^1(x) \dis B_2^1(x) \dis A_2^1(x); \smallskip\\
  B_2^2(x) & = & B_1^1(x) \imp A_1^1(x) \dis B_2^1(x) \dis A_3^1(x). \smallskip\\
%  B_3^2(x) & = & B_1^1(x) \imp A_1^1(x) \dis A_3^1(x) \dis B_2^1(x). \smallskip\\
\end{array}
$$ 
Next, consider the finite intuitionistic frame $\kframe{F}_0 = \otuple{W_0, R_0}$, where
$$
\begin{array}{lcl}
W_0 
  & = 
  &  \{\delta_1,\delta_2,\delta'_2,\delta_3,
       \alpha^0_1,\alpha^0_2,\beta^0_1,\beta^0_2\,
       \alpha^1_1,\alpha^1_2,\alpha^1_3,\beta^1_1,\beta^1_2,\beta^1_3,
       \alpha^2_1,\alpha^2_2,%\alpha^2_3,
       \beta^2_1,\beta^2_2%,\beta^2_3
       \}
\end{array}
$$
and $R_0$ is reflexive transitive closure of the relation
\settowidth{\templength}{\mbox{$\langle \alpha^0_0, \alpha^0_0 \rangle$}}
$$
\begin{array}{l}
   \{ 
   \parbox{\templength}{\hfill$\langle \alpha^0_1, \gamma_1  \rangle$}, 
   \parbox{\templength}{\hfill$\langle \alpha^0_1, \gamma_3  \rangle$}, 
   \parbox{\templength}{\hfill$\langle \alpha^0_2, \gamma_1  \rangle$},
   \parbox{\templength}{\hfill$\langle \alpha^0_2, \gamma_2  \rangle$}, 
   \parbox{\templength}{\hfill$\langle \alpha^0_2, \gamma'_2 \rangle$}, 
   \parbox{\templength}{\hfill$\langle \beta^0_1,  \gamma_2  \rangle$}, 
   \parbox{\templength}{\hfill$\langle \beta^0_1,  \gamma'_2 \rangle$}, 
   \parbox{\templength}{\hfill$\langle \beta^0_1,  \gamma_3   \rangle$}, 
   \parbox{\templength}{\hfill$\langle \beta^0_2,  \gamma_1   \rangle$}, 
   \smallskip\\
   \phantom{\{}
   \parbox{\templength}{\hfill$\langle \beta^0_2,  \gamma_2   \rangle$}, 
   \parbox{\templength}{\hfill$\langle \beta^0_2,  \gamma'_2  \rangle$}, 
   \parbox{\templength}{\hfill$\langle \beta^0_2,  \gamma_3   \rangle$}, 
   \parbox{\templength}{\hfill$\langle \beta^1_1,  \alpha^0_1 \rangle$}, 
   \parbox{\templength}{\hfill$\langle \beta^1_1,  \beta^0_2  \rangle$}, 
   \parbox{\templength}{\hfill$\langle \alpha^1_1, \beta^0_1  \rangle$}, 
   \parbox{\templength}{\hfill$\langle \alpha^1_1, \beta^0_2  \rangle$}, 
   \parbox{\templength}{\hfill$\langle \alpha^1_2, \alpha^0_2 \rangle$},
   \parbox{\templength}{\hfill$\langle \alpha^1_2, \alpha^0_2 \rangle$}, 
   \smallskip\\
   \phantom{\{}
   \parbox{\templength}{\hfill$\langle \alpha^1_3, \alpha^0_2 \rangle$}, 
   \parbox{\templength}{\hfill$\langle \alpha^1_3, \beta^0_1  \rangle$}, 
   \parbox{\templength}{\hfill$\langle \beta^1_2, \alpha^0_1  \rangle$},
   \parbox{\templength}{\hfill$\langle \beta^1_2, \alpha^0_1 \rangle$},
   \parbox{\templength}{\hfill$\langle \beta^1_3, \alpha^0_1 \rangle$}, 
   \parbox{\templength}{\hfill$\langle \beta^1_3, \alpha^0_2 \rangle$},
   \parbox{\templength}{\hfill$\langle \alpha^2_1, \beta^1_1  \rangle$},
   \parbox{\templength}{\hfill$\langle \alpha^2_1, \alpha^1_2 \rangle$},
   \parbox{\templength}{\hfill$\langle \alpha^2_1, \beta^1_2  \rangle$},
   \smallskip\\
   \phantom{\{}
   \parbox{\templength}{\hfill$\langle \alpha^2_2, \beta^1_1  \rangle$},
   \parbox{\templength}{\hfill$\langle \alpha^2_2, \alpha^1_2 \rangle$},
   \parbox{\templength}{\hfill$\langle \alpha^2_2, \beta^1_3  \rangle$},
%   \parbox{\templength}{\hfill$\langle \alpha^2_3, \beta^1_1  \rangle$},
%   \parbox{\templength}{\hfill$\langle \alpha^2_3, \alpha^1_3 \rangle$},
%   \parbox{\templength}{\hfill$\langle \alpha^2_3, \beta^1_2  \rangle$},
   \parbox{\templength}{\hfill$\langle \beta^2_1, \alpha^1_1  \rangle$},
   \parbox{\templength}{\hfill$\langle \beta^2_1, \beta^1_2  \rangle$},
   \parbox{\templength}{\hfill$\langle \beta^2_1, \alpha^1_2   \rangle$},
   \parbox{\templength}{\hfill$\langle \beta^2_2, \alpha^1_1  \rangle$},
   \parbox{\templength}{\hfill$\langle \beta^2_2, \beta^1_2  \rangle$},
   \parbox{\templength}{\hfill$\langle \beta^2_2, \alpha^1_3   \rangle$}
%   \parbox{\templength}{\hfill$\langle \beta^2_3, \alpha^1_1  \rangle$},
%   \parbox{\templength}{\hfill$\langle \beta^2_3, \alpha^1_3  \rangle$},
%   \parbox{\templength}{\hfill$\langle \beta^2_3, \beta^1_2   \rangle$}
   \}.
\end{array}
$$
The frame is depicted in Figure~\ref{fig:F-2}; the worlds are presented by white circles, the accessibility relation between the worlds is presented by arrows (we omit the arrows inferrable by reflexivity and transitivity).
%
%(the arrows that can be inferred by reflexivity and transitivity are omitted). 

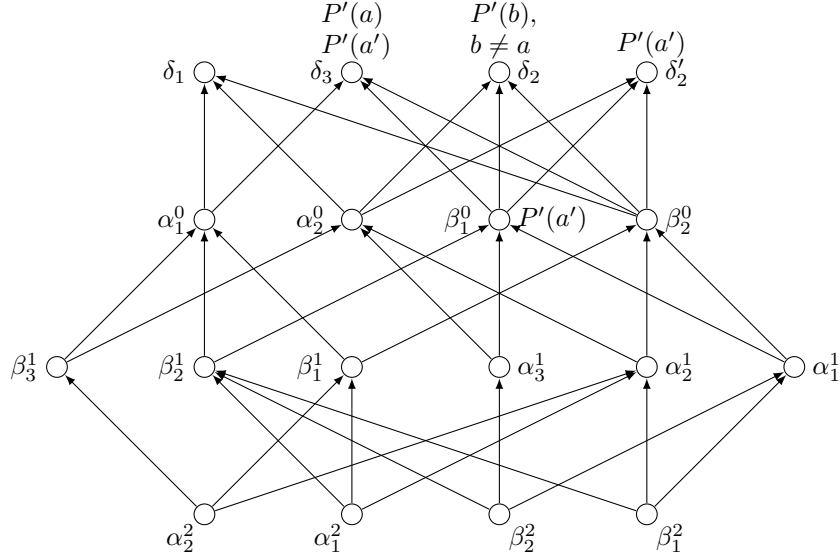
\begin{figure}
  \centering
  \begin{tikzpicture}[scale=1.96]

\coordinate (d1)  at (2, 7);
\coordinate (d3)  at (3, 7);
\coordinate (d2)  at (4, 7);
\coordinate (d2a) at (5, 7);
\coordinate (a01) at (2, 6);
\coordinate (a02) at (3, 6);
\coordinate (b01) at (4, 6);
\coordinate (b02) at (5, 6);
\coordinate (b13) at (1, 5);
\coordinate (b12) at (2, 5);
\coordinate (b11) at (3, 5);
\coordinate (a13) at (4, 5);
\coordinate (a12) at (5, 5);
\coordinate (a11) at (6, 5);
%\coordinate (a23) at (1, 4);
\coordinate (a22) at (2, 4);
\coordinate (a21) at (3, 4);
%\coordinate (b23) at (4, 4);
\coordinate (b22) at (4, 4); %(5, 4);
\coordinate (b21) at (5, 4); %(6, 4);

%\coordinate (dots1) at (3.5, 4.25);
%
%\coordinate (ak1) at (1, 3.5);
%\coordinate (aki) at (2, 3.5);
%\coordinate (bkj) at (5, 3.5);
%\coordinate (bk1) at (6, 3.5);
%
%\coordinate (dots1ak) at (1.5, 3.5);
%\coordinate (dots2ak) at (2.5, 3.5);
%\coordinate (dots1bk) at (5.5, 3.5);
%\coordinate (dots2bk) at (4.5, 3.5);
%
%\coordinate (akp1m) at (5, 2.5);
%\coordinate (bkp1m) at (2, 2.5);
%
%\coordinate (dots1akp1m) at (4.5, 2.5);
%\coordinate (dots2akp1m) at (5.5, 2.5);
%\coordinate (dots1bkp1m) at (1.5, 2.5);
%\coordinate (dots2bkp1m) at (2.5, 2.5);
%
%\coordinate (dots2) at (3.5, 1.75);
%
%\coordinate (as1)  at (1, 1);
%\coordinate (as2)  at (2, 1);
%\coordinate (asns) at (3, 1);
%\coordinate (bs1)  at (6, 1);
%\coordinate (bs2)  at (5, 1);
%\coordinate (bsns) at (4, 1);
%
%\coordinate (dots1as) at (2.5, 1);
%\coordinate (dots1bs) at (4.5, 1);

\draw [] (d1)  circle [radius=2.0pt] ;
\draw [] (d3)  circle [radius=2.0pt] ;
\draw [] (d2)  circle [radius=2.0pt] ;
\draw [] (d2a) circle [radius=2.0pt] ;
\draw [] (a01) circle [radius=2.0pt] ;
\draw [] (a02) circle [radius=2.0pt] ;
\draw [] (b01) circle [radius=2.0pt] ;
\draw [] (b02) circle [radius=2.0pt] ;
\draw [] (b13) circle [radius=2.0pt] ;
\draw [] (b12) circle [radius=2.0pt] ;
\draw [] (b11) circle [radius=2.0pt] ;
\draw [] (a13) circle [radius=2.0pt] ;
\draw [] (a12) circle [radius=2.0pt] ;
\draw [] (a11) circle [radius=2.0pt] ;
%\draw [] (b23) circle [radius=2.0pt] ;
\draw [] (b22) circle [radius=2.0pt] ;
\draw [] (b21) circle [radius=2.0pt] ;
%\draw [] (a23) circle [radius=2.0pt] ;
\draw [] (a22) circle [radius=2.0pt] ;
\draw [] (a21) circle [radius=2.0pt] ;

%\draw [] (dots1) node {$\ldots$};
%
%\draw [] (ak1) circle [radius=2.0pt] ;
%\draw [] (aki) circle [radius=2.0pt] ;
%\draw [] (bkj) circle [radius=2.0pt] ;
%\draw [] (bk1) circle [radius=2.0pt] ;
%
%\draw [] (dots1ak) node {$\ldots$};
%\draw [] (dots2ak) node {$\ldots$};
%\draw [] (dots1bk) node {$\ldots$};
%\draw [] (dots2bk) node {$\ldots$};
%
%\draw [] (akp1m) circle [radius=2.0pt] ;
%\draw [] (bkp1m) circle [radius=2.0pt] ;
%
%\draw [] (dots1akp1m) node {$\ldots$};
%\draw [] (dots2akp1m) node {$\ldots$};
%\draw [] (dots1bkp1m) node {$\ldots$};
%\draw [] (dots2bkp1m) node {$\ldots$};
%
%\draw [] (dots2) node {$\ldots$};
%
%\draw [] (as1)  circle [radius=2.0pt] ;
%\draw [] (as2)  circle [radius=2.0pt] ;
%\draw [] (asns) circle [radius=2.0pt] ;
%\draw [] (bs1)  circle [radius=2.0pt] ;
%\draw [] (bs2)  circle [radius=2.0pt] ;
%\draw [] (bsns) circle [radius=2.0pt] ;
%
%\draw [] (dots1as) node {$\ldots$};
%\draw [] (dots1bs) node {$\ldots$};
%
\begin{scope}[>=latex, ->, shorten >= 4.25pt, shorten <= 4.25pt]
\draw [ ] (a01) -- (d1);
\draw [ ] (a01) -- (d3);
\draw [ ] (a02) -- (d1);
\draw [ ] (a02) -- (d2);
\draw [ ] (b01) -- (d3);
\draw [ ] (b01) -- (d2);
\draw [ ] (b02) -- (d1);
\draw [ ] (b02) -- (d2);
\draw [ ] (b02) -- (d3);
\draw [ ] (a02) -- (d2a);
\draw [ ] (b01) -- (d2a);
\draw [ ] (b02) -- (d2a);

\draw [ ] (b13) -- (a01);
\draw [ ] (b13) -- (a02);
\draw [ ] (b12) -- (a01);
\draw [ ] (b12) -- (b01);
\draw [ ] (b11) -- (a01);
\draw [ ] (b11) -- (b02);
\draw [ ] (a13) -- (a02);
\draw [ ] (a13) -- (b01);
\draw [ ] (a12) -- (a02);
\draw [ ] (a12) -- (b02);
\draw [ ] (a11) -- (b01);
\draw [ ] (a11) -- (b02);

\draw [ ] (a21) -- (b11);
\draw [ ] (a21) -- (a12);
\draw [ ] (a21) -- (b12);
\draw [ ] (a22) -- (b11);
\draw [ ] (a22) -- (a12);
\draw [ ] (a22) -- (b13);
%\draw [ ] (a23) -- (b11);
%\draw [ ] (a23) -- (a13);
%\draw [ ] (a23) -- (b12);
\draw [ ] (b21) -- (a11);
\draw [ ] (b21) -- (b12);
\draw [ ] (b21) -- (a12);
\draw [ ] (b22) -- (a11);
\draw [ ] (b22) -- (b12);
\draw [ ] (b22) -- (a13);
%\draw [ ] (b23) -- (a11);
%\draw [ ] (b23) -- (a13);
%\draw [ ] (b23) -- (b12);

%\draw [->, shorten >= 2.75pt, shorten <= 2.75pt] (bkp1m) -- (ak1);
%\draw [->, shorten >= 2.75pt, shorten <= 2.75pt] (bkp1m) -- (aki);
%\draw [->, shorten >= 2.75pt, shorten <= 2.75pt] (bkp1m) -- (bkj);
%\draw [->, shorten >= 2.75pt, shorten <= 2.75pt] (akp1m) -- (bk1);
%\draw [->, shorten >= 2.75pt, shorten <= 2.75pt] (akp1m) -- (aki);
%\draw [->, shorten >= 2.75pt, shorten <= 2.75pt] (akp1m) -- (bkj);

\end{scope}

\node [      left ] at (d1)  {$\delta_1\mbox{~}$}    ;
\node [      right] at (d2)  {$\mbox{~}\delta_2$}    ;
\node [      right] at (d2a) {$\mbox{~}\delta'_2$}   ;
\node [      left ] at (d3)  {$\delta_3\mbox{~}$}    ;
\node [      left ] at (a01) {$\alpha^0_1\mbox{~}$}  ;
\node [      left ] at (a02) {$\alpha^0_2\mbox{~~}$} ;
\node [      left ] at (b01) {$\beta^0_1\mbox{~~}$}  ;
\node [      right] at (b02) {$\mbox{~}\beta^0_2$}   ;

\node [      right] at (a11) {$\mbox{~}\alpha^1_1$} ;
\node [      right] at (a12) {$\mbox{~}\alpha^1_2$} ;
\node [      right] at (a13) {$\mbox{~}\alpha^1_3$} ;
\node [      left ] at (b11) {$\beta^1_1\mbox{~~}$}  ;
\node [      left ] at (b12) {$\beta^1_2\mbox{~}$}  ;
\node [      left ] at (b13) {$\beta^1_3\mbox{~}$}  ;

\node [below left ] at (a21) {$\alpha^2_1$} ;
\node [below left ] at (a22) {$\alpha^2_2$} ;
%\node [below left ] at (a23) {$\alpha^2_3$} ;
\node [below right] at (b21) {$\beta^2_1$}  ;
\node [below right] at (b22) {$\beta^2_2$}  ;
%\node [below right] at (b23) {$\beta^2_3$}  ;

%\node [above right] at (ak1) {$\alpha^k_1$} ;
%\node [above right] at (aki) {$\alpha^k_i$} ;
%\node [above right] at (bk1) {$\beta^k_1$}  ;
%\node [above right] at (bkj) {$\beta^k_j$}  ;
%
%\node [below right] at (akp1m) {$\alpha^{k+1}_m$} ;
%\node [below right] at (bkp1m) {$\beta^{k+1}_m$}  ;
%
%\node [above right] at (as1)  {$\alpha^{s+1}_1$}         ;
%\node [above right] at (as2)  {$\alpha^{s+1}_2$}         ;
%\node [above right] at (asns) {$\alpha^{s+1}_{n_{s+1}}$} ;
%\node [above right] at (bs1)  {$\beta^{s+1}_1$}          ;
%\node [above right] at (bs2)  {$\beta^{s+1}_2$}          ;
%\node [above right] at (bsns) {$\beta^{s+1}_{n_{s+1}}$}  ;

\node [above      ] at (d3)   {$\mbox{~}\begin{array}{l}P'(a)\\P'(a')\end{array}$}   ;
\node [above      ] at (d2)   {$\mbox{~}\begin{array}{l}P'(b),\\b\ne a\end{array}$} ;
\node [      right] at (b01)  {$\mbox{~}P'(a')$}         ;
\node [above = 2pt] at (d2a)  {$\mbox{~}P'(a')$}         ;

\end{tikzpicture}

\caption{An $a$-suitable model based on frame $\kframe{F}_0$}
  \label{fig:F-2}
\end{figure} 

Let $\mathcal{A}$ be a set such that $|\mathcal{A}|\geqslant 3$ and let $a$ be an element of~$\mathcal{A}$.
We call an intuitionistic model $\kmodel{M} = \otuple{\otuple{W_0, R_0}\odot\mathcal{A}, I}$  \defnotion{$a$\nobreakdash-suitable} if there exists $a' \in \mathcal{A} \setminus \{a\}$ such that, for every $w\in W_0$ and every $b\in\mathcal{A}$,
$$
\begin{array}{lcl}
\kmodel{M},w \imodels P'(b) 
  & \iff & \mbox{either $w=\delta_2$ and $b\ne a$;} \\
  &      & \mbox{or $w\in\{\delta'_2,\beta_1^0\}$ and $b = a'$;} \\
  &      & \mbox{or $w=\delta_3$ and $b \in\{a,a'\}$,}
\end{array} 
$$
see Figure~\ref{fig:F-2}.

\begin{lemma}
\label{lem:frame-F:1}
Let $\kmodel{M} = \otuple{\otuple{W_0, R_0}\odot\mathcal{A}, I}$ be an $a$-suitable model.  Then, for
every $w \in W_0$\/ and all $k,m \in \{1,2\}$,
$$
\begin{array}{lcl}
\kmodel{M},w\not\imodels {A_m^k}(a) & \iff & w R_0 \alpha_m^k; 
\smallskip\\
\kmodel{M},w\not\imodels {B_m^k}(a) & \iff & w R_0 \beta_m^k.
\end{array}
$$
\end{lemma}

\begin{proof} 
A routine check. Also, the statement follows directly from~\cite[Lemma~3.9]{RShJLC21b}.
\end{proof}

\begin{lemma}
\label{lem:frame-F:2}
Let $\kmodel{M} = \otuple{\otuple{W_0, R_0}\odot\mathcal{A}, I}$ be an $a$-suitable model.  Then, for
every $b\in\mathcal{A}\setminus\{a\}$, every $w \in W_0$, and every $m \in \{1,2\}$,
$$
\begin{array}{lcl}
\kmodel{M},w\imodels {A_m^2}(b) 
  & \mbox{and\/} 
  & \kmodel{M},w\imodels {B_m^2}(b).
\end{array}
$$
\end{lemma}

\begin{proof} 
A routine check. Also, the statement follows directly from~\cite[Lemma~3.10]{RShJLC21b}.
\end{proof}

Now, we are ready to simulate the letters $Q$ and $G$ in $S_8 \mathit{M}^{\mathit{int}}\mathit{Tiling}_n^{\mathbb{X}}$ by positive formulas that contain a single unary predicate letter and two individual variables. Let $S_9$ be a function that replaces in formulas
\begin{itemize}
\item each occurrence of $Q(x)$ or $Q(y)$ with $A^2_1(x)\vee B^2_1(x)$ or $A^2_1(y)\vee B^2_1(y)$, respectively;
\item each occurrence of $G(x)$ or $G(y)$ with $A^2_2(x)\vee B^2_2(x)$ or $A^2_2(y)\vee B^2_2(y)$, respectively.
\end{itemize}

\begin{lemma}
\label{lem:MIntTiling:1:Q:positive}
If $n\in\mathbb{X}$, then $S_9 S_8 \mathit{M}^{\mathit{int}}\mathit{Tiling}_n^{\mathbb{X}} \in \logic{QInt}$.
\end{lemma}

\begin{proof}
Follows from Lemma~\ref{lem:MIntTiling:1:QG:positive}, since $S_9 S_8 \mathit{M}^{\mathit{int}}\mathit{Tiling}_n^{\mathbb{X}}$ is a substitution instance of $S_8 \mathit{M}^{\mathit{int}}\mathit{Tiling}_n^{\mathbb{X}}$.
\end{proof}

\begin{lemma}
\label{lem:MIntTiling:2:Q:positive}
Let $L\in\{\logic{QInt}_{\mathit{dfin}} + \bm{bd}_7 + \bm{cd},\logic{QKC}_{\mathit{dfin}} + \bm{bd}_8 + \bm{cd}\}$.
If $n\in\mathbb{Y}$, then $S_9 S_8 \mathit{M}^{\mathit{int}}\mathit{Tiling}_n^{\mathbb{X}} \not\in L$.
%
%If $n\in\mathbb{Y}$, then both $S_9 S_8 \mathit{M}^{\mathit{int}}\mathit{Tiling}_n^{\mathbb{X}} \not\in \logic{QInt}_{\mathit{dfin}} + \bm{bd}_7 + \bm{cd}$ and\/ $S_9 S_8 \mathit{M}^{\mathit{int}}\mathit{Tiling}_n^{\mathbb{X}} \not\in \logic{QKC}_{\mathit{dfin}} + \bm{bd}_8 + \bm{cd}$.
\end{lemma}

\begin{proof}
Let $n\in\mathbb{Y}$. Observe that $\ckf (\logic{QInt} + \bm{bd}_3)$ is an swKKZ class; then, by Lemma~\ref{lem:MIntTiling:2:QG:positive}, $S_8 \mathit{M}^{\mathit{int}}\mathit{Tiling}_n^{\mathbb{X}} \not\in \QSILcd \ckf (\logic{QInt} + \bm{bd}_3)$. Therefore, there are a Kripke frame $\kframe{F}=\otuple{W,R}$ of depth at most~$3$, a model $\kModel{M}=\otuple{\kframe{F}\odot\mathcal{D},I}$ with a finite domain~$\mathcal{D}$, and a world~$w_0\in W$ such that $\kModel{M},w_0\not\imodels S_8 \mathit{M}^{\mathit{int}}\mathit{Tiling}_n^{\mathbb{X}}$; without a loss of generality, we may assume that $|\mathcal{D}|\geqslant 3$.

For every $a\in\mathcal{D}$, let us take a copy $\kframe{F}_0^a = \otuple{W_a,R_a}$ of $\kframe{F}_0$ defined as follows:
$$
\begin{array}{lcl}
W_a & = & W_0 \times\{a\}; \\
R_a & = & \{\otuple{\otuple{v,a},\otuple{u,a}} : v R_0 u\}.
\end{array}
$$

Define an intuitionistic Kripke frame $\bar{\kframe{F}} = \otuple{\bar{W},\bar{R}}$ by
$$
\begin{array}{lcl}
\bar{W} & = & \displaystyle W \cup {\bigcup\{W_a : a\in\mathcal{D}\}}\\
\end{array}
$$
and $\bar{R}$ be the reflexive transitive closure of the relation
$$
\begin{array}{lcl}
%\bar{W} & = & 
\displaystyle R \cup  
\bigcup \{R_a : a\in\mathcal{D}\} \cup 
\{\otuple{w,\otuple{\alpha^2_1,a}},\otuple{w,\otuple{\beta^2_1,a}} : \mbox{$w\in W$, $a\in\mathcal{D}$, $\kModel{M},w\not\imodels Q(a)$}\} 
\\
\displaystyle \phantom{R \cup \bigcup \{R_a : a\in\mathcal{D}\}}{} \cup 
\{\otuple{w,\otuple{\alpha^2_2,a}},\otuple{w,\otuple{\beta^2_2,a}} : \mbox{$w\in W$, $a\in\mathcal{D}$, $\kModel{M},w\not\imodels G(a)$}\} 
\\
\displaystyle \phantom{R \cup \bigcup \{R_a : a\in\mathcal{D}\}}{} \cup 
\{\otuple{w,\otuple{\alpha^1_1,a}},\otuple{w,\otuple{\beta^1_1,a}} : \mbox{$w\in W$, $a\in\mathcal{D}$}\}.
\\
\end{array}
$$
Observe that $\kframe{F}_0^a$ is a subframe of $\bar{\kframe{F}}$, for every $a\in\mathcal{D}$.
Let $\bar{\kModel{M}} = \otuple{\bar{\kframe{F}}\odot\mathcal{D},\bar{I}}$ be a model such that 
\begin{itemize}
\item
for every $w\in W$, $\bar{I}(P') = \varnothing$;
\item
for every $a\in\mathcal{D}$, its submodel $\bar{\kModel{M}}_a = \otuple{\kframe{F}_0^a\odot\mathcal{D},\bar{I}_a}$, where $\bar{I}_a(v,P') = \bar{I}(v,P')$ for every $v\in W_a$, is $a$-suitable.
\end{itemize}

Let us make some useful observations about model~$\bar{\kModel{M}}$ (cf.~\cite[Lemma~3.11]{RShJLC21b}).

\begin{sublemma}
\label{sublem:1:lem:MIntTiling:2:Q:positive}
If $w\in W$ and $a\in\mathcal{D}$, then\/ $\bar{\kModel{M}},w\not\imodels A^1_1(a)\vee B^1_1(a)$.
\end{sublemma}

\begin{proof}
Notice that $w\bar{R}\otuple{\alpha^1_1,a}$ and $w\bar{R}\otuple{\beta^1_1,a}$ by the definition of $\bar{\kframe{F}}$; then apply Lemma~\ref{lem:frame-F:1} (and heredity), since $\bar{\kModel{M}}_a$ is $a$\nobreakdash-suitable.
\end{proof}

\begin{sublemma}
\label{sublem:2:lem:MIntTiling:2:Q:positive}
If $\varphi$ is a positive formula whose predicate letters are in $\{Q,G\}$ and individual variables are in $\{x,y\}$, $v\in \bar{W}\setminus W$, and $g$ is an assignment in\/~$\bar{\kModel{M}}$, then\/ $\bar{\kModel{M}},v\imodels^g S_9\varphi$.
\end{sublemma}

\begin{proof}
By Lemmas~\ref{lem:frame-F:1} and~\ref{lem:frame-F:2}, $\bar{\kModel{M}}_a\imodels A^2_k(b)\vee B^2_k(b)$, for every $k\in\{1,2\}$ and all $a,b\in \mathcal{D}$. Then apply induction on~$\varphi$ (taking into account that $\varphi$ is positive).
\end{proof}

\begin{sublemma}
\label{sublem:3:lem:MIntTiling:2:Q:positive}
For every positive formula $\varphi$ whose predicate letters are in $\{Q,G\}$ and individual variables are in $\{x,y\}$, every $w\in W$, and every assignment~$g$ in~$\kModel{M}$,
$$
\begin{array}{lcl}
\kModel{M},w\imodels^g \varphi & \iff & \bar{\kModel{M}},w\imodels^g S_9\varphi.
\end{array}
$$
\end{sublemma}

\begin{proof}
Induction on~$\varphi$.

Let $\varphi$ be $Q(x)$ or $Q(y)$. 

If $\kModel{M},w\not\imodels Q(a)$, for some $w\in W$ and $a\in\mathcal{D}$, then $w\bar{R}\otuple{\alpha^2_1,a}$ and $w\bar{R}\otuple{\beta^2_1,a}$. It follows from Lemma~\ref{lem:frame-F:1} that $\bar{\kModel{M}}_a,\otuple{\alpha^2_1,a}\not\imodels A^2_1(a)$ and $\bar{\kModel{M}}_a,\otuple{\alpha^2_1,a}\not\imodels B^2_1(a)$, since $\bar{\kModel{M}}_a$ is $a$-suitable. Then, by heredity, 
$\bar{\kModel{M}},w\not\imodels A^2_1(a)\vee B^2_1(a)$.

Suppose that $\bar{\kModel{M}},w\not\imodels A^2_1(a)\vee B^2_1(a)$, for some $w\in W$ and $a\in\mathcal{D}$.
Then $\bar{\kModel{M}},w\not\imodels A^2_1(a)$ and $\bar{\kModel{M}},w\not\imodels B^2_1(a)$; hence, there exist $u,v\in \bar{R}(w)$ such that
$$
\begin{array}{llll}
\bar{\kModel{M}},u\imodels A^1_1(a),
  & \bar{\kModel{M}},u\not\imodels B^1_1(a),
  & \bar{\kModel{M}},u\not\imodels A^1_2(a),
  & \bar{\kModel{M}},u\not\imodels B^1_2(a);
  \smallskip\\
\bar{\kModel{M}},v\imodels B^1_1(a),
  & \bar{\kModel{M}},v\not\imodels A^1_1(a),
  & \bar{\kModel{M}},v\not\imodels B^1_2(a),
  & \bar{\kModel{M}},v\not\imodels A^1_2(a).
\end{array}
$$
By Sublemma~\ref{sublem:1:lem:MIntTiling:2:Q:positive}, $u,v\not\in W$. Observe that $\bar{\kModel{M}},u\not\imodels A^2_1(a)$ and $\bar{\kModel{M}},v\not\imodels B^2_1(a)$; hence, by Sublemma~\ref{sublem:2:lem:MIntTiling:2:Q:positive}, $u,v\not\in W_b$ with $b\ne a$. Thus, $u,v\in W_a$. Then, Lemma~\ref{lem:frame-F:1} implies that $u=\otuple{\alpha^2_1,a}$ and $v=\otuple{\beta^2_1,a}$. By the definition of~$\bar{\kModel{M}}$, this is possible only if $\kModel{M},w\not\imodels Q(a)$.

Let $\varphi$ be $G(x)$ or $G(y)$. Then use a similar argumentation.

Let $\varphi=\varphi'\to\varphi''$ with the statement of the sublemma being true for $\varphi'$ and~$\varphi''$.

If $\kModel{M},w\not\imodels^g \varphi$, for some $w\in W$ and some assignment~$g$, then $\kModel{M},w'\imodels^g \varphi'$ and $\kModel{M},w'\not\imodels^g \varphi''$, for some $w'\in R(w)$. Then
$\bar{\kModel{M}},w'\imodels^g S_9\varphi'$ and $\bar{\kModel{M}},w'\not\imodels^g S_9\varphi''$; hence, $\bar{\kModel{M}},w\not\imodels^g S_9\varphi$.

Suppose that $\bar{\kModel{M}},w\not\imodels^g S_9\varphi$, for some $w\in W$ and some assignment~$g$. Then $\bar{\kModel{M}},w'\imodels^g S_9\varphi'$ and $\bar{\kModel{M}},w'\not\imodels^g S_9\varphi''$, for some $w'\in\bar{R}(w)$. By Sublemma~\ref{sublem:2:lem:MIntTiling:2:Q:positive}, $w'\in W$. Then $\kModel{M},w'\imodels^g \varphi'$ and $\kModel{M},w'\not\imodels^g \varphi''$; hence, $\kModel{M},w\not\imodels^g \varphi$.

The cases $\varphi = \varphi'\wedge \varphi''$, $\varphi = \varphi'\vee \varphi''$, $\varphi = \forall z\,\varphi'$, and $\varphi = \exists z\,\varphi'$, where $z\in\{x,y\}$, are similar and left to the reader.
\end{proof}

By Sublemma~\ref{sublem:3:lem:MIntTiling:2:Q:positive}, $\bar{\kModel{M}},w_0\not\imodels S_9 S_8 \mathit{M}^{\mathit{int}}\mathit{Tiling}_n^{\mathbb{X}}$. Observe that $\bar{\kframe{F}}$ is of depth at most~$7$. Thus, $S_9 S_8 \mathit{M}^{\mathit{int}}\mathit{Tiling}_n^{\mathbb{X}} \not\in \logic{QInt}_{\mathit{dfin}} + \bm{bd}_7 + \bm{cd}$.

In order to prove that $S_9 S_8 \mathit{M}^{\mathit{int}}\mathit{Tiling}_n^{\mathbb{X}} \not\in \logic{QKC}_{\mathit{dfin}} + \bm{bd}_8 + \bm{cd}$, just extend $\bar{\kModel{M}}$ with a new world accessible from all the worlds of the model and then make the formula $\forall x\,P'(x)$ being true at this world.
\end{proof}

\begin{lemma}
\label{lem:MIntTiling:3:Q:positive}
Let $L\in\{\logic{QInt}_{\mathit{wfin}} + \bm{bd}_7 + \bm{cd},\logic{QKC}_{\mathit{wfin}} + \bm{bd}_8 + \bm{cd}\}$.
If $n\in\mathbb{Y}$, then $S_9 S_8 \mathit{M}^{\mathit{int}}\mathit{Tiling}_n^{\mathbb{X}} \not\in L$.
%
%If $n\in\mathbb{Y}$, then both $S_9 S_8 \mathit{M}^{\mathit{int}}\mathit{Tiling}_n^{\mathbb{X}} \not\in \logic{QInt}_{\mathit{wfin}} + \bm{bd}_7 + \bm{cd}$ and\/ $S_9 S_8 \mathit{M}^{\mathit{int}}\mathit{Tiling}_n^{\mathbb{X}} \not\in \logic{QKC}_{\mathit{wfin}} + \bm{bd}_8 + \bm{cd}$.
\end{lemma}

\begin{proof}
Similarly to the proof of Lemma~\ref{lem:MIntTiling:2:Q:positive}. We give just some comments.

Let $n\in\mathbb{Y}$. Observe that $\fin \ckf (\logic{QInt} + \bm{bd}_3)$ is an swKKZ class; then, by Lemma~\ref{lem:MIntTiling:3:QG:positive}, $S_8 \mathit{M}^{\mathit{int}}\mathit{Tiling}_n^{\mathbb{X}} \not\in \QSILcw \ckf (\logic{QInt} + \bm{bd}_3)$. Therefore, there are a finite Kripke frame $\kframe{F}=\otuple{W,R}$ of depth at most~$3$, a model $\kModel{M}=\otuple{\kframe{F}\odot\mathcal{D},I}$ with a domain~$\mathcal{D}$, and a world~$w_0\in W$ such that $\kModel{M},w_0\not\imodels S_8 \mathit{M}^{\mathit{int}}\mathit{Tiling}_n^{\mathbb{X}}$.

Observe that we may assume $\mathcal{D}$ being finite (and, additionally, with $|\mathcal{D}|\geqslant 3$); we can extract this from Proposition~\ref{prop:finite:frame} or from the proof of Lemma~\ref{lem:MIntTiling:3:QG:positive}. Then, applying the same argumentation as in the proof of Lemma~\ref{lem:MIntTiling:2:Q:positive}, we refute $S_9 S_8 \mathit{M}^{\mathit{int}}\mathit{Tiling}_n^{\mathbb{X}}$ in a finite intuitionistic model of the depth required, and similarly for~$\logic{QKC}$; we leave the details to the reader.
\end{proof}

\begin{theorem}
\label{th:QInt:positive:2var:1unary}
The positive fragments of logics\/ 
$$
\begin{array}{lcl}
\logic{QInt} & \mbox{and\/} & \logic{QInt}_{\mathit{dfin}}\hfill +\bm{bd}_7+\bm{cd}, \\
\logic{QInt} & \mbox{and\/} & \logic{QInt}_{\mathit{wfin}}\hfill+\bm{bd}_7+\bm{cd}, \\
\logic{QInt} & \mbox{and\/} & \logic{QKC}_{\mathit{dfin}}\hfill+\bm{bd}_8+\bm{cd}, \\
\logic{QInt} & \mbox{and\/} & \logic{QKC}_{\mathit{wfin}}\hfill+\bm{bd}_8+\bm{cd}\phantom{,} 
\end{array}
$$
are recursively inseparable in the language with a single unary predicate letter and two individual variables.
\end{theorem}

\begin{proof}
Follows from Lemmas~\ref{lem:MIntTiling:1:Q:positive}, \ref{lem:MIntTiling:2:Q:positive}, and~\ref{lem:MIntTiling:3:Q:positive}.
\end{proof}

\begin{corollary}
\label{cor:1:th:QInt:positive:2var:1unary}
Let\/ $\logic{QInt}\subseteq L\subseteq L'$ and also either $L'\subseteq \logic{QInt}+\bm{bd}_7+\bm{cd}$ or\/ $L'\subseteq \logic{QKC}+\bm{bd}_8+\bm{cd}$. Then the positive fragments of $L$ and $L'_{\mathit{wfin}}$ as well as $L$ and $L'_{\mathit{dfin}}$ are recursively inseparable in the language with a single unary predicate letter and two individual variables. 
\end{corollary}

\begin{corollary}
\label{cor:2:th:QInt:positive:2var:1unary}
Let $L$ be one of\/ $\logic{QInt}$, $\logic{QKP}$, $\logic{QKC}$, $\logic{QInt}+\bm{bd}_n$ with $n\geqslant 7$, $\logic{QKC}+\bm{bd}_n$ with $n\geqslant 8$. Then the positive fragments of $L$ and $L_{\mathit{wfin}}$ as well as $L$ and $L_{\mathit{dfin}}$ are recursively inseparable in the language with a single unary predicate letter and two individual variables. 
\end{corollary}

\section{Computational complexity}
\label{sec:complexity}
\setcounter{equation}{0}

\subsection{Lower bounds for monadic fragments}
\label{subsec:lbounds:monadic}

Notice that the technique proving the recursive inseparability of fragments of modal and superintuitionistic logics also provides us with lower bounds for computational complexity of the fragments. We can argue similar to Section~\ref{subsec:remarks-on-complexity}. We omit the argumentation and only give the resulting statements.

For modal predicate logics, we obtain the following statements.

\begin{theorem}
\label{th:complexity:modal:1}
Let $L$ be a modal predicate logic that contains\/ $\logic{QK}$ and is contained in at least one of the logics\/ $\logic{QS5}$, $\logic{QGL.3.bf}$, $\logic{QGrz.3.bf}$, $\logic{QGL.bf}\oplus\bm{bd}_2$, $\logic{QGrz.bf}\oplus\bm{bd}_2$. Then the fragment of $L$ in the language with a single unary predicate letter and three individual variables is\/ $\Sigma^0_1$\nobreakdash-hard.
\end{theorem}

\begin{theorem}
\label{th:complexity:modal:2}
Let $L$ be a modal predicate logic that contains\/ $\logic{QK}$ and is contained in at least one of the logics\/ $\logic{QS5}$, $\logic{QGL.3}$, $\logic{QGrz.3}$, $\logic{QGL}\oplus\bm{bd}_2$, $\logic{QGrz}\oplus\bm{bd}_2$. Then the fragments of $L_{\mathit{dfin}}$, $L_{\mathit{wfin}}$, $L_{\mathit{dfin}}\oplus\bm{bf}$, and $L_{\mathit{wfin}}\oplus\bm{bf}$ in the language with a single unary predicate letter and three individual variables are\/ $\Pi^0_1$\nobreakdash-hard.
\end{theorem}

Notice that Theorems~\ref{th:complexity:modal:1} and~\ref{th:complexity:modal:2} remain true if we additionally require that the language does not contain~$\bot$ (and hence, $\neg$ and~$\Diamond$), see Remark~\ref{rem:modal:bot-free}. 

For superintuitionistic predicate logics, we obtain similar statements.

\begin{theorem}
\label{th:complexity:int:1}
Let $L$ be a superintuitionistic predicate logic that contains\/ $\logic{QInt}$ and is contained in either\/ $\logic{QInt}+\bm{bd}_2+\bm{cd}$ or\/ $\logic{QKC}+\bm{bd}_3+\bm{cd}$. Then the positive fragment of $L$ in the language with a single unary predicate letter and three individual variables is\/ $\Sigma^0_1$\nobreakdash-hard.
\end{theorem}

\begin{theorem}
Let $L$ be a superintuitionistic predicate logic that contains\/ $\logic{QInt}$ and is contained in either\/ $\logic{QInt}+\bm{bd}_2+\bm{cd}$ or\/ $\logic{QKC}+\bm{bd}_3+\bm{cd}$. Then the positive fragments of $L_{\mathit{dfin}}$, $L_{\mathit{wfin}}$, $L_{\mathit{dfin}}+\bm{cd}$, and $L_{\mathit{wfin}}+\bm{cd}$ in the language with a single unary predicate letter and three individual variables are\/ $\Pi^0_1$\nobreakdash-hard.
\end{theorem}

Of course, the fragments of such logics as $\logic{QInt}$, $\logic{QKP}$, $\logic{QKC}$, $\logic{QK}$, $\logic{QT}$, $\logic{QD}$, $\logic{QKB}$, $\logic{QKTB}$, $\logic{QK4}$, $\logic{QS4}$, $\logic{QK5}$, $\logic{QS5}$, $\logic{QK45}$, $\logic{QKD45}$, $\logic{QK4B}$, $\logic{QGL}$, $\logic{QGrz}$, $\logic{QwGrz}$, etc.\ are $\Sigma^0_1$-complete, since the logics are recursively enumerable. Observe that the classes of their finite Kripke frames are also recursively enumerable, therefore, the monadic fragments of $\logic{QInt}_{\mathit{wfin}}$, $\logic{QKP}_{\mathit{wfin}}$, $\logic{QKC}_{\mathit{wfin}}$, $\logic{QK}_{\mathit{wfin}}$, $\logic{QT}_{\mathit{wfin}}$, $\logic{QD}_{\mathit{wfin}}$, $\logic{QKB}_{\mathit{wfin}}$, $\logic{QKTB}_{\mathit{wfin}}$, $\logic{QK4}_{\mathit{wfin}}$, $\logic{QS4}_{\mathit{wfin}}$, $\logic{QK5}_{\mathit{wfin}}$, $\logic{QS5}_{\mathit{wfin}}$, $\logic{QK45}_{\mathit{wfin}}$, $\logic{QKD45}_{\mathit{wfin}}$, $\logic{QK4B}_{\mathit{wfin}}$, $\logic{QGL}_{\mathit{wfin}}$, $\logic{QGrz}_{\mathit{wfin}}$, $\logic{QwGrz}_{\mathit{wfin}}$, etc.\ are $\Pi^0_1$-complete, see~\cite[Theorem~4.4]{RShsubmitted}.

Similar results are true for the monadic fragments of logics with two individual variables in the language. Notice that we then loose some logics.

\begin{theorem}
\label{th:complexity:modal:3}
Let $L$ be a modal predicate logic that contains\/ $\logic{QK}$ and is contained in at least one of the logics\/ $\logic{QGL.bf}\oplus\bm{bd}_4$, $\logic{QGrz.bf}\oplus\bm{bd}_4$, and $\logic{QwGrz.bf}\oplus\bm{bd}_3$. Then the fragment of $L$ in the language with a single unary predicate letter and two individual variables is\/ $\Sigma^0_1$\nobreakdash-hard.
\end{theorem}

\begin{theorem}
\label{th:complexity:modal:4}
Let $L$ be a modal predicate logic that contains\/ $\logic{QK}$ and is contained in at least one of the logics\/ $\logic{QGL.bf}\oplus\bm{bd}_4$, $\logic{QGrz.bf}\oplus\bm{bd}_4$, and $\logic{QwGrz.bf}\oplus\bm{bd}_3$.  Then the fragments of $L_{\mathit{dfin}}$, $L_{\mathit{wfin}}$, $L_{\mathit{dfin}}\oplus\bm{bf}$, and $L_{\mathit{wfin}}\oplus\bm{bf}$ in the language with a single unary predicate letter and two individual variables are\/ $\Pi^0_1$\nobreakdash-hard.
\end{theorem}

\begin{theorem}
Let $L$ be a superintuitionistic predicate logic that contains\/ $\logic{QInt}$ and is contained in either\/ $\logic{QInt}+\bm{bd}_7+\bm{cd}$ or\/ $\logic{QKC}+\bm{bd}_8+\bm{cd}$. Then the positive fragment of $L$ in the language with a single unary predicate letter and two individual variables is\/ $\Sigma^0_1$\nobreakdash-hard.
\end{theorem}

\begin{theorem}
Let $L$ be a superintuitionistic predicate logic that contains\/ $\logic{QInt}$ and is contained in either\/ $\logic{QInt}+\bm{bd}_7+\bm{cd}$ or\/ $\logic{QKC}+\bm{bd}_8+\bm{cd}$. Then the positive fragments of $L_{\mathit{dfin}}$, $L_{\mathit{wfin}}$, $L_{\mathit{dfin}}+\bm{cd}$, and $L_{\mathit{wfin}}+\bm{cd}$ in the language with a single unary predicate letter and two individual variables are\/ $\Pi^0_1$\nobreakdash-hard.
\end{theorem}

\subsection{Lower bounds for dyadic fragments}

The situation with lower bounds for computational complexity of dyadic fragments differs from the one for monadic fragments. Let us make some observations.

\begin{proposition}
Let $L$ be a modal predicate logic containing\/ $\logic{QK}$ and contained in at least one of the logics\/ $\logic{QS5}$, $\logic{QGL.3}$, $\logic{QGrz.3}$, $\logic{QGL}\oplus\bm{bd}_2$, $\logic{QGrz}\oplus\bm{bd}_2$. Then the fragments of $L_{\mathit{wfin}}$ and $L_{\mathit{wfin}}\oplus\bm{bf}$ in the language with a single binary predicate letter and three individual variables are both\/ $\Sigma^0_1$\nobreakdash-hard and\/ $\Pi^0_1$\nobreakdash-hard.
\end{proposition}

\begin{proof}
Indeed, $\Sigma^0_1$\nobreakdash-hardness follows from the Church Theorem (or from the Trakhtenbrot Theorem, see Theorem~\ref{th:Trakhtenbrot:bin:P}) and $\Pi^0_1$\nobreakdash-hardness follows from Theorem~\ref{th:complexity:modal:2}, since we can simulate $Q(x)$ by~$P(x,x)$.
\end{proof}

\begin{proposition}
Let $L$ be a superintuitionistic predicate logic that contains\/ $\logic{QInt}$ and is contained in either\/ $\logic{QInt}+\bm{bd}_2+\bm{cd}$ or\/ $\logic{QKC}+\bm{bd}_3+\bm{cd}$. Then the positive fragments of $L_{\mathit{wfin}}$ and $L_{\mathit{wfin}}+\bm{cd}$ in the language with a single binary predicate letter and three individual variables are both\/ $\Sigma^0_1$\nobreakdash-hard and\/ $\Pi^0_1$\nobreakdash-hard.
\end{proposition}

\begin{proof}
Again, $\Sigma^0_1$\nobreakdash-hardness follows from the Church Theorem (or Theorem~\ref{th:Trakhtenbrot:bin:P}) with use of the Kolmogorov embedding (or some similar) and $\Pi^0_1$\nobreakdash-hardness follows from Theorem~\ref{th:complexity:int:1}.
\end{proof}

\subsection{Remarks on high undecidability}

The technique presented here makes it possible to prove only $\Sigma^0_1$\nobreakdash-hardness or $\Pi^0_1$\nobreakdash-hardness of fragments of logics. However, the monadic fragments of some logics are not contained in either $\Sigma^0_1$ or~$\Pi^0_1$. We give some examples.

For a modal predicate logic $L$, let us define\footnote{In general, it is more appropriate to define Kripke completion of~$L$ as the logic of the class of e\nobreakdash-augmented frames of~$L$; the use of this notion will be such that we will not feel the difference.} its \defnotion{Kripke completion} $L^{\mathit{K}}$ as the logic of the class of Kripke frames validating~$L$,
i.e., $L^{\mathit{K}} = \QMLext{\mathit{e}}{\mathit{all}} \ckf L$.
%:
%$$
%\begin{array}{lcl}
%L^{\mathit{K}} & = & \QMLext{\mathit{e}}{\mathit{all}} \ckf L.
%\end{array}
%$$
Let us consider logics $\logic{QGL}^{\mathit{K}}$, $\logic{QGrz}^{\mathit{K}}$, and $\logic{QwGrz}^{\mathit{K}}$. 
It is known~\cite{Montagna84,MR:2001:LI,RybIGPL22} that $\logic{QGL}^{\mathit{K}}\ne \logic{QGL}$, $\logic{QGrz}^{\mathit{K}}\ne \logic{QGrz}$, and $\logic{QwGrz}^{\mathit{K}}\ne \logic{QwGrz}$, i.e., $\logic{QGL}$, $\logic{QGrz}$, and $\logic{QwGrz}$ are Kripke incomplete. Using an estimate of computational complexity, we shall show this along the way, even for the monadic fragments; but our main goal is to draw attention to how high the lower estimates for the computational complexity of $\logic{QGL}^{\mathit{K}}$, $\logic{QGrz}^{\mathit{K}}$, and $\logic{QwGrz}^{\mathit{K}}$~are. 

\Rem{
It is not hard to see that
$$
\begin{array}{lclcl}
\logic{QGL} 
  & \subseteq & \logic{QGL}^{\mathit{K}} 
  & \subseteq & \logic{QGL}_{\mathit{wfin}};
  %\smallskip
  \\
\logic{QGrz} 
  & \subseteq & \logic{QGrz}^{\mathit{K}} 
  & \subseteq & \logic{QGrz}_{\mathit{wfin}};
  %\smallskip
  \\
\logic{QwGrz} 
  & \subseteq & \logic{QwGrz}^{\mathit{K}} 
  & \subseteq & \logic{QwGrz}_{\mathit{wfin}}.
  %\smallskip
  \\
\end{array}
$$
}

Logics $\logic{QGL}$, $\logic{QGrz}$, and $\logic{QwGrz}$ are recursively enumerable; since they are KHC-friendly, their monadic fragments are $\Sigma^0_1$\nobreakdash-hard, and hence $\Sigma^0_1$\nobreakdash-complete.
The monadic fragments of $\logic{QGL}_{\mathit{wfin}}$, $\logic{QGrz}_{\mathit{wfin}}$, and $\logic{QwGrz}_{\mathit{wfin}}$ are $\Pi^0_1$\nobreakdash-complete~\cite[Theorem~4.4]{RShsubmitted}. 
As for the monadic fragments of $\logic{QGL}^{\mathit{K}}$, $\logic{QGrz}^{\mathit{K}}$, and $\logic{QwGrz}^{\mathit{K}}$, they are $\Sigma^0_1$\nobreakdash-hard, since the logics are KHC-friendly. But the fragments are also $\Pi^0_1$\nobreakdash-hard~\cite{MR:2002:LI} and even $\Pi^1_1$\nobreakdash-hard in the language with two unary predicate letters and two individual variables~\cite[Theorem~3]{RybIGPL22}. As a corollary, indeed, the monadic fragments of $\logic{QGL}$, $\logic{QGrz}$, and $\logic{QwGrz}$ are Kripke incomplete.

Thus, the lower bounds for the computational complexity of the monadic fragments of some ``natural'' KHC-friendly modal predicate logics can be much higher than those that can be obtained using the theorems of Section~\ref{subsec:lbounds:monadic}. 
Near examples can be found in~\cite{RSh20AiML,RShJLC21c}; in particular, for $\QMLe \otuple{\numN,\leqslant}$ and $\QMLc \otuple{\numN,\leqslant}$. 

%Let us give another example. Consider the logic $\QMLe \otuple{\numN,\leqslant}$. Clearly, the Kripke frame $\otuple{\numN,\leqslant}$ satisfies KHC; hence, $\QMLe \otuple{\numN,\leqslant}$ is KHC-friendly, and therefore, its monadic fragment is $\Sigma^0_1$-hard. But it is known~\cite{RSh20AiML} that the monadic fragment of $\QMLe \otuple{\numN,\leqslant}$ with two individual variables is $\Pi^1_1$-hard. Notice that the same is true for logics $\QMLe \otuple{\numN,R}$ and $\QMLc \otuple{\numN,R}$ with $R$ lying between $<$ and~$\leqslant$, as well as for modal predicate logics of some infinite ordinals, even if the language contain a single unary predicate letter and a single proposition letter~\cite{RShJLC21c}.

\section{Conclusion}
\label{sec:conclusion}
\setcounter{equation}{0}

\subsection{Remarks and generalizations}

Notice that the methods and the results presented here can be quite easily transferred to some other classes of predicate logics; let us mention some of them here.

The ``classical'' part of the results is expandable to the free logic and its theories. This logic rejects the law $\forall x\,P(x)\to \exists x\,P(x)$ but this feature does not affect the construction described.  Also, we can apply the construction to logics and theories in the languages enriched with some non-elementary tools (for example, the transitive closure operator); in this case sometimes two individual variables are sufficient to obtain similar results~\cite{GOR-1999,MR:2022:DoklMath,MR:2023:LI} and even to prove $\Sigma^1_1$\nobreakdash-hardness of the satisfiability problem.

The ``modal'' part can be easily expanded to polymodal logics, to logics defined via varying domain semantics, and to non-normal modal logics (which are not extensions of logic~$\logic{QK}$). Moreover, in polymodal logics, it is possible to refine some results obtained for their monomodal counterparts. For example, in $\logic{QS5}_2$ with modalities $\Box_1$ and $\Box_2$, we can define the modality $\Box_1\Box_2$; clearly, on Kripke frames it corresponds to the relations being reflexive but not necessarily symmetric or transitive. So, the logic of its finite Kripke frames is undecidable in the language with two individual variables and a single unary predicate letter. In contrast, to prove the undecidability of $\logic{QS5}_{\mathit{wfin}}$, we used either three individual variables or three unary predicate letters. 

The ``intuitionistic'' part can be transferred in a straightforward way on logics with near semantics as well. For example, on predicate counterparts of basic propositional logic and formal propositional logic introduced by A.\,Visser~\cite{Visser81}. A~distinctive feature of the Kripke semantics for them is that it allows irreflexive worlds for the former logic and requires irreflexivity\footnote{More exactly, the Kripke frames for the logic are exactly strict Noetherian orders.} for the latter one. However, constructions similar to the one described here can work for it with slight technical modifications~\cite{RSh19SL}.

Also, observe that the results on recursive inseparability allow us to obtain then undecidability as $\Sigma^0_1$\nobreakdash-hardness or $\Pi^0_1$\nobreakdash-hardness only. As we have seen, some logics have higher complexity, even in languages with two-three individual variables and one-two unary predicate letters~\cite{RSh20AiML,RShJLC21c,RybIGPL22}. So, the results for them, similar to those presented here, may appear too weak and, in some cases, could be improved.

\subsection{Summary of results}

\settowidth{\templengtha}{\parbox{0.85\linewidth}{~}}
\settowidth{\templengthb}{\parbox{0.48\templengtha}{~}}
\settowidth{\templengthc}{\parbox{0.049\templengtha}{~}}
\settowidth{\templengthd}{\parbox{0.049\templengtha}{~}}
\settowidth{\templengthe}{\parbox{0.049\templengtha}{~}}
\settowidth{\templengthf}{\parbox{0.37\templengtha}{~}}

Let us summarize the main results presented here. 
%Below, we mention only the results on dyadic fragments of superclassical predicate logics (and theories) and on monadic fragments of modal and superintuitionistic predicate logics. 
Let, in the tables below, ``bpl'', ``upl'', and ``iv'' mean, respectively, ``binary predicate letters'', ``unary predicate letters'', and ``individual variables''.

Let us begin with superclassical logics and theories. Recall that $\logic{SIB}$ and $\logic{SRB}$ are the theories of the classes of models of, respectively, the symmetric irreflexive and symmetric reflexive binary relation; $T_{\mathit{fin}}$ is the theory of the class of finite models of a theory~$T$.
\begin{center}
\begin{tabular}{|c|c|c|c|c|}
\hline
\multirow{2}{*}
{
\parbox[c][\totalheight+5pt][c]{\templengthb}
       {\centering\bf Theories or logics}
} 
  & \multicolumn{3}{|c|}
    {
      \parbox[c][\totalheight+5pt][c]{0.15\templengtha}
             {\centering\bf Number of}
    }
  & \multirow{2}{*}
    {\parbox[c][\totalheight+5pt][c]{\templengthf}
      {\centering\bf Result obtained}
    }
\\
\cline{2-4}
  & \parbox[c][\totalheight+2pt][c]{\templengthc}
    {\centering\bf bpl}
  & \parbox[c][\totalheight+2pt][c]{\templengthd}
    {\centering\bf upl}
  & \parbox[c][\totalheight+2pt][c]{\templengthe}
    {\centering\bf $\!\!$\phantom{p}iv\phantom{p}$\!\!$}
  & 
\\
\hline
\hline
  \parbox[c][\totalheight+5pt][c]{\templengthb}
  {%\centering
    \begin{tabular}{ll}
    %{either} 
    & $\logic{QCl}^{\mathit{bin}}\subseteq \Gamma \subseteq \Gamma' \subseteq \logic{SIB}_{\mathit{fin}}$ \\
    {or} 
    & $\logic{QCl}^{\mathit{bin}}\subseteq \Gamma \subseteq \Gamma' \subseteq \logic{SRB}_{\mathit{fin}}$
    \end{tabular}
  }
  & 1
  & ---
  & 3
  & \parbox[c][\totalheight+5pt][c]{\templengthf}
    {\centering
      $\Gamma$ and $\Gamma'$ \\
      are 
      recursively 
      inseparable
    }
  \\
\hline
  \parbox[c][\totalheight+5pt][c]{\templengthb}
  {%\centering
    \begin{tabular}{ll}
    %{either} 
    & $\logic{QCl}^{\mathit{bin}}\subseteq \Gamma \subseteq \logic{SIB}$ \\
    {or} 
    & $\logic{QCl}^{\mathit{bin}}\subseteq \Gamma \subseteq \logic{SRB}$
    \end{tabular}
  }
  & 1
  & ---
  & 3
  & \parbox[c][\totalheight+5pt][c]{\templengthf}
    {\centering
      $\Gamma$ is $\Sigma^0_1$-hard
    }
  \\
\hline
  \parbox[c][\totalheight+5pt][c]{\templengthb}
  {%\centering
    \begin{tabular}{ll}
    %{either} 
    & $\logic{QCl}^{\mathit{bin}}_{\mathit{fin}}\subseteq \Gamma \subseteq \logic{SIB}_{\mathit{fin}}$ \\[2pt]
    {or} 
    & $\logic{QCl}^{\mathit{bin}}_{\mathit{fin}}\subseteq \Gamma \subseteq  \logic{SRB}_{\mathit{fin}}$
    \end{tabular}
  }
  & 1
  & ---
  & 3
  & \parbox[c][\totalheight+5pt][c]{\templengthf}
    {\centering
      $\Gamma$ is $\Pi^0_1$-hard
    }
  \\
\hline
\end{tabular}
\end{center}

Next, let us mention results on modal predicate logics.\footnote{See an addition in Section~\ref{sec:note}.} 
\begin{center}
\begin{tabular}{|c|c|c|c|c|}
\hline
\multirow{2}{*}
{
\parbox[c][\totalheight+5pt][c]{\templengthb}
       {\centering\bf Theories or logics}
} 
  & \multicolumn{3}{|c|}
    {
      \parbox[c][\totalheight+5pt][c]{0.15\templengtha}
             {\centering\bf Number of}
    }
  & \multirow{2}{*}
    {\parbox[c][\totalheight+5pt][c]{\templengthf}
      {\centering\bf Result obtained}
    }
\\
\cline{2-4}
  & \parbox[c][\totalheight+2pt][c]{\templengthc}
    {\centering\bf bpl}
  & \parbox[c][\totalheight+2pt][c]{\templengthd}
    {\centering\bf upl}
  & \parbox[c][\totalheight+2pt][c]{\templengthe}
    {\centering\bf $\!\!$\phantom{p}iv\phantom{p}$\!\!$}
  & 
\\
\hline
\hline
  \parbox[c][\totalheight+5pt][c]{\templengthb}
  {%\centering
    \begin{tabular}{ll}
    %{either} 
    & $\logic{QK} \subseteq L \subseteq \logic{QTriv}$ \\
    {or}
    & $\logic{QK} \subseteq L \subseteq \logic{QVer}$ \\
	\end{tabular}
    \\
    \begin{tabular}{l}
	with $\fin \ckf L$ be recursive 
    \end{tabular}
  }
  & ---
  & $\infty$
  & $\infty$
  & \parbox[c][\totalheight+5pt][c]{\templengthf}
    {\centering
	  \begin{tabular}{ll}
      both & \!\!\!$\QMLe\ckf(L\oplus\bm{alt}_n)$ \\ 
	  and  & \!\!\!$\QMLc\ckf(L\oplus\bm{alt}_n)$ \\ 
	  \end{tabular}
	  \\
	  \begin{tabular}{l}
	  are decidable
	  \end{tabular}
    }
  \\
\hline
\hline
\multicolumn{5}{|c|}
{ 
  \parbox[c][\totalheight+5pt][c]{1.12\templengtha}
  {\centering 
    Let $\logic{QL}$ be one of
    $\logic{QwGrz.bf}\oplus\bm{bd}_3$,
    $\logic{QGL.bf}\oplus\bm{bd}_4$, 
	$\logic{QGrz.bf}\oplus\bm{bd}_4$
  }
} 
\\
\hline
  \parbox[c][\totalheight+5pt][c]{\templengthb}
  {%\centering
    \begin{tabular}{ll}
    %{either} 
    & $\logic{QK} \subseteq L \subseteq L' \subseteq \logic{QL}_{\mathit{dfin}}$ \\
    {or}
    & $\logic{QK} \subseteq L \subseteq L' \subseteq \logic{QL}_{\mathit{wfin}}$
    \end{tabular}
%    \\
%    \begin{tabular}{l}
%		with $\logic{L}$ be one of $\logic{QS5}$, \\
%    $\logic{QGL.bf}\oplus\bm{bw}_1$, $\logic{QGrz.bf}\oplus\bm{bw}_1$, \\ 
%    $\logic{QGL.bf}\oplus\bm{bd}_2$, $\logic{QGrz.bf}\oplus\bm{bd}_2$\phantom{$,$} 
%    \end{tabular}
  }
  & ---
  & 1
  & 2
  & \parbox[c][\totalheight+5pt][c]{\templengthf}
    {\centering
      $L$ and $L'$ \\
      are 
      recursively 
      inseparable
    }
  \\
\hline
  \parbox[c][\totalheight+5pt][c]{\templengthb}
  {%\centering
    \begin{tabular}{ll}
    %{either} 
    \phantom{or}
    & $\logic{QK} \subseteq L \subseteq \logic{QL}$
    \end{tabular}
%    \\
%    \begin{tabular}{l}
%		with $\logic{L}$ be one of $\logic{QS5}$, \\
%    $\logic{QGL.bf}\oplus\bm{bw}_1$, $\logic{QGrz.bf}\oplus\bm{bw}_1$, \\ 
%    $\logic{QGL.bf}\oplus\bm{bd}_2$, $\logic{QGrz.bf}\oplus\bm{bd}_2$\phantom{$,$} 
%    \end{tabular}
  }
  & ---
  & 1
  & 2
  & \parbox[c][\totalheight+5pt][c]{\templengthf}
    {\centering
      $L$ is $\Sigma^0_1$-hard
    }
  \\
\hline
  \parbox[c][\totalheight+5pt][c]{\templengthb}
  {%\centering
    \begin{tabular}{ll}
    %{either} 
    & $\logic{QK}_{\mathit{dfin}} \subseteq L \subseteq \logic{QL}_{\mathit{dfin}}$ \\
    {or}
    & $\logic{QK}_{\mathit{wfin}} \subseteq L \subseteq \logic{QL}_{\mathit{wfin}}$
    \end{tabular}
%    \\
%    \begin{tabular}{l}
%		with $\logic{L}$ be one of $\logic{QS5}$, \\
%    $\logic{QGL.bf}\oplus\bm{bw}_1$, $\logic{QGrz.bf}\oplus\bm{bw}_1$, \\ 
%    $\logic{QGL.bf}\oplus\bm{bd}_2$, $\logic{QGrz.bf}\oplus\bm{bd}_2$\phantom{$,$} 
%    \end{tabular}
  }
  & ---
  & 1
  & 2
  & \parbox[c][\totalheight+5pt][c]{\templengthf}
    {\centering
      $L$ is $\Pi^0_1$-hard
    }
  \\
\hline
\hline
\multicolumn{5}{|c|}
{ 
  \parbox[c][\totalheight+5pt][c]{1.12\templengtha}
  {\centering 
    Let $\logic{QL}$ be one of $\logic{QS5}$, 
    $\logic{QGL.bf}\oplus\bm{bw}_1$, $\logic{QGrz.bf}\oplus\bm{bw}_1$,  
    $\logic{QGL.bf}\oplus\bm{bd}_3$, $\logic{QGrz.bf}\oplus\bm{bd}_3$
  }
} 
\\
\hline
  \parbox[c][\totalheight+5pt][c]{\templengthb}
  {%\centering
    \begin{tabular}{ll}
    %{either} 
    & $\logic{QK} \subseteq L \subseteq L' \subseteq \logic{QL}_{\mathit{dfin}}$ \\
    {or}
    & $\logic{QK} \subseteq L \subseteq L' \subseteq \logic{QL}_{\mathit{wfin}}$
    \end{tabular}
%    \\
%    \begin{tabular}{l}
%		with $\logic{L}$ be one of $\logic{QS5}$, \\
%    $\logic{QGL.bf}\oplus\bm{bw}_1$, $\logic{QGrz.bf}\oplus\bm{bw}_1$, \\ 
%    $\logic{QGL.bf}\oplus\bm{bd}_2$, $\logic{QGrz.bf}\oplus\bm{bd}_2$\phantom{$,$} 
%    \end{tabular}
  }
  & ---
  & 3
  & 2
  & \parbox[c][\totalheight+5pt][c]{\templengthf}
    {\centering
      $L$ and $L'$ \\
      are 
      recursively 
      inseparable
    }
  \\
\hline
  \parbox[c][\totalheight+5pt][c]{\templengthb}
  {%\centering
    \begin{tabular}{ll}
    %{either} 
    \phantom{or}
    & $\logic{QK} \subseteq L \subseteq \logic{QL}$
    \end{tabular}
%    \\
%    \begin{tabular}{l}
%		with $\logic{L}$ be one of $\logic{QS5}$, \\
%    $\logic{QGL.bf}\oplus\bm{bw}_1$, $\logic{QGrz.bf}\oplus\bm{bw}_1$, \\ 
%    $\logic{QGL.bf}\oplus\bm{bd}_2$, $\logic{QGrz.bf}\oplus\bm{bd}_2$\phantom{$,$} 
%    \end{tabular}
  }
  & ---
  & 3
  & 2
  & \parbox[c][\totalheight+5pt][c]{\templengthf}
    {\centering
      $L$ is $\Sigma^0_1$-hard
    }
  \\
\hline
  \parbox[c][\totalheight+5pt][c]{\templengthb}
  {%\centering
    \begin{tabular}{ll}
    %{either} 
    & $\logic{QK}_{\mathit{dfin}} \subseteq L \subseteq \logic{QL}_{\mathit{dfin}}$ \\
    {or}
    & $\logic{QK}_{\mathit{wfin}} \subseteq L \subseteq \logic{QL}_{\mathit{wfin}}$
    \end{tabular}
%    \\
%    \begin{tabular}{l}
%		with $\logic{L}$ be one of $\logic{QS5}$, \\
%    $\logic{QGL.bf}\oplus\bm{bw}_1$, $\logic{QGrz.bf}\oplus\bm{bw}_1$, \\ 
%    $\logic{QGL.bf}\oplus\bm{bd}_2$, $\logic{QGrz.bf}\oplus\bm{bd}_2$\phantom{$,$} 
%    \end{tabular}
  }
  & ---
  & 3
  & 2
  & \parbox[c][\totalheight+5pt][c]{\templengthf}
    {\centering
      $L$ is $\Pi^0_1$-hard
    }
  \\
\hline
\hline
\multicolumn{5}{|c|}
{ 
  \parbox[c][\totalheight+5pt][c]{1.12\templengtha}
  {\centering 
    Let $\logic{QL}$ be one of $\logic{QS5}$, 
    $\logic{QGL.bf}\oplus\bm{bw}_1$, $\logic{QGrz.bf}\oplus\bm{bw}_1$,  
    $\logic{QGL.bf}\oplus\bm{bd}_2$, $\logic{QGrz.bf}\oplus\bm{bd}_2$
  }
} 
\\
\hline
  \parbox[c][\totalheight+5pt][c]{\templengthb}
  {%\centering
    \begin{tabular}{ll}
    %{either} 
    & $\logic{QCl}^\Box \subseteq L \subseteq L' \subseteq \logic{QL}_{\mathit{dfin}}$ \\
    {or}
    & $\logic{QCl}^\Box \subseteq L \subseteq L' \subseteq \logic{QL}_{\mathit{wfin}}$
    \end{tabular}
%    \\
%    \begin{tabular}{l}
%		with $\logic{L}$ be one of $\logic{QS5}$, \\
%    $\logic{QGL.bf}\oplus\bm{bw}_1$, $\logic{QGrz.bf}\oplus\bm{bw}_1$, \\ 
%    $\logic{QGL.bf}\oplus\bm{bd}_2$, $\logic{QGrz.bf}\oplus\bm{bd}_2$\phantom{$,$} 
%    \end{tabular}
  }
  & ---
  & 1
  & 3
  & \parbox[c][\totalheight+5pt][c]{\templengthf}
    {\centering
      $L$ and $L'$ \\
      are 
      recursively 
      inseparable
    }
  \\
\hline
  \parbox[c][\totalheight+5pt][c]{\templengthb}
  {%\centering
    \begin{tabular}{ll}
    %{either} 
    \phantom{or}
    & $\logic{QK} \subseteq L \subseteq \logic{QL}$
    \end{tabular}
%    \\
%    \begin{tabular}{l}
%		with $\logic{L}$ be one of $\logic{QS5}$, \\
%    $\logic{QGL.bf}\oplus\bm{bw}_1$, $\logic{QGrz.bf}\oplus\bm{bw}_1$, \\ 
%    $\logic{QGL.bf}\oplus\bm{bd}_2$, $\logic{QGrz.bf}\oplus\bm{bd}_2$\phantom{$,$} 
%    \end{tabular}
  }
  & ---
  & 1
  & 3
  & \parbox[c][\totalheight+5pt][c]{\templengthf}
    {\centering
      $L$ is $\Sigma^0_1$-hard
    }
  \\
\hline
  \parbox[c][\totalheight+5pt][c]{\templengthb}
  {%\centering
    \begin{tabular}{ll}
    %{either} 
    & $\logic{QK}_{\mathit{dfin}} \subseteq L \subseteq \logic{QL}_{\mathit{dfin}}$ \\
    {or}
    & $\logic{QK}_{\mathit{wfin}} \subseteq L \subseteq \logic{QL}_{\mathit{wfin}}$
    \end{tabular}
%    \\
%    \begin{tabular}{l}
%		with $\logic{L}$ be one of $\logic{QS5}$, \\
%    $\logic{QGL.bf}\oplus\bm{bw}_1$, $\logic{QGrz.bf}\oplus\bm{bw}_1$, \\ 
%    $\logic{QGL.bf}\oplus\bm{bd}_2$, $\logic{QGrz.bf}\oplus\bm{bd}_2$\phantom{$,$} 
%    \end{tabular}
  }
  & ---
  & 1
  & 3
  & \parbox[c][\totalheight+5pt][c]{\templengthf}
    {\centering
      $L$ is $\Pi^0_1$-hard
    }
  \\
\hline
  \parbox[c][\totalheight+5pt][c]{\templengthb}
  {%\centering
    \begin{tabular}{ll}
    %{either} 
    %& $\logic{QK}_{\mathit{dfin}} \subseteq L \subseteq \logic{QL}_{\mathit{dfin}}$ \\
    \phantom{or}
    & $\logic{QK}_{\mathit{wfin}} \subseteq L \subseteq \logic{QL}_{\mathit{wfin}}$
    \end{tabular}
%    \\
%    \begin{tabular}{l}
%		with $\logic{L}$ be one of $\logic{QS5}$, \\
%    $\logic{QGL.bf}\oplus\bm{bw}_1$, $\logic{QGrz.bf}\oplus\bm{bw}_1$, \\ 
%    $\logic{QGL.bf}\oplus\bm{bd}_2$, $\logic{QGrz.bf}\oplus\bm{bd}_2$\phantom{$,$} 
%    \end{tabular}
  }
  & 1
  & ---
  & 3
  & \parbox[c][\totalheight+5pt][c]{\templengthf}
    {\centering
      $L$ is $\Sigma^0_1$-hard and $\Pi^0_1$-hard
    }
  \\
\hline
\end{tabular}
\end{center}

For convenience, let us recall that $\logic{QCl}^\Box$ is the classical predicate logic in the language enriched with~$\Box$; $L_{\mathit{dfin}}$ is the logic of the class of e-augmented frames of logic~$L$ with finite local domains; $L_{\mathit{wfin}}$ is the logic of the class of finite Kripke frames of~$L$; $\fin\ckf L$ is the class of finite Kripke frames of~$L$; $\QMLe\ckf L$ and $\QMLc\ckf L$ are the logics of the classes of, respectively, e\nobreakdash-augmented and c\nobreakdash-augmented frames defined on Kripke frames of~$L$.

Finally, let us turn to the class of superintuitionistic predicate logics. The notation is similar to the modal case.
\begin{center}
\begin{tabular}{|c|c|c|c|c|}
\hline
\multirow{2}{*}
{
\parbox[c][\totalheight+5pt][c]{\templengthb}
       {\centering\bf Theories or logics}
} 
  & \multicolumn{3}{|c|}
    {
      \parbox[c][\totalheight+5pt][c]{0.15\templengtha}
             {\centering\bf Number of}
    }
  & \multirow{2}{*}
    {\parbox[c][\totalheight+5pt][c]{\templengthf}
      {\centering\bf Result obtained}
    }
\\
\cline{2-4}
  & \parbox[c][\totalheight+2pt][c]{\templengthc}
    {\centering\bf bpl}
  & \parbox[c][\totalheight+2pt][c]{\templengthd}
    {\centering\bf upl}
  & \parbox[c][\totalheight+2pt][c]{\templengthe}
    {\centering\bf $\!\!$\phantom{p}iv\phantom{p}$\!\!$}
  & 
\\
\hline
\hline
\multicolumn{5}{|c|}
{ 
  \parbox[c][\totalheight+5pt][c]{1.12\templengtha}
  {\centering 
    Let $\logic{QL}$ be 
    $\logic{QInt}+\bm{cd}+\bm{bd}_7$ or
    $\logic{QKC}+\bm{cd}+\bm{bd}_8$ 
  }
} 
\\
\hline
  \parbox[c][\totalheight+5pt][c]{\templengthb}
  {%\centering
    \begin{tabular}{ll}
    %{either} 
    & $\logic{QInt} \subseteq L \subseteq L' \subseteq \logic{QL}_{\mathit{dfin}}$ \\
    {or}
    & $\logic{QInt} \subseteq L \subseteq L' \subseteq \logic{QL}_{\mathit{wfin}}$
    \end{tabular}
%    \\
%    \begin{tabular}{l}
%		with $\logic{L}$ be one of $\logic{QS5}$, \\
%    $\logic{QGL.bf}\oplus\bm{bw}_1$, $\logic{QGrz.bf}\oplus\bm{bw}_1$, \\ 
%    $\logic{QGL.bf}\oplus\bm{bd}_2$, $\logic{QGrz.bf}\oplus\bm{bd}_2$\phantom{$,$} 
%    \end{tabular}
  }
  & ---
  & 1
  & 2
  & \parbox[c][\totalheight+5pt][c]{\templengthf}
    {\centering
      $L$ and $L'$ \\
      are 
      recursively 
      inseparable
    }
  \\
\hline
  \parbox[c][\totalheight+5pt][c]{\templengthb}
  {%\centering
    \begin{tabular}{ll}
    %{either} 
    \phantom{or}
    & $\logic{QInt} \subseteq L \subseteq \logic{QL}$
    \end{tabular}
%    \\
%    \begin{tabular}{l}
%		with $\logic{L}$ be one of $\logic{QS5}$, \\
%    $\logic{QGL.bf}\oplus\bm{bw}_1$, $\logic{QGrz.bf}\oplus\bm{bw}_1$, \\ 
%    $\logic{QGL.bf}\oplus\bm{bd}_2$, $\logic{QGrz.bf}\oplus\bm{bd}_2$\phantom{$,$} 
%    \end{tabular}
  }
  & ---
  & 1
  & 2
  & \parbox[c][\totalheight+5pt][c]{\templengthf}
    {\centering
      $L$ is $\Sigma^0_1$-hard
    }
  \\
\hline
  \parbox[c][\totalheight+5pt][c]{\templengthb}
  {%\centering
    \begin{tabular}{ll}
    %{either} 
    & $\logic{QInt}_{\mathit{dfin}} \subseteq L \subseteq \logic{QL}_{\mathit{dfin}}$ \\
    {or}
    & $\logic{QInt}_{\mathit{wfin}} \subseteq L \subseteq \logic{QL}_{\mathit{wfin}}$
    \end{tabular}
%    \\
%    \begin{tabular}{l}
%		with $\logic{L}$ be one of $\logic{QS5}$, \\
%    $\logic{QGL.bf}\oplus\bm{bw}_1$, $\logic{QGrz.bf}\oplus\bm{bw}_1$, \\ 
%    $\logic{QGL.bf}\oplus\bm{bd}_2$, $\logic{QGrz.bf}\oplus\bm{bd}_2$\phantom{$,$} 
%    \end{tabular}
  }
  & ---
  & 1
  & 2
  & \parbox[c][\totalheight+5pt][c]{\templengthf}
    {\centering
      $L$ is $\Pi^0_1$-hard
    }
  \\
\hline
\hline
\multicolumn{5}{|c|}
{ 
  \parbox[c][\totalheight+5pt][c]{1.12\templengtha}
  {\centering 
    Let $\logic{QL}$ be 
    $\logic{QInt}+\bm{cd}+\bm{bd}_3$ or
    $\logic{QKC}+\bm{cd}+\bm{bd}_4$ 
  }
} 
\\
\hline
  \parbox[c][\totalheight+5pt][c]{\templengthb}
  {%\centering
    \begin{tabular}{ll}
    %{either} 
    & $\logic{QInt} \subseteq L \subseteq L' \subseteq \logic{QL}_{\mathit{dfin}}$ \\
    {or}
    & $\logic{QInt} \subseteq L \subseteq L' \subseteq \logic{QL}_{\mathit{wfin}}$
    \end{tabular}
%    \\
%    \begin{tabular}{l}
%		with $\logic{L}$ be one of $\logic{QS5}$, \\
%    $\logic{QGL.bf}\oplus\bm{bw}_1$, $\logic{QGrz.bf}\oplus\bm{bw}_1$, \\ 
%    $\logic{QGL.bf}\oplus\bm{bd}_2$, $\logic{QGrz.bf}\oplus\bm{bd}_2$\phantom{$,$} 
%    \end{tabular}
  }
  & ---
  & 2
  & 2
  & \parbox[c][\totalheight+5pt][c]{\templengthf}
    {\centering
      $L$ and $L'$ \\
      are 
      recursively 
      inseparable
    }
  \\
\hline
  \parbox[c][\totalheight+5pt][c]{\templengthb}
  {%\centering
    \begin{tabular}{ll}
    %{either} 
    \phantom{or}
    & $\logic{QInt} \subseteq L \subseteq \logic{QL}$
    \end{tabular}
%    \\
%    \begin{tabular}{l}
%		with $\logic{L}$ be one of $\logic{QS5}$, \\
%    $\logic{QGL.bf}\oplus\bm{bw}_1$, $\logic{QGrz.bf}\oplus\bm{bw}_1$, \\ 
%    $\logic{QGL.bf}\oplus\bm{bd}_2$, $\logic{QGrz.bf}\oplus\bm{bd}_2$\phantom{$,$} 
%    \end{tabular}
  }
  & ---
  & 2
  & 2
  & \parbox[c][\totalheight+5pt][c]{\templengthf}
    {\centering
      $L$ is $\Sigma^0_1$-hard
    }
  \\
\hline
  \parbox[c][\totalheight+5pt][c]{\templengthb}
  {%\centering
    \begin{tabular}{ll}
    %{either} 
    & $\logic{QInt}_{\mathit{dfin}} \subseteq L \subseteq \logic{QL}_{\mathit{dfin}}$ \\
    {or}
    & $\logic{QInt}_{\mathit{wfin}} \subseteq L \subseteq \logic{QL}_{\mathit{wfin}}$
    \end{tabular}
%    \\
%    \begin{tabular}{l}
%		with $\logic{L}$ be one of $\logic{QS5}$, \\
%    $\logic{QGL.bf}\oplus\bm{bw}_1$, $\logic{QGrz.bf}\oplus\bm{bw}_1$, \\ 
%    $\logic{QGL.bf}\oplus\bm{bd}_2$, $\logic{QGrz.bf}\oplus\bm{bd}_2$\phantom{$,$} 
%    \end{tabular}
  }
  & ---
  & 2
  & 2
  & \parbox[c][\totalheight+5pt][c]{\templengthf}
    {\centering
      $L$ is $\Pi^0_1$-hard
    }
  \\
\hline
\hline
\multicolumn{5}{|c|}
{ 
  \parbox[c][\totalheight+5pt][c]{1.12\templengtha}
  {\centering 
    Let $\logic{QL}$ be 
    $\logic{QInt}+\bm{cd}+\bm{bd}_2$ or
    $\logic{QKC}+\bm{cd}+\bm{bd}_3$ 
  }
} 
\\
\hline
  \parbox[c][\totalheight+5pt][c]{\templengthb}
  {%\centering
    \begin{tabular}{ll}
    %{either} 
    & $\logic{QInt} \subseteq L \subseteq L' \subseteq \logic{QL}_{\mathit{dfin}}$ \\
    {or}
    & $\logic{QInt} \subseteq L \subseteq L' \subseteq \logic{QL}_{\mathit{wfin}}$
    \end{tabular}
%    \\
%    \begin{tabular}{l}
%		with $\logic{L}$ be one of $\logic{QS5}$, \\
%    $\logic{QGL.bf}\oplus\bm{bw}_1$, $\logic{QGrz.bf}\oplus\bm{bw}_1$, \\ 
%    $\logic{QGL.bf}\oplus\bm{bd}_2$, $\logic{QGrz.bf}\oplus\bm{bd}_2$\phantom{$,$} 
%    \end{tabular}
  }
  & ---
  & 1
  & 3
  & \parbox[c][\totalheight+5pt][c]{\templengthf}
    {\centering
      $L$ and $L'$ \\
      are 
      recursively 
      inseparable
    }
  \\
\hline
  \parbox[c][\totalheight+5pt][c]{\templengthb}
  {%\centering
    \begin{tabular}{ll}
    %{either} 
    \phantom{or}
    & $\logic{QInt} \subseteq L \subseteq \logic{QL}$
    \end{tabular}
%    \\
%    \begin{tabular}{l}
%		with $\logic{L}$ be one of $\logic{QS5}$, \\
%    $\logic{QGL.bf}\oplus\bm{bw}_1$, $\logic{QGrz.bf}\oplus\bm{bw}_1$, \\ 
%    $\logic{QGL.bf}\oplus\bm{bd}_2$, $\logic{QGrz.bf}\oplus\bm{bd}_2$\phantom{$,$} 
%    \end{tabular}
  }
  & ---
  & 1
  & 3
  & \parbox[c][\totalheight+5pt][c]{\templengthf}
    {\centering
      $L$ is $\Sigma^0_1$-hard
    }
  \\
\hline
  \parbox[c][\totalheight+5pt][c]{\templengthb}
  {%\centering
    \begin{tabular}{ll}
    %{either} 
    & $\logic{QInt}_{\mathit{dfin}} \subseteq L \subseteq \logic{QL}_{\mathit{dfin}}$ \\
    {or}
    & $\logic{QInt}_{\mathit{wfin}} \subseteq L \subseteq \logic{QL}_{\mathit{wfin}}$
    \end{tabular}
%    \\
%    \begin{tabular}{l}
%		with $\logic{L}$ be one of $\logic{QS5}$, \\
%    $\logic{QGL.bf}\oplus\bm{bw}_1$, $\logic{QGrz.bf}\oplus\bm{bw}_1$, \\ 
%    $\logic{QGL.bf}\oplus\bm{bd}_2$, $\logic{QGrz.bf}\oplus\bm{bd}_2$\phantom{$,$} 
%    \end{tabular}
  }
  & ---
  & 1
  & 3
  & \parbox[c][\totalheight+5pt][c]{\templengthf}
    {\centering
      $L$ is $\Pi^0_1$-hard
    }
  \\
\hline
  \parbox[c][\totalheight+5pt][c]{\templengthb}
  {%\centering
    \begin{tabular}{ll}
    %{either} 
    %& $\logic{QK}_{\mathit{dfin}} \subseteq L \subseteq \logic{QL}_{\mathit{dfin}}$ \\
    \phantom{or}
    & $\logic{QInt}_{\mathit{wfin}} \subseteq L \subseteq \logic{QL}_{\mathit{wfin}}$
    \end{tabular}
%    \\
%    \begin{tabular}{l}
%		with $\logic{L}$ be one of $\logic{QS5}$, \\
%    $\logic{QGL.bf}\oplus\bm{bw}_1$, $\logic{QGrz.bf}\oplus\bm{bw}_1$, \\ 
%    $\logic{QGL.bf}\oplus\bm{bd}_2$, $\logic{QGrz.bf}\oplus\bm{bd}_2$\phantom{$,$} 
%    \end{tabular}
  }
  & 1
  & ---
  & 3
  & \parbox[c][\totalheight+5pt][c]{\templengthf}
    {\centering
      $L$ is $\Sigma^0_1$-hard and $\Pi^0_1$-hard
    }
  \\
\hline
\end{tabular}
\end{center}

Note that the results shown in the tables have clarifications. So, in many cases, the results remain true for languages without the constant~$\bot$, which means without~$\neg$ and, in the modal case, without~$\Diamond$.

\subsection{Future work}

There are several questions related to the algorithmic complexity of fragments of predicate logics that remain unanswered so far. It seems that the solutions to some of them are a matter of technique, but for some of them, apparently, this is not the case.

Let us start with $\logic{QS5}$ and near logics such as $\logic{QK5}$, $\logic{QK45}$, $\logic{QKD45}$, and $\logic{QK4B}$. Our results do not give an answer on the decidability of $\logic{QS5}$ if its language contains only two individual variables and one or two unary predicate letters. However, as we have seen, $\logic{QS5}$ with three individual variables and a single unary predicate letter is undecidable, and the same for two individual variables and three unary predicate letters. It seems that one of three unary predicate letters can be eliminated; but all attempts of the author to do this have so far failed.

\begin{conjecture}
\label{conj:1}
The fragment of\/ $\logic{QS5}$ in the language with two individual variables and two unary predicate letters is algorithmically undecidable.
\end{conjecture}

As for two variables and a single unary predicate letter, the fragment of $\logic{QS5}$ in this language seems to be decidable. Moreover, it seems to the author that the fragment is finitely approximable, in particular, $\logic{QS5}$, $\logic{QS5}_{\mathit{dfin}}$, and $\logic{QS5}_{\mathit{wfin}}$ do not differ by formulas of the fragment.

\begin{conjecture}
\label{conj:2}
The fragment of\/ $\logic{QS5}$ in the language with two individual variables and a single unary predicate letter is algorithmically decidable.
\end{conjecture}

In both conjectures, $\logic{QS5}$ can be replaced with  $\logic{QK5}$, $\logic{QK45}$, $\logic{QKD45}$ or $\logic{QK4B}$. Also observe that if one of the conjectures is wrong, then so is the other.

Next, let us concern the logics of Kripke frames of bounded depth. What happens if we bound the depth of frames by a quite small number? For example, what can we say about the decidability of the fragment of $\logic{QK}\oplus\bm{bd}_2$ in the language with a single unary predicate letter and two individual variables? What about similar fragments of $\logic{QGL}\oplus\bm{bd}_3$, $\logic{QGrz}\oplus\bm{bd}_3$, $\logic{QInt}\oplus\bm{bd}_6$ or $\logic{QKC}\oplus\bm{bd}_7$? It seems to the author that at least some of them are undecidable.

Another class we would like to talk about is the class of predicate logics defined by frames of bounded width. There are known results showing that such logics can be undecidable~--- and even highly undecidable~--- in languages with one-two unary predicate letters and two-three individual variables. It would be interesting to expand the ideas presented here on such logics in a regular way.

Finally, let us turn to the semantics used. 
To obtain the results presented here, we did not use all v\nobreakdash-augmented frames but only e\nobreakdash-augmented and c\nobreakdash-augmented ones. However, v\nobreakdash-augmented frames allow us to distinguish more subtle properties described by modal predicate formulas. For example, the \defnotion{converse Barcan formula} $\text{\textbf{\textit{cbf}}}=\Box\forall x\,P(x)\to\forall x\,\Box P(x)$ is valid on the class e\nobreakdash-augmented frames (and hence, on the class of c-augmented ones) but it is refuted on a v\nobreakdash-augmented frame containing a world seeing a different world. It seems that the results obtained and the methods used should be transferred to this semantics, however, this requires a separate study.

In general, it is known that Kripke semantics is a rather weak tool for investigating predicate non-classical logics. There are a lot of results showing that predicate counterparts of Kripke complete propositional logics can be (and often are) Kripke incomplete. Of course, the technique presented here uses mainly soundness, not Kripke completeness; nevertheless, it seems reasonable to expect that using other semantics will yield more general results.

%At the same time, at a certain stage of the proof, we will use the same method that was applied earlier~\cite{??} and originates in studies related to the complexity of superintuitionistic propositional logics~\cite{??}. Initially, it was not assumed that it could be applied in logics not contained in~$\logic{QKC}$ (or $\logic{KC}$, in propositional case), but this turned out not to be the case, and perhaps this observation will be able to be applied in future studies.

\section{Note}
\label{sec:note}
\setcounter{equation}{0}

After the paper was prepared, the author obtained a result from which it follows that Conjecture~\ref{conj:1} is true, but Conjecture~\ref{conj:2} is~not. The result is presented in the following theorem.

\begin{theorem}
\label{th:ml:insep:QK:QS5wfin}
Logics\/ $\logic{QK}$ and\/ $\logic{QS5}_{\mathit{wfin}}$ as well as\/ $\logic{QK}$ and\/ $\logic{QS5}_{\mathit{dfin}}$ are recursively inseparable in the language with a single unary predicate letter and two individual variables.
\end{theorem}

The proof is based on similar methods, but requires a different modeling; the author suggests presenting it in a separate paper.

In Theorem~\ref{th:ml:insep:QK:QS5wfin}, logic $\logic{QK}$ can be replaced with any logic between $\logic{QK}$ and $\logic{QS5}$; so, let us replace it with $\logic{QS5}$ in order to deduce one interesting corollary.

\begin{theorem}
\label{th:ml:insep:QS5:QS5wfin}
Logics\/ $\logic{QS5}$ and\/ $\logic{QS5}_{\mathit{wfin}}$ as well as\/ $\logic{QK}$ and\/ $\logic{QS5}_{\mathit{dfin}}$ are recursively inseparable in the language with a single unary predicate letter and two individual variables.
\end{theorem}

Note that the atomic formulas $Q(x)$ and $Q(y)$ can be modeled in classical predicate logic by formulas $P(x,z)$ and $P(y,z)$, respectively. The intended meaning of $P(x,z)$ and $P(y,z)$ is that $Q(x)$ and $Q(y)$ are true in a world~$z$. Let us define the following translation $\mathit{ST}$ of the modal formulas built from $Q(x)$ and $Q(y)$ into the classical ones:
$$
\begin{array}{lcll}
\mathit{ST}(\bot) & = & \bot; \\
\mathit{ST}(Q(v)) & = & P(v,z), & \mbox{where $v\in\set{x,y}$}; \\
\mathit{ST}(\varphi\wedge\psi) & = & \mathit{ST}(\varphi)\wedge \mathit{ST}(\psi); \\
\mathit{ST}(\varphi\vee\psi) & = & \mathit{ST}(\varphi)\vee \mathit{ST}(\psi); \\
\mathit{ST}(\varphi\to\psi) & = & \mathit{ST}(\varphi)\to \mathit{ST}(\psi); \\
\mathit{ST}(\Box\varphi) & = & \forall z\,\mathit{ST}(\varphi); \\
\mathit{ST}(\exists v\,\varphi) & = & \exists v\,\mathit{ST}(\varphi), & \mbox{where $v\in\set{x,y}$}; \\
\mathit{ST}(\forall v\,\varphi) & = & \forall v\,\mathit{ST}(\varphi), & \mbox{where $v\in\set{x,y}$}. \\
\end{array}
$$

In fact, $\mathit{ST}$ is the standard translation, see~\cite{Benthem85,ChRyb:2000,GShS}, adapted to~$\logic{QS5}$; so, the following statement should be obvious.

\begin{lemma}
\label{lem:QS5-Q2var:to:QCl}
For every\/ $\lang{ML}$-formula~$\varphi$ with a single unary predicate letter~$Q$ and individual variables $x$ and $y$, the following equivalences hold:
$$
\begin{array}{lcl}
\varphi\in\logic{QS5} & \iff & \mathit{ST}(\varphi) \in \logic{QCl}; \\
\varphi\in\logic{QS5}_{\mathit{wfin}} & \iff & \mathit{ST}(\varphi) \in \logic{QCl}_{\mathit{fin}}. \\
\end{array}
$$
\end{lemma}

Then, we obtain a refinement of the Trakhtenbrot Theorem for the fragment of $\lang{L}$ defined by the atomic formulas $P(x,z)$ and~$P(y,z)$.

\begin{theorem}
\label{th:Trakhtenbrot:P(x,z):P(y,z)}
Logics\/ $\logic{QCl}$ and\/ $\logic{QCl}_{\mathit{fin}}$ are recursively inseparable in the language with $P(x,z)$ and $P(y,z)$ as the only atomic formulas.
\end{theorem}

\begin{proof}
Immediately follows from Theorem~\ref{th:ml:insep:QS5:QS5wfin} and Lemma~\ref{lem:QS5-Q2var:to:QCl}.
\end{proof}

At the time of writing, the author does not know whether it is possible to replace $\logic{QCl}_{\mathit{fin}}$ with $\logic{SIB}_{\mathit{fin}}$ or $\logic{SRB}_{\mathit{fin}}$ in Theorem~\ref{th:Trakhtenbrot:P(x,z):P(y,z)}.  

%%%%%%%%%%%%%%%%%%%%%%%%%%%%%%%%%%%%%%%

\pagebreak[3]

\addcontentsline{toc}{section}{References} 
%\bibliography{sources,sources2}

\end{document}